\documentclass[12pt,reqno]{amsart}

\RequirePackage{amsthm, amsmath, natbib, amssymb, graphicx}

\RequirePackage[OT1]{fontenc}
\RequirePackage{amsthm,amsmath,natbib, amssymb, mathrsfs, graphicx, subfigure, listings}

\RequirePackage{amsthm,amsmath,natbib, amssymb, mathrsfs, graphicx, listings}

\usepackage{comment}
\usepackage{setspace}
\usepackage{lscape}
\usepackage{boxedminipage}
\usepackage{booktabs,caption}
\usepackage[flushleft]{threeparttable} % To add comment below tables

\newtheoremstyle{theorem}%name
{10pt} % space above
{10pt} % space below
{\sl} % bofy font
{\parindent} % ident - empty=no indent, \parindent= paragraph indent
{\bf} % thm head font
{. } % punctuation after thm head
{ } % space after thm head: `` ``=normal \newline=linebreak
{} % thm head specification
\newtheorem{prop}{Proposition}%[section]
\newtheorem{lem}{Lemma}%[section]
\newtheorem{theorem}{Theorem}%[section]

\newtheorem{cor}{Corollary}
\newtheorem{defn}{Definition}%[section]
\newtheorem{assp}{Assumption}
\theoremstyle{remark}

\newtheorem{rk}{Remark}

\renewcommand{\contentsname}{Content of the paper and the online appendix}
\def\d{{\rm d}}
\def\tr{{\rm tr}}
\def\intr{\mathrm{int}}

\def\E{\mathbb{E}}

\def\P{\mathbb{P}}
\def\V{\mathbb{V}}

\def\Prm{\mathrm{P}}

\def\Qrm{\mathrm{Q}}

\def\ldsb{[\![}
\def\rdsb{]\!]}

\def\T{\mathbf{\Theta}}

\def\d{\mathrm{d}}
\def\e{\mathrm{e}}
\def\sp{\mathrm{sp}}\def\sp{\mathrm{sp}}
\def\R{\mathbf{R}}
\def\N{\mathbf{N}}

\def\Omegabf{\mathbf{\Omega}}
\def\Gammabf{\mathbf{\Gamma}}
\def\Sbf{\mathbf{S}}

\def\Tbf{\mathbf{T}}
\def\Gbf{\mathbf{G}}
\def\balltheta0{\overline{B_{\epsilon_{\theta_0}}(\theta_0)}}

\def\oBL{\overline{B_L}}

\newlength{\dhatheight}

\title[ESP estimator]{The Empirical Saddlepoint Estimator}

\author{Benjamin Holcblat*}
\address{University of Luxembourg, LSF\\
6 Rue Richard Coudenhove-Kalergi, L-1359, Luxembourg.}
\email{Benjamin.Holcblat@uni.lu}

\author{Fallaw Sowell$^{\#}$}
\address{Carnegie Mellon University \\
Tepper School of Business\\
5000 Forbes Ave. Pittsburgh, PA, USA}
\email{fs0v@andrew.cmu.edu}

\begin{document}

\begin{center} Version with online appendix included.\end{center}

\medskip

\begin{abstract}
We define a moment-based estimator that maximizes the empirical saddlepoint (ESP) approximation of the distribution of solutions to empirical moment conditions. We call it the ESP estimator. We prove its existence, consistency and asymptotic normality, and we propose novel test statistics. We also show that the ESP estimator corresponds to the MM (method of moments) estimator shrunk toward parameter values with lower estimated variance, so it reduces the documented instability of existing moment-based estimators. In the case of just-identified moment conditions, which is the case we focus on, the ESP estimator is different from the MM estimator, unlike the recently proposed alternatives, such as the empirical-likelihood-type estimators.

\bigskip


\medskip

%{\noindent \textit{JEL classification}: C1.}

\end{abstract}

\date{\today}

\maketitle
\setcounter{tocdepth}{3}
%\tableofcontents

%\doublespace

\section{Introduction }\label{sec1}
The saddlepoint (SP)  approximation has been developed to  approximate distributions. Because of its accuracy it is regularly used in several fields, such as  numerical analysis  (e.g., \cite{2000Loader}'s algorithm to approximate  binomial distributions, and which is notably used in the statistical software R) and actuarial sciences (e.g., \cite{1932Ess}'s approximation for distributions tails). In statistics, the SP approximation and its empirical version ---the empirical saddlepoint (ESP) approximation---  have been used to approximate  finite-sample distributions  \citep[e.g.,][]{1954Dan,1988DavisonHinkley}.\footnote{Standard monographs and introductions about the ESP and the SP approximation for statistics include \cite{1990FieRon}, \cite{1994Kolassa},  \cite{1995Jen},   \cite{1999GouCas}  and is \cite{2007Butler}.\\
$^*$University of Luxembourg, 6 rue Coudenhove-Kalergi, L-1359 Luxembourg\\
${ }^{\#}$Carnegie Mellon University, 5000 Forbes Ave, Pittsburgh, PA 15213, USA. 
}

In the present paper, we propose to use the ESP  approximation to define a \textit{point} estimator $\hat{\theta}_T$. 
We call it the ESP estimator. It  maximizes the   \cite{1994RonWel}'s ESP approximation, i.e., 
\begin{eqnarray}
\hat{\theta}_T \in \arg \max_{\theta \in \T}  \hat{f}_{\theta^{*}_T}(\theta)
\end{eqnarray}
where  $ \hat{f}_{\theta^{*}_T}(.)$  is the ESP approximation of  the distribution of  solutions to  empirical  moment conditions  \begin{eqnarray}
\frac{1}{T}\sum_{t=1}^T \psi\left(X_t,\theta\right) =0  
\label{Eq:EstimatingEquation}
\end{eqnarray}
and where $\psi(.,.)$ denotes the moment function s.t. $\E [ \psi(X_1, \theta_0) ]=0_{m \times 1}$  an $m$-dimensional vector of zeros,   $(X_t)_{t=1}^T$  i.i.d. data, $\theta_{0}\negmedspace\in\negmedspace \T\negthickspace\subset \negthickspace\R^m$ the  unknown parameter of interest, and $T$ the sample size. The exact formula for  $\hat{f}_{\theta^{*}_T}(.)$ is reminded below in equation \eqref{Eq:ESPApproximationDefnMText} on p. \pageref{Eq:ESPApproximationDefnMText}. 

The ESP estimator  is a moment-based estimator.  Since \cite{1894Pea,1902Pearson}'s  method of moment (MM),   moment-based estimators   have been found useful in a variety of applications (e.g., covariance structure analysis in psychology, and   asset pricing in economics). Their two main advantages are  (i) they do not require a parametric family of probability distributions for the data so they are less prone to model misspecification, and (ii) they allow complex models for which the likelihood function is  intractable.

Nevertheless, the  increase use of the MM and its extensions has     revealed that they  can be unstable   and  perform poorly    in  finite samples (e.g., July 1996 special issue of JBES). The idea of the ESP estimator to improve on the MM estimator is the following. By definition, the MM estimate $\theta^*_T(\omega)$  solves  a realization of the empirical moment conditions \eqref{Eq:EstimatingEquation}, but it typically does not solve  the empirical moment condition for another realization $\omega$ of the data. Thus, we might want an estimate that does only take into account the realized empirical moment conditions, but also their  other potential realizations. More precisely, we want an estimate that accounts for all the potential realizations of the empirical moment conditions according to  their probability weight of occurrence.   This leads to the ESP estimate, which is  a \textit{maximum-probability} estimate. The ESP estimate maximizes the estimated probability weights of solving the  empirical moment conditions.\footnote{This is in contrast to the traditional motivation for ML estimators, which  maximize the  probability weights  of obtaining a sample equal to the observed sample. In other words, the support of the ESP distribution is the parameter space, while the support of the distribution associated with a likelihood is the data space. Thus, if we are looking for relevant  parameter values instead of data values ---as it is typically the case---, a maximum-probability motivation appears more appealing than the traditional ML motivation. }   If the empirical  moment conditions \eqref{Eq:EstimatingEquation}  have a unique solution with a continuous distribution,  the ESP estimator maximizes the ESP approximation of a probability density function of the solution $\theta^*_T$.\footnote{Another motivation for maximum probability estimators is decision theoretic. Maximum probability estimators follows from the minimization of the expectation of   a loss ``function'' that equals zero  when $\theta$ solves the empirical moment conditions and one otherwise by normalization. This motivation is similar to the  decision-theoretic justification for the  Bayesian maximum a posteriori estimator \citep[e.g.,][sec. 4.1.2]{1994Robert}.  As in Bayesian analysis, the choice of other loss functions is possible. It is left for future research.} We rely on the ESP approximation  because  simulation and theoretical evidence shows the  ESP approximation can be  very accurate in small sample \citep[e.g.,][]{1988DavisonHinkley,1994RonWel}.

Besides the maximum-probability motivation, we show that the ESP estimator corresponds to an MM estimator shrunk toward parameter values with lower estimated variance. More precisely, we decompose the logarithm of the ESP approximation as the sum of a term, which is maximized at the  MM  estimator, and a variance penalty, which discounts  parameter values with high estimated variance.   Under   assumptions adapted from the entropy literature, we establish  the ESP estimator has the same good asymptotic properties as the MM estimator, so the variance penalization is a  finite-sample correction. We also    derive the ESP counterparts of the Wald, Lagrange multiplier (LM),  analogue likelihood-ratio (ALR) test statistics, as well as another test statistic. Then, we investigate  the ESP estimator through Monte-Carlo simulations.
We compare its performance with the exponential tilting (ET) estimator, which is equal to the MM estimator  in the just-identified case (i.e., when the number of parameters is the same as the number of moment conditions).   Results show that the variance penalization of the ESP estimator reduces the finite-sample instability of the ET estimator (or equivalently, of the MM estimator). An empirical application  illustrates the gain from this greater stability in terms of inference. 

The ESP estimator is not the first proposal   to improve on the MM and its extensions. Alternative   moment-based approaches have been proposed    such as the empirical likelihood approach of Owen  \citep[][]{1994QinLawless}, the continuously updating approach \citep{1996HanHeaYar}, the already-mentioned exponential tilting (ET) approach \citep{1997KitStu,1998ImbSpaJoh}, and combinations of the aforementioned approaches \cite[e.g.,][]{2007Schennach}. All these approaches yield an estimator  closely related to the empirical likelihood estimator, so we call them empirical-likelihood-type estimators.  In  the just-identified case, when well-defined, all of these empirical-likelihood-type estimators are numerically equal to the original Pearson's  MM estimator  $\theta^*_T$.
  Because we focus on  the  just-identified case,  it thus is sufficient for us to compare the ESP estimator with the MM estimator, or with one of any   of these more recent estimators.

In addition to the already cited papers, the present paper, which supersedes the unpublished manuscript \cite{2009Sow}, is related to many other ones. We clarify these relations in Section \ref{Sec:Literature} (p. \pageref{Sec:Literature}). To the best of our knowledge, none of the prior papers use the SP or the ESP to propose a novel moment-based point estimator. Overall,  the present paper brings together the literature on the saddlepoint approximation and the literature on moment-based estimation.

%\section{A small-sample asymptotic alternative to the MM and others}
\section{Finite-sample analysis}
In the present section, we remind the formula for the ESP approximation, and   analyze its finite-sample structure. Then, we decompose the log-ESP into two terms and show that  the ESP estimator is a MM estimator shrunk toward parameter values with lower estimated variance. 
\subsection{The ESP approximation}\label{Sec:InsightsAboutESP}
%\section{The objective function : The ESP approximation}

  Formalizing and generalizing  prior works \citep{1988DavisonHinkley,1989Feu,1990Wang,1990YoungDaniels}, \cite{1994RonWel} propose the following ESP approximation to estimate the distribution of a solution to  the empirical moment conditions \eqref{Eq:EstimatingEquation}
\begin{eqnarray}
\hat{f}_{\theta^{*}_T}(\theta) :=
\exp\left\{T\ln\left[ \frac{1}{T}\sum_{t=1}^T\e^{\tau_T(\theta) ' \psi_t(\theta)}\right]\right\}\left(\frac{T}{2\pi}\right)^{m/2}\left|\Sigma_T(\theta) \right|_{\det}^{-\frac{1}{2}}
   \label{Eq:ESPApproximationDefnMText}
\end{eqnarray}
where $|.|_{\det}$ denotes the determinant function, $\theta^*_T$ a solution to  \eqref{Eq:EstimatingEquation},  $\psi_t(.):=\psi (X_t,.)$, and
\begin{eqnarray}
\Sigma_T(\theta) & := & \left[ \sum_{t=1}^T w_{t,\theta}\frac{\partial \psi_{t} (\theta)}{\partial \theta'}\right]^{-1}\left[\sum_{t=1}^T w_{t,\theta}\psi_{t}(\theta)\psi_{t}(\theta)'\right]\left[ \sum_{t=1}^T w_{t,\theta}\frac{\partial \psi_{t} (\theta)'}{\partial \theta}\right]^{-1}, \label{Eq:TiltedSigma_TMText}\\
w_{t,\theta} & := & \frac{\exp\left[\tau_T(\theta) ' \psi_t(\theta)\right]}{\sum_{i=1}^T\exp\left[\tau_T(\theta) ' \psi_i(\theta)\right]}  \text{ , }\label{Eq:TiltedWeightMText} \\
\tau_T(\theta) &\text{such that }&\sum_{t=1}^T \psi_{t}(\theta)\frac{\exp\left[\tau_T(\theta) ' \psi_t(\theta)\right]}{\sum_{i=1}^T\exp\left[\tau_T(\theta) ' \psi_i(\theta)\right]}\times \frac{1}{T}=0 \text{.  }\label{Eq:ESPTiltingEquationMText}
\end{eqnarray}
%whenever it exists.

The ESP approximation \eqref{Eq:ESPApproximationDefnMText} is the empirical counterpart of the SP approximation  of \cite{1982Fie}.  From  a computational point of view, the ESP approximation \eqref{Eq:ESPApproximationDefnMText} is not  complicated.\footnote{We do not claim that the ESP estimator is as easy to compute as the MM estimator, but that its additional complexity is similar to the recently proposed empirical-likelihood-type estimators (e.g., ET estimator), and that it is worthwhile in several applications (e.g., Section \ref{Sec:Examples}). Moreover, it seems to make sense to develop novel estimation methods that take advantage of  the increasingly available computational power.} The only implicit quantity is $\tau_T(\theta)$, which solves the tilting equation \eqref{Eq:ESPTiltingEquationMText}, which, in turn,  is just the FOC (first-order condition) of the unconstrained convex problem  $\min_{\tau\in \R^m}\sum_{t=1}^T\e^{\tau ' \psi_t(\theta)}  $. A full understanding of the ESP approximation \eqref{Eq:ESPApproximationDefnMText} arguably requires to work  through higher-order asymptotic expansions along the lines of \cite{1982Fie}. However, direct inspection of the ESP approximation \eqref{Eq:ESPApproximationDefnMText}  also provides  insight for how it incorporates  information from the data through   two channels.

The first channel is the   \textit{ET (exponential tilting) term} $\exp\left\{T\ln\left[ \frac{1}{T}\sum_{t=1}^T\e^{\tau_T(\theta) ' \psi_t(\theta)}\right]\right\}  $.  In equation  \eqref{Eq:ESPTiltingEquationMText}, for  any $\theta\in \T $, the terms $\frac{\exp\left[\tau_T(\theta) ' \psi_t(\theta)\right]}{\sum_{i=1}^T\exp\left[\tau_T(\theta) ' \psi_i(\theta)\right]} $ tilt (i.e., reweight) the empirical weights $ 1/T$, so the finite-sample moment conditions \eqref{Eq:ESPTiltingEquationMText} holds.  This tilting determines, through equation \eqref{Eq:TiltedWeightMText}, the multinomial distribution $(w_{t,\theta})_{t=1}^T$ that is the closest to the empirical distribution ---in the sense of the Kullback-Leibler divergence criterion--- s.t. the finite-sample moment conditions \eqref{Eq:ESPTiltingEquationMText} holds: The tilting equation \eqref{Eq:ESPTiltingEquationMText} is the FOC w.r.t. (with respect to) $\tau$ of the Lagrangian dual problem of  the minimization problem
\begin{eqnarray}
\min_{(w_{1,\theta},w_{2,\theta} ,\cdots,w_{T,\theta}) \in ]0,1]^T} \sum_{t=1}^T w_{t,\theta}\log\left(\frac{w_{t,\theta}}{1/T}\right) \nonumber
 \\\text{ s.t.} \sum_{t=1}^Tw_{t,\theta} \psi_t(\theta)=0 \text{ and } \sum_{t=1}^T w_{t,\theta}=1, \label{Eq:KLEmpMomentCond}
\end{eqnarray}
where $\sum_{t=1}^T w_{t,\theta}\log[w_{t,\theta}/(1/T)]$ is the Kullback-Leibler divergence criterion
between
the empirical distribution and the multinomial distribution $(w_{t,\theta})_{t=1}^T$ with the same support \cite[e.g.,][]{1981Efron,1997KitStu}. Then,
    for  the given $\theta\in \T $, in the ESP approximation \eqref{Eq:ESPApproximationDefnMText}, the   ET  term $\exp\left\{T\ln\left[ \frac{1}{T}\sum_{t=1}^T\e^{\tau_T(\theta) ' \psi_t(\theta)}\right]\right\}  $ indicates the extent  of the tilting   needed to set the  finite-sample moment conditions \eqref{Eq:KLEmpMomentCond} (or equivalently,  equation \eqref{Eq:ESPTiltingEquationMText}) to zero.  The bigger  is the tilting of the empirical distribution, the less compatible are the data   with $\theta$ solving  the empirical moment conditions,  and the smaller should be the ET term $\exp\left\{T\ln\left[ \frac{1}{T}\sum_{t=1}^T\e^{\tau_T(\theta) ' \psi_t(\theta)}\right]\right\}  $. It can be easily seen that $\frac{1}{T}\sum_{t=1}^T\e^{\tau_T(\theta) ' \psi_t(\theta)}$ reaches its maximum when $\theta$ is a solution $\theta^*_T$ of  the empirical moment conditions \eqref{Eq:EstimatingEquation}, i.e., when      $\tau_T(\theta^*_T)=0_{m \times 1} $ and no tilting is needed.\footnote{For a complete proof, one can follow the same reasoning as in the proof of Lemma \ref{Lem:AsTiltingFct} in \citet[p. \pageref{Lem:AsTiltingFct}]{2019HolcblatSowell}  with the empirical distribution in lieu of $\P$.}

In the ESP approximation on equation \eqref{Eq:ESPApproximationDefnMText}, the second term $\left(\frac{T}{2\pi}\right)^{m/2} $ comes from the multivariate Gaussian distribution that  is the leading term of the Edgeworth's asymptotic expansions underlying ESP approximations. However,  because it is constant w.r.t. $\theta$, it does not affect the maximization of the ESP approximation, so it  is \textit{not} an information channel for the ESP estimator. The  remaining term $\left|\Sigma_T(\theta) \right|_{\det}^{-\frac{1}{2}}$   , which we call the \textit{variance term}, is the second channel through which the ESP approximation incorporates information from data. The variance term discounts the ET term according to the tilted estimated variance of the solution to the finite-sample moment conditions. Under standard assumptions, a consistent estimator of the asymptotic variance of   $\sqrt{T}(\theta^*_T-\theta_0)$  is $\left[ \frac{1}{T}\sum_{t=1}^T \frac{\partial \psi_{t} (\theta^*_T)}{\partial \theta}\right]^{-1}\left[\frac{1}{T}\sum_{t=1}^T \psi_{t}(\theta^*_T)\psi_{t}(\theta^*_T)'\right]\left[ \frac{1}{T}\sum_{t=1}^T \frac{\partial \psi_{t} (\theta^*_T)'}{\partial \theta}\right]^{-1} $.  The bigger the variance term is, the less plausible a solution takes exactly this value, and the smaller is $\left|\Sigma_T(\theta) \right|_{\det}^{-\frac{1}{2}}$  ---note the negative power. Therefore, overall, for a given  $\theta\in \T $, the bigger the tilting  or  the estimated variance,  the smaller the ESP approximation, i.e., the estimated probability weight that $\theta$ solves the empirical moment conditions \eqref{Eq:EstimatingEquation}.

\subsection{The ESP estimator as a shrinkage estimator}\label{Sec:ESPvsMMandOthers} As explained in the introduction, the recently proposed moment-based estimators  are numerically equal to the Pearson's MM estimator in the just-identified case. Thus, it is sufficient to compare the ESP estimator with  one of them in order to understand the difference between the former and the other proposed moment-based estimators. The ET estimator of \cite{1997KitStu} and \cite{1998ImbSpaJoh} is particularly convenient for this purpose.
Taking the logarithm of the ESP approximation \eqref{Eq:ESPApproximationDefnMText}, and removing the terms constant w.r.t. $\theta$, it can be seen that,   $\P$-a.s. for $T$ big enough, the ESP estimator $\hat{\theta}_T$ maximizes the  objective function
\begin{eqnarray} \ln  \left[\frac{1}{T}\sum_{t=1}^T \e^{\tau_T(\theta)'\psi_t(\theta)}\right]- \frac{1}{2T} \ln \vert \Sigma_T(\theta)\vert_{\det}, \label{Eq:LogESP}
\end{eqnarray}
 where $\ln  \left[\frac{1}{T}\sum_{t=1}^T \e^{\tau_T(\theta)'\psi_t(\theta)}\right]$ is an increasing transformation of the objective function of the ET estimator. Thus, the difference between the ESP estimator and  the ET estimators  comes only from the log-variance term $-\frac{1}{2T} \ln \vert \Sigma_T(\theta)\vert_{\det}$. The latter does not only incorporates additional information from data as explained in Section \ref{Sec:InsightsAboutESP}, but  it also  penalizes parameter values with higher estimated variance. Thus, the ESP estimator is an ET estimator ---or equivalently, a MM estimator--- shrunk toward parameter values with lower estimated variance. Now, as the factor $\frac{1}{2T} $ suggests and the proofs of Section \ref{Sec:Asymptotic} show, the log-variance term vanishes asymptotically, so the shrinkage is a finite-sample correction. 
\section{Asymptotic properties}\label{Sec:Asymptotic}

In the present section, we investigate the asymptotic properties of the ESP estimator.  Good asymptotic properties can be regarded as a
minimal requirement for the
ESP estimator, which is based on a small-sample asymptotic approximation.
 All the proofs and assumptions are in the online Appendix \cite{2019HolcblatSowell}.
\subsection{Existence, consistency and asymptotic normality }
  Under  assumptions adapted from the entropy literature,  the following theorem establishes the existence,  the strong consistency, and the asymptotic normality of the ESP estimator $\hat{\theta}_T$.

\begin{theorem}[Existence, consistency and asymptotic normality]\label{theorem:ConsistencyAsymptoticNormality} Under Assumption \ref{Assp:ExistenceConsistency}, $\P$-a.s. for $T$ big enough, there exists $\hat{\theta}_T$ s.t.   
\begin{enumerate}
\item[(i)] $\P$-a.s. as $T \rightarrow \infty$, $\hat{\theta}_T \rightarrow \theta_0$ ; and
\item[(ii)] under the additional Assumption \ref{Assp:AsymptoticNormality}, as $T \rightarrow \infty$,  $\sqrt{T} ( \hat{\theta}_T  - \theta_0 )
 \underset{}{\stackrel{D}{\longrightarrow}}   \mathcal{N}\left(0, \Sigma(\theta_0)\right)$.
% \begin{eqnarray*}
% %
%  \sqrt{T} ( \hat{\theta}_T  - \theta_0 )
% & \underset{}{\stackrel{D}{\longrightarrow}}  & \mathcal{N}\left(0, \Sigma(\theta_0)\right).
% \end{eqnarray*}

\end{enumerate}
where $\Sigma(\theta_0):=\negthickspace  \left[\E  \frac{\partial \psi(X_{1},\theta_{0})}{\partial \theta' }\right]^{-1}\negthickspace\negthickspace \E\left[\psi(X_{1},\theta_{0}) \psi(X_{1},\theta_{0})' \right] \left[ \E  \frac{\partial \psi(X_{1},\theta_{0})'}{\partial \theta }\right]^{-1} $, $\stackrel{D}{\rightarrow} $ denotes the convergence in distribution.
\end{theorem}

% \begin{proof} See online Appendices \ref{Ap:PfExistenceConsistency} (p. \pageref{Ap:PfExistenceConsistency})  and \ref{Ap:PfAsymptoticNormality} (p. \pageref{Ap:PfAsymptoticNormality}).
%\end{proof}
Theorem \ref{theorem:ConsistencyAsymptoticNormality} shows that the ESP estimator has the same first-order asymptotic properties  as the MM and hence the recently proposed moment-based estimators. Although the asymptotic properties of the ESP estimator are standard, the  proof of Theorem \ref{theorem:ConsistencyAsymptoticNormality} is quite involved. The crux of the proof is to show that the variance penalization  $-\frac{1}{2T} \ln \vert \Sigma_T(\theta)\vert_{\det}$  vanishes sufficiently quickly asymptotically, so it does not distort the first-order asymptotic.

\subsection{More on inference\,: The trinity$+1$}

The ESP estimator provides different ways to test  parameter restrictions
\begin{eqnarray}
\mathrm{H}_0: r(\theta_{0})=0_{q \times 1} \label{Eq:HypParameterRestriction}
\end{eqnarray}
where $r: \T \rightarrow \R^q$ with $q \in \ldsb 1, \infty\ldsb$.
More precisely, within the ESP framework,  there exist the usual trinity of  Wald, LM and ALR tests statistics, plus  another test statistic, which we call the exponential tilting (ET)  test statistic.  Our ET test has a structure similar to a test for over-identifyied moment conditions in \cite{1998ImbSpaJoh}.

Under a mild  standard additional assumption, the following theorem shows that the Wald, LM  ALR, and ET statistics asymptotically follow a chi-squared distribution with $q$ degrees of freedom.

\begin{theorem}[The trinity$+1$: Wald, LM, ALR  and ET tests] \label{theorem:TrinityPlus1} Define $R(\theta):=\frac{\partial r(\theta)}{\partial \theta'}$, and the following Wald, LM,  ALR and ET test statistics
\begin{eqnarray*}
\mathrm{Wald}_T& :=& Tr(\hat{\theta}_T)'[R(\hat{\theta}_T)\widehat{\Sigma(\theta_0)}_TR(\hat{\theta}_T)']^{-1}r(\hat{\theta}_T)\\
\mathrm{LM}_T &:=& T \check{\gamma}_T'[R(\check{\theta}_T)  \widehat{\Sigma(\theta_0)}_TR(\check{\theta}_T)']\check{\gamma}_T=\frac{\partial\ln[\hat{f}_{\theta^{*}_T}(\check{\theta}_T)]}{\partial \theta' }\widehat{\Sigma(\theta_0)}_T^{-1}\frac{\partial\ln[\hat{f}_{\theta^{*}_T}(\check{\theta}_T)]}{\partial \theta }\\
\mathrm{ALR}_T & := & 2\{\ln[\hat{f}_{\theta^{*}_T}(\hat{\theta}_T)]-\ln[ \hat{f}_{\theta^{*}_T}(\check{\theta}_T)]\}\\
\mathrm{ET}_T&:=& T \tau_T(\check{\theta}_T)'\widehat{V}_T\tau_T(\check{\theta}_T)
\end{eqnarray*}
where $\widehat{\Sigma(\theta_0)}_T$ and $\widehat{V}_T$ are symmetric matrices that converge in probability to  $\Sigma(\theta_0)$ and  $ \E[ \psi(X_1, \theta_0)\psi(X_1, \theta_0)'] $, respectively; and where $\check{\gamma}_T$ and $\check{\theta}_T$ respectively denote the Lagrange multiplier and a solution to the maximization of  $\hat{f}_{\theta^{*}_T}(\theta)$ w.r.t. $\theta\in \T$ under the constraint that  $r(\theta)=0_{q \times 1}$.\footnote{In mathematical terms, $\check{\theta}_T\in \arg \max_{\theta \in \check{\Theta}} \hat{f}_{\theta^{*}_T}(\theta)$ where $\check{\Theta}:=\{\theta \in \Theta: r(\theta)=0_{q \times 1}\}$ and $\check{\gamma}_T$ is the Lagrangian multiplier s.t.  $\frac{1}{T} \frac{\partial \ln[\hat{f}_{\theta^{*}_T}(\check{\theta}_T)]}{\partial \theta} + \frac{\partial r(\check{\theta}_T)'}{\partial \theta}\check{\gamma}_T=0_{m \times 1}$.}
Under Assumptions \ref{Assp:ExistenceConsistency}, \ref{Assp:AsymptoticNormality} and \ref{Assp:Trinity}, if the test hypothesis \eqref{Eq:HypParameterRestriction} holds,  as $T \rightarrow \infty$,
\begin{eqnarray*}
\mathrm{Wald}_T, \mathrm{LM}_T, \mathrm{ALR}_T, \mathrm{ET}_T \stackrel{D}{\rightarrow } \chi^2_{q}.
\end{eqnarray*}

\end{theorem}

% \begin{proof} See online Appendix \ref{Ap:PfTrinityPlus1} on p. \pageref{Ap:PfTrinityPlus1}.
% \end{proof}

  Theorem \ref{theorem:TrinityPlus1} can also be used to obtain valid confidence regions by  the inversion of the  test statistics with  $\check{\theta}_T=\theta_0$.   Our  Wald, LM  and ALR test statistics share some similarity with the test statistics proposed in \cite{1997KitStu}, \cite{1998ImbSpaJoh} and \cite{2003RobRonYou}. The main difference is that the latter are built around (possibly constrained) maximizers of the ET term, while our tests statistics are based on the (possibly constrained) ESP estimator, which maximizes the whole ESP approximation including the variance term.

\section{Examples}\label{Sec:Examples}

In the present section, we  further investigate and illustrate the finite-sample properties of the ESP estimator.\footnote{In addition to our finite-sample analysis of the ESP objective function (Section \ref{sec1}), our derivation of the first-order asymptotic properties (Section \ref{Sec:Asymptotic}),    our Monte-Carlo simulations and  empirical application (present section),  another way  to shed light on the finite-sample properties of the ESP estimator would be to derive its higher-order asymptotic properties such as its second-order bias  \citep[e.g.,][]{1996RilstoneSrivastavaUllah}.
In the present paper, we do not follow this way because it would add several dozens of pages of proofs without much insight: Our preliminary derivations yield a long and complicated structure for the second-order bias, from which we struggle to gain insight. The length and the complexity of the second-order bias mainly comes from (i) the derivatives of the variance $\left|\Sigma_T(\theta) \right|_{\det}^{-1/2}$; and (ii) the reliance on the exact FOCs instead of approximate FOCs.   A mild preview of this complexity can be seen in \citet[Appendix \ref{Sec:LTAndDerivatives}]{2019HolcblatSowell}.}  We focus on the comparison  with the   ET estimator, as previously noted, (i) in the just-identified case, which is the case addressed in the present paper, the MM estimator and the recently proposed moment-based estimators are  equal to the ET estimator so  there is no loss of generality in terms of point estimation, and (ii) the ESP objective function nests the ET objective function, so that the source of the difference between the two is easily understood ---it necessarily comes from the variance term (see Section \ref{Sec:ESPvsMMandOthers}). For brevity,  we  present the main results for a numerical and an empirical example that are known to be challenging for moment-based estimation.

\subsection{Numerical example\,: Monte-Carlo simulations}\label{Sec:NumericalExp}

We simulate the just-identified version of the \cite{1996HalHor}  model, which has become  a standard benchmark to compare  the performance of moment-based estimators in statistics \citep[e.g.,][]{2007Schennach,2012LoRonchetti}   and econometrics \citep[e.g.,][]{1998ImbSpaJoh,2001Kitamura}. This model can be interpreted as
a simplified consumption-based asset pricing model where $\beta $ is the relative risk aversion (RRA) parameter \citep{2002GreLamSmi}.
  In the simulations, we estimate the two parameters $(\mu, \beta)$ with the  moment function  
\begin{eqnarray*}
\psi_t(\beta , \mu ) = \left[ \begin{array}{c}
                        \exp\left\{\mu- \beta \left( X_{t} + Y_{t} \right) + 3 Y_{t}  \right\} - 1 \\
    Y_t  \left( \frac{}{} \exp\left\{\mu  - \beta \left( X_t + Y_t \right) + 3 Y_t  \right\} - 1  \right)
                      \end{array}
  \right]
\end{eqnarray*}
where  $\mu_0 = -.72 $, $ \beta_0 = 3$,
and 
$X_{t}$ and $Y_{t}$ are jointly i.i.d. random variables with distribution $\mathcal{N}(0, .16)$. 

\begin{table}[htp]\caption{\textbf{ESP vs. ET  estimator  for the just-identified Hall and Horowitz model.}     }\label{table:HH}
\begin{center}

\begin{tabular}{c r r r r r}
\hline
\hline
$T$ & & \multicolumn{2}{c}{$\beta$} & \multicolumn{2}{c}{$\mu$} \\
\hline
%\hline
& & \multicolumn{1}{c}{ET} & \multicolumn{1}{c}{ESP}  & \multicolumn{1}{c}{ET} & \multicolumn{1}{c}{ESP}  \\
\cline{2-6}
 & MSE   & 3.6228  &  0.7065  &  1.5391  &  0.2319    \\ 
25 & Bias   & 0.4782  &  -0.0048  &  -0.1855  &   0.1089      \\ 
 & Var.   &  3.3941  &  0.7065  &  1.5047  &  0.2200      \\ 
\cline{2-6}
 & MSE    &   1.7024  &  0.3344  &  0.9959  &  0.1292     \\ 
50 & Bias   & 0.2670  &  -0.0160  &  -0.1330  &   0.0619       \\ 
 & Var.   &  1.6311  &  0.3342  &  0.9782  &  0.1254     \\ 
\cline{2-6}
 & MSE    & 0.6812  &  0.1742  &  0.4780  &  0.0645    \\ 
100 & Bias   &   0.1429  &  -0.0119  &  -0.0735  &   0.0388      \\ 
 & Var.   &  0.6608  &  0.1741  &  0.4726  &  0.0630      \\ 
\cline{2-6}
 & MSE    &  0.2162  &  0.0830  &  0.1457  &  0.0324    \\ 
200 & Bias   &  0.0684  &  -0.0113  &  -0.0340  &   0.0223     \\ 
 & Var.   & 0.2115  &  0.0829  &  0.1445  &  0.0319    \\ \hline
\end{tabular}

\end{center}
\raggedright \begin{flushleft}\begin{tiny} Note: The reported statistics are based on 10,000 simulated samples of sample size equal to the indicated $T$. For ET, the parameter space is restricted to $\beta <15$ in order to limit  the erratic behaviour of the estimator  at sample sizes $T=25$ and $50$. No  such parameter restriction is imposed for ESP.      \end{tiny}\end{flushleft}\begin{center}\begin{tiny}     \end{tiny} \end{center}
\end{table}

Table \ref{table:HH} reports the mean-squarred error (MSE), bias and variance of the ESP and ET estimators for different sample sizes.  The   MSE, the variance and the bias of the ESP estimator are always smaller than for the ET   estimator, and the differences are notable, especially for small sample sizes. In fact,
Table \ref{table:HH} understates the improvement delivered by the variance penalization of the ESP objective function. We help the ET estimator (or equivalently,  the MM estimator),\footnote{We numerically check that they deliver the same estimates even for the small sample sizes $T=25$ and $50$.  } by restricting its parameter space to $\beta< 15$. Without this parameter restriction, the behaviour of the ET estimator is  very unstable.
%\footnote{Per se, the average squared errors cannot diverge to $\infty$ because (i) we assume a bounded parameter space in line with Assumption \ref{Assp:ExistenceConsistency}(d) on p. \pageref{Assp:ExistenceConsistency} and (ii) a computer can only handle a bounded parameter space. However,  if the parameter restriction on $\beta$ is too loose, it seems to take much more than 10,000 simulated samples to observe the convergence of the average squared errors for $T=25$ and $T=50$.}
  An analysis of the typical shape of the objective functions for small sample size explains this phenomenon.    The typical ET objective function has a ridge that follows from around the population parameter values ($\beta_{0}=3$, $\mu_{0}=-.72$) towards $(1000, -600)$.  The  ridgeline is not totally flat, and it often   has a gentle downward slope as we move away from the area near the population parameter values.  However, regularly, for some simulated samples, the very top  of the ridge   is  extremely far from the population parameter values, so that  ET estimates  are very far from the population parameter values. This does not happen for the ESP estimator. The variance term of the ESP objective function  ensures that the ridge drops sufficiently as we move away from the maximum that is near the population parameter value. Thus, in line with our finite-sample analysis of the ESP objective (Section \ref{Sec:ESPvsMMandOthers}), the ESP estimator is much more stable.

\subsection{Empirical example} \label{Sec:EmpExample}

In this section, we present an empirical example from   asset pricing. Since \cite{1982HanSin}, moment-based estimation is standard in consumption-based asset pricing.  For brevity, we focus on the key features of the example.  See \citet[Appendix \ref{Ap:EmpiricalExample} on p. \pageref{Ap:EmpiricalExample}]{2019HolcblatSowell}  for  additional information and comparisons.    
\begin{table}[ht!] \caption{\textbf{ET  vs. ESP inference (1890--2009)}} \label{Tab:ETvsESP1890Short}
 \medskip
 \begin{tabular}{ll}
% \hline
%  \hline
% \multicolumn{2}{c}{\textbf{GMM inference with CRRA preferences} }  \\ 
% %
\hline\multicolumn{2}{l}{Empirical moment condition: $\frac{1}{2009-1889}\sum_{t=1890}^{2009}\left[  \left(\frac{C_{t}}{C_{t-1}}\right)^{-\theta}(R_{m,t}-R_{f,t})\right]=0 $, where} \\
\multicolumn{2}{l}{$R_{m,t}:=$ gross market  return,\, $R_{f,t}:=$risk-free asset gross return,\, $C_t:=$ consumption, }\\
\multicolumn{2}{l}{ and $\theta:=$relative risk aversion;}\\
\multicolumn{2}{l}{\text{Normalized ET:=}$\exp\negthickspace\left\{T\ln\left[ \frac{1}{T}\sum_{t=1}^T\e^{\tau_T(.) ' \psi_t(.)}\right]\right\}\negthickspace/\negthickspace\int_{\Theta} \exp\negthickspace\left\{T\ln\left[ \frac{1}{T}\sum_{t=1}^T\e^{\tau_T(\theta) ' \psi_t(\theta)}\right]\right\}\d \theta$;    }\\
\multicolumn{2}{l}{\text{Normalized ESP:=}$ \hat{f}_{\theta^*_T}(.)/\negthickspace\int_{\Theta} \hat{f}_{\theta^*_T}(\theta)\d \theta$;    }\\
% \multicolumn{2}{l}{Case with support restricted to $\R_{+}$: $\hat{I}_{.05}\negthickspace %=\negthickspace[10.50,  188.85]$ (stripe on A), ESP  support $=\negthickspace[0,289.0]$}\\
\multicolumn{2}{l}{$\hat{\theta}_{\mathrm{ET},T}=\hat{\theta}_{\mathrm{MM},T}=50.3$ (bullet) and $\hat{\theta}_{\mathrm{ESP},T}=32.21$ (bullet); }\\
\multicolumn{2}{l}{  ET and ESP  support $=[-218.2,289.0]$; 95\% ET ALR conf. region=$[18.3, 289.0] $ (stripe);  }\\
\multicolumn{2}{l}{   95\% ESP ALR conf. region=$[15.0,112.7]$ (stripe). }
  \\ 
% \hline
%  \multicolumn{2}{c}{ \includegraphics[scale=1.85]{1890_2009_LogET_LogESP.png}
%  %  \includegraphics[scale=1.85]{1890_2009_LogET_LogESP.png}
%   }  \\ 
%   \multicolumn{2}{c}{\begin{small}(A)  Translated LogET (light green) vs. LogESP (dark blue)  with  point estimates. \end{small} }\\ 
\hline
\includegraphics[scale=.351]{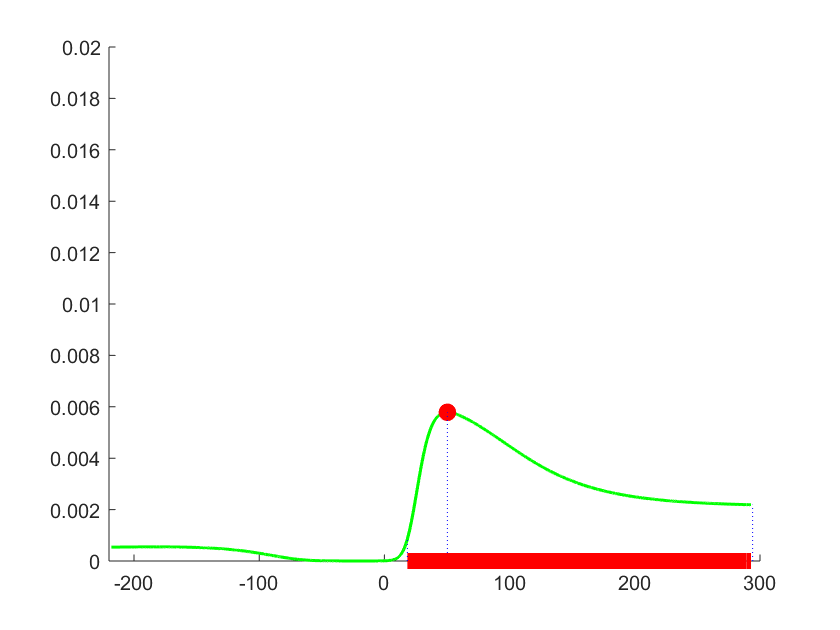} & \includegraphics[scale=.351]{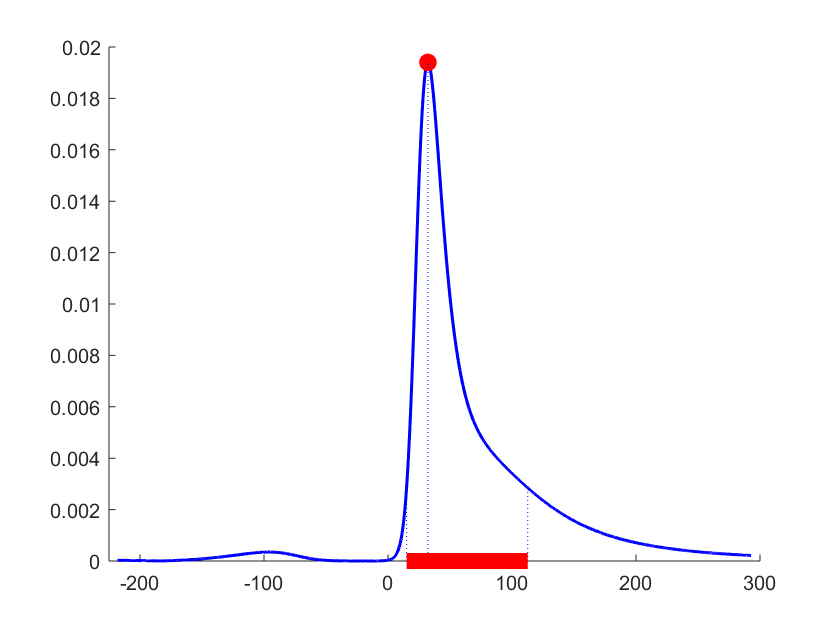} \\ 
  \begin{small}(A) ET est. and ALR conf. region.\end{small} &  \begin{small}(B) ESP est. and ALR conf. region.\end{small} \\ 
\hline
\hline
\end{tabular}
\end{table} 

We  estimate the relative risk aversion  (RRA) $\theta$ of a representative agent of the US economy.   Previous studies have shown that existing moment-based estimation approaches often produce   unstable  RRA  parameter estimates. We rely on the following   moment condition
\begin{eqnarray}
\E\left[\left(\frac{C_{t}}{C_{t-1}}\right)^{-\theta}(R_{m,t}-R_{f,t})\right]=0,\label{Eq:KeyEmpMomCond}
\end{eqnarray}
where $\frac{C_{t}}{C_{t-1}}$ is the growth consumption and $(R_{m,t}-R_{f,t})$ the market return in excess of the risk-free rate.
%  The moment condition (\ref{Eq:KeyEmpMomCond}), which is common to many consumption-based asset pricing models, states that the  expected return on a riskless asset  should be the same as the market return once discounted  for risk ---the multiplication by  $\left(\frac{C_{t}}{C_{t-1}}\right)^{-\theta} $ discounts for risk. 
 The moment condition, which is common to many consumption-based asset pricing models, and the data are similar to \cite{2012JulGho} corresponding to standard US data at yearly frequency from Shiller's website spanning from 1890 to 2009. We report ET and ESP estimates as well as confidence regions based on the inversion of the ALR test statistics of Theorem \ref{theorem:TrinityPlus1} (p. \pageref{theorem:TrinityPlus1}) with $\check{\theta}_T=\theta_0$.  The latter have the advantage to take into account the whole shape of the objective function unlike $t$-statistics-based confidence regions, which only account
for the shape of the objective function in a neighborhood of the estimate  through its standard errors.

In Table \ref{Tab:ETvsESP1890Short}, Figures (A) and (B)  respectively display the ET term and the ESP approximation. For ease of comparison, the scale is the same, and we normalize both of them so they integrate to one. The normalized ET term is much   flatter around its maximum than the normalized ESP approximation. Flatness of the objective function  around the estimate has  been documented for other existing moment-based estimators, and it has often been regarded as one of the main sources of the instability of the RRA estimates  \citep[e.g.,][]{2000StoWri,2001NeeRoyWhi}. Figure (B)  shows that the normalized ESP is  sharp around the ESP estimator. The relative sharpness of the ESP yields  sharper confidence regions\,: The ESP confidence region is less than half its ET counterpart.  In light of the variance penalization term in the ESP objective function (Section \ref{Sec:ESPvsMMandOthers} on p. \pageref{Sec:ESPvsMMandOthers}) and the shrinkage-like behavior of the ESP estimator in  the Monte-Carlo simulations (Section \ref{Sec:NumericalExp}), the relative sharpness of the ESP inference is not surprising. In \cite{2019HolcblatSowell}, additional empirical evidences  corroborate  the increased stability and precision of the ESP estimator w.r.t. the ET estimator (or equivalently, MM estimator).

\section{Connection to the literature and further research directions}\label{Sec:Literature}

 The present paper demonstrates a previously unknown connection between  the SP approximation and   moment-based estimation, and hence it is related to many papers in addition to the ones already cited.
   Following  \cite{1954Dan},  the literature in
statistics
\citep[e.g.,][]{1986EastonRonchetti,1991Spady,1992Jen,2012LaVecchiaRonchettiTrojani, 2015BrodaKan,2018FasioloEtAl}   and econometrics \citep[e.g.,][]{1978Phi,1979HolPhi,1982Phi,1994Lieberman,2006AitSahYu}   has 
used  the SP (saddlepoint) and ESP approximations  to obtain  accurate  approximations  of  distributions, especially in the tails.  The strand of the SP literature that is closest to our paper derives SP approximations to   the distribution of statistics that correspond to solutions of nonlinear estimating equations. The latter strand of literature started with \cite{1982Fie} and continued with \cite{1990Sko,1993MontiRonchetti,1997Imb,1998JenWoo,2000AlmFieRob,2003RobRonYou},  and
\cite{2003RonTro}, among others.
More recently, \cite{2010CzeRon}, \cite{2011MaRon}, and  \cite{2012LoRonchetti,2013KunRil,2015KundhiRilstone}   propose more accurate  tests for indirect inference, functional measurement error models, moment condition models, nonlinear estimators and GEL (generalized empirical likelihood) estimators, respectively.
To the best of our knowledge, unlike the present paper, none of the prior papers use the SP or the ESP to develop an estimation method that yields a novel moment-based estimator.
In ongoing work, we  generalize the ESP approximation  to the over-identified case, and  establish further good mathematical properties.

\bibliographystyle{kluwer}
\bibliography{general}

\section*{Notes and acknowledgements}
Parts of the present paper have previously circulated under the title ``The Empirical Saddlepoint Likelihood Estimator Applied to Two-Step GMM" \citep{2009Sow}. Some proofs of the present paper also borrow technical results from \cite{2012Hol}. Helpful comments were provided by Philipp Ketz (discussant), Eric Renault, Aman Ullah and seminar/conference participants at Carnegie Mellon University, CFE-CMStatistics 2017,  Swiss Finance Institute (EPFL and the University of Lausanne), 10th French Econometrics Conference (Paris School of Economics),  at the Econometric Society European Winter Meeting 2018 (University of Naples Federico II), and at the University of Luxembourg.
\newpage
\appendix

\begin{center}\begin{LARGE}\textbf{ ONLINE APPENDIX:\\
\smallskip
The ESP estimator}\end{LARGE}\\
\medskip
Benjamin Holcblat and Fallaw Sowell\end{center}
\begin{small}
\bigskip

 This appendix mainly consists of the proofs of Theorem \ref{theorem:ConsistencyAsymptoticNormality}, existence and consistency and asymptotic normality of the ESP estimator, and Theorem \ref{theorem:TrinityPlus1}, asymptotic distributions of the Trintiy+1 test statistics.   The proof of Theorem  \ref{theorem:ConsistencyAsymptoticNormality}  builds on the traditional uniform convergence proof technique of  \citet{1949Wal}.   The proof of Theorem \ref{theorem:TrinityPlus1} adapts the  usual  way of deriving the trinity tests.
The length of the proofs is mainly due to the variance term $\vert\Sigma_T(\theta)\vert_{\det}^{-\frac{1}{2}} $  and the high-level of details.   The latter should make the proofs more transparent, and should ease the use of the intermediary results in further research.

In addition to the proofs, this appendix contains a table of contents, some formal definitions, the precise assumptions of the paper, a discussion thereof, and additional information regarding the examples.

\tableofcontents
\section{Definitions and assumptions}

\begin{defn}[ESP approximation; \citealp{1994RonWel}]\label{Def:ESPApproximation}
The ESP approximation of the distribution of the solution to the empirical moment conditions \eqref{Eq:EstimatingEquation} is
\begin{eqnarray}
\hat{f}_{\theta^{*}_T}(\theta):=\exp\left\{T\ln\left[ \frac{1}{T}\sum_{t=1}^T\e^{\tau_T(\theta) ' \psi_t(\theta)}\right]\right\}  \left(\frac{T}{2\pi}\right)^{m/2}\left|\Sigma_T(\theta) \right|_{\det}^{-\frac{1}{2}}
\label{Eq:ESPApproximationDefn}
\end{eqnarray}
where $|.|_{\det}$ denotes the determinant function, $\theta^*_T$ a solution to the empirical moment conditions \eqref{Eq:EstimatingEquation},  $\psi_t(.):=\psi (X_t,.) $,  and
\begin{eqnarray}
\Sigma_T(\theta) & := & \left[ \sum_{t=1}^T \hat{w}_{t,\theta}\frac{\partial \psi_{t} (\theta)}{\partial \theta'}\right]^{-1}\left[\sum_{t=1}^T \hat{w}_{t,\theta}\psi_{t}(\theta)\psi_{t}(\theta)'\right]\left[ \sum_{t=1}^T \hat{w}_{t,\theta}\frac{\partial \psi_{t} (\theta)'}{\partial \theta}\right]^{-1}, \label{Eq:TiltedSigma_T}\\
\hat{w}_{t,\theta} & := & \frac{\exp\left[\tau_T(\theta) ' \psi_t(\theta)\right]}{\sum_{i=1}^T\exp\left[\tau_T(\theta) ' \psi_i(\theta)\right]}  \text{ , }\label{Eq:TiltedWeight} \\
\tau_T(\theta) &\text{such that\ (s.t.)}&\sum_{t=1}^T \psi_{t}(\theta)\exp\left[\tau_T(\theta) ' \psi_t(\theta)\right]=0_{m\times 1} \text{.  }\label{Eq:ESPTiltingEquation}
\end{eqnarray}
%whenever it exists.

\end{defn}

\begin{defn}[ESP estimator]\label{Defn:ESPEstimator}
The  ESP estimator $\hat{\theta}_T$ is a maximizer of the ESP approximation \eqref{Eq:ESPApproximationDefn}, i.e.,
\begin{eqnarray}
\hat{\theta}_T \in \arg \max_{\theta \in \T}  \hat{f}_{\theta^{*}_T}(\theta).
\end{eqnarray}
\end{defn}

We require the following assumption to prove the existence and the consistency of the ESP estimator.

\begin{assp}\label{Assp:ExistenceConsistency}
\textbf{\emph{(a)}} The data $(X_t)_{t=1}^{\infty}$ are a sequence of  i.i.d. random  vectors of dimension p  on the complete probability sample space $(\Omegabf, \mathcal{E}, \P)$.     \emph{\textbf{ (b)}} Let the moment function $\psi: \R^p \times \T^{\epsilon} \mapsto \R^m$  be  s.t.  $\theta \mapsto \psi(X_{1},\theta)$ is continuously differentiable $\P$-a.s., and $\forall \theta\in \T^{\epsilon} $, $x \mapsto \psi(x,\theta) $ is $\mathcal{B}(\R^p) /\mathcal{B}(\R^m)$-measurable, where, for  $\epsilon>0$,  $\T^{\epsilon} $ denotes the $\epsilon$-neighborhood of $\T$. \emph{\textbf{ (c)}}  In the parameter space $\T$, there exists a unique $\theta_0 \in \intr(\T)$ s.t. $\E\left[\psi(X_{1},\theta_0)\right]=0_{m\times 1}$ where $\E$ denotes the expectation under $\P$. \emph{\textbf{ (d)}} Let the parameter space $\T\subset\R^{m}$  be a compact set, s.t.,  for all $\theta \in \T$,  there exists   $\tau(\theta)\in \R^m$    that solves the equation    $\E\left[ \e^{\tau'\psi(X_1,\theta)}\psi(X_1,\theta)\right]=0 $ for $\tau$.  \emph{\textbf{ (e)}} $\E \left[\sup_{(\theta, \tau) \in \Sbf^\epsilon}  \e^{2\tau'\psi(X_{1}, \theta)}\right]< \infty$ where $\Sbf:=\{(\theta, \tau):\theta \in \T \& \tau \in \Tbf(\theta)\}$ and $\Tbf(\theta):= \overline{B_{\epsilon_{\Tbf}}(\tau(\theta))}$ with $\overline{B_{\epsilon_{\Tbf}}(\tau(\theta))} $ the closed ball of radius $\epsilon_{\Tbf}>0$ and center $\tau(\theta)$.   \emph{\textbf{ (f)}} $\E \left[ \sup_{\theta \in \T} \vert \frac{\partial \psi(X_{1}, \theta)}{\partial \theta'}\vert^2\right]            < \infty  $, where $\vert .\vert$ denotes the Euclidean norm. \emph{\textbf{ (g)}} $\E \left[ \sup_{\theta \in \T^{\epsilon}}\vert \psi(X_1, \theta)\psi(X_1, \theta)'\vert^2\right]< \infty$. \emph{\textbf{ (h)}}
For all $\theta \in \T$, the matrices $\left[\E\e^{\tau(\theta)' \psi(X_{1},\theta)}  \frac{\partial \psi(X_{1},\theta)}{\partial \theta' }\right]$ and $\E\left[\e^{\tau(\theta)' \psi(X_{1},\theta)}\psi(X_{1},\theta) \psi(X_{1},\theta)' \right] $ are invertible, so  $\Sigma(\theta):=\negthickspace  \left[\E\e^{\tau(\theta)' \psi(X_{1},\theta)}  \frac{\partial \psi(X_{1},\theta)}{\partial \theta' }\right]^{-1}\negthickspace\negthickspace\E\left[\e^{\tau(\theta)' \psi(X_{1},\theta)}\psi(X_{1},\theta) \psi(X_{1},\theta)' \right]\negthickspace \left[\E\e^{\tau(\theta)' \psi(X_{1},\theta)}  \frac{\partial \psi(X_{1},\theta)'}{\partial \theta }\right]^{-1}$ is also invertible.
\end{assp}

We require the following additional assumption to prove the asymptotic normality  of the ESP estimator.

\begin{assp}\label{Assp:AsymptoticNormality} \textbf{\emph{(a)}} The function
$\theta \mapsto \psi(X_{1},\theta)$ is  three times continuously differentiable in  a neighborhood $\mathcal{N} $ of  $\theta_0$ in $\T$ $\P$-a.s.  \textbf{\emph{(b)}} There exists a  $\mathcal{B}(\R^p) /\mathcal{B}(\R)$-measurable function $b(.)$ satisfying $\E\left[ \sup_{\theta\in \mathcal{N} } \sup_{\tau \in \Tbf(\theta)}\e^{k_1 \tau'\psi(X_1, \theta)} b(X_1)^{k_2}\right]< \infty$ for $k_1\in \ldsb 1 ,2\rdsb$ and $k_2 \in \ldsb 1,4\rdsb$ s.t., for all $j \in\ldsb 0,3\rdsb $, $\sup_{ \theta\in \mathcal{N} }\vert \nabla^j \psi(X_1, \theta)\vert \leqslant   b(X_1)$ where $\nabla^j \psi(X_1, \theta)$ denotes a vector of all  partial derivatives of $\theta \mapsto \psi(X_1, \theta)$ of order $j$.

\end{assp}
Assumptions \ref{Assp:ExistenceConsistency} and \ref{Assp:AsymptoticNormality} are stronger than the usual assumptions in the MM literatureè, but are similar to assumptions used
in the entropy  literature and related literatures. Assumptions \ref{Assp:ExistenceConsistency} and  \ref{Assp:AsymptoticNormality} are essentially adapted from  \cite{1984Haberman,1997KitStu}, and \citet[Assumption 3]{2007Schennach}. See  also  \cite{2018ChibShinSimoni} for similar assumptions.
The Appendix \ref{Sec:DiscussionAsspSchennach} (p. \pageref{Sec:DiscussionAsspSchennach}) contains a detailed discussion of Assumptions \ref{Assp:ExistenceConsistency} and  \ref{Assp:AsymptoticNormality}.

In addition to Assumptions \ref{Assp:ExistenceConsistency} and \ref{Assp:AsymptoticNormality}, we require the following  standard and mild assumption to establish the asymptotic distribution of the Wald, LM,  ALR, and ET statistics.
\begin{assp}[For the trinity$+1$]\label{Assp:Trinity}\textbf{\emph{(a)}} The function $r: \T \rightarrow \R^q$ in the null hypothesis \eqref{Eq:HypParameterRestriction}
is continuously differentiable. \textbf{\emph{(b)}} The derivative $R(\theta):=\frac{\partial r(\theta)}{\partial \theta'}$ is full  rank at $\theta_0$.
\end{assp}

\section{Proofs}
% This appendix contains very detailed proofs. The high level of detail  should make  the proofs more transparent, and should  ease the use of the intermediary results in  further research.

\subsection{Proof of Theorem \ref{theorem:ConsistencyAsymptoticNormality}(i)\,: Existence and consistency}\label{Ap:PfExistenceConsistency}
 The  proof of Theorem \ref{theorem:ConsistencyAsymptoticNormality}(i) (i.e., consistency) adapts the Wald's approach to consistency \citep{1949Wal} along the lines of  \cite{1997KitStu}, \citet{2007Schennach},   \cite{2018ChibShinSimoni} and others. More precisely, standardizing
the logarithm of the ESP approximation, we show that,  $\P$-a.s. for $T$ big enough, the ESP estimator maximizes the LogESP function  \eqref{Eq:LogESP} on p. \pageref{Eq:LogESP},
 where 
 
\noindent
$\sup_{\theta \in \T}\left\vert\ln  \left[\frac{1}{T}\sum_{t=1}^T \e^{\tau_T(\theta)'\psi_t(\theta)}\right]-\ln \E[ \e^{\tau(\theta)'\psi(X_1, \theta)}] \right\vert=o(1)$ and

\noindent
$\sup_{\theta \in \T}\left\vert \frac{1}{2T} \ln \vert \Sigma_T(\theta)\vert_{\det}\right\vert=O(T^{-1}) $. The two main differences between our proof of Theorem \ref{theorem:ConsistencyAsymptoticNormality}(i)  and the proofs  available in  the entropy  literature are the following. Firstly, we need to ensure that, for $T$ big enough, for all $\theta \in \T$, $\vert \Sigma_T(\theta)\vert_{\det}$ is bounded away from zero, so that the LogESP function \eqref{Eq:LogESP} on p. \pageref{Eq:LogESP} does not diverge to $\infty$ on parts of the parameter space. Secondly, we prove that the joint parameter space for $\theta$ and $\tau$ (i.e., $\Sbf$) is a compact set. %\footnote{Unlike what has been some time suggested in the entropy literature, the compactness of $\T$ and of $\Tbf(\theta) $, for all $\theta \in \T$, is not sufficient to ensure the compactness of $\{(\theta, \tau):\theta \in \T \& \tau \in \Tbf(\theta)\}$: See Lemma \ref{Lem:TauCorrespondence} (p. \pageref{Lem:TauCorrespondence}).  }   
 \begin{proof}[Core of the proof of Theorem \ref{theorem:ConsistencyAsymptoticNormality}i]
 Under Assumption \ref{Assp:ExistenceConsistency}(a)(b) and (d)-(h), by Lemma \ref{Lem:ESPExistence} (p. \pageref{Lem:ESPExistence}), $\P$-a.s. for $T$ big enough, the ESP approximation and the ESP estimator exist. Moreover,  under Assumption \ref{Assp:ExistenceConsistency}(a)-(b) and (d)-(h),  by Lemma \ref{Lem:AsVarianceTermBehaviour}iv (p. \pageref{Lem:AsVarianceTermBehaviour}), $\P$-a.s. for $T$ big enough, $\vert \Sigma_T(\theta)\vert_{\det}>0$, for all $\theta \in \T$. Thus, we can apply the strictly increasing transformation $x \mapsto \frac{1}{T}[ \ln(x)-\frac{m}{2}\ln(\frac{T}{2\pi})] $ to   the ESP approximation in equation \eqref{Eq:ESPApproximationDefn} on p. \pageref{Eq:ESPApproximationDefn},  so that, $\P$-a.s. for $T$ big enough,
\begin{eqnarray}
& & \hat{\theta}_T \in \arg \max_{\theta \in \T} \hat{f}_{\theta^{*}_T}(\theta) \nonumber\\
& \Leftrightarrow  & \hat{\theta}_T \in \arg \max_{\theta \in \T}\left\{  \ln  \left[\frac{1}{T}\sum_{t=1}^T \e^{\tau_T(\theta)'\psi_t(\theta)}\right]- \frac{1}{2T} \ln \vert \Sigma_T(\theta)\vert_{\det}  \right\}.\label{Eq:ArgMaxLogESP}
\end{eqnarray}

Now, by the triangle inequality,
\begin{eqnarray}
& & \sup_{\theta \in \T}\left\vert\ln  \left[\frac{1}{T}\sum_{t=1}^T \e^{\tau_T(\theta)'\psi_t(\theta)}\right]- \frac{1}{2T} \ln \vert \Sigma_T(\theta)\vert_{\det}-\ln \E[ \e^{\tau(\theta)'\psi(X_1, \theta)}]\right\vert  \nonumber\\
& \leqslant & \sup_{\theta \in \T}\left\vert\ln  \left[\frac{1}{T}\sum_{t=1}^T \e^{\tau_T(\theta)'\psi_t(\theta)}\right]-\ln \E[ \e^{\tau(\theta)'\psi(X_1, \theta)}] \right\vert+\sup_{\theta \in \T}\left\vert \frac{1}{2T} \ln \vert \Sigma_T(\theta)\vert_{\det}\right\vert \nonumber\\
& = & o(1) \text{ $\P$-a.s. as $T \rightarrow \infty$} \label{Eq:UniformCVOfESPObjFct}
\end{eqnarray}
where the last equality follow from Lemma \ref{Lem:Schennachtheorem10PfFirstSteps}iv (p. \pageref{Lem:Schennachtheorem10PfFirstSteps}) and Lemma  \ref{Lem:AsVarianceTermBehaviour}v (p. \pageref{Lem:AsVarianceTermBehaviour}) under Assumption \ref{Assp:ExistenceConsistency}(a)-(b) and (d)-(h).
Thus, regarding $\hat\theta_T$, it is now sufficient to check the assumptions of the standard consistency theorem \citep[e.g.][pp. 2121-2122 Theorem 2.1, which is also valid in an almost-sure sense]{1994NewMcF}. Firstly,  under Assumption \ref{Assp:ExistenceConsistency} (a)-(e) and (g)-(h), by Lemma \ref{Lem:AsTiltingFct}iv (p. \pageref{Lem:AsTiltingFct}), $\theta \mapsto \ln \E[ \e^{\tau(\theta)'\psi(X_1, \theta)}] \vert$ is uniquely maximized at $\theta_0$, i.e.,  for all $\theta \in \T \setminus\{ \theta_0\} $, $\ln \E[ \e^{\tau(\theta)'\psi(X_1, \theta)}]<\ln \E[ \e^{\tau(\theta_0)'\psi(X_1, \theta_{0})}]=0$. Secondly, under Assumptions \ref{Assp:ExistenceConsistency} (a)(b)(d)(e)(g) and (h), by Lemma \ref{Lem:ExpTauPsiStrictlyPositive} (p. \pageref{Lem:ExpTauPsiStrictlyPositive}),    $\theta \mapsto \ln \E[ \e^{\tau(\theta)'\psi(X_1, \theta)}] $  is continuous in $\T$. Finally,  by Assumption \ref{Assp:ExistenceConsistency}(d), the parameter space $\T$ is compact.
%Regarding $\tau_T(\hat{\theta}_T)$,  under Assumption \ref{Assp:ExistenceConsistency}(a)(b) and (d)-(h), by Lemma \ref{Lem:Schennachtheorem10PfFirstSteps}iii (p. \pageref{Lem:Schennachtheorem10PfFirstSteps}), $\P\text{-a.s.}$ as $T \rightarrow \infty$,  $ \sup_{\theta \in \T}\left\vert \tau_{T}(\theta)- \tau(\theta) \right\vert =o(1)$, so that $\tau_T(\hat{\theta}_T)\rightarrow  \tau(\theta_0)$. Moreover,  $\tau(\theta_0)=0_{m \times 1}$ by Assumption \ref{Assp:ExistenceConsistency}(c) and Lemma \ref{Lem:AsTiltingFct}ii (p. \pageref{Lem:AsTiltingFct}), under Assumption \ref{Assp:ExistenceConsistency}(a)(b), (d)(e)(g) and (h).
\end{proof}

\begin{lem}[Existence  of the ESP approximation and estimator]\label{Lem:ESPExistence} Under Assumption \ref{Assp:ExistenceConsistency}(a)(b) and (d)-(h), $\P$-a.s. for $T$ big enough,
\begin{enumerate}
\item[(i)] the ESP approximation $ \hat{f}_{\theta^{*}_T}(.) $  exists;
\item[(ii)]  $\theta \mapsto \tau_T(\theta)$ is unique and continuously differentiable in $\T$, so that the ESP approximation $\theta \mapsto \hat{f}_{\theta^{*}_T}(\theta) $ is also unique and continuous in $\T$;

\item[(iii)] for all $\theta \in \T$, the ESP approximation $\omega \mapsto \hat{f}_{\theta^{*}_T}(\theta) $ is $\mathcal{E}/\mathcal{B}(\R)$-measurable; and
\item[(iv)] there exists an  ESP estimator $\hat{\theta}_T \in \arg \max_{\theta \in \T} \hat{f}_{\theta^{*}_T}(\theta)$ that
is $\mathcal{E}/\mathcal{B}(\R^m)$-measurable.
\end{enumerate}
\end{lem}
\begin{proof} The result follows from Lemmas \ref{Lem:Schennachtheorem10PfFirstSteps} (p. \pageref{Lem:Schennachtheorem10PfFirstSteps}), \ref{Lem:ExpTauPsiStrictlyPositive} (p. \pageref{Lem:ExpTauPsiStrictlyPositive}) and \ref{Lem:AsVarianceTermBehaviour} (p. \pageref{Lem:AsVarianceTermBehaviour}) and standard arguments. For completeness, a detailed proof is provided.

\textit{(i)} Under  Assumption \ref{Assp:ExistenceConsistency}(a)(b), (d)-(e)(g) and (h), by Lemma \ref{Lem:Schennachtheorem10PfFirstSteps}ii (p. \pageref{Lem:Schennachtheorem10PfFirstSteps}), $\P$-a.s. there exists a $\mathcal{B}(\T)\otimes \mathcal{E}/\mathcal{B}(\R^m)$-measurable function $\tau_T(.)$ s.t., for $T$ big enough,  for all $\theta \in \T$,  $\frac{1}{T}\sum_{t=1}^T \e^{\tau_T(\theta)'\psi_t(\theta)}\psi_t(\theta)=0_{m \times 1}$ and $\tau_T(\theta)\in \intr[\Tbf(\theta)]$. Moreover, under Assumption \ref{Assp:ExistenceConsistency} (a)(b)(d) (e)(g) and (h), by Lemma \ref{Lem:ExpTauPsiStrictlyPositive} (p. \pageref{Lem:ExpTauPsiStrictlyPositive}) with $\Prm=\frac{1}{T}\sum_{t=1}^T\delta_{X_t}$, for all $T \in \ldsb 1, \infty \ldsb$, for all $(\theta, \tau) \in \Sbf$, $0<\frac{1}{T}\sum_{t=1}^T \e^{\tau'\psi_t(\theta)} $, so that, for all $\theta \in \T$,  $0<\frac{1}{T}\sum_{t=1}^T \e^{\tau_T(\theta)'\psi_t(\theta)} $. Thus, the ET term exists. Now, under Assumption  \ref{Assp:ExistenceConsistency}(a)-(b) and (d)-(h), by Lemma \ref{Lem:AsVarianceTermBehaviour}iv (p. \pageref{Lem:AsVarianceTermBehaviour}), $\P$-a.s.  for $T$ big enough, $ \inf_{\theta \in \T} \vert \Sigma_T(\theta)\vert_{\det} >0$, so that the variance term of the ESP approximation exists. Thus,  the ESP approximation exists.

\textit{(ii)} By Assumption \ref{Assp:ExistenceConsistency}(b),   $\theta \mapsto \psi(X_{1},\theta)$ is continuously differentiable in $\T^\epsilon$ $\P$-a.s., so that it is sufficient to show that $\tau_T(.)$ is unique and continuous, which we prove at once with the standard implicit function theorem. Check its assumptions.    Firstly,  under  Assumption \ref{Assp:ExistenceConsistency}(a)(b), (d)-(e)(g) and (h), by Lemma \ref{Lem:Schennachtheorem10PfFirstSteps}ii (p. \pageref{Lem:Schennachtheorem10PfFirstSteps}), $\P$-a.s. there exists a  function $\tau_T(.)$ s.t., for $T$ big enough,  for all $\theta \in \T$,  $\frac{1}{T}\sum_{t=1}^T \e^{\tau_T(\theta)'\psi_t(\theta)}\psi_t(\theta)=0_{m \times 1}$ and $\tau_T(\theta)\in \intr[\Tbf(\theta)]$.  Secondly, for all $\dot{\theta} \in \T$,   $\left.\frac{\partial\left[ \frac{1}{T}\sum_{t=1}^T\e^{\tau' \psi_t(\theta)}\psi_t(\theta)\right]}{\partial \tau' }\right\vert_{(\theta, \tau)=(\dot{\theta},\tau_T(\dot{\theta}))}= \frac{1}{T}\sum_{t=1}^T\e^{\tau' \psi_t(\dot\theta)}\psi_t(\dot\theta)\psi_t(\dot\theta)'$, which is full rank $\P$-a.s. for $T$ big enough for all $\theta \in \T$,  because under Assumption  \ref{Assp:ExistenceConsistency}(a)-(b) and (d)-(h), by Lemma \ref{Lem:AsVarianceTermBehaviour}iv (p. \pageref{Lem:AsVarianceTermBehaviour}), $\P$-a.s.  for $T$ big enough, $ \inf_{\theta \in \T} \vert \Sigma_T(\theta)\vert_{\det} >0$.  Finally, by Assumption \ref{Assp:ExistenceConsistency}(b), $(\theta, \tau) \mapsto \frac{1}{T}\sum_{t=1}^T\e^{\tau' \psi_t(\theta)}\psi_t(\theta)$ is continuously differentiable in  $\Sbf^\epsilon$.

\textit{(iii)}    By Assumption \ref{Assp:ExistenceConsistency}(b), for all $\theta \in \T$,  $x \mapsto \psi(x,\theta) $ is $\mathcal{B}(\R^p) /\mathcal{B}(\R^m)$-measurable. Moreover, under  Assumption \ref{Assp:ExistenceConsistency}(a)(b), (d)-(e)(g) and (h), by Lemma \ref{Lem:Schennachtheorem10PfFirstSteps}ii (p. \pageref{Lem:Schennachtheorem10PfFirstSteps}), $\P$-a.s. $\tau_T(.)$ is a $\mathcal{B}(\T)\otimes \mathcal{E}/\mathcal{B}(\R^m)$-measurable function. Thus, the result follows.

\textit{(iv)} By Assumption \ref{Assp:ExistenceConsistency}(d), $\T$ is compact, so that, by the statements (i)-(iii) of the present lemma, the result follows from the Schmetterer-Jennrich lemma (\citealt{1966Sch} Chap. 5 Lemma 3.3;   \citealt{1969Jen} Lemma 2).
\end{proof}

\begin{lem}[Asymptotic limit of the ET term] \label{Lem:Schennachtheorem10PfFirstSteps} Under  Assumption \ref{Assp:ExistenceConsistency}(a)(b), (d)-(e)(g) and (h),
\begin{enumerate}
\item[(i)] $\P\text{-a.s.}$ as $T \rightarrow \infty$,
$ \sup_{(\theta, \tau) \in\Sbf}\left\vert\frac{1}{T}\sum_{t=1}^T \e^{\tau'\psi_t(\theta)}- \E[ \e^{\tau'\psi(X_1, \theta)}] \right\vert =o(1)$, which implies that $\P\text{-a.s.}$ as $T \rightarrow \infty$,   $ \sup_{(\theta, \tau) \in\Sbf}\left\vert\ln  \left[\frac{1}{T}\sum_{t=1}^T \e^{\tau'\psi_t(\theta)}\right]-\ln \E[ \e^{\tau'\psi(X_1, \theta)}] \right\vert =o(1) \ $;

\item[(ii)]  $\P\text{-a.s.}$  there exists a $\mathcal{B}(\T)\otimes \mathcal{E}/\mathcal{B}(\R^m)$-measurable function $\tau_T(.)$ s.t., for $T$ big enough,  for all $\theta \in \T$, $\tau_T(\theta) \in \arg \min_{\tau \in \R^m} \frac{1}{T}\sum_{t=1}^T \e^{\tau'\psi_t(\theta)}$,  $\frac{1}{T}\sum_{t=1}^T \e^{\tau_T(\theta)'\psi_t(\theta)}\psi_t(\theta)=0_{m \times 1}$ and $\tau_T(\theta)\in \intr[\Tbf(\theta)]$;

\item[(iii)]   $\P\text{-a.s.}$ as $T \rightarrow \infty$,  $ \sup_{\theta \in \T}\left\vert \tau_T(\theta)- \tau(\theta) \right\vert =o(1)$;
\item[(iv)]  $\P\text{-a.s.}$ as $T \rightarrow \infty$,
$ \sup_{ \theta \in\T}\left\vert\frac{1}{T}\sum_{t=1}^T \e^{\tau_{T}(\theta)'\psi_t(\theta)}- \E[ \e^{\tau(\theta)'\psi(X_1, \theta)}] \right\vert =o(1)$, which implies that $\P\text{-a.s.}$ as $T \rightarrow \infty$, $ \sup_{\theta \in \T}\left\vert\ln  \left[\frac{1}{T}\sum_{t=1}^T \e^{\tau_T(\theta)'\psi_t(\theta)}\right]-\ln \E[ \e^{\tau(\theta)'\psi(X_1, \theta)}] \right\vert =o(1). \ $
\end{enumerate}
\end{lem}

\begin{proof} \textit{(i)} Under Assumptions \ref{Assp:ExistenceConsistency} (a)-(b)(d)(e)(g) and (h), by Lemma \ref{Lem:TauCorrespondence}iii (p. \pageref{Lem:TauCorrespondence}), $\Sbf:=\{(\theta, \tau):\theta \in \T \wedge \tau \in \Tbf(\theta)\}$ is a compact set.\footnote{Note that, unlike what has been sometimes  suggested in the entropy literature, if $\Tbf(\theta)$ is an  unspecified  compact set, $\{(\theta, \tau): \theta \in \T \wedge \tau \in \Tbf(\theta)\}$ does not need to be a compact set\,: $\{(\theta, \tau):\theta \in \T \wedge \tau \in \Tbf(\theta)\}$ is not a Cartesian product, but the graph of a correspondence. See  Lemma \ref{Lem:TauCorrespondence}  (p. \pageref{Lem:TauCorrespondence}) for more details. } Thus, under Assumption \ref{Assp:ExistenceConsistency}(a)-(b), (d)  (e) and (h),   the  ULLN (uniform law of large numbers) \`a la  Wald  \citep[e.g.,][pp. 24-25, Theorem 1.3.3]{2003GhoRam} yields the first part of the result. Now, under Assumptions \ref{Assp:ExistenceConsistency} (a)(b)(d)(e)(g) and (h),  by Lemma \ref{Lem:ExpTauPsiStrictlyPositive} (p. \pageref{Lem:ExpTauPsiStrictlyPositive}), $(\theta, \tau)\mapsto \E[ \e^{\tau'\psi(X_1, \theta)}] $ is continuous, so that $\{\E[ \e^{\tau'\psi(X_1, \theta)}]:(\theta, \tau)\in \Sbf  \}$ is a compact set by Assumption  \ref{Assp:ExistenceConsistency}(d) ---continuous mappings preserve  compactness \citep[e.g.,][Theorem 4.14]{1953Rudin}. Moreover,    $x \mapsto \ln x$ is continuous, and, under Assumptions \ref{Assp:ExistenceConsistency} (a)(b)(d)(e)(g) and (h),  again by Lemma \ref{Lem:ExpTauPsiStrictlyPositive} (p. \pageref{Lem:ExpTauPsiStrictlyPositive}),
$0<\inf_{(\theta, \tau)\in \Sbf }\E[ \e^{\tau'\psi(X_1, \theta)}]$. Thus, we can choose an $\eta \in \left]0,  \inf_{(\theta, \tau)\in \Sbf }\E[ \e^{\tau'\psi(X_1, \theta)}]\right[$  s.t. $x \mapsto \ln x$ is uniformly continuous on the closed $\eta$-neighborhood of $\{\E[ \e^{\tau'\psi(X_1, \theta)}]:(\theta, \tau)\in \Sbf  \}$ ---continuous mappings on a compact set  are uniformly continuous \citep[e.g.,][Theorem 4.19]{1953Rudin}. Then, the second part follows from the first part of the result: By the first part,  $\P$-a.s. there exists a $\dot T\in \N $ s.t., $\forall T \in \ldsb \dot{T}, \infty \ldsb$,   $\sup_{(\theta, \tau) \in \Sbf }\vert\frac{1}{T}\sum_{t=1}^T \e^{\tau'\psi_t(\theta)}- \E[ \e^{\tau'\psi(X_1, \theta)}] \vert< \eta/2$.

\textit{(ii)-(iii)} The proof follows the overall strategy of \citet[Step 1 in the proof of Theorem 10]{2007Schennach}. For completeness and in order to justify our different assumptions, we provide  a detailed  proof. In particular, note that we formally prove that $0<\inf_{\theta \in \T}\inf_{\tau \in \Tbf(\theta):\vert \tau -\tau(\theta)\vert \geqslant \eta}\vert \E [\e^{\tau'\psi(X_1, \theta)}]-  \E[ \e^{\tau(\theta)'\psi(X_1, \theta)}] \vert$:  See  Lemma \ref{Lem:StrictPositivityOfEpsilon} (p. \pageref{Lem:StrictPositivityOfEpsilon}).  Let $\eta\in ]0,\epsilon_{\Tbf} ]$ be a fixed constant. By Assumption \ref{Assp:ExistenceConsistency}(a)(b), $(\theta, \omega )\mapsto \frac{1}{T}\sum_{t=1}^T \e^{\tau' \psi_t(\theta)}$ is continuous w.r.t $\theta$ and $\mathcal{E}/\mathcal{B}(\R)$-measurable w.r.t to $\omega$, so that it is  $\mathcal{B}(\T)\otimes \mathcal{E}/\mathcal{B}(\R)$-measurable \citep[e.g.,][Lemma 4.51]{1999AliprantisBorder}. Moreover, under Assumptions \ref{Assp:ExistenceConsistency} (a)-(b)(d)(e)(g) and (h), by Lemma \ref{Lem:TauCorrespondence}ii (p. \pageref{Lem:TauCorrespondence}),  $\theta \mapsto \Tbf(\theta)$ is a nonempty compact valued measurable correspondence.  Then, by a generalization of the Schmetterer-Jennrich lemma  \citep[e.g.,][Theorem 18.19]{1999AliprantisBorder}, we can define a $\mathcal{B}(\T)\otimes \mathcal{E}/\mathcal{B}(\R)$-measurable function $\tilde{\tau}_T(\theta)$ s.t., for all $\theta \in \T$, $\tilde{\tau}_T(\theta)\in \arg \min_{\tau  \in \Tbf(\theta)}\frac{1}{T}\sum_{t=1}^T \e^{\tau' \psi_t(\theta)}$.
%\footnote{We do not invoke the implicit function theorem to reach this conclusion because if $\tilde{\tau}_T(\theta)$  is on the boundary of $\Tbf(\theta) $, then $\frac{1}{T}\sum_{t=1}^T \e^{\tilde{\tau}_T(\theta)' \psi_t(\theta)}\psi_t(\theta)$ may not equal zero. % If it was the case, we could prove that (i.e., $\tilde{\tau}_T(\theta)\in \intr \Tbf$ ),   for $T$ big enough, $\frac{1}{T}\sum_{t=1}^T \e^{\tilde{\tau}_T(\theta)' \psi_t(\theta)}\psi_t(\theta)\psi_t(\theta)'>0 $, so that it is unique. } %
For the present proof, put $\varepsilon := \inf_{\theta \in \T}\inf_{\tau \in \Tbf(\theta):\vert \tau -\tau(\theta)\vert \geqslant \eta}\vert \E [\e^{\tau'\psi(X_1, \theta)}]-  \E[ \e^{\tau(\theta)'\psi(X_1, \theta)}] \vert$, which is strictly positive\footnote{The argument requires $\varepsilon>0 $.  If $\varepsilon =0$, then the upcoming inequality \eqref{Eq:UniformCVOfTildeTau} is not sufficient to show that  $\sup_{\theta \in \T} \vert\E [\e^{\tilde{\tau}_{T}(\theta)'\psi(X_1, \theta)}]-  \E[ \e^{\tau(\theta)'\psi(X_1, \theta)}]\vert < \varepsilon$.   } by Lemma \ref{Lem:StrictPositivityOfEpsilon} (p. \pageref{Lem:StrictPositivityOfEpsilon}) under  Assumptions \ref{Assp:ExistenceConsistency} (a)(b)(d)(e) and (h).\footnote{Strict convexity of $\tau \mapsto \E[\e^{\tau'\psi(X_1, \theta)}]$  and compactness of $\T$ are not sufficient to ensure that $\varepsilon>0$: We also need the continuity of the value function of the first infimum, which we obtain through Berge's maximum theorem. See  Lemma \ref{Lem:StrictPositivityOfEpsilon} (p. \pageref{Lem:StrictPositivityOfEpsilon}).} Then,  by the definition of $\varepsilon $, whenever $\sup_{\theta \in \T} \vert\E [\e^{\tilde{\tau}_{T}(\theta)'\psi(X_1, \theta)}]-  \E[ \e^{\tau(\theta)'\psi(X_1, \theta)}]\vert < \varepsilon$, then  $\sup_{\theta \in \T}\vert \tilde{\tau}_{T}(\theta)- \tau(\theta)  \vert\leqslant \eta $. We now show that it is happening $\P$-a.s. as $T \rightarrow \infty$.  Under Assumptions \ref{Assp:ExistenceConsistency} (a)(b)(d)(e)(g) and (h), by Lemma \ref{Lem:AsTiltingFct} (p. \pageref{Lem:AsTiltingFct}), $\tau(\theta)=\arg \min_{\tau \in \R^m} \E[ \e^{\tau'\psi(X_1,\theta)}]$, so that
\begin{eqnarray}
& &\sup_{\theta \in \T}\left\vert  \E[ \e^{\tilde{\tau}_T(\theta)'\psi(X_1,\theta)}]-\E[\e^{\tau(\theta)'\psi(X_1,\theta)}]\right\vert \notag\\
& = & \sup_{\theta \in \T}  \left\{\E[ \e^{\tilde{\tau}_T(\theta)'\psi(X_1,\theta)}]-\E[\e^{\tau(\theta)'\psi(X_1,\theta)}]\right\} \notag\\
& \stackrel{(a)}{=} &  \sup_{\theta \in \T}  \left\{\E[ \e^{\tilde{\tau}_T(\theta)'\psi(X_1,\theta)}]-\frac{1}{T}\sum_{t=1}^T\e^{\tilde{\tau}_T(\theta)' \psi_t(\theta)}+\frac{1}{T}\sum_{t=1}^T\e^{\tilde{\tau}_T(\theta)' \psi_t(\theta)}-\frac{1}{T}\sum_{t=1}^T\e^{\tau(\theta)' \psi_t(\theta)}\right. \notag \\
&  &\left.  +\frac{1}{T}\sum_{t=1}^T\e^{\tau(\theta)' \psi_t(\theta)}-\E[\e^{\tau(\theta)'\psi(X_1,\theta)}]\right\}\notag\\
& \stackrel{(b)}{\leqslant} & \sup_{\theta \in \T}  \left\{\E[ \e^{\tilde{\tau}_T(\theta)'\psi(X_1,\theta)}]-\frac{1}{T}\sum_{t=1}^T\e^{\tilde{\tau}_T(\theta)' \psi_t(\theta)} +\frac{1}{T}\sum_{t=1}^T\e^{\tau(\theta)' \psi_t(\theta)}-\E[\e^{\tau(\theta)'\psi(X_1,\theta)}]\right\} \notag\\
& \stackrel{(c)}{\leqslant} & \sup_{\theta \in \T}  \left\vert\E[ \e^{\tilde{\tau}_T(\theta)'\psi(X_1,\theta)}]-\frac{1}{T}\sum_{t=1}^T\e^{\tilde{\tau}_T(\theta)' \psi_t(\theta)}\right\vert +\sup_{\theta \in \T}  \left\vert\frac{1}{T}\sum_{t=1}^T\e^{\tau(\theta)' \psi_t(\theta)}-\E[\e^{\tau(\theta)'\psi(X_1,\theta)}]\right\vert \notag\\
& \stackrel{(d)}{=} & o(1) \text{ $\P$-a.s. as $T \rightarrow \infty$.} \label{Eq:UniformCVOfTildeTau}
\end{eqnarray}
\textit{(a)} Add and subtract $\frac{1}{T}\sum_{t=1}^T\e^{\tilde{\tau}_T(\theta)' \psi_t(\theta)} $ and $\frac{1}{T}\sum_{t=1}^T\e^{\tau(\theta)' \psi_t(\theta)} $. \textit{(b)} Note that, under Assumption \ref{Assp:ExistenceConsistency}(d) and (e), by definition, $\tau(\theta)\in \Tbf(\theta)$ and   $\tilde{\tau}_T(\theta)\in \arg \min_{\tau  \in \Tbf(\theta)}\frac{1}{T}\sum_{t=1}^T \e^{\tau' \psi_t(\theta)}$ so that $\frac{1}{T}\sum_{t=1}^T\e^{\tilde{\tau}_T(\theta)' \psi_t(\theta)}-\frac{1}{T}\sum_{t=1}^T\e^{\tau(\theta)' \psi_t(\theta)} \leqslant0$. \textit{(c)} Triangle inequality w.r.t. the uniform norm. \textit{(d)} Under Assumption \ref{Assp:ExistenceConsistency}(d)(e), by definition, for all $\theta \in \T $, $\tau(\theta)\in \Tbf(\theta)$ and $\tilde{\tau}_T(\theta)\in \Tbf(\theta)$ so that the conclusion follows from statement (i).

Inequality \eqref{Eq:UniformCVOfTildeTau} implies that $\sup_{\theta \in \T}\vert \tilde{\tau}_{T}(\theta)- \tau(\theta)  \vert=o(1) $ $\P$-a.s. as $T \rightarrow \infty$. Moreover, by Assumption \ref{Assp:ExistenceConsistency}(e), for all $\theta \in \T$, $\Tbf(\theta)=\overline{B_{\epsilon_{\Tbf}}(\tau(\theta))}$ where $\epsilon_\Tbf>0$. Thus, $\P$-a.s., for $T$ big enough, for all $\theta \in \T$, $\tilde{\tau}_{T}(\theta)\in\intr[\Tbf(\theta)] $.  Now,  \ for all $\theta \in \T$, $\tau \mapsto \frac{1}{T}\sum_{t=1}^T\e^{\tau'\psi_t(\theta)}$ is a convex function (Lemma \ref{Lem:ChgOfMeasureInvertibilityPDP}i on p. \pageref{Lem:ChgOfMeasureInvertibilityPDP} with $\Prm=\frac{1}{T}\sum_{t=1}^T\delta_{X_t}$ ensures that $\frac{\partial^2[\frac{1}{T}\sum_{t=1}^T\e^{\tau'\psi_t(\theta)}]}{\partial \tau \partial \tau'}=\frac{1}{T}\sum_{t=1}^T\e^{\tau'\psi_t(\theta)}\psi_t(\theta)\psi_t(\theta)'\geqslant0$), and the local minimum of a convex function is a global minimum \citep[e.g.,][p. 253]{1993Hir-UrLem}. Therefore, $\P$-a.s. for $T$ big enough, for all $\theta \in \T$, $\tilde{\tau}_T(\theta)$ minimizes $\frac{1}{T}\sum_{t=1}^T \e^{\tau' \psi_t(\theta)}$  not only over $\Tbf(\theta)$, but also over $\R^m$, which means that we can put $\tilde{\tau}_T(\theta)=\tau_T(\theta)$.

\textit{(iv)} Addition  and subtraction of $  \E[ \e^{\tau_{T}(\theta)'\psi(X_1, \theta)}]$, and   the triangle inequality yield $\P$-a.s. for $T$ big enough
\begin{eqnarray*}
& & \sup_{ \theta \in\T}\left\vert\frac{1}{T}\sum_{t=1}^T \e^{\tau_{T}(\theta)'\psi_t(\theta)}- \E[ \e^{\tau(\theta)'\psi(X_1, \theta)}] \right\vert\\
& \leqslant & \sup_{ \theta \in\T}\left\vert\frac{1}{T}\sum_{t=1}^T \e^{\tau_{T}(\theta)'\psi_t(\theta)}-  \E[ \e^{\tau_{T}(\theta)'\psi(X_1, \theta)}]\right\vert + \sup_{ \theta \in\T}\left\vert\E[ \e^{\tau_{T}(\theta)'\psi(X_1, \theta)}] - \E[ \e^{\tau(\theta)'\psi(X_1, \theta)}] \right\vert\\
& = & o(1) \text{ , as $T \rightarrow \infty$, }
\end{eqnarray*}
where the explanations for the last equality are as follows. By the statement (i) of the present lemma, $\P\text{-a.s.}$ as $T \rightarrow \infty$,
$ \sup_{(\theta, \tau) \in\Sbf}\left\vert\frac{1}{T}\sum_{t=1}^T \e^{\tau'\psi_t(\theta)}- \E[ \e^{\tau'\psi(X_1, \theta)}] \right\vert =o(1)$. Moreover, by the statement (ii) of the present lemma,  $\P\text{-a.s.}$ for $T$ big enough, $\tau_T(\theta)\in \intr[\Tbf(\theta)]$, so that, for all $\theta \in \T$, $( \theta, \tau_T(\theta))\in \Sbf$.   Thus, the first supremum is $o(1)$  as $T \rightarrow \infty$. Regarding the second supremum, under Assumption \ref{Assp:ExistenceConsistency} (a)(b)(d)(e)(g) and (h), by Lemma \ref{Lem:ExpTauPsiStrictlyPositive} (p. \pageref{Lem:ExpTauPsiStrictlyPositive}),     $(\theta, \tau) \mapsto \E[ \e^{\tau ' \psi(X_1, \theta)}]$ is  continuous in $\Sbf$. Now, under Assumptions \ref{Assp:ExistenceConsistency} (a)(b)(d)(e)(g) and (h), by Lemma \ref{Lem:TauCorrespondence}iii (p. \pageref{Lem:TauCorrespondence}), $\Sbf$ is compact, so that      $(\theta, \tau) \mapsto \E[ \e^{\tau ' \psi(X_1, \theta)}]$  is also uniformly continuous in $\Sbf$ ---continuous functions on compact sets are uniformly continuous \citep[e.g.,][Theorem 4.19]{1953Rudin}. Thus,   under Assumption \ref{Assp:ExistenceConsistency}(a)(b),  (d)-(e), (g) and (h), by the statement (iii) of the present lemma, which states  that $\sup_{\theta \in \T}\left\vert \tau_T(\theta)- \tau(\theta) \right\vert =o(1)$ $\P$-a.s. as $T \rightarrow \infty$, the second supremum is also $o(1)$  $\P$-a.s. as $T \rightarrow \infty$.

The second part of the result follows from the first part as in the proof of the statement (i) of the present lemma.
 \end{proof}

\begin{lem}\label{Lem:ExpTauPsiStrictlyPositive} Let $\Prm$ be any probability measure, and $\E_{\Prm}$ denote the expectation under $\Prm$. Under Assumption \ref{Assp:ExistenceConsistency} (a)(b)(d)(e)(g) and (h), if $\E_{\Prm} [\sup_{(\theta, \tau)\in \Sbf } \e^{\tau'\psi(X_1, \theta)}]< \infty$, then

\noindent
$0<\inf_{(\theta, \tau)\in \Sbf }\E_{\Prm}[ \e^{\tau'\psi(X_1, \theta)}]$, so that $0<\inf_{ \theta\in \T }\E_{\Prm}[ \e^{\tau(\theta)'\psi(X_1, \theta)}]$. Moreover, $(\theta, \tau)\mapsto \E_{\Prm}[ \e^{\tau'\psi(X_1, \theta)}]$ and $\theta \mapsto\E_{\Prm}[ \e^{\tau(\theta)'\psi(X_1, \theta)}]  $ are continuous in $\Sbf$ and $\T$, respectively. All of these results hold for $\Prm=\P$ under the aforementioned assumptions.
\end{lem}
\begin{proof}Under Assumption \ref{Assp:ExistenceConsistency} (a) and (b), the Lebesgue dominated convergence theorem and the lemma's assumption $\E_{\Prm} [\sup_{(\theta, \tau)\in \Sbf } \e^{\tau'\psi(X_1, \theta)}]< \infty$ imply that $(\theta, \tau)\mapsto\E_{\Prm}[ \e^{\tau'\psi(X_1, \theta)}] $ is continuous. Moreover, under Assumptions \ref{Assp:ExistenceConsistency} (a)(b)(d)(e)(g) and (h), by Lemma \ref{Lem:TauCorrespondence} (p. \pageref{Lem:TauCorrespondence}), $\Sbf$ is compact, and continuous functions over compact sets reach a minimum \citep[e.g., ][Theorem 4.16]{1953Rudin}. Now,   if there exist $(\dot\tau, \dot\theta)\in \Sbf  $ s.t.  $0=\E_{\Prm}[ \e^{\dot\tau'\psi(X_1, \dot\theta)}] $, then $ \e^{\dot\tau'\psi(X_1, \dot\theta)}=0$ $\Prm$-a.s.  \citep[e.g.,][Lemma 1.24]{2002Kal}, which  is impossible by definition of the exponential function. Thus,  $0<\inf_{(\theta, \tau)\in \Sbf }\E_{\Prm}[ \e^{\tau'\psi(X_1, \theta)}]$, so that $0<\inf_{ \theta\in \T }\E_{\Prm}[ \e^{\tau(\theta)'\psi(X_1, \theta)}]$ because by the  definition of $\Sbf$ in Assumption \ref{Assp:ExistenceConsistency}(e), for all $\theta \in \T$, $( \theta, \tau(\theta))\in \Sbf$. Regarding the second part of the result, it immediately follows from the Lebesgue dominated convergence theorem, the lemma's assumption that $\E_{\Prm} [\sup_{(\theta, \tau)\in \Sbf } \e^{\tau'\psi(X_1, \theta)}]< \infty$, and the continuity of $\tau: \T \rightarrow \R^m$ by Lemma \ref{Lem:AsTiltingFct}iii (p. \pageref{Lem:AsTiltingFct}) under Assumptions \ref{Assp:ExistenceConsistency} (a)(b)(d)(e)(g) and (h). Regarding the third part of the result, it is sufficient to note that, under Assumption \ref{Assp:ExistenceConsistency} (a)(b), by the Cauchy-Schwarz inequality, $ \E[\sup_{(\theta, \tau)\in \Sbf} \e^{\tau'\psi(X_1, \theta)}]\leqslant \E[\sup_{(\theta, \tau)\in \Sbf} \e^{2\tau'\psi(X_1, \theta)}]^{1/2}< \infty$, where the last inequality follows from Assumption \ref{Assp:ExistenceConsistency}(e).
\end{proof}
\begin{lem}[Compactness of $\Sbf$]\label{Lem:TauCorrespondence}
Under Assumptions \ref{Assp:ExistenceConsistency} (a)(b)(d)(e)(g) and (h),
\begin{itemize}
\item[(i)] The closure of the $\epsilon_{\Tbf}$-neighborhood of $\tau(\T)$ (i.e., $ \overline{\tau(\T)^{\epsilon_{\Tbf}}}$) is compact
\item[(ii)] For all $\theta \in \T$, the correspondence $\theta \mapsto \Tbf(\theta)$ is nonempty compact-valued and uhc (upper hemi-continuous), and thus measurable;
\item[(iii)] The set $\Sbf:=\{(\theta, \tau):\theta \in \T \wedge \tau \in \Tbf(\theta)\}$ is compact.
\end{itemize}
\end{lem}
\begin{proof}\textit{(i)} Under Assumptions \ref{Assp:ExistenceConsistency} (a)(b)(d)(e)(g) and (h),
 by Lemma \ref{Lem:AsTiltingFct}iii (p. \pageref{Lem:AsTiltingFct}), $\tau: \T\rightarrow \R^m $ is continuous. Moreover, by Assumption \ref{Assp:ExistenceConsistency}(d),  $\T$ is compact. Thus,  $\tau(\T)$ is bounded ---continuous mappings preserve  compactness \citep[e.g.,][Theorem 4.14]{1953Rudin}. Consequently, $ \tau(\T)^{\epsilon_{\Tbf}}=:\{\tau \in \R^m:\inf_{\tilde{\tau} \in \tau(\T)}\vert \tau- \tilde{\tau}\vert< \epsilon_{\Tbf}  \}$ is bounded, which means that its closure $ \overline{\tau(\T)^{\epsilon_{\Tbf}}}$ is closed and bounded, i.e., compact.

\textit{(ii)} \textit{Proof that $\Tbf$ is nonempty and compact valued.} By Assumption \ref{Assp:ExistenceConsistency}(d), for all $\theta \in \T$, there exists $\tau(\theta)$ s.t. $\E[ \e^{\tau(\theta)'\psi(X_1, \theta)}\psi(X_1, \theta)]=0$. Thus, for all $\theta \in \T$, $\Tbf(\theta)=\overline{B_{\epsilon_{\Tbf}}(\tau(\theta))}$ is nonempty. Moreover, by construction,  $\overline{B_{\epsilon_{\Tbf}}(\tau(\theta))}$ is compact, so that it is nonempty compact valued.

\textit{Proof that $\Tbf$ is uhc.}  Because $\Tbf$ is compact valued, we can use the sequential characterization of upper hemicontinuity \cite[e.g.,][Theorem 17.20]{1999AliprantisBorder}. Let $((\theta_n, \tau_n))_{n \in \N}\in (\Sbf)^\N $ be a sequence s.t., for all $n \in \N$, $\tau_n \in \Tbf(\theta_n)$ and  $\theta_n \rightarrow \bar{\theta}\in \T$ as $n \rightarrow \infty$. By construction, for all $n \in \N$, $\tau_n\in \overline{B_{\epsilon_{\Tbf}}(\tau(\theta_{n}))} \subset \overline{\tau(\T)^{\epsilon_{\Tbf}}}$. Moreover, by statement (i),  $ \overline{\tau(\T)^{\epsilon_{\Tbf}}}$ is compact, so that  there exists a subsequence $(\tau_{\alpha(n)})_{n \in \N}$  s.t.   $\tau_{\alpha(n)}\rightarrow \bar{\tau}\in \overline{\tau(\T)^{\epsilon_{\Tbf}}}$, as $n \rightarrow \infty$. Again, by construction, for all $n \in \N$, $(\theta_n, \tau_n)\in \Sbf $, so that $ \vert \tau_{\alpha(n)}- \tau(\theta_{\alpha(n)})\vert\leqslant\epsilon_{\Tbf} $. Now,  under Assumptions \ref{Assp:ExistenceConsistency} (a)(b)(d)(e)(g) and (h), by Lemma \ref{Lem:AsTiltingFct}iii,
$\tau:\T \rightarrow \R^m $ is continuous. Thus,   $\vert \tau_{\alpha(n)}- \tau(\theta_{\alpha(n)})\vert \rightarrow \vert\bar{\tau}- \tau(\bar{\theta})\vert $ as $n \rightarrow \infty$. Thus,  $ \vert\bar{\tau}- \tau(\bar{\theta})\vert\leqslant\epsilon_{\Tbf}$, which means that $\bar{\tau} \in \overline{B_{\epsilon_{\Tbf}}(\tau(\bar{\theta}))}=\Tbf(\bar{\theta})$.

\textit{Proof that $\Tbf$ is measurable.}  Let $F$ be a closed subset of $\R^m$. Then, its complement $F^c$ is an open subset of $\R^m$.
Now, a correspondence is uhc iff  the upper inverse image of a open set is an open set \citep[e.g.,][Lemma 17.4]{1999AliprantisBorder}.
Thus, by the previous paragraph, $\Tbf^u(F^c)\in \mathcal{B}(\T) $, where $\Tbf^u$ denotes the upper inverse of $\Tbf$. Now, denoting the lower inverse of $\Tbf$ with $\Tbf^l$, notice that $\Tbf^u(F^c)=[\Tbf^l(F)]^c $ \citep[e.g.,][p. 557]{1999AliprantisBorder}, so that $[\Tbf^l(F)]^c \in \mathcal{B}(\T)$, which, in turn implies that $\Tbf^l(F) \in \mathcal{B}(\T) $ because of the stability of $\sigma$-algebras under complementation.

\textit{(iii)} Note that the compactness of $\T$ and $\Tbf(\theta)$ are not sufficient to ensure the compactness of  $\Sbf$ because $\Sbf$ is not a Cartesian product. By the statement (ii) of the present lemma,  $\Tbf$ is uhc and closed valued, so that it has a closed graph \citep[e.g.,][Theorem 17.10]{1999AliprantisBorder}, i.e., $\Sbf$ is closed. Now, by construction, $\Sbf$ is a subset $[\overline{\tau(\T)^{\epsilon_{\Tbf}}} \times \T ]$, which is compact by statement (i) and Assumption \ref{Assp:ExistenceConsistency}(d). Thus, $\Sbf$ is also compact ---in metric spaces, closed subsets of compact sets are compact \citep[e.g.,][Theorem 2.35]{1953Rudin}.
\end{proof}

\begin{lem} \label{Lem:StrictPositivityOfEpsilon}Under Assumptions \ref{Assp:ExistenceConsistency} (a)(b)(d)(e) and (h),
\begin{enumerate}
\item[(i)]  for any constant $\eta\in ]0,\epsilon_\Tbf ] $, there exists a continuous value function   $v: \T \rightarrow \R_+ $ s.t., for all $\theta \in \T$, $v(\theta)= \inf_{\tau \in \Tbf(\theta):\vert \tau -\tau(\theta)\vert \geqslant \eta}\vert \E [\e^{\tau'\psi(X_1, \theta)}]-  \E[ \e^{\tau(\theta)'\psi(X_1, \theta)}] \vert$;
\item[(ii)] for any constant $\eta\in ]0,\epsilon_\Tbf ] $, $0 <\inf_{\theta \in \T}\inf_{\tau \in \Tbf(\theta):\vert \tau -\tau(\theta)\vert \geqslant \eta}\vert \E [\e^{\tau'\psi(X_1, \theta)}]-  \E[ \e^{\tau(\theta)'\psi(X_1, \theta)}] \vert$.
\end{enumerate}

\end{lem}
\begin{proof}

\textit{(i)} It is a consequence  of Berge's  maximum theorem \citep[e.g.,][Theorem 17.31]{1999AliprantisBorder}. Thus, it remains to check its assumptions. For the present proof, define the correspondence $\varphi: \T \twoheadrightarrow \R^m  $ s.t. $\varphi(\theta)=\{ \tau \in \Tbf(\theta): \vert \tau- \tau(\theta) \vert\geqslant \eta \}$, and the function $f: \Sbf  \rightarrow \R_+$ s.t. $f(\theta, \tau)=\vert \E [\e^{\tau'\psi(X_1, \theta)}]-  \E[ \e^{\tau(\theta)'\psi(X_1, \theta)}] \vert$.

\textit{Proof of the continuity of $f$.} Under Assumption \ref{Assp:ExistenceConsistency} (a)(b)(d)(e)(g) and (h), by Lemma \ref{Lem:ExpTauPsiStrictlyPositive} (p. \pageref{Lem:ExpTauPsiStrictlyPositive}),  $(\theta, \tau)\mapsto \E[ \e^{\tau'\psi(X_1, \theta)}] $ and $ \theta \mapsto \E[ \e^{\tau(\theta)'\psi(X_1, \theta)}]$ are  continuous in $\Sbf$ and $\T$, respectively, so that  the continuity of $f$ follows immediately.

\textit{ Proof that $\varphi$ is nonempty compact valued.}  By the definition of $\Tbf$ in Assumption \ref{Assp:ExistenceConsistency}(e), for all $\theta \in \T$, $\Tbf(\theta)=\overline{B_{\epsilon_\Tbf}(\tau(\theta))}$, so that, for any $\eta\in ]0,\epsilon_\Tbf ] $, $\varphi(\theta)=\overline{B_{\epsilon_\Tbf}(\tau(\theta))}\cap \{\tau \in \R^m:\eta \leqslant \vert \tau - \tau(\theta) \vert \}\neq \emptyset$, i.e., $\varphi$ is nonempty valued.  Moreover, for all $\theta \in \T$, $\overline{B_{\epsilon_\Tbf}(\tau(\theta))}$ is a compact set and  $\{\tau \in \R^m:\eta \leqslant \vert \tau - \tau(\theta) \vert \}$ is a closed set, so that $\varphi(\theta)$, which is their intersection,    is compact  \citep[e.g.,][Theorem 2.35 and the following Corollary]{1953Rudin}.

\textit{Proof of the upper hemicontinuity of $\varphi$.} Because $\varphi$ is compact valued, we can use the sequential characterization of upper hemicontinuity \cite[e.g.,][Theorem 17.20]{1999AliprantisBorder}. Let $((\theta_n, \tau_n))_{n \in \N}\in \Sbf^\N $ be a sequence s.t., for all $n \in \N$, $\tau_n \in \varphi(\theta_n)$ and  $\theta_n \rightarrow \bar{\theta}\in \T$ as $n \rightarrow \infty$. Now, under Assumptions \ref{Assp:ExistenceConsistency} (a)(b)(d)(e)(g) and (h), Lemma \ref{Lem:TauCorrespondence}iii (p. \pageref{Lem:TauCorrespondence}),   $\Sbf$ is a compact set, so that there exists a subsequence $((\theta_{\alpha(n)},\tau_{\alpha(n)}))_{n \in \N}$  s.t.   $(\theta_{\alpha(n)},\tau_{\alpha(n)})\rightarrow (\bar{\theta}, \bar{\tau})\in \Sbf$, as $n \rightarrow \infty$. The definition of $\Sbf$ implies that $ \bar{\tau}\in \Tbf(\bar{\theta})$. Thus, it remains to show that  $\eta \leqslant \vert\bar{\tau}- \tau(\bar{\theta})\vert$ in order to conclude that  $\bar{\tau} \in \varphi(\bar{\theta})$. By construction, for all $n \in \N$, $\eta \leqslant \vert \tau_{\alpha(n)}- \tau(\theta_{\alpha(n)})\vert$. Moreover,  under Assumptions \ref{Assp:ExistenceConsistency} (a)(b)(d)(e)(g) and (h), by Lemma \ref{Lem:AsTiltingFct}iii (p. \pageref{Lem:AsTiltingFct}),
$\tau:\T \rightarrow \R^m $ is continuous, so that  $\vert \tau_{\alpha(n)}- \tau(\theta_{\alpha(n)})\vert \rightarrow \vert\bar{\tau}- \tau(\bar{\theta})\vert $ as $n \rightarrow \infty$, which means that   $\eta \leqslant \vert\bar{\tau}- \tau(\bar{\theta})\vert$.

\textit{Proof of the lower hemicontinuity of $\varphi$.}  Use the sequential characterization of the lower hemicontinuity \cite[e.g.,][Theorem 17.21]{1999AliprantisBorder}. Let $(\theta_n)_{n \in \N}\in \T^{\N}$ be a sequence s.t. $\theta_n \rightarrow \bar{\theta}\in \T$ and $\bar{\tau} \in \varphi(\bar{\theta}) $. Define the sequence  $(\tau_{n})_{n \in \N}$   s.t., for all $n \in \N$, $\tau_n=\tau(\theta_n)+\bar{\tau}-\tau(\bar{\theta})$.   By definition of the correspondence $\varphi$, for all $n \in \N$, $\vert \tau_n- \tau(\theta_n)\vert=\vert\bar{\tau}-\tau(\bar{\theta}) \vert \in [ \eta, \epsilon_{\Tbf}] $, which implies that  $\tau_n \in \varphi(\theta_n) $.  Moreover, under Assumptions \ref{Assp:ExistenceConsistency} (a)(b)(d)(e)(g) and (h), by Lemma \ref{Lem:AsTiltingFct}iii (p. \pageref{Lem:AsTiltingFct}),
$\tau:\T \rightarrow \R^m $ is continuous, so that $\lim_{n \rightarrow \infty} \tau_n=\lim_{n \rightarrow \infty}\tau(\theta_n)+\bar{\tau}-\tau(\bar{\theta})=\tau(\bar{\theta})+\bar{\tau}-\tau(\bar{\theta})=\bar{\tau} $.

\textit{(ii)} Under Assumptions \ref{Assp:ExistenceConsistency} (a)(b)(d)(e)(g) and (h), by Lemma \ref{Lem:AsTiltingFct} (p. \pageref{Lem:AsTiltingFct}), for all $\theta \in \T$, $\tau(\theta)$ is the unique minimum of the strictly convex minimization problem $\inf_{\tau \in \R^m}\E[ \e^{\tau'\psi(X_1, \theta)}] $. Thus, for all $\theta \in \T$, $v(\theta)>0$. Moreover, by Assumption \ref{Assp:ExistenceConsistency}(d), $\T$ is compact, and by statement (i) of the present lemma, $v(.)$ is continuous. Thus, there exists $\varepsilon_v>0$ s.t. $\min_{\theta \in \T}v(\theta )> \varepsilon_v$ because a continuous function over a compact set reaches a minimum \citep[e.g., ][Theorem 4.16]{1953Rudin}.
\end{proof}

\begin{lem}[Asymptotic limit of the variance term] \label{Lem:AsVarianceTermBehaviour}
Under Assumption \ref{Assp:ExistenceConsistency}(a)-(b) and (d)-(h),
\begin{enumerate}
\item[(i)] $\P$-a.s. for $T$ big enough,  $0<\inf_{\theta \in \T}\left\vert\left\vert \left[\sum_{t=1}^T \hat{w}_{t,\theta}\frac{\partial \psi_{t} (\theta)'}{\partial \theta}\right]\right\vert_{\det}\right\vert$;
\item[(ii)]     $\P$-a.s. as $T \rightarrow \infty$, $\sup_{\theta \in \T}\left\vert \Sigma_T(\theta) -  \E[ \e^{\tau(\theta)'\psi(X_1, \theta)}]\Sigma(\theta) \right\vert=o(1)$
\item[(iii)] $\theta \mapsto \Sigma(\theta)  $ and  $\theta \mapsto\E[ \e^{\tau(\theta)'\psi(X_1, \theta)}]\Sigma(\theta) $ are continuous in $\T$

\item[(iv)]  $\P$-a.s. for $T$ big enough, $ \inf_{\theta \in \T} \vert \Sigma_T(\theta)\vert_{\det} >0$;
\item[(v)] $\P$-a.s. as $T \rightarrow \infty$, $\sup_{\theta \in \T}\left\vert \ln \vert \Sigma_T(\theta)\vert_{\det}-\ln \left\vert  \E[ \e^{\tau(\theta)'\psi(X_1, \theta)}]\Sigma(\theta)\right\vert_{\det}\right\vert=o(1)$, so that, for all $\eta>0$, $\P$-a.s. as $T \rightarrow \infty$, $\sup_{\theta \in \T}\vert \frac{1}{2T^\eta} \ln \vert \Sigma_T(\theta)\vert_{\det}\vert=o(1)$.
\end{enumerate}

\end{lem}
\begin{proof} \textit{(i)}
Under Assumption \ref{Assp:ExistenceConsistency}(a)-(b) and (d)-(h),
by Lemma \ref{Lem:TiltedDerivative} (p. \pageref{Lem:TiltedDerivative}),  $\P$-a.s. as $T \rightarrow \infty$,
$\sup_{ \theta\in   \T}\left\vert\left[ \frac{1}{T}\sum_{t=1}^T \hat{w}_{t,\theta}\frac{\partial \psi_{t} (\theta)'}{\partial \theta}\right] -\frac{1}{\E[ \e^{\tau(\theta)'\psi(X_1, \theta)}]}\E \left[\e^{\tau(\theta)'\psi(X_1, \theta) }\frac{\partial \psi(X_1, \theta)'}{\partial \theta}\right]\right\vert=o(1)$, so that it is sufficient to check the invertibility of $\frac{1}{\E[ \e^{\tau(\theta)'\psi(X_1, \theta)}]}\E \left[\e^{\tau(\theta)'\psi(X_1, \theta) }\frac{\partial \psi(X_1, \theta)'}{\partial \theta}\right]$ for all $\theta \in \T$ and the continuity of $\theta \mapsto\frac{1}{\E[ \e^{\tau(\theta)'\psi(X_1, \theta)}]}\E \left[\e^{\tau(\theta)'\psi(X_1, \theta) }\frac{\partial \psi(X_1, \theta)'}{\partial \theta}\right]$   (Lemma \ref{Lem:UniFiniteSampleInvertibilityFromUniAsInvertibility}  on p. \pageref{Lem:UniFiniteSampleInvertibilityFromUniAsInvertibility}).  Firstly, by Assumption \ref{Assp:ExistenceConsistency}(h),
for all $\theta \in \T$, $\Sigma(\theta):=\negthickspace \left[\E\e^{\tau(\theta)' \psi(X_{1},\theta)}  \frac{\partial \psi(X_{1},\theta)}{\partial \theta' }\right]^{-1}\negthickspace\E\left[\e^{\tau(\theta)' \psi(X_{1},\theta)}\psi(X_{1},\theta) \psi(X_{1},\theta)' \right]\negthickspace \\\left[\E\e^{\tau(\theta)' \psi(X_{1},\theta)}  \frac{\partial \psi(X_{1},\theta)'}{\partial \theta }\right]^{-1}$  is a positive-definite symmetric matrix, and thus $\left[\E\e^{\tau(\theta)' \psi(X_{1},\theta)}  \frac{\partial \psi(X_{1},\theta)}{\partial \theta }'\right]$  is invertible.  Moreover, under Assumption \ref{Assp:ExistenceConsistency} (a)(b)(d)(e)(g) and (h), by Lemma \ref{Lem:ExpTauPsiStrictlyPositive} (p. \pageref{Lem:ExpTauPsiStrictlyPositive}), and Assumption \ref{Assp:ExistenceConsistency}(e), for all $\theta \in \T$, $0<\E[\e^{\tau(\theta)'\psi(X_1,\theta)}]<\infty$, so that $\frac{1}{\E[ \e^{\tau(\theta)'\psi(X_1, \theta)}]}\E \left[\e^{\tau(\theta)'\psi(X_1, \theta) }\frac{\partial \psi(X_1, \theta)'}{\partial \theta}\right]$   is  invertible for all $\theta \in \T$. Secondly, under Assumption \ref{Assp:ExistenceConsistency}(a)-(b), (e)-(f), by Lemma \ref{Lem:TiltedVariance}i (p. \pageref{Lem:TiltedVariance}), $ \E \left[\sup_{(\theta, \tau) \in \Sbf} \vert \e^{\tau'\psi(X_1, \theta) }\frac{\partial \psi(X_1, \theta)'}{\partial \theta} \vert\right]< \infty$, so that  the Lebesgue dominated convergence theorem and Assumption \ref{Assp:ExistenceConsistency}(b) imply the continuity of  $(\theta, \tau)\mapsto \E \left[ \e^{\tau'\psi(X_1, \theta) }\frac{\partial \psi(X_1, \theta)'}{\partial \theta} \right] $  in $\Sbf$. Moreover, by definition in Assumption \ref{Assp:ExistenceConsistency}(e), for all $\theta \in \T$, $(\tau(\theta), \theta)\in \Sbf$, and  under Assumption \ref{Assp:ExistenceConsistency} (a)(b)(d)(e)(g) and (h),
by Lemma \ref{Lem:AsTiltingFct}iii (p. \pageref{Lem:AsTiltingFct}),
$\tau:\T \rightarrow \R^m $ is continuous.
  Thus, $\theta \mapsto \left[\E\e^{\tau(\theta)' \psi(X_{1},\theta)}  \frac{\partial \psi(X_{1},\theta)}{\partial \theta }'\right]$  is continuous. Then, the  continuity of 
 
\noindent 
$\theta \mapsto\frac{1}{\E[ \e^{\tau(\theta)\psi(X_1, \theta)}]}\E \left[\e^{\tau(\theta)'\psi(X_1, \theta) }\frac{\partial \psi(X_1, \theta)'}{\partial \theta}\right]$ follows from Lemma \ref{Lem:ExpTauPsiStrictlyPositive} (p. \pageref{Lem:ExpTauPsiStrictlyPositive}) under Assumption \ref{Assp:ExistenceConsistency} (a)(b)(d)(e)(g) and (h).

\textit{(ii)}  On one hand, by definition, $\Sigma(\theta)\negthickspace:=\negthickspace \left[\E\e^{\tau(\theta)' \psi(X_{1},\theta)}  \frac{\partial \psi(X_{1},\theta)}{\partial \theta' }\right]^{-1}\negthickspace\negthickspace\E\left[\e^{\tau(\theta)' \psi(X_{1},\theta)}\psi(X_{1},\theta) \psi(X_{1},\theta)' \right]\negthickspace\\ \left[\E\e^{\tau(\theta)' \psi(X_{1},\theta)}  \frac{\partial \psi(X_{1},\theta)'}{\partial \theta }\right]^{-1}$, which is symmetric positive definite   by Assumption \ref{Assp:ExistenceConsistency} (h), and $\Sigma_T(\theta)  :=  \left[ \sum_{t=1}^T \hat{w}_{t,\theta}\frac{\partial \psi_{t} (\theta)}{\partial \theta'}\right]^{-1}\left[\sum_{t=1}^T \hat{w}_{t,\theta}\psi_t(\theta)\psi_t(\theta)'\right]\left[ \sum_{t=1}^T \hat{w}_{t,\theta}\frac{\partial \psi_{t} (\theta)'}{\partial \theta}\right]^{-1}$, which is well defined $\P$-a.s. for $T$ big enough by the statement (i) of the present lemma. On the other hand,
under   Assumption \ref{Assp:ExistenceConsistency}(a)-(b) and (d)-(h),
by Lemma \ref{Lem:TiltedDerivative}iii (p. \pageref{Lem:TiltedDerivative}),  $\P$-a.s. as $T \rightarrow \infty$,
$\sup_{ \theta\in   \T}\vert\left[ \frac{1}{T}\sum_{t=1}^T \hat{w}_{t,\theta}\frac{\partial \psi_{t} (\theta)'}{\partial \theta}\right] -\frac{1}{\E[ \e^{\tau(\theta)'\psi(X_1, \theta)}]}\E \left[\e^{\tau(\theta)'\psi(X_1, \theta) }\frac{\partial \psi(X_1, \theta)'}{\partial \theta}\right]\vert=o(1)$, and, under Assumptions \ref{Assp:ExistenceConsistency}(a)-(b), (d)-(e) and (g)-(h), by Lemma \ref{Lem:TiltedVariance} (p. \pageref{Lem:TiltedVariance}), $\P$-a.s. as $T \rightarrow \infty$, $\sup_{ \theta\in   \T}\vert\left[ \frac{1}{T}\sum_{t=1}^T \e^{\tau_T(\theta)'\psi_t(\theta)}\psi_t(\theta) \psi_t(\theta)'\right] -\frac{1}{\E[ \e^{\tau(\theta)'\psi(X_1, \theta)}]}\E\left[ \e^{\tau(\theta)'\psi(X_1, \theta) }\psi(X_1,\theta) \psi(X_1,\theta)'\right]\vert=o(1)$. Thus, the claim follows from the continuity of the inverse transformation \citep[e.g.,][Theorem 9.8]{1953Rudin} and the limiting functions,  and the compactness of $\T$.

\textit{(iii)} Under Assumption \ref{Assp:ExistenceConsistency}(a)-(b), (e)-(g), by Lemma \ref{Lem:TiltedDerivative}i (p. \pageref{Lem:TiltedDerivative}) and \ref{Lem:TiltedVariance} (p. \pageref{Lem:TiltedVariance}), \\$ \E \left[\sup_{(\theta, \tau) \in \Sbf} \vert \e^{\tau'\psi(X_1, \theta) }\frac{\partial \psi(X_1, \theta)'}{\partial \theta} \vert\right]< \infty$ and $\E \left[\sup_{(\theta, \tau) \in \Sbf}\vert \e^{\tau'\psi(X_1, \theta) }\psi(X_{1},\theta) \psi(X_{1},\theta)' \vert\right]< \infty$, so that, by  the Lebesgue dominated convergence theorem and Assumption \ref{Assp:ExistenceConsistency}(b), 

\noindent
 $(\theta, \tau)\mapsto \E \left[ \e^{\tau'\psi(X_1, \theta) }\frac{\partial \psi(X_1, \theta)'}{\partial \theta} \right] $ and $(\theta, \tau)\mapsto \E \left[ \e^{\tau'\psi(X_1, \theta) }\psi(X_{1},\theta) \psi(X_{1},\theta)' \right]$
are continuous in $\Sbf$. Moreover, by definition in Assumption \ref{Assp:ExistenceConsistency}(e), for all $\theta \in \T$, $( \theta,\tau(\theta))\in \Sbf$, and  under Assumption \ref{Assp:ExistenceConsistency} (a)(b)(d)(e)(g) and (h),
by Lemma \ref{Lem:AsTiltingFct}iii (p. \pageref{Lem:AsTiltingFct}),
$\tau:\T \rightarrow \R^m $ is continuous.
 Thus $\theta \mapsto \Sigma(\theta)  $  is continuous, which is the first result. Under  Assumption \ref{Assp:ExistenceConsistency} (a)(b)(d)(e)(g) and (h), the second result follows from  Lemma \ref{Lem:ExpTauPsiStrictlyPositive} (p. \pageref{Lem:ExpTauPsiStrictlyPositive}), which states  that $\theta \mapsto \E[ \e^{\tau(\theta)'\psi(X_1, \theta)}] $ is also continuous.

\textit{(iv)} By construction, $\Sigma_T(\theta)$ is a symmetric positive semi-definite matrix (Lemma \ref{Lem:ChgOfMeasureInvertibilityPDP}i  on p. \pageref{Lem:ChgOfMeasureInvertibilityPDP} with $\Prm=\frac{1}{T}\sum_{t=1}^T\delta_{X_t}$), so that $\vert\Sigma_T(\theta) \vert_{\det} \geqslant 0$.
Thus, by the statement (ii) and (iii) of present lemma, it is sufficient to check the  invertibility of $   \E[ \e^{\tau(\theta)'\psi(X_1, \theta)}]\Sigma(\theta)$ for all $\theta \in \T$  (Lemma \ref{Lem:UniFiniteSampleInvertibilityFromUniAsInvertibility}  on p. \pageref{Lem:UniFiniteSampleInvertibilityFromUniAsInvertibility}).   By Assumption \ref{Assp:ExistenceConsistency} (h),
for all $\theta \in \T$, 

\noindent
$\Sigma(\theta):=\negthickspace \left[\E\e^{\tau(\theta)' \psi(X_{1},\theta)}  \frac{\partial \psi(X_{1},\theta)}{\partial \theta' }\right]^{-1}\negthickspace\negthickspace\E\left[\e^{\tau(\theta)' \psi(X_{1},\theta)}\psi(X_{1},\theta) \psi(X_{1},\theta)' \right]  \left[\E\e^{\tau(\theta)' \psi(X_1,\theta)}  \frac{\partial \psi(X_1,\theta)'}{\partial \theta }\right]^{-1}$

\noindent
is a positive-definite symmetric matrix, and thus a fortiori  invertible. Moreover, under Assumption \ref{Assp:ExistenceConsistency} (a)(b)(d)(e)(g) and (h), by Lemma \ref{Lem:ExpTauPsiStrictlyPositive} (p. \pageref{Lem:ExpTauPsiStrictlyPositive}), and Assumption \ref{Assp:ExistenceConsistency}(e), for all $\theta \in \T$, $0<\E[\e^{\tau(\theta)'\psi(X_{1},\theta)}]<\infty$, so that it is also invertible.

\textit{(v)} Under Assumption  \ref{Assp:ExistenceConsistency} (a)(b)(d)(e)(g) and (h),   by Lemma \ref{Lem:ExpTauPsiStrictlyPositive} (p. \pageref{Lem:ExpTauPsiStrictlyPositive}) with $\Prm=\sum_{t=1}^T\delta_{X_t}$ and by the statement (iv) of the present lemma, $\P$-a.s. for $T$ big enough,  $\ln \vert \Sigma_T(\theta)\vert_{\det} $  is well-defined in  $\T$. Similarly, under Assumption  \ref{Assp:ExistenceConsistency} (a)(b)(d)(e)(g) and (h),   by Lemma \ref{Lem:ExpTauPsiStrictlyPositive} (p. \pageref{Lem:ExpTauPsiStrictlyPositive}) and  Assumption  \ref{Assp:ExistenceConsistency}  (h), $\ln \left\vert  \E[ \e^{\tau(\theta)'\psi(X_1, \theta)}]\Sigma(\theta)\right\vert_{\det}$ is well-defined in  $\T$. Then, the first part of the result follows from the statement (ii) of the present lemma. Regarding the second part, by the triangle inequality, $\P$-a.s. as $T \rightarrow \infty$,
\begin{eqnarray*}
& & \frac{1}{T^\eta}\sup_{\theta \in \T}\left\vert  \ln \vert \Sigma_T(\theta)\vert_{\det}\right\vert\\
& \leqslant & \frac{1}{T^\eta}\sup_{\theta \in \T}\left\vert  \ln\left[ \vert \Sigma_T(\theta)\vert_{\det}\right]-\ln\left[  \left\vert \E[ \e^{\tau(\theta)'\psi(X_1, \theta)}]\Sigma(\theta)\right\vert_{\det} \right] \right\vert +\frac{1}{T^\eta}\sup_{\theta \in \T}\left\vert  \ln\left[  \left\vert \E[ \e^{\tau(\theta)'\psi(X_1, \theta)}]\Sigma(\theta)\right\vert_{\det} \right] \right\vert   \\
&= & o(1)
\end{eqnarray*}
where the explanations of the last equality are as follows. Under Assumption \ref{Assp:ExistenceConsistency}(a)-(b) and (d)-(h),  by the statement (iii) of  the present lemma $\theta \mapsto   \E[ \e^{\tau(\theta)'\psi(X_1, \theta)}]\Sigma(\theta) $ is continuous in $\T$, which is a compact set by Assumption \ref{Assp:ExistenceConsistency}(d).  Now, continuous functions over compact sets are bounded \citep[e.g.,][Theorem 4.16]{1953Rudin}, so that $\sup_{\theta \in \T}\left\vert  \ln\left[  \left\vert \E[ \e^{\tau(\theta)'\psi(X_1, \theta)}]\Sigma(\theta)\right\vert_{\det} \right]\right\vert$  is bounded, which, in turn,  implies that $\frac{1}{T^\eta}\sup_{\theta \in \T}\left\vert  \ln\left[  \left\vert \E[ \e^{\tau(\theta)'\psi(X_1, \theta)}]\Sigma(\theta)\right\vert_{\det} \right] \right\vert=o(1)$, as $T \rightarrow \infty$. Now the last equality follows from the statement (iv) of the present lemma.\end{proof}

\begin{lem}\label{Lem:TiltedDerivative} Under Assumptions \ref{Assp:ExistenceConsistency}(a)-(b) and (e)-(f),
\begin{enumerate}
\item[(i)] $ \E \left[\sup_{(\theta, \tau) \in \Sbf} \vert \e^{\tau'\psi(X_1, \theta) }\frac{\partial \psi(X_1, \theta)'}{\partial \theta} \vert\right]< \infty$;
\item[(ii)] under additional Assumption  \ref{Assp:ExistenceConsistency}(d)(g) and (h),  $\P$-a.s. as $T \rightarrow \infty$, \\ $\sup_{ (\theta, \tau)\in   \Sbf}\left\vert\left[ \frac{1}{T}\sum_{t=1}^T \e^{\tau'\psi_t(\theta)}\frac{\partial \psi_{t} (\theta)'}{\partial \theta}\right] -\E\left[ \e^{\tau '\psi(X_1, \theta) }\frac{\partial \psi(X_1, \theta)'}{\partial \theta}\right]\right\vert=o(1)$, so that\\ $\sup_{ \theta\in   \T}\left\vert\left[ \frac{1}{T}\sum_{t=1}^T \e^{\tau_T(\theta)'\psi_t(\theta)}\frac{\partial \psi_{t} (\theta)'}{\partial \theta}\right] -\E\left[ \e^{\tau(\theta)'\psi(X_1, \theta) }\frac{\partial \psi(X_1, \theta)'}{\partial \theta}\right]\right\vert=o(1)$; and
\item[(iii)] under additional Assumption  \ref{Assp:ExistenceConsistency}(d)(g) and (h),  $\P$-a.s. as $T \rightarrow \infty$,  \\$\sup_{ \theta\in   \T}\left\vert\left[ \sum_{t=1}^T \hat{w}_{t,\theta}\frac{\partial \psi_{t} (\theta)'}{\partial \theta}\right] -\frac{1}{\E[ \e^{\tau(\theta)'\psi(X_1, \theta)}]}\E \left[\e^{\tau(\theta)'\psi(X_1, \theta) }\frac{\partial \psi(X_1, \theta)'}{\partial \theta}\right]\right\vert=o(1)$
\end{enumerate}
\end{lem}
\begin{proof} \textit{(i)} The supremum of the absolute value of the product is smaller than the product of the suprema of the absolute values. Thus,
\begin{eqnarray}
& & \E \left[\sup_{(\theta, \tau) \in \Sbf}\vert \e^{\tau'\psi(X_1, \theta) }\frac{\partial \psi(X_1, \theta)'}{\partial \theta} \vert\right] \notag\\
& \leqslant & \E \left[\sup_{(\theta, \tau) \in \Sbf}\vert \e^{\tau'\psi(X_1, \theta) }\vert\sup_{(\theta, \tau) \in \Sbf}\left\vert\frac{\partial \psi(X_1, \theta)'}{\partial \theta} \right\vert\right] \notag\\
& \stackrel{(a)}{\leqslant} & \E \left[\sup_{(\theta, \tau) \in \Sbf}\vert \e^{\tau'\psi(X_1, \theta) }\vert^2\right]^{1/2} \E\left[\sup_{\theta \in \T}\left\vert\frac{\partial \psi(X_1, \theta)'}{\partial \theta} \right\vert^2\right]^{1/2}\notag\\
& \stackrel{(b)}{<} & \infty \label{Eq:BoundedDerivativeExpectationTilted}
\end{eqnarray}
\textit{(a)} Firstly,  note that the expression in the second supremum does not depend on $\tau$,  so that $\sup_{(\theta, \tau) \in \Sbf}\left\vert\frac{\partial \psi(X_1, \theta)'}{\partial \theta} \right\vert=\sup_{\theta \in \T}\left\vert\frac{\partial \psi(X_1, \theta)'}{\partial \theta} \right\vert $. Secondly apply the Cauchy-Schwarz inequality. Finally, note that $[\sup_{(\theta, \tau) \in \Sbf}\vert \e^{\tau(\theta)'\psi(X_1, \theta) }\vert]^2 =\sup_{(\theta, \tau) \in \Sbf}\vert \e^{\tau(\theta)'\psi(X_1, \theta) }\vert^{2}$ and $[\sup_{\theta \in \T}\left\vert\frac{\partial \psi(X_1, \theta)'}{\partial \theta} \right\vert]^2=\sup_{\theta \in \T}\left\vert\frac{\partial \psi(X_1, \theta)'}{\partial \theta} \right\vert^2$ because $x \mapsto x^2$ is increasing on $\R_{+}$.
\textit{(b)} Note that $\vert \e^{\tau(\theta)'\psi(X_1, \theta) }\vert^2= \e^{2\tau(\theta)'\psi(X_1, \theta) }$, and then apply Assumption \ref{Assp:ExistenceConsistency}(e) to the first term. Then, application of Assumption \ref{Assp:ExistenceConsistency}(f) to the second term yields the result.

\textit{(ii)} By the triangle inequality, as $T \rightarrow \infty $ $\P$-a.s.,
\begin{eqnarray*}
& & \sup_{ \theta\in   \T}\left\vert\left[ \frac{1}{T}\sum_{t=1}^T \e^{\tau_T(\theta)'\psi_t(\theta)}\frac{\partial \psi_{t} (\theta)'}{\partial \theta}\right] -\E\left[ \e^{\tau(\theta)'\psi(X_1, \theta) }\frac{\partial \psi(X_1, \theta)'}{\partial \theta}\right]\right\vert \\
& \leqslant & \sup_{ \theta\in   \T}\left\vert\left[ \frac{1}{T}\sum_{t=1}^T \e^{\tau_T(\theta)'\psi_t(\theta)}\frac{\partial \psi_{t} (\theta)'}{\partial \theta}\right] -\E \left[ \e^{\tau_{T}(\theta)'\psi(X_1, \theta) }\frac{\partial \psi(X_1, \theta)'}{\partial \theta}\right]\right\vert \\
& & \qquad +\sup_{ \theta\in   \T}\left\vert \E\left[ \e^{\tau_{T}(\theta)'\psi(X_1, \theta) }\frac{\partial \psi(X_1, \theta)'}{\partial \theta}\right]-\E  \left[\e^{\tau(\theta)'\psi(X_1, \theta) }\frac{\partial \psi(X_1, \theta)'}{\partial \theta}\right]\right\vert \\
& = & o(1)
\end{eqnarray*}
where the explanations for the last equality are as follows. Regarding the first supremum,  under Assumptions \ref{Assp:ExistenceConsistency} (a)-(b)(d)(e)(g) and (h), by Lemma \ref{Lem:TauCorrespondence}iii (p. \pageref{Lem:TauCorrespondence}), $\Sbf:=\{(\theta, \tau):\theta \in \T \wedge \tau \in \Tbf(\theta)\}$ is a compact set, so that Assumptions \ref{Assp:ExistenceConsistency}(a)-(b),  the statement   (i) of the present lemma and  the  ULLN (uniform law of large numbers) \`a la Wald  \citep[e.g.,][pp. 24-25, Theorem 1.3.3]{2003GhoRam} imply that,  $\P$-a.s.  as $T \rightarrow \infty$,
\begin{eqnarray*}
& & \sup_{(\theta, \tau) \in \Sbf}\left\vert\left[ \frac{1}{T}\sum_{t=1}^T \e^{\tau'\psi_t(\theta)}\frac{\partial \psi_{t} (\theta)'}{\partial \theta}\right] -\E \left[ \e^{\tau '\psi(X_1, \theta) }\frac{\partial \psi(X_1, \theta)'}{\partial \theta}\right]\right\vert=o(1).
\end{eqnarray*}
Now,  by Assumption \ref{Assp:ExistenceConsistency}(e), for all $\theta \in \T$, $ \tau(\theta) \in  \Tbf(\theta)$, and under Assumption \ref{Assp:ExistenceConsistency}(a)(b),  (d)-(e), (g) and (h), by Lemma \ref{Lem:Schennachtheorem10PfFirstSteps}ii (p. \pageref{Lem:Schennachtheorem10PfFirstSteps}), $\P$-a.s. for $T$ big enough,  for all $\theta \in \T$, $\tau_T(\theta) \in \Tbf(\theta)$. Moreover, under Assumption \ref{Assp:ExistenceConsistency}(a)(b),  (d)-(e), (g) and (h),  by Lemma \ref{Lem:Schennachtheorem10PfFirstSteps}iii (p. \pageref{Lem:Schennachtheorem10PfFirstSteps}), $\sup_{\theta \in \T}\left\vert \tau_T(\theta)- \tau(\theta) \right\vert =o(1)$ $\P$-a.s. as $T \rightarrow \infty$. Thus,   the first supremum is $o(1)$, i.e.,  $ \sup_{ \theta\in   \T}\vert \frac{1}{T}\sum_{t=1}^T \e^{\tau_T(\theta)'\psi_t(\theta)}\frac{\partial \psi_{t} (\theta)'}{\partial \theta} -\E \e^{\tau_{T}(\theta)'\psi(X_1, \theta) }\frac{\partial \psi(X_1, \theta)'}{\partial \theta}\vert=o(1)$, as $T \rightarrow \infty$ $\P$-a.s. Regarding the second supremum, by Assumption \ref{Assp:ExistenceConsistency}(b), $(\theta, \tau) \mapsto  \e^{\tau ' \psi(X_1, \theta)}\frac{\partial \psi(X_1, \theta)'}{\partial \theta}$ is continuous. Moreover under Assumptions \ref{Assp:ExistenceConsistency}(a)-(b), and (e)-(f), by the statement (i) of the present lemma, $ \E \left[\sup_{(\theta, \tau)\in \Sbf }\vert \e^{\tau'\psi(X_1, \theta) }\frac{\partial \psi(X_1, \theta)'}{\partial \theta} \vert\right]< \infty$. Thus, by the Lebesgue dominated convergence theorem and Assumption \ref{Assp:ExistenceConsistency}(b),   $(\theta, \tau) \mapsto \E\left[ \e^{\tau ' \psi(X_1, \theta)}\frac{\partial \psi(X_1, \theta)'}{\partial \theta}\right]$ is also continuous in $\Sbf$. Now, under Assumptions \ref{Assp:ExistenceConsistency} (a)(b)(d)(e)(g) and (h), by Lemma \ref{Lem:TauCorrespondence}iii (p. \pageref{Lem:TauCorrespondence}), $\Sbf$ is compact, so that  $(\theta, \tau) \mapsto \E\left[ \e^{\tau ' \psi(X_1, \theta)}\frac{\partial \psi(X_1, \theta)'}{\partial \theta}\right]$ is uniformly continuous in $\Sbf$ ---continuous functions on compact sets are uniformly continuous \citep[e.g.,][Theorem 4.19]{1953Rudin}. Thus,   under Assumption \ref{Assp:ExistenceConsistency}(a)(b),  (d)-(e), (g) and (h), by Lemma \ref{Lem:Schennachtheorem10PfFirstSteps}iii (p. \pageref{Lem:Schennachtheorem10PfFirstSteps}), which states  that $\sup_{\theta \in \T}\left\vert \tau_T(\theta)- \tau(\theta) \right\vert =o(1)$ $\P$-a.s. as $T \rightarrow \infty$, the second supremum is also $o(1)$  $\P$-a.s. as $T \rightarrow \infty$.

\textit{(iii)} Under Assumptions \ref{Assp:ExistenceConsistency} (a)(b)(d)(e)(g) and (h),  Lemma \ref{Lem:ExpTauPsiStrictlyPositive} (p. \pageref{Lem:ExpTauPsiStrictlyPositive}) yields

\noindent
$0<\inf_{(\theta, \tau)\in \Sbf }\frac{1}{T}\sum_{t=1}^T \e^{\tau'\psi_t(\theta)}$ with $\Prm=\frac{1}{T}\sum_{t=1}^T\delta_{X_t}$, and
$0<\inf_{(\theta, \tau)\in \Sbf }\E[ \e^{\tau'\psi(X_1, \theta)}]$ with $\Prm=\P$.  Consequently, under  Assumption \ref{Assp:ExistenceConsistency}(a)(b),  (d)-(f), (g) and (h), by Lemma \ref{Lem:Schennachtheorem10PfFirstSteps}iii and iv  (p. \pageref{Lem:Schennachtheorem10PfFirstSteps}) and the statement (ii) of the present lemma,  as $T \rightarrow \infty$, $\P$-a.s., uniformly w.r.t. $\theta$
\begin{eqnarray*}
\sum_{t=1}^T \hat{w}_{t,\theta}\frac{\partial \psi_{t} (\theta)'}{\partial \theta} &= &\frac{1}{\frac{1}{T} \sum_{i=1}^T\e^{ \tau_T(\theta) ' \psi_i(\theta)}}\frac{1}{T}\sum_{t=1}^T \frac{}{}\e^{\tau_T(\theta) ' \psi_t(\theta)}\frac{\partial \psi_{t} (\theta)'}{\partial \theta}\\
&\rightarrow  & \frac{1}{\E[ \e^{\tau(\theta)'\psi(X_1, \theta)}]}\E\left[\e^{\tau(\theta)' \psi(X_{1},\theta)}  \frac{\partial \psi(X_{1},\theta)}{\partial \theta }'\right].
\end{eqnarray*}
\end{proof}

\begin{lem}\label{Lem:TiltedVariance} Under Assumptions \ref{Assp:ExistenceConsistency}(a)-(b), (e) and (g),
\begin{enumerate}
\item[(i)] $ \E \left[\sup_{(\theta, \tau) \in \Sbf^{\epsilon}}\vert \e^{\tau'\psi(X_1, \theta) }\psi(X_{1},\theta) \psi(X_{1},\theta)' \vert\right]< \infty$
\item[(ii)] under additional Assumption  \ref{Assp:ExistenceConsistency}(d) and (h), $\P$-a.s. as $T \rightarrow \infty$,\\$\sup_{ (\theta, \tau)\in   \Sbf}\left\vert \frac{1}{T}\sum_{t=1}^T \e^{\tau'\psi_t(\theta)}\psi_t(\theta) \psi_t(\theta)' -\E \e^{\tau'\psi(X_1, \theta) }\psi(X_1,\theta) \psi(X_1,\theta)'\right\vert=o(1)$, so that \\$\sup_{ \theta\in   \T}\left\vert \frac{1}{T}\sum_{t=1}^T \e^{\tau_T(\theta)'\psi_t(\theta)}\psi_t(\theta) \psi_t(\theta)' -\E \e^{\tau(\theta)'\psi(X_1, \theta) }\psi(X_1,\theta) \psi(X_1,\theta)'\right\vert=o(1)$
\item[(iii)] under additional Assumption  \ref{Assp:ExistenceConsistency}(d)(f) and (h),  $\P$-a.s. as $T \rightarrow \infty$, \\$\sup_{ \theta\in   \T}\left\vert \sum_{t=1}^T \hat{w}_{t,\theta}\psi_t(\theta) \psi_t(\theta)' -\frac{1}{\E[ \e^{\tau(\theta)'\psi(X_1, \theta)}]}\E \e^{\tau(\theta)'\psi(X_1, \theta) }\psi(X_1,\theta) \psi(X_1,\theta)'\right\vert=o(1)$.
\end{enumerate}
\end{lem}
\begin{proof} The proof is the same as for Lemma \ref{Lem:TiltedDerivative} with $\psi(X_1,\theta) \psi(X_1,\theta)'  $ and $\psi_t(\theta) \psi_t(\theta)'$ in lieu of $\frac{\partial \psi(X_1, \theta)'}{\partial \theta}$ and $\frac{\partial \psi_{t} (\theta)'}{\partial \theta}$, respectively. For completeness, we provide a proof.

\textit{(i)} The supremum of the absolute value of the product is smaller than the product of the suprema of the absolute values. Thus,
\begin{eqnarray}
& & \E \left[\sup_{(\theta, \tau) \in \Sbf^{\epsilon}}\vert \e^{\tau'\psi(X_1, \theta) }\psi(X_1,\theta) \psi(X_1,\theta)' \vert\right] \notag\\
& \leqslant & \E \left[\sup_{(\theta, \tau) \in \Sbf^{\epsilon}}\vert \e^{\tau'\psi(X_1, \theta) }\vert\sup_{(\theta, \tau) \in \Sbf^{\epsilon}}\left\vert\psi(X_1,\theta) \psi(X_1,\theta)' \right\vert\right] \notag\\
& \stackrel{(a)}{\leqslant} & \E \left[\sup_{(\theta, \tau) \in \Sbf^{\epsilon}}\vert \e^{\tau'\psi(X_1, \theta) }\vert^2\right]^{1/2} \E\left[\sup_{\theta \in \T^{\epsilon}}\left\vert\psi(X_1,\theta) \psi(X_1,\theta)' \right\vert^2\right]^{1/2}\notag\\
& \stackrel{(b)}{<} & \infty. \notag%\label{Eq:BoundedVarianceTilted}
\end{eqnarray}
\textit{(a)} Firstly,  for any $(\theta, \tau) \in \Sbf^\epsilon$, $\theta \in \T^\epsilon$ because, for all  $(\tilde{\tau}, \tilde{\theta})\in  \Sbf$, $\vert \theta- \tilde{\theta} \vert=\sqrt{\sum_{k=1}^m(\theta_k -\tilde{\theta}_k)^2}\leqslant\sqrt{\sum_{k=1}^m(\theta_k -\tilde{\theta}_k)^2+\sum_{k=1}^m(\tau_k -\tilde{\tau}_k)^2}=\vert (\theta, \tau)-(\tilde{\tau}, \tilde{\theta})\vert< \epsilon$. Thus, as  the second supremum does not depend on $\tau$,  $\sup_{(\theta, \tau) \in \Sbf^{\epsilon}}\left\vert\psi(X_1,\theta) \psi(X_1,\theta)' \right\vert \leqslant \sup_{\theta \in \T^{\epsilon}}\left\vert\psi(X_1,\theta) \psi(X_1,\theta)' \right\vert $. Secondly apply the Cauchy-Schwarz inequality. Finally, $[\sup_{(\theta, \tau) \in \Sbf^{\epsilon}}\vert \e^{\tau(\theta)'\psi(X_1, \theta) }\vert]^2 =\sup_{(\theta, \tau) \in \Sbf^{\epsilon}}\vert \e^{\tau(\theta)'\psi(X_1, \theta) }\vert^{2}$ and $[\sup_{\theta \in \T^{\epsilon}}\left\vert\psi(X_1,\theta) \psi(X_1,\theta)' \right\vert]^2=\sup_{\theta \in \T^{\epsilon}}\left\vert\psi(X_1,\theta) \psi(X_1,\theta)' \right\vert^2$ because $x \mapsto x^2$ is increasing on $\R_{+}$.
\textit{(b)} Note that $\vert \e^{\tau(\theta)'\psi(X_1, \theta) }\vert^2= \e^{2\tau(\theta)'\psi(X_1, \theta) }$, and then apply Assumption \ref{Assp:ExistenceConsistency}(e) to the first term. Then, application of Assumption \ref{Assp:ExistenceConsistency} (g) to the second term yields the result.

\textit{(ii)} By the triangle inequality,  $\P$-a.s. as $T \rightarrow \infty $,
\begin{eqnarray*}
& & \sup_{ \theta\in   \T}\left\vert \left[\frac{1}{T}\sum_{t=1}^T \e^{\tau_T(\theta)'\psi_t(\theta)}\psi_t(\theta) \psi_t(\theta)' \right]-\E\left[ \e^{\tau(\theta)'\psi(X_1, \theta) }\psi(X_1,\theta) \psi(X_1,\theta)'\right]\right\vert \\
& \leqslant & \sup_{ \theta\in   \T}\left\vert \left[ \frac{1}{T}\sum_{t=1}^T \e^{\tau_T(\theta)'\psi_t(\theta)}\psi_t(\theta) \psi_t(\theta)' \right]-\E\left[ \e^{\tau_{T}(\theta)'\psi(X_1, \theta) }\psi(X_1,\theta) \psi(X_1,\theta)'\right]\right\vert \\
& & \qquad +\sup_{ \theta\in   \T}\left\vert \E \left[\e^{\tau_{T}(\theta)'\psi(X_1, \theta) }\psi(X_1,\theta) \psi(X_1,\theta)'\right]-\E\left[ \e^{\tau(\theta)'\psi(X_1, \theta) }\psi(X_1,\theta) \psi(X_1,\theta)'\right]\right\vert \\
& = & o(1)
\end{eqnarray*}
where the explanations for the last equality are as follows. Regarding the first supremum, under Assumptions \ref{Assp:ExistenceConsistency} (a)-(b)(d)(e)(g) and (h), by Lemma \ref{Lem:TauCorrespondence}iii (p. \pageref{Lem:TauCorrespondence}), $\Sbf:=\{(\theta, \tau):\theta \in \T \wedge \tau \in \Tbf(\theta)\}$ is a compact set, so that  Assumption \ref{Assp:ExistenceConsistency}(a)-(b), the statement  (i) of the present lemma and  the  ULLN (uniform law of large numbers) \`a la Wald yields that   \citep[e.g.,][pp. 24-25, Theorem 1.3.3]{2003GhoRam},  $\P$-a.s. as $T \rightarrow \infty$,
\begin{eqnarray*}
& & \sup_{(\theta, \tau) \in \Sbf}\left\vert \left[\frac{1}{T}\sum_{t=1}^T \e^{\tau'\psi_t(\theta)}\psi_t(\theta) \psi_t(\theta)' \right]-\E\left[ \e^{\tau '\psi(X_1, \theta) }\psi(X_1,\theta) \psi(X_1,\theta)'\right]\right\vert=o(1).
\end{eqnarray*}
Now, by Assumption \ref{Assp:ExistenceConsistency}(e), for all $\theta \in \T$, $ \tau(\theta) \in  \Tbf(\theta)$, and under Assumption \ref{Assp:ExistenceConsistency}(a)(b),  (d)-(e), (g) and (h), by Lemma \ref{Lem:Schennachtheorem10PfFirstSteps}ii (p. \pageref{Lem:Schennachtheorem10PfFirstSteps}), $\P$-a.s. for $T$ big enough,  for all $\theta \in \T$, $\tau_T(\theta) \in \Tbf(\theta)$. Moreover, under Assumption \ref{Assp:ExistenceConsistency}(a)(b),  (d)-(e), (g) and (h),  by Lemma \ref{Lem:Schennachtheorem10PfFirstSteps}iii (p. \pageref{Lem:Schennachtheorem10PfFirstSteps}), $\sup_{\theta \in \T}\left\vert \tau_T(\theta)- \tau(\theta) \right\vert =o(1)$ $\P$-a.s. as $T \rightarrow \infty$. Thus, the first supremum is $o(1)$, i.e.,  $ \sup_{ \theta\in   \T}\vert \frac{1}{T}\sum_{t=1}^T \e^{\tau_T(\theta)'\psi_t(\theta)}\psi_t(\theta) \psi_t(\theta)' -\E \e^{\tau_{T}(\theta)'\psi(X_1, \theta) }\psi(X_1,\theta) \psi(X_1,\theta)'\vert=o(1)$, as $T \rightarrow \infty$ $\P$-a.s. Regarding the second supremum, by Assumption \ref{Assp:ExistenceConsistency}(b), $(\theta, \tau) \mapsto  \e^{\tau ' \psi(X_1, \theta)}\psi(X_1,\theta) \psi(X_1,\theta)'$ is continuous in $\Sbf$. Moreover under Assumptions \ref{Assp:ExistenceConsistency}(a)-(b), (e) and (g), by the statement (i) of the present lemma, $ \E \left[\sup_{(\theta, \tau)\in \Sbf }\vert \e^{\tau'\psi(X_1, \theta) }\psi(X_1,\theta) \psi(X_1,\theta)' \vert\right]< \infty$. Thus, by the Lebesgue dominated convergence theorem and Assumption \ref{Assp:ExistenceConsistency}(b),   $(\theta, \tau) \mapsto \E\left[ \e^{\tau ' \psi(X_1, \theta)}\psi(X_1,\theta) \psi(X_1,\theta)'\right]$ is also continuous. Now, under Assumptions \ref{Assp:ExistenceConsistency} (a)(b)(d)(e)(g) and (h), by Lemma \ref{Lem:TauCorrespondence}iii (p. \pageref{Lem:TauCorrespondence}), $\Sbf$ is compact, so that   $(\theta, \tau) \mapsto \E\left[ \e^{\tau ' \psi(X_1, \theta)}\psi(X_1,\theta) \psi(X_1,\theta)'\right]$ is uniformly continuous ---continuous functions on compact sets are uniformly continuous \citep[e.g.,][Theorem 4.19]{1953Rudin}. Thus,   under Assumption \ref{Assp:ExistenceConsistency}(a)(b),  (d)-(e), (g) and (h), by Lemma \ref{Lem:Schennachtheorem10PfFirstSteps}iii (p. \pageref{Lem:Schennachtheorem10PfFirstSteps}), which states  that $\sup_{\theta \in \T}\left\vert \tau_T(\theta)- \tau(\theta) \right\vert =o(1)$ $\P$-a.s. as $T \rightarrow \infty$, the second supremum is also $o(1)$  $\P$-a.s. as $T \rightarrow \infty$.

\textit{(iii)} Under Assumptions \ref{Assp:ExistenceConsistency} (a)(b)(d)(e)(g) and (h),  Lemma \ref{Lem:ExpTauPsiStrictlyPositive} (p. \pageref{Lem:ExpTauPsiStrictlyPositive}) yields

\noindent
$0<\inf_{(\theta, \tau)\in \Sbf }\frac{1}{T}\sum_{t=1}^T \e^{\tau'\psi_t(\theta)}$ with $\Prm=\frac{1}{T}\sum_{t=1}^T\delta_{X_t}$, and
$0<\inf_{(\theta, \tau)\in \Sbf }\E[ \e^{\tau'\psi(X_1, \theta)}]$ with $\Prm=\P$.  Consequently, under  Assumption \ref{Assp:ExistenceConsistency}(a)(b),  (d)-(e), (g) and (h), by Lemma \ref{Lem:Schennachtheorem10PfFirstSteps}iii and iv (p. \pageref{Lem:Schennachtheorem10PfFirstSteps}) and the statement (ii) of the present lemma,  as $T \rightarrow \infty$, $\P$-a.s., uniformly w.r.t. $\theta$,
\begin{eqnarray*}
\sum_{t=1}^T \hat{w}_{t,\theta}\psi_t(\theta)\psi_t(\theta)'& = &  \frac{1}{\frac{1}{T} \sum_{i=1}^T\e^{ \tau_T(\theta) ' \psi_i(\theta)}}\frac{1}{T}\sum_{t=1}^T \frac{}{}\e^{\tau_T(\theta) ' \psi_t(\theta)}\psi_t(\theta)\psi_t(\theta)'\\
& \rightarrow & \frac{1}{\E[ \e^{\tau(\theta)'\psi(X_1, \theta)}]} \E \left[ \e^{\tau(\theta)'\psi(X_1, \theta) }\psi(X_1,\theta) \psi(X_1,\theta)'\right].
\end{eqnarray*}

\end{proof}

\begin{lem} \label{Lem:BoundAsTiltingEquation} Under Assumptions \ref{Assp:ExistenceConsistency}(a)(b)(g),
\begin{itemize}
\item[(i)] $\E\left[\sup_{\theta \in \T^{\epsilon}}\left\vert\psi(X_1,\theta)\right\vert^4\right]< \infty$, so that $\E\left[\sup_{\theta \in \T^{\epsilon}}\left\vert\psi(X_1,\theta)\right\vert^2\right]< \infty$; and
\item[(ii)] under additional Assumption \ref{Assp:ExistenceConsistency}(e), $ \E \left[\sup_{(\theta, \tau) \in \Sbf^{\epsilon}} \vert \e^{\tau'\psi(X_1, \theta) }\psi(X_{1},\theta) \vert\right]< \infty$
\end{itemize}
\end{lem}
\begin{proof} \textit{(i)} Put $\psi(X_1, \theta)=:(\psi_{1}(X_1, \theta)\; \psi_{2}(X_1, \theta) \; \cdots \; \psi_{m}(X_1, \theta))'  $. Note that   $\sup_{\theta \in \T^{\epsilon}}\left\vert\psi(X_1,\theta)\right\vert^{4}=[\sup_{\theta \in \T^{\epsilon}}\left\vert\psi(X_1,\theta)\vert^2\right]^{2}$ because $x \mapsto x^2$ is an increasing function. Thus, by the Cauchy-Schwarz inequality, $\E\left[\sup_{\theta \in \T^{\epsilon}}\left\vert\psi(X_1,\theta)\right\vert^2\right] \leqslant \sqrt{ \E\left\{\left[\sup_{\theta \in \T^{\epsilon}}\left\vert\psi(X_1,\theta)\right\vert^2\right]^2\right\}}  \stackrel{}{=}  \sqrt{ \E\left\{\sup_{\theta \in \T^{\epsilon}}\left\vert\psi(X_1,\theta)\right\vert^4\right\}}$, so that it remains to show the first part of the statement. On one hand,
by the definition of the Euclidean norm,
\begin{eqnarray}
\sqrt{ \E\left\{\sup_{\theta \in \T^{\epsilon}}\left\vert\psi(X_1,\theta)\right\vert^4\right\}} %
& \stackrel{}{=} & \sqrt{ \E\left\{\sup_{\theta \in \T^{\epsilon}}\left[\left(\sum_{k=1}^m\psi_k(X_1,\theta)^2\right)^2\right]\right\}} \notag\\
& \stackrel{}{\leqslant} &  \sqrt{m\E\left\{\sup_{\theta \in \T^{\epsilon}}\left[\sum_{k=1}^m\psi_k(X_1,\theta)^4\right]\right\}} \label{Eq:PsiSquaredBound}
\end{eqnarray}
where the explanation for the last inequality is as follows. By the Jensen's inequality, 

\noindent
$\left(\frac{1}{m}\sum_{k=1}^m a_k\right)^2\leqslant\frac{1}{m}\sum_{k=1}^m a_k^2 $, so that $\left(\sum_{k=1}^m a_k\right)^2\leqslant m\sum_{k=1}^m a_k^2 $. Apply the later inequality with $\psi_k(X_1, \theta)^{2}=a_{k}$.

% \begin{eqnarray*}
% & &\E\left[\sup_{\theta \in \T^{\epsilon}}\left\vert\psi(X_1,\theta)\right\vert^2\right] %= \E\left\{\sup_{\theta \in \T^{\epsilon}}\left[\sum_{k=1}^m\psi_k(X_1,\theta)^2\right]\right\}
% \end{eqnarray*}

On the other hand,
\begin{eqnarray*}
& & \E\left[\sup_{\theta \in \T^{\epsilon}}\left\vert\psi(X_1,\theta) \psi(X_1,\theta)' \right\vert^2\right] \\
& = &\E\left[\sup_{\theta \in \T^{\epsilon}}\left\vert\begin{pmatrix}\psi_{1}(X_1,\theta)^{2} & \psi_{1}(X_1,\theta)\psi_{2}(X_1,\theta) & \cdots & \psi_{1}(X_1,\theta)\psi_{m}(X_1,\theta) \\
\psi_{2}(X_1,\theta)\psi_{1}(X_1,\theta) & \psi_{2}(X_1,\theta)^{2}  & \cdots & \psi_{2}(X_1,\theta)\psi_{m}(X_1,\theta) \\
\vdots & \vdots & \ddots  & \vdots \\
\psi_{m}(X_1,\theta)\psi_{1}(X_1,\theta) & \psi_{m}(X_1,\theta)\psi_{2}(X_1,\theta) & \cdots & \psi_{m}(X_1,\theta)^2 \\
\end{pmatrix} \right\vert^2\right]\\
& = & \E\left\{\sup_{\theta \in \T^{\epsilon}}\left[\sum_{(i,j)\in \ldsb 1,m\rdsb^2}\left[\psi_i(X_1,\theta)\psi_j(X_1,\theta)\right]^2\right]\right\}\\
& = & \E\left\{\sup_{\theta \in \T^{\epsilon}}\left[\sum_{k=1}^m\psi_k(X_1,\theta)^4+\sum_{(i,j)\in \ldsb 1,m\rdsb^2:i\neq j}\left[\psi_i(X_1,\theta)\psi_j(X_1,\theta)\right]^2\right]\right\}
\end{eqnarray*}

Therefore, $\sum_{k=1}^m\psi_k(X_1,\theta)^4\leqslant \sum_{k=1}^m\psi_k(X_1,\theta)^4+\sum_{(i,j)\in \ldsb 1,m\rdsb^2:i\neq j}\left[\psi_i(X_1,\theta)\psi_j(X_1,\theta)\right]^2$, 

the later equality and inequality \eqref{Eq:PsiSquaredBound} yield
\begin{eqnarray*}
\E\left[\sup_{\theta \in \T^{\epsilon}}\left\vert\psi(X_1,\theta)\right\vert^2\right] & \leqslant & \sqrt{m\E\left[\sup_{\theta \in \T^{\epsilon}}\left\vert\psi(X_1,\theta) \psi(X_1,\theta)' \right\vert^2\right]}\\
& < &  \infty
\end{eqnarray*}
where the last inequality follows from Assumption \ref{Assp:ExistenceConsistency}(g).

\textit{(ii)} The supremum of the absolute value of the product is smaller than the product of the suprema of the absolute values. Thus,
\begin{eqnarray}
& & \E \left[\sup_{(\theta, \tau) \in \Sbf^{\epsilon}} \vert \e^{\tau'\psi(X_1, \theta) }\psi(X_1,\theta) \vert\right] \notag\\
& \leqslant & \E \left[\sup_{(\theta, \tau) \in \Sbf^{\epsilon}} \vert \e^{\tau'\psi(X_1, \theta) }\vert\sup_{(\theta, \tau) \in \Sbf^{\epsilon}} \left\vert\psi(X_1,\theta) \right\vert\right] \notag\\
& \stackrel{(a)}{\leqslant} & \E \left[\sup_{(\theta, \tau) \in \Sbf^{\epsilon}} \vert \e^{\tau'\psi(X_1, \theta) }\vert^2\right]^{1/2} \E\left[\sup_{\theta \in \T^{\epsilon}}\vert\psi(X_1,\theta)\vert^2\right]^{1/2}\notag\\
& \stackrel{(b)}{<} & \infty 
\end{eqnarray}
\textit{(a)} Firstly, for any $(\theta, \tau) \in \Sbf^\epsilon$, $\theta \in \T^\epsilon$ because, for all  $(\tilde{\tau}, \tilde{\theta})\in  \Sbf$, $\vert \theta- \tilde{\theta} \vert=\sqrt{\sum_{k=1}^m(\theta_k -\tilde{\theta}_k)^2}\leqslant\sqrt{\sum_{k=1}^m(\theta_k -\tilde{\theta}_k)^2+\sum_{k=1}^m(\tau_k -\tilde{\tau}_k)^2}=\vert (\theta, \tau)-(\tilde{\tau}, \tilde{\theta})\vert< \epsilon$. Thus,  as the expression in the second supremum does not depend on $\tau$, $\sup_{(\theta, \tau) \in \Sbf^{\epsilon}} \left\vert\psi(X_1,\theta) \psi(X_1,\theta)' \right\vert \leqslant \sup_{\theta \in \T^{\epsilon}}\left\vert\psi(X_1,\theta) \psi(X_1,\theta)' \right\vert $. Secondly apply the Cauchy-Schwarz inequality. Finally, note that $[\sup_{(\theta, \tau) \in \Sbf^{\epsilon}} \vert \e^{\tau(\theta)'\psi(X_1, \theta) }\vert]^2 =\sup_{(\theta, \tau) \in \Sbf^{\epsilon}} \vert \e^{\tau(\theta)'\psi(X_1, \theta) }\vert^{2}$ and $[\sup_{\theta \in \T^{\epsilon}}\left\vert\psi(X_1,\theta)\right\vert]^2=\sup_{\theta \in \T^{\epsilon}}\left\vert\psi(X_1,\theta)\right\vert^2$ because $x \mapsto x^2$ is increasing on $\R_{+}$.
\textit{(b)} Note that $\vert \e^{\tau(\theta)'\psi(X_1, \theta) }\vert^2= \e^{2\tau(\theta)'\psi(X_1, \theta) }$, and then apply Assumption \ref{Assp:ExistenceConsistency}(e) to the first term. Then, application of the statement (i) of the present lemma to the second term yields the result.
\end{proof}
\begin{rk}

The first step of the proof shows that even the fourth moment is uniformly bounded. \hfill $\diamond$ \end{rk}

\begin{lem}[Implicit function $\tau(.)$] \label{Lem:AsTiltingFct}% Old Lem:AsTiltingFct
 Under Assumption \ref{Assp:ExistenceConsistency} (a)(b)(e)(g) and (h),
\begin{itemize}
\item[(i)] for all $\theta\in  \T $, $\tau \mapsto \E\left[\e^{\tau'\psi(X_{1},\theta)} \right]$ is a strictly convex function s.t. $\frac{\partial \E \left[\e^{\tau'\psi(X_{1},\theta)} \right ]}{\partial \tau}=\E \left[\e^{\tau'\psi(X_{1},\theta)}\psi(X_{1},\theta) \right ]$;
\item[(ii)] under additional Assumption \ref{Assp:ExistenceConsistency}(d), for all $\theta\in  \T $, there exists a unique $\tau(\theta) $ such that $\E\left[\e^{\tau(\theta)'\psi(X_{1},\theta)}\psi(X_{1},\theta) \right]=0 $; and
\item[(iii)]  under additional Assumption \ref{Assp:ExistenceConsistency}(d),
$\tau:\T \rightarrow \R^m $ is continuous; and
\item[(iv)]  under additional Assumption \ref{Assp:ExistenceConsistency}(c) and (d), for all $ \theta \in \T\setminus\{\theta_0\}$, $\E\left[\e^{\tau(\theta)'\psi(X_{1},\theta)} \right]<\E\left[\e^{\tau(\theta_0)'\psi(X_{1},\theta_{0})} \right]=1 $ where $\tau(\theta_0)=0_{m \times 1}$.
\end{itemize}
\end{lem}
\begin{proof} \textit{(i)}  Under Assumption \ref{Assp:ExistenceConsistency}(a) and (b), by  the Cauchy-Schwarz inequality,

\noindent
 $\E \left[\sup_{(\theta, \tau) \in \Sbf^{\epsilon}} \e^{\tau'\psi(X_1, \theta) } \right] \leqslant \E \left[\sup_{(\theta, \tau) \in \Sbf^{\epsilon}}\e^{2\tau'\psi(X_1, \theta) } \right]^{1/2}$, which is finite  by Assumption \ref{Assp:ExistenceConsistency}(e). Now, by Assumption \ref{Assp:ExistenceConsistency}(e), for all $\dot{\theta} \in \T$, $\tau(\dot \theta)\in \intr[\Tbf(\dot{\theta})]$. Then,  by a standard result on Laplace's transform \citep[e.g.,][Theorems  3 on p. 183]{1996Mon},   $\tau \mapsto \E \left[\e^{\tau'\psi(X_{1},\dot\theta)} \right ]$ is $C^{\infty} $ in a neighborhood of $\tau(\dot\theta)$,  and   $\tau \mapsto \frac{\partial \E \left[\e^{\tau'\psi(X_{1},\dot\theta)} \right ]}{\partial \tau}=\E \left[\e^{\tau'\psi(X_{1},\dot\theta)}\psi(X_{1},\theta) \right ]$ and $\tau \mapsto \frac{\partial^2 \E \left[\e^{\tau'\psi(X_{1},\dot\theta)} \right ]}{\partial \tau\partial \tau'}=\E\left[\e^{\tau' \psi(X_{1},\dot\theta)}\psi(X_{1},\dot\theta) \psi(X_{1},\dot\theta)' \right]$.  Moreover, under Assumptions \ref{Assp:ExistenceConsistency}(a)-(b), (e) and (g), Assumption \ref{Assp:ExistenceConsistency}(h) implies that,

\noindent
$\E\left[\e^{\tau' \psi(X_{1},\dot\theta)}\psi(X_{1},\dot\theta) \psi(X_{1},\dot\theta)' \right] $  is a symmetric positive-definite matrix because   a well-defined covariance matrix is invertible iff it is invertible under an equivalent probability measure (Lemma \ref{Lem:ChgOfMeasureInvertibilityPDP} and Corollary \ref{Cor:ChgOfMeasureInvertibilityPDP}i on p. \pageref{Cor:ChgOfMeasureInvertibilityPDP}).

\textit{(ii)}  Assumption \ref{Assp:ExistenceConsistency}(d) ensures existence, while the statement (i) of the present lemma ensures that $\tau(\theta)$ is the solution of a strictly convex problem, so that it is unique.

\textit{(iii)} Note that, under our assumptions,  an application  of the standard implicit function \citep[e.g.,][Theorem 9.28]{1953Rudin} is not directly possible as it  requires $(\theta, \tau)\mapsto \E\left[\e^{\tau'\psi(X_{1},\theta)}\psi(X_{1},\theta) \right]$ to be continuously differentiable in $\Sbf^\epsilon$, which, in turn, typically requires to uniformly bound the  derivative of the latter in $\Sbf^\epsilon$   \citep[e.g.,][Theorem 9.31]{1994Dav}.
Thus, we  apply  the sufficiency part of Kumagai's implicit function theorem \citep{1980Kum}.  Check its assumptions.  Firstly, under  Assumptions \ref{Assp:ExistenceConsistency}(a)(b)(e) and (g), by Lemma \ref{Lem:BoundAsTiltingEquation}ii (p. \pageref{Lem:BoundAsTiltingEquation}) and the  Lebesgue dominated convergence theorem,     $(\theta , \tau) \mapsto \E \left[\e^{\tau'\psi(X_{1},\theta)}\psi(X_{1},\theta) \right] $ is continuous in $\Sbf^\epsilon$, i.e., in an open neighborhood of every $(\theta, \tau) \in \Sbf$.  Secondly, by the inverse function theorem applied to $\tau \mapsto\E \left[\e^{\tau'\psi(X_{1},\theta)}\psi(X_{1},\theta) \right]$ \citep[e.g.,][Theorem 9.24]{1953Rudin},  for all $\theta \in \T^\epsilon $, $\tau \mapsto\E \left[\e^{\tau'\psi(X_{1},\theta)}\psi(X_{1},\theta) \right]$ is locally one-to-one\,:\footnote{Here it is necessary to work in an $\epsilon$-neighborhood of $\T$ in order to satisfy the assumption of Kumagai's implicit function theorem \citep{1980Kum}. The standard implicit function theorem would also require the existence of open neighborhoods around the parameter values at which the function is zero.} As explained in the proof of (i),
  under Assumption \ref{Assp:ExistenceConsistency}(a)(b)(e) and (h), $\tau \mapsto\E \left[\e^{\tau'\psi(X_{1},\theta)}\psi(X_{1},\theta) \right]$ is continuously differentiable and, under Assumption \ref{Assp:ExistenceConsistency}(a)(b)(e)(g) and (h), for all $\theta \in \T $,  $\frac{ \partial \E \left[\e^{\tau'\psi(X_{1},\theta)}\psi(X_{1},\theta) \right ]}{\partial \tau'}=\E\left[\e^{\tau' \psi(X_{1},\theta)}\psi(X_{1},\theta) \psi(X_{1},\theta)' \right] $  is invertible, so that the assumptions of the inverse function theorem are valid.

\textit{(iv)} By the statements (i) and (ii) of the present lemma, for all $\theta\in \T$,  for all $\tau\neq \tau(\theta)$, $\E [\e^{\tau(\theta)'\psi(X_1,\theta)}]< \E[\e^{\tau'\psi(X_1,\theta)}]$. Now, for all $ \theta \in \T\setminus\{\theta_0\}$, $\tau(\theta)\neq 0_{m \times 1} $: If there existed  $ \dot{\theta} \in \T\setminus\{\theta_0\}$  s.t. $\tau(\dot{\theta})=0_{m \times 1}$, then $0=\E [\e^{\tau(\dot{\theta})'\psi(X_1,\dot{\theta})}\psi(X_1,\dot{\theta})]=\E [\psi(X_1,\dot{\theta})] $, which would contradict Assumption \ref{Assp:ExistenceConsistency}(c). Thus, for all $ \theta \in \T\setminus\{\theta_0\}$, $\E [\e^{\tau(\theta)'\psi(X_1,\theta)}]< \E[\e^{0_{1 \times m}\psi(X_1,\theta)}]=1$. Then, the result follows by the statement (ii) of the present lemma because $0_{m \times 1}=\E[\psi(X_1, \theta_0)]=\E[\e^{0_{1 \times m}\psi(X_1,\theta)}\psi(X_1, \theta_0)]$.
\end{proof}

\subsection{Decomposition and derivatives of the log-ESP $L_T(., .)$}\label{Sec:LTAndDerivatives}

In this section, we simplify $L_T(\theta, \tau)$ and  study its derivatives. Such results are needed for the proof of Theorem \ref{theorem:ConsistencyAsymptoticNormality}ii and other results afterwards.

\begin{lem}\label{Lem:LogESPExistenceInS} Under Assumption \ref{Assp:ExistenceConsistency}(a)-(e) and (g)(h), by Lemma \ref{Lem:AsTiltingFct} (p. \pageref{Lem:AsTiltingFct}), define $\tau(\theta_0)=\tau_0=0_{m\times 1}$.  Under Assumption \ref{Assp:ExistenceConsistency}(a)-(b),   (e) and (h),
\begin{enumerate}
\item[(i)] under additional Assumption  \ref{Assp:ExistenceConsistency}(d) and (g), there exist $(\underline{M}_{\e}, \overline{M}_{\e})\in \R_+\setminus \{ 0\}$ s.t.   $\P$-a.s. for $T$ big enough, $\underline{M}_{\e}<\inf_{(\theta, \tau)\in \Sbf}\frac{1}{T}\sum_{t=1}^T\e^{\tau'\psi_{t}(\theta)}$ and $\sup_{(\theta, \tau)\in \Sbf}\frac{1}{T}\sum_{t=1}^T\e^{\tau'\psi_{t}(\theta)} < \overline{M}_{\e} $;
\item[(ii)] under additional Assumption \ref{Assp:ExistenceConsistency} (c)(d) and (g), there exists an open ball $B_{r}(\theta_0, \tau_0)$ centered at $(\theta_0, \tau_0)$ of radius $r>0$, which is a subset of $\Sbf$;

\item[(iii)] under additional Assumption  \ref{Assp:ExistenceConsistency}(c)(d)(f) and (g), for all $(\theta, \tau)$ in  a closed ball $\overline{B_{r_\partial}(\theta_0, \tau_0)}\subset \Sbf$ centered at $(\theta_0, \tau_0)$ with  radius $r_\partial >0$,  $\vert \E \e^{\tau'\psi(X_1, \theta)}\frac{\partial \psi(X_1, \theta)}{\partial \theta'}\vert_{\det}^2 > 0$, so that, $\P$-a.s. for $T$ big enough, $\vert\frac{1}{T}\sum_{t=1}^T \e^{\tau'\psi_t(\theta)}\frac{\partial \psi_t( \theta)}{\partial \theta'} \vert_{\det}^2>0$;
\item[(iv)]  under additional Assumption  \ref{Assp:ExistenceConsistency}(g),  $\P$-a.s. for $T$ big enough, 

$\inf_{(\theta, \tau)\in \Sbf}\negthickspace\vert \frac{1}{T}\negthickspace\sum_{t=1}^T\negthickspace\e^{\tau'\psi_t(\theta)}\psi_t(\theta)\psi_t(\theta)' \vert_{\det} >0$.
\end{enumerate}
\end{lem}
\begin{proof} \textit{(i)} Under Assumption \ref{Assp:ExistenceConsistency}(a)(b)(d)(e)(g) and (h), by Lemma \ref{Lem:Schennachtheorem10PfFirstSteps}i (p. \pageref{Lem:Schennachtheorem10PfFirstSteps}), which states that, $\P\text{-a.s.}$ as $T \rightarrow \infty$, $ \sup_{(\theta, \tau) \in\Sbf}\left\vert\frac{1}{T}\sum_{t=1}^T \e^{\tau'\psi_t(\theta)}- \E[ \e^{\tau'\psi(X_1, \theta)}] \right\vert =o(1)$, and Lemma \ref{Lem:ExpTauPsiStrictlyPositive} (p. \pageref{Lem:ExpTauPsiStrictlyPositive}) with $\Prm=\P$, which states that $0<\inf_{(\theta, \tau)\in \Sbf }\E[ \e^{\tau'\psi(X_1, \theta)}]$,   the result follows.

\textit{(ii)} First of all, note that the result is not completely immediate, as  $\Sbf:=\{(\theta, \tau):\theta \in \T \wedge \tau \in \overline{B_{\epsilon_{\Tbf}}(\tau(\theta))}\}$   is not a Cartesian product. Under Assumption \ref{Assp:ExistenceConsistency} (a)(b)(d)(e)(g) and (h), by Lemma \ref{Lem:AsTiltingFct}iii (p. \pageref{Lem:AsTiltingFct}), $\tau: \T \rightarrow \R^m $ is continuous. Thus, by the topological definition of continuity, $\tau^{-1}[B_{\epsilon_\Tbf/2}(\tau_0)]$ is an open set of $\T$. Moreover, by the definition of $\tau_0$, $\theta_0 \in\tau^{-1}[B_{\epsilon_\Tbf/2}(\tau_0)] $, and, by Assumption \ref{Assp:ExistenceConsistency}(c), $\theta_0 \in \intr(\T) $,\footnote{This assumption forbids  $\theta_0$ to be on the boundary of  $\tau^{-1}[B_{\epsilon_\Tbf/2}(\tau_0)]$, which is an open set of $\T$, but not necessarily of $\R^m$.} so that there exists $r_0>0$ s.t. $B_{r_0}(\theta_0)\subset \tau^{-1}[B_{\epsilon_\Tbf/2}(\tau_0)]$ and $B_{r_0}(\theta_0)\subset \T$.
 Now, for this proof, put $r=\min\{ r_{0}, \epsilon_\Tbf/2\}$. Then, it remains to show that $B_{r}(\theta_0, \tau_0)\subset \Sbf $, i.e., for all $(\dot{\theta}, \dot{\tau})\in B_{r}(\theta_0, \tau_0) $, $\vert \dot{\tau}-\tau(\dot{\theta})\vert \leqslant\epsilon_{\Tbf}$.  By the triangle inequality, for  any $(\dot{\theta}, \dot{\tau})\in B_{r}(\theta_0, \tau_0)$,
\begin{eqnarray*}
\vert \dot{\tau}-\tau(\dot{\theta})\vert &\leqslant& \vert \dot{\tau}-\tau_0 \vert + \vert \tau_0-\tau(\dot{\theta}) \vert\\
& \leqslant & \frac{\epsilon_{\Tbf}}{2}+\frac{\epsilon_{\Tbf}}{2}=\epsilon_{\Tbf}
\end{eqnarray*}
 where the explanations for the last inequality are as follows. Firstly, $\vert \dot{\tau}-\tau_0 \vert<\sqrt{ \sum_{k=1}^m (\dot{\tau}_k-\tau_{0,k})^2}\leqslant \sqrt{\sum_{k=1}^m (\dot{\theta}_k-\theta_{0,k})^2+\sum_{k=1}^m (\dot{\tau}_k-\tau_{0,k})^2}<r\leqslant\frac{\epsilon_{\Tbf}}{2}$  by definition of $r$. Secondly, and similarly,$\vert \dot{\theta}-\theta_0 \vert<\sqrt{ \sum_{k=1}^m (\dot{\theta}_k-\theta_{0,k})^2}\leqslant \sqrt{\sum_{k=1}^m (\dot{\theta}_k-\theta_{0,k})^2+\sum_{k=1}^m (\dot{\tau}_k-\tau_{0,k})^2}< r \leqslant r_0$, so that $\vert\tau_0-  \tau(\dot{\theta})\vert< \frac{\epsilon_{\Tbf}}{2}$ because $B_{r_0}(\theta_0)\subset \tau^{-1}[B_{\epsilon_\Tbf/2}(\tau_0)]$.

\textit{(iii)}
Under Assumption  \ref{Assp:ExistenceConsistency} (a)-(b) and (e)-(f), by Lemma \ref{Lem:TiltedDerivative}i (p. \pageref{Lem:TiltedDerivative}), Assumption \ref{Assp:ExistenceConsistency}(b)   and the Lebesgue dominated convergence theorem, $(\theta, \tau)\mapsto \E \left[ \e^{\tau'\psi(X_1, \theta)}\frac{\partial \psi(X_1, \theta)}{\partial \theta'}\right]$ is continuous in $\Sbf$, and thus in a neighborhood of $(\theta_0, \tau_0)$ in $\Sbf$ by Assumption \ref{Assp:ExistenceConsistency}(c) and (e). Then,  $(\theta, \tau)\mapsto \vert \E \left[ \e^{\tau'\psi(X_1, \theta)}\frac{\partial \psi(X_1, \theta)}{\partial \theta'}\right]\vert_{\det}^2$ is also continuous. Now, by  Assumption \ref{Assp:ExistenceConsistency}(h),

\noindent
$\vert \E[ \e^{\tau(\theta_0)'\psi(X_1, \theta_0)}\frac{\partial \psi(X_1, \theta_{0})}{\partial \theta'}]\vert_{\det}^2>0$, so that, under Assumption \ref{Assp:ExistenceConsistency}(a)-(e) and (g)-(h), by the statement (ii) of the present lemma, there exists a closed ball $\overline{B_{r_\partial}(\theta_0, \tau_0)}\subset \Sbf$ centered at $(\theta_0, \tau_0)$ with  radius $r_\partial >0$,   s.t., for all $(\theta, \tau )\in  \overline{B_{r_\partial}(\theta_0, \tau_0)}$, $0 < \vert \E \left[ \e^{\tau'\psi(X_1, \theta)}\frac{\partial \psi(X_1, \theta)}{\partial \theta'}\right]\vert_{\det}^2 $, which is the first part of the result. By Lemma \ref{Lem:UniFiniteSampleInvertibilityFromUniAsInvertibility} (p. \pageref{Lem:UniFiniteSampleInvertibilityFromUniAsInvertibility}), the second part of the result follows from     the continuity of $(\theta, \tau)\mapsto \E \left[ \e^{\tau'\psi(X_1, \theta)}\frac{\partial \psi(X_1, \theta)}{\partial \theta'}\right]$, the invertibility of $\E \left[ \e^{\tau'\psi(X_1, \theta)}\frac{\partial \psi(X_1, \theta)}{\partial \theta'}\right]$ for all $(\theta, \tau )\in  \overline{B_{r_\partial}(\theta_0, \tau_0)}$, and Lemma \ref{Lem:TiltedDerivative}ii (p. \pageref{Lem:TiltedDerivative}), which, under Assumption \ref{Assp:ExistenceConsistency}(a)-(b) and (e)-(f), implies that \\$\sup_{ (\theta, \tau)\in\overline{B_{r_\partial}(\theta_0, \tau_0)}}\vert\left[ \frac{1}{T}\sum_{t=1}^T \e^{\tau'\psi_t(\theta)}\frac{\partial \psi_{t} (\theta)'}{\partial \theta}\right] -\E\left[ \e^{\tau '\psi(X_1, \theta) }\frac{\partial \psi(X_1, \theta)'}{\partial \theta}\right]\vert=o(1)$, $\P$-a.s. as $T \rightarrow \infty$.

\textit{(iv)} It follows from   Lemma \ref{Lem:UniFiniteSampleInvertibilityFromUniAsInvertibility} (p. \pageref{Lem:UniFiniteSampleInvertibilityFromUniAsInvertibility}), so that it is sufficient to check its assumptions.  Firstly, under Assumptions \ref{Assp:ExistenceConsistency}(a)-(b), (e), (g) and (h), by Corollary \ref{Cor:ChgOfMeasureInvertibilityPDP} (p. \pageref{Cor:ChgOfMeasureInvertibilityPDP}), for all $(\theta, \tau)\in \Sbf $, $\E\left[\e^{\tau' \psi(X_1, \theta)}\psi(X_1, \theta) \psi(X_1, \theta)' \right] $  is a positive definite symmetric matrix, and thus it is invertible. Secondly, under Assumption \ref{Assp:ExistenceConsistency}(a)-(b), (e) and (g), by Lemma \ref{Lem:TiltedVariance}i (p. \pageref{Lem:TiltedVariance}),

\noindent
$\E\left[\sup_{(\theta, \tau)\in \Sbf}\vert\e^{\tau' \psi(X_1, \theta)}\psi(X_1, \theta) \psi(X_1, \theta)' \vert\right]<\infty $, so that by the Lebesgue dominated convergence theorem and Assumption \ref{Assp:ExistenceConsistency}(b), $(\theta,\tau)\mapsto \E\left[\e^{\tau' \psi(X_1, \theta)}\psi(X_1, \theta) \psi(X_1, \theta)' \right] $ is continuous in $\Sbf$. Finally, under Assumptions \ref{Assp:ExistenceConsistency}(a)-(b), (d), (e),  (g) and (h), $\P$-a.s. as $T \rightarrow \infty$, 

\noindent
$\sup_{ (\theta, \tau)\in   \Sbf}\left\vert \frac{1}{T}\sum_{t=1}^T \e^{\tau'\psi_t(\theta)}\psi_t(\theta) \psi_t(\theta)' -\E [\e^{\tau'\psi(X_1, \theta) }\psi(X_1,\theta) \psi(X_1,\theta)']\right\vert=o(1)$.

\end{proof}

In order to simplify the analysis of the asymptotic properties of the ESP estimator, we decompose the  LogESP into three terms.

\begin{lem}[LogESP decomposition]\label{Lem:LogESPDecomposition} Under Assumption \ref{Assp:ExistenceConsistency}, $\P$-a.s. for $T$ big enough, define, for all $(\theta, \tau)\in\overline{B_{r_\partial}(\theta_0, \tau_0)}$,
$L_T(\theta, \tau):= \ln  \left[\frac{1}{T}\sum_{t=1}^T \e^{\tau'\psi_t(\theta)}\right]\\-\frac{1}{2T}\ln\negthickspace \left\{ \left\vert\left[ \sum_{t=1}^T \frac{        \e^{\tau'\psi_t(\theta)}}
     { \sum_{i=1}^T \e^{\tau'\psi_i(\theta)}}\frac{\partial \psi_{t} (\theta)}{\partial \theta'}\right]^{-1}\negthickspace\left[\sum_{t=1}^T \frac{        \e^{\tau'\psi_t(\theta)}}
     { \sum_{i=1}^T \e^{\tau'\psi_i(\theta)}}\psi_{t}(\theta)\psi_{t}(\theta)'\right]\negthickspace\left[ \sum_{t=1}^T \frac{        \e^{\tau'\psi_t(\theta)}}
     { \sum_{i=1}^T \e^{\tau'\psi_i(\theta)}}\frac{\partial \psi_{t} (\theta)'}{\partial \theta}\right]^{-1}\right\vert_{\det}\negthickspace\right\}$, which exists by Lemma \ref{Lem:LogESPExistenceInS} (p. \pageref{Lem:LogESPExistenceInS}), and where $\overline{B_{r_\partial}(\theta_0, \tau_0)}$ is defined as in the aforementioned lemma. Then, under Assumption \ref{Assp:ExistenceConsistency}(a)(b) and (d)-(h), $\P$-a.s. for $T$ big enough, for all $(\theta, \tau)\in\overline{B_{r_\partial}(\theta_0, \tau_0)}$,
\begin{eqnarray*}
L_T(\theta, \tau)=M_{1,T}(\theta, \tau)+M_{2,T}(\theta, \tau)+M_{3,T}(\theta, \tau)\text{ where}
\end{eqnarray*}
 $M_{1,T}(\theta, \tau)\negthickspace:=\negthickspace\left(1 \negthickspace-\negthickspace \frac{m}{2T}  \right)\ln\negthickspace \left[    \frac{1}{T} \sum_{t=1}^T \e^{\tau'\psi_t(\theta)}   \right]$, $M_{2,T}(\theta, \tau)\negthickspace:=\negthickspace\frac{1}{2T} \ln\negthickspace \left[   \left|  \frac{1}{T} \sum_{t=1}^T                \e^{\tau'\psi_t(\theta)}
      \frac{\partial \psi_t(\theta)}{\partial \theta' }  \right|_{\det}^2          \right]$, and 
         
         \noindent
         $M_{3,T}(\theta, \tau)\negthickspace:=\negthickspace-\frac{1}{2T } \ln\negthickspace \left[ \left|  \frac{1}{T} \sum_{t=1}^T                  \e^{\tau'\psi_t(\theta)}
    \psi_t(\theta)  \psi_t(\theta)  ' \right|_{\det}         \right] $.
\end{lem}
\begin{proof} First of all, note that, under Assumption \ref{Assp:ExistenceConsistency}, by Lemma \ref{Lem:LogESPExistenceInS} (p. \pageref{Lem:LogESPExistenceInS}), $L_T(.)$ is well-defined $\P$-a.s. for $T$ big enough, for all  $(\theta, \tau)\in\overline{B_{r_\partial}(\theta_0, \tau_0)}$. Thus, under Assumption \ref{Assp:ExistenceConsistency},    $\P$-a.s. for $T$ big enough,  for all  $(\theta, \tau)\in\overline{B_{r_\partial}(\theta_0, \tau_0)}$.
\begin{eqnarray}
& & L_T( \theta, \tau) \nonumber\\
& = & \ln \left[    \frac{1}{T} \sum_{t=1}^T \e^{\tau'\psi_t(\theta)}   \right]\nonumber\\
& &\negthickspace-\frac{1}{2T}\ln\negthickspace \left\{ \left\vert\left[ \sum_{t=1}^T \frac{        \e^{\tau'\psi_t(\theta)}}
     { \sum_{i=1}^T \e^{\tau'\psi_i(\theta)}}\frac{\partial \psi_{t} (\theta)}{\partial \theta'}\right]^{-1}\negthickspace\left[\sum_{t=1}^T \frac{        \e^{\tau'\psi_t(\theta)}}
     { \sum_{i=1}^T \e^{\tau'\psi_i(\theta)}}\psi_{t}(\theta)\psi_{t}(\theta)'\right]
\right. \right.     
\\     
& & \hspace{2.2in} \left. \left.  \times   \left[ \sum_{t=1}^T \frac{        \e^{\tau'\psi_t(\theta)}}
     { \sum_{i=1}^T \e^{\tau'\psi_i(\theta)}}\frac{\partial \psi_{t} (\theta)'}{\partial \theta}\right]^{-1}\right\vert_{\det}\negthickspace\right\}\nonumber \\
 & \stackrel{(a)}{=} &  \ln \left[    \frac{1}{T} \sum_{t=1}^T \e^{\tau'\psi_t(\theta)}   \right]  +  \frac{1}{2T} \ln \left[ \left|  \sum_{t=1}^T         \frac{        \e^{\tau'\psi_t(\theta)}}
     { \sum_{i=1}^T \e^{\tau'\psi_i(\theta)}}  \frac{\partial \psi_t(\theta)}{\partial \theta' }  \right|_{\det}^2          \right] \nonumber  \\
& &\
        -  \frac{1}{2T } \ln \left[  \left|  \frac{1}{T} \sum_{t=1}^T                  \frac{        \e^{\tau'\psi_t(\theta)}}
     { \sum_{i=1}^T \e^{\tau'\psi_i(\theta)}}
    \psi_t(\theta)  \psi_t(\theta)  ' \right|_{\det}         \right]
         \nonumber \\
& \stackrel{(b)}{=} &    \ln \left[    \frac{1}{T} \sum_{t=1}^T \e^{\tau'\psi_t(\theta)}   \right]       \nonumber \\
 & &  +  \frac{1}{2T} \ln \left[  \left( \frac{ 1 }
     {  \frac{1}{T} \sum_{i=1}^T \e^{\tau'\psi_i(\theta)}} \right)^{2m} \left|  \frac{1}{T} \sum_{t=1}^T                \e^{\tau'\psi_t(\theta)}
      \frac{\partial \psi_t(\theta)}{\partial \theta' }  \right|_{\det}^2 \right] \nonumber\\
& &    -  \frac{1}{2T } \ln \left[  \left( \frac{ 1 }
     {  \frac{1}{T} \sum_{i=1}^T \e^{\tau'\psi_i(\theta)}} \right)^m\left|  \sum_{t=1}^T          \e^{\tau'\psi_t(\theta)}\psi_t(\theta)  \psi_t(\theta)  ' \right|_{\det}         \right] \    \nonumber \\
& \stackrel{(c)}{=} &    \left(1 - \frac{m}{2T}  \right)\  \ln \left[    \frac{1}{T} \sum_{t=1}^T \e^{\tau'\psi_t(\theta)}   \right]
   +  \frac{1}{2T} \ln \left[   \left|  \frac{1}{T} \sum_{t=1}^T                \e^{\tau'\psi_t(\theta)}
      \frac{\partial \psi_t(\theta)}{\partial \theta' }  \right|_{\det}^2          \right]  \nonumber \\
& &     -  \frac{1}{2T } \ln \left[ \left|  \frac{1}{T} \sum_{t=1}^T                  \e^{\tau'\psi_t(\theta)}
    \psi_t(\theta)  \psi_t(\theta)  ' \right|_{\det}         \right]. \label{proof_spd}
\end{eqnarray}
\textit{(a)} Firstly,  use that the determinant of the product is the product of the determinants \citep[e.g.][Theorem 9.35]{1953Rudin}. Secondly, the determinant of an inverse is the inverse of the determinant \citep[e.g.][p. 233]{1953Rudin}. Finally, use basic properties of the logarithm, and note that we keep the square in the second logarithm in order to ensure the  positivity of the argument (then the strict positivity is ensured by Lemma \ref{Lem:LogESPExistenceInS} on p. \pageref{Lem:LogESPExistenceInS}). \textit{(b)} Use multilinearity of determinant.  \textit{(c)} Note that $1+\frac{-2m}{2T}-\frac{-m}{2T}=1-\frac{m}{2T} $.
\end{proof}

\subsubsection{Derivatives of $M_{1,T}(\theta, \tau):=\left(1 - \frac{m}{2T}  \right)\ln \left[    \frac{1}{T} \sum_{t=1}^T \e^{\tau'\psi_t(\theta)}   \right]$}

\textit{First derivative $\frac{\partial M_{1,T}(\theta, \tau)}{\partial \theta_j} $.}
 By  Assumption \ref{Assp:ExistenceConsistency}(b),  $\theta \mapsto \psi(X_1,\theta)$ is differentiable in $\T$  $\P$-a.s.  Thus,   for all $(\theta, \tau)\in \Sbf$, for all $j \in \ldsb 1,m\rdsb$,
 \begin{eqnarray}
\frac{\partial M_{1,T}(\theta, \tau)}{\partial \theta_j}=\left(1 - \frac{m}{2T}  \right)
 \frac{
        \frac{1}{T}\sum_{t=1}^T \e^{\tau'\psi_t(\theta)} \tau'
                \frac{\partial \psi_t(\theta)}{\partial\theta_j}  }{\frac{1}{T}\sum_{t=1}^T \e^{\tau'\psi_t(\theta)}}\label{Eq:LogETTerm1stDerivative}.
\end{eqnarray}

\textit{Second derivative $\frac{\partial^2 M_{1,T}(\theta, \tau) }{\partial \theta_\ell \partial\theta_{j} }$.}
By  Assumption \ref{Assp:AsymptoticNormality}(a),  $\theta \mapsto \psi(X_1,\theta)$ are three times  continuously differentiable in  a neighborhood of $\theta_0$ $\P$-a.s. Thus, by equation \eqref{Eq:LogETTerm1stDerivative} on p. \pageref{Eq:LogETTerm1stDerivative}, under Assumptions \ref{Assp:ExistenceConsistency}(a)-(e), (g)-(h) and  \ref{Assp:AsymptoticNormality}(a), by Lemma \ref{Lem:LogESPExistenceInS}ii (p. \pageref{Lem:LogESPExistenceInS}),  $\P$-a.s.,    for all $(\theta, \tau)$ in a neighborhood of $(\theta_0, \tau_0)$, for all $( \ell, j)\in \ldsb 1,m\rdsb^2$,
\begin{eqnarray}
& & \frac{\partial^2 M_{1,T}(\theta, \tau) }{\partial\theta_{\ell} \partial \theta_j}\nonumber \\
 & = & \frac{(1 - \frac{m}{2T} ) }{\left[ \frac{1}{T}\sum_{t=1}^T \e^{\tau'\psi_t(\theta)}\right]^2}
        \left\{
                \left\{\frac{1}{T}\sum_{t=1}^T \e^{\tau'\psi_t(\theta)} \left[\tau' \frac{\partial \psi_t(\theta)}{\partial \theta_\ell}\right] \left[\tau' \frac{\partial \psi_t(\theta)}{\partial\theta_j}\right]+\e^{\tau'\psi_t(\theta)}\left[\tau'\frac{\partial^2 \psi_t(\theta)}{\partial \theta_j \partial\theta_\ell}\right] \right\} \right.  \nonumber \\
& &\times\left\{ \frac{1}{T}\sum_{t=1}^T \e^{\tau'\psi_t(\theta)}\right\} - \left\{\frac{1}{T}\sum_{t=1}^T \e^{\tau'\psi_t(\theta)}\left[\tau'\frac{\partial \psi_t(\theta)}{\partial\theta_j } \right]\right\}\nonumber  \\
& & \left. \times \left\{\frac{1}{T}\sum_{t=1}^T \e^{\tau'\psi_t(\theta)}\left[\tau'\frac{\partial \psi_t(\theta)}{\partial \theta_\ell } \right]\right\}\right\}\nonumber  \\
 & = & \frac{(1 - \frac{m}{2T} ) }{\left[ \frac{1}{T}\sum_{t=1}^T \e^{\tau'\psi_t(\theta)}\right]}\left\{\frac{1}{T}\sum_{t=1}^T \e^{\tau'\psi_t(\theta)} \left[\tau' \frac{\partial \psi_t(\theta)}{\partial \theta_\ell}\right] \left[\tau' \frac{\partial \psi_t(\theta)}{\partial\theta_j}\right] +\e^{\tau'\psi_t(\theta)}\left[\tau'\frac{\partial^2 \psi_t(\theta)}{\partial \theta_j \partial \theta_\ell}\right] \right\} \nonumber   \\
& & -\frac{(1 - \frac{m}{2T} ) }{\left[ \frac{1}{T}\sum_{t=1}^T \e^{\tau'\psi_t(\theta)}\right]^2} \left\{\frac{1}{T}\sum_{t=1}^T \e^{\tau'\psi_t(\theta)}\tau'\frac{\partial \psi_t(\theta)}{\partial\theta_j } \right\}  \times \left\{\frac{1}{T}\sum_{t=1}^T \e^{\tau'\psi_t(\theta)}\tau'\frac{\partial \psi_t(\theta)}{\partial \theta_\ell } \right\}.\label{Eq:LogETTerm2ndDerivative}
\end{eqnarray}

\textit{Second derivative $\frac{\partial^2 M_{1,T}(\theta, \tau)}{ \partial \tau_{k} \partial \theta_{j}}$.} Under Assumption \ref{Assp:ExistenceConsistency}(a)-(b), by equation \eqref{Eq:LogETTerm1stDerivative} on p. \pageref{Eq:LogETTerm1stDerivative}, $\P$-a.s., for all $(\theta, \tau)\in \Sbf$, for all $(k, j)\in \ldsb 1,m\rdsb^2$,
\begin{eqnarray}
& &
\frac{\partial^2 M_{1,T}(\theta, \tau)}{ \partial \tau_{k} \partial \theta_{j}}\nonumber \\
& = &
\left(1 - \frac{m}{2T}  \right) \frac{1}{ \left[ \frac{1}{T} \sum_{i=1}^T \e^{\tau'\psi_i(\theta)}  \right]^2 }\nonumber
\\
& & \times
\left\{
  \left[\frac{1}{T} \sum_{i=1}^T \e^{\tau'\psi_i(\theta)} \right]\frac{1}{T} \sum_{t=1}^T \left\{\e^{\tau'\psi_t(\theta)} \tau' \frac{\partial \psi_t(\theta)}{\partial \theta_{j}}
 \psi_{t,k} (\theta) \nonumber
+
\e^{\tau'\psi_t(\theta)}  \frac{\partial \psi_{t,k} (\theta)}{\partial \theta_{j}} \right\}
\right.\nonumber
\\
& &  \left.
\hspace{.5in}
 - \left[\frac{1}{T} \sum_{t=1}^T \e^{\tau'\psi_t(\theta)} \tau' \frac{\partial \psi_t(\theta)}{\partial \theta_{j}}\right]\left[ \frac{1}{T} \sum_{i=1}^T \e^{\tau'\psi_i(\theta)}   \psi_{i,k}(\theta) \right]
\right\}.\label{Eq:LogETTerm2ndDerivativeThetaTau}
\end{eqnarray}

\textit{First derivative $\frac{\partial M_{1,T}(\theta, \tau)}{\partial \tau_{k}} $.} By definition of $M_{1,T}(\theta, \tau)$ in Lemma \ref{Lem:LogESPDecomposition} (p. \pageref{Lem:LogESPDecomposition}), for all $(\theta, \tau)\in \Sbf$, for all $k \in \ldsb 1,m\rdsb$,\begin{eqnarray}
\frac{\partial M_{1,T}(\theta, \tau)}{\partial \tau_{k}}
& = &
\left(1 - \frac{m}{2T}  \right) \frac{ \frac{1}{T} \sum_{t=1}^T \e^{\tau'\psi_t(\theta)}  \psi_{t, k} (\theta) }
{\frac{1}{T} \sum_{i=1}^T \e^{\tau'\psi_i(\theta)} }.\label{Eq:LogETTermDTau}
\end{eqnarray}

\textit{Second derivative $ \frac{\partial^2 M_{1,T}(\theta, \tau)}{ \partial \tau_{h}\partial \tau_{k} } $.} By the above equation \eqref{Eq:LogETTermDTau}, for all $(\theta, \tau)\in \Sbf$, for all $(h,k) \in \ldsb 1,m\rdsb^2$,
\begin{eqnarray}
& &
\frac{\partial^2 M_{1,T}(\theta, \tau)}{  \partial \tau_{h} \partial \tau_{k}}\nonumber \\
& = &
\left(1 - \frac{m}{2T}  \right) \frac{1}{ \left[ \frac{1}{T} \sum_{i=1}^T \e^{\tau'\psi_i(\theta)}  \right]^2 }\times
\left\{
 \left[ \frac{1}{T} \sum_{i=1}^T \e^{\tau'\psi_i(\theta)}\right] \left[ \frac{1}{T} \sum_{t=1}^T \e^{\tau'\psi_t(\theta)}
 \psi_{t,h} (\theta)
 \psi_{t,k} (\theta) \right]\right.\nonumber
\\
& & \left.
 - \left[\frac{1}{T} \sum_{t=1}^T \e^{\tau'\psi_t(\theta)} \psi_{t,h} (\theta)\right]\left[ \frac{1}{T} \sum_{i=1}^T \e^{\tau'\psi_i(\theta)}   \psi_{i,k}(\theta)\right]
\right\}.\label{Eq:LogETTermDTauDTau}
\end{eqnarray}

\subsubsection{Derivatives of $M_{2,T}(\theta, \tau):=\frac{1}{2T} \ln \left[   \left|  \frac{1}{T} \sum_{t=1}^T                \e^{\tau'\psi_t(\theta)}
      \frac{\partial \psi_t(\theta)}{\partial \theta' }  \right|_{\det}^2          \right]$ }

\textit{First derivative $\frac{\partial M_{2,T}(\theta ,\tau)}{\partial \theta_j}$.}
If $F(.)$ is a differentiable matrix function s.t. $\vert F(x) \vert_{\det}\neq 0$, then
$D \ln[\vert F(x)\vert_{\det}^2] =2\tr[ F(x)^{-1}DF(x)]$ (Lemma \ref{Lem:DeterminantDifferential}ii on p. \pageref{Lem:DeterminantDifferential}) where $DF(x)$ denotes the derivative of $F(.)$ at $x$.
Now, under Assumption \ref{Assp:ExistenceConsistency}, by Lemma \ref{Lem:LogESPExistenceInS}iii (p. \pageref{Lem:LogESPExistenceInS}), $\P$-a.s. for $T$ big enough, for all $(\theta, \tau)$ in a neighborhood of $(\theta_0, \tau_0)$,  $\frac{1}{T} \sum_{t=1}^T                \e^{\tau'\psi_t(\theta)}
      \frac{\partial \psi_t(\theta)}{\partial \theta' }$ is invertible. In addition, under Assumption \ref{Assp:ExistenceConsistency}(a), by Assumption \ref{Assp:AsymptoticNormality}(a),  $\theta \mapsto \psi(X_1, \theta)$ is twice differentiable in a neighborhood of $\theta_0$ $\P$-a.s., so that, under Assumption  \ref{Assp:ExistenceConsistency}(a)-(e) and (g)-(h), by Lemma \ref{Lem:LogESPExistenceInS}ii (p. \pageref{Lem:LogESPExistenceInS}),   $(\theta, \tau) \mapsto\frac{1}{T} \sum_{t=1}^T                \e^{\tau'\psi_t(\theta)}
      \frac{\partial \psi_t(\theta)}{\partial \theta' } $ is also differentiable  in a neighborhood of $(\theta_0, \tau_0)$ $\P$-a.s.  Thus, under Assumptions \ref{Assp:ExistenceConsistency}  and  \ref{Assp:AsymptoticNormality}(a), $\P$-a.s. for $T$ big enough,  for all $(\theta, \tau)$ in a neighborhood of $(\theta_0, \tau_0)$, for all $j \in \ldsb 1,m\rdsb$,
\begin{eqnarray}
& & \frac{\partial M_{2,T}(\theta ,\tau)}{\partial \theta_j}
 =   \frac{1}{T} {\rm tr}\left\{ \left[   \frac{1}{T} \sum_{t=1}^T \e^{\tau'\psi_t(\theta)}
\frac{\partial \psi_t(\theta)}{\partial \theta' }    \right]^{-1}  \right. \nonumber
\\
& &   \negthickspace\times\negthickspace   \left[  \frac{1}{T} \sum_{t=1}^T \e^{\tau'\psi_t(\theta)}
\frac{\partial^2 \psi_t(\theta)}{\partial \theta_j \partial\theta'}    %\right.  %\\
%& & \hspace{.25in} \left.
\left.   \negthickspace+  \frac{1}{T} \sum_{t=1}^T \e^{\tau'\psi_t(\theta)}\negthickspace  \left(
\tau' \frac{\partial \psi_t(\theta)}{\partial \theta_j } \right)
\frac{\partial \psi_t(\theta)}{\partial \theta' }   \right]   \right\} \label{Eq:DerivativeTerm1stDerivative}
\end{eqnarray}

\textit{Second derivative $\frac{\partial^2 M_{2,T}(\theta, \tau)}{\partial \theta_\ell \partial \theta_j}$.} The trace of a derivative is the derivative of the trace because both the trace and derivative operators are linear \cite[e.g.,][chap. 9 sec. 9]{1988MagnusNeudecker}. Moreover, if $F(.)$ is a differentiable matrix function s.t., for all $x$ in a neighborhood of $\dot{x}$, $\vert F(x) \vert_{\det}\neq 0$, then
$D \left[F(\dot{x})^{-1}\right] =-F(\dot{x})^{-1} [DF(\dot{x})]F(\dot{x})^{-1}$ \citep[e.g.,][ chap. 8 sec. 4]{1988MagnusNeudecker}. Now, as explained for the first derivative, under Assumption \ref{Assp:ExistenceConsistency}, by Lemma \ref{Lem:LogESPExistenceInS}iii (p. \pageref{Lem:LogESPExistenceInS}), $\P$-a.s. for $T$ big enough, for all $(\theta, \tau)$ in a neighborhood of $(\theta_0, \tau_0)$,   $\frac{1}{T} \sum_{t=1}^T                \e^{\tau'\psi_t(\theta)}
      \frac{\partial \psi_t(\theta)}{\partial \theta' }$ is invertible.  In addition, by Assumption \ref{Assp:AsymptoticNormality}(a), $\P$-a.s. $\theta \mapsto \psi(X_1, \theta)$ is three times continuously differentiable in a neighborhood of $\theta_0$, so that, under Assumption  \ref{Assp:ExistenceConsistency} and  \ref{Assp:AsymptoticNormality}(a),   $\theta\mapsto \frac{\partial M_{2,T}(\theta ,\tau)}{\partial \theta_j} $ is differentiable in a neighborhood of $(\theta_0, \tau_0)$.  Thus, under Assumptions \ref{Assp:ExistenceConsistency} and   \ref{Assp:AsymptoticNormality}(a), by the above equation \eqref{Eq:DerivativeTerm1stDerivative}, $\P$-a.s. for $T$ big enough, for all $(\theta, \tau)$ in a neighborhood of $(\theta_0, \tau_0)$, for all $(\ell, j)\in \ldsb 1,m \rdsb^2 $,
\begin{eqnarray}
& & \frac{\partial^2 M_{2,T}(\theta, \tau)}{\partial \theta_\ell \partial \theta_j}  =   \frac{1}{T} {\rm tr}\left\{ -\left[   \frac{1}{T} \sum_{t=1}^T \e^{\tau'\psi_t(\theta)}
\frac{\partial \psi_t(\theta)}{\partial \theta' }    \right]^{-1}  \right. \nonumber\\
& & \hspace{.1in} \left\{   \frac{1}{T} \sum_{t=1}^T \e^{\tau'\psi_t(\theta)}
\frac{\partial^2 \psi_t(\theta)'}{\partial \theta_{\ell} \partial \theta'}
%\right. \\
%
% & & \hspace{.3in} \left.
+  \frac{1}{T } \sum_{t=1}^T \e^{\tau'\psi_t(\theta)}   \left(  \tau' \frac{\partial \psi_t(\theta)}{\partial \theta_\ell } \right)
\frac{\partial \psi_t(\theta)}{\partial \theta' }    \right\}\nonumber \\
& & \hspace{.1in} \left[   \frac{1}{T} \sum_{t=1}^T \e^{\tau'\psi_t(\theta)}
\frac{\partial \psi_t(\theta)}{\partial \theta' }    \right]^{-1}   \nonumber\\
& & \hspace{.2in}   \times   \left[  \frac{1}{T} \sum_{t=1}^T \e^{\tau'\psi_t(\theta)}
\frac{\partial^2 \psi_t(\theta)}{\partial \theta_j \partial \theta'}
%\right.  \\
%& & \hspace{.25in} \left.
+  \frac{1}{T} \sum_{t=1}^T \e^{\tau'\psi_t(\theta)}  \left(
\tau' \frac{\partial \psi_t(\theta)}{\partial \theta_j } \right)
\frac{\partial \psi_t(\theta)}{\partial \theta' }   \right]  \nonumber\\
 & & \hspace{.02in} + \left[   \frac{1}{T} \sum_{t=1}^T \e^{\tau'\psi_t(\theta)}
 \frac{\partial \psi_t(\theta)}{\partial \theta' }    \right]^{-1} \nonumber  \\
% %
 & & \hspace{.2in}   \times   \left[  \frac{1}{T} \sum_{t=1}^T \e^{\tau'\psi_t(\theta)}
 \frac{\partial^3 \psi_t(\theta)}{\partial \theta_\ell \partial \theta_j \partial \theta'}
\right.
%\\
%& &  \hspace{.5in}
 +  \frac{1}{T} \sum_{t=1}^T \e^{\tau'\psi_t(\theta)}   \left(
\tau' \frac{\partial \psi_t(\theta)}{\partial \theta_\ell } \right)
 \frac{\partial^2 \psi_t(\theta)}{\partial \theta_j \partial \theta'}    \nonumber  \\
& & \hspace{.5in}  +  \frac{1}{T} \sum_{t=1}^T \e^{\tau'\psi_t(\theta)}  \left(
\tau' \frac{\partial \psi_t(\theta)}{\partial \theta_j } \right)
\frac{\partial^2 \psi_t(\theta)}{\partial \theta_{\ell} \partial \theta'} \nonumber \\
& & \hspace{.5in} +  \frac{1}{T} \sum_{t=1}^T \e^{\tau'\psi_t(\theta)}
\left(
\tau' \frac{\partial^2 \psi_t(\theta)}{\partial \theta_{\ell} \partial \theta_j}
\right)
\frac{\partial \psi_t(\theta)}{\partial \theta' }  \nonumber \\
& & \hspace{.5in} \left. \left. +  \frac{1}{T} \sum_{t=1}^T \e^{\tau'\psi_t(\theta)}\negthickspace
 \left(
\tau' \frac{\partial \psi_t(\theta)}{\partial \theta_\ell } \right)
%\\
%& & \left. \left.  \hspace{2in} \times
\negthickspace \left(  \tau' \frac{\partial \psi_t(\theta)}{\partial \theta_j } \right)
\negthickspace\frac{\partial \psi_t(\theta)}{\partial \theta' }   \right]   \right\}\label{Eq:DerivativeTerm2ndDerivative}
\end{eqnarray}

\textit{Second derivative $ \frac{\partial^{2} M_{2,T}(\theta, \tau)}{ \partial \tau_{k} \partial \theta_{j}}  $.} By a reasoning similar to the one for the derivative $\frac{\partial^2 M_{2,T}(\theta, \tau)}{\partial \theta_\ell \partial \theta_j}$, under Assumptions \ref{Assp:ExistenceConsistency} and   \ref{Assp:AsymptoticNormality}(a), by the above equation \eqref{Eq:DerivativeTerm1stDerivative}, $\P$-a.s. for $T$ big enough, for all $(\theta, \tau)$ in a neighborhood of $(\theta_0, \tau_0)$, for all $(k,j)\in \ldsb 1,m \rdsb^2$,

\begin{eqnarray}
& & \frac{\partial^{2} M_{2,T}(\theta, \tau)}{ \partial \tau_{k} \partial \theta_{j}}
\\
& &  = \frac{1}{T}
{\rm tr}\left\{ - \left[ \frac{1}{T} \sum_{t=1}^T                \e^{\tau'\psi_t(\theta)}
      \frac{\partial \psi_t(\theta)}{\partial \theta' }  \right]^{-1}
      \left[ \frac{1}{T} \sum_{t=1}^T                \e^{\tau'\psi_t(\theta)}
      \psi_{t,k}(\theta)
      \frac{\partial \psi_t(\theta)}{\partial \theta' }  \right]
      \left[ \frac{1}{T} \sum_{t=1}^T                \e^{\tau'\psi_t(\theta)}
      \frac{\partial \psi_t(\theta)}{\partial \theta' }  \right]^{-1}
      \right.
       \nonumber \\
& & \hspace{.5in}
      \left.
      \times
      \left[
      \frac{1}{T} \sum_{t=1}^T                \e^{\tau'\psi_t(\theta)} \left(
      \frac{\partial \psi_t(\theta)}{\partial \theta' }
\tau' \frac{\partial \psi_t(\theta)}{\partial \theta_{j} } +
\frac{\partial^2 \psi_t(\theta)}{ \partial \theta_{j} \partial \theta' }
      \right)
      \right]
\right\} \nonumber
\\
& &  + \frac{1}{T}
{\rm tr}\left\{  \left[ \frac{1}{T} \sum_{t=1}^T                \e^{\tau'\psi_t(\theta)}
      \frac{\partial \psi_t(\theta)}{\partial \theta' }  \right]^{-1}
      \left[
      \frac{1}{T} \sum_{t=1}^T                \e^{\tau'\psi_t(\theta)} \psi_{t,k}(\theta) \left(
      \frac{\partial \psi_t(\theta)}{\partial \theta' }
\tau' \frac{\partial \psi_t(\theta)}{\partial \theta_{j} } +
\frac{\partial^2 \psi_t(\theta)}{ \partial \theta_{j} \partial \theta' }
      \right)
      \right]
\right\}\nonumber
\\
& &  + \frac{1}{T}
{\rm tr}\left\{  \left[ \frac{1}{T} \sum_{t=1}^T                \e^{\tau'\psi_t(\theta)}
      \frac{\partial \psi_t(\theta)}{\partial \theta' }  \right]^{-1}
      \left[
      \frac{1}{T} \sum_{t=1}^T                \e^{\tau'\psi_t(\theta)} \left(
      \frac{\partial \psi_t(\theta)}{\partial \theta' }
 \frac{\partial \psi_{t,k}(\theta)}{\partial \theta_{j} }
      \right)
      \right]
\right\}. \label{Eq:DerivativeTerm2ndDerivativeTauTheta}
\end{eqnarray}

\textit{First derivative $ \frac{\partial M_{2,T}(\theta, \tau)}{\partial \tau_{k}} $.} If $F(.)$ is a differentiable matrix function s.t. $\vert F(x) \vert_{\det}\neq 0$, then
$D \ln[\vert F(x)\vert_{\det}^2] =2\tr[ F(x)^{-1}DF(x)]$ (Lemma \ref{Lem:DeterminantDifferential}ii on p. \pageref{Lem:DeterminantDifferential}) where $DF(x)$ denotes the derivative of $F(.)$ at $x$.
Now, under Assumption \ref{Assp:ExistenceConsistency}, by Lemma \ref{Lem:LogESPExistenceInS}iii (p. \pageref{Lem:LogESPExistenceInS}), $\P$-a.s. for $T$ big enough, for all $(\theta, \tau)$ in a neighborhood of $(\theta_0, \tau_0)$,  $\frac{1}{T} \sum_{t=1}^T                \e^{\tau'\psi_t(\theta)}
      \frac{\partial \psi_t(\theta)}{\partial \theta' }$ is invertible.  Thus, under Assumption \ref{Assp:ExistenceConsistency}, by definition of $M_{2,T}(\theta, \tau)$ in Lemma \ref{Lem:LogESPDecomposition} (p. \pageref{Lem:LogESPDecomposition}),   $\P$-a.s. for $T$ big enough,  for all $(\theta, \tau)$ in a neighborhood of $(\theta_0, \tau_0)$, for all $k \in \ldsb 1,m\rdsb$,
\begin{eqnarray}
\frac{\partial M_{2,T}(\theta, \tau)}{\partial \tau_{k}} & = & \frac{1}{T}
{\rm tr}\left\{  \left[ \frac{1}{T} \sum_{t=1}^T                \e^{\tau'\psi_t(\theta)}
      \frac{\partial \psi_t(\theta)}{\partial \theta' }  \right]^{-1}
      \left[
      \frac{1}{T} \sum_{t=1}^T                \e^{\tau'\psi_t(\theta)}
    \psi_{t,k}(\theta)  \frac{\partial \psi_t(\theta)}{\partial \theta' }
      \right]
\right\} . \label{Eq:DerivativeTermDTau}
\end{eqnarray}

\textit{Second derivative $ \frac{\partial^{2} M_{2,T}(\theta, \tau)}{\partial \tau_{h}\partial \tau_{k}} $.} The trace of a derivative is the derivative of the trace because both the trace and derivative operators are linear \cite[e.g.,][chap. 9 sec. 9]{1988MagnusNeudecker}. Moreover, if $F(.)$ is a differentiable matrix function s.t., for all $x$ in a neighborhood of $\dot{x}$, $\vert F(x) \vert_{\det}\neq 0$, then
$D \left[F(\dot{x})^{-1}\right] =-F(\dot{x})^{-1} [DF(\dot{x})]F(\dot{x})^{-1}$ \citep[e.g.,][ chap. 8 sec. 4]{1988MagnusNeudecker}. Now, as explained for the first derivative, under Assumption \ref{Assp:ExistenceConsistency}, by Lemma \ref{Lem:LogESPExistenceInS}iii (p. \pageref{Lem:LogESPExistenceInS}), $\P$-a.s. for $T$ big enough, for all $(\theta, \tau)$ in a neighborhood of $(\theta_0, \tau_0)$,   $\frac{1}{T} \sum_{t=1}^T                \e^{\tau'\psi_t(\theta)}
      \frac{\partial \psi_t(\theta)}{\partial \theta' }$ is invertible.   Thus, under Assumption \ref{Assp:ExistenceConsistency}, by the above equation \eqref{Eq:DerivativeTermDTau}, $\P$-a.s. for $T$ big enough, for all $(\theta, \tau)$ in a neighborhood of $(\theta_0, \tau_0)$, for all $( h, k)\in \ldsb 1,m \rdsb^2 $,\begin{eqnarray}
& & \frac{\partial^2 M_{2,T}(\theta, \tau)}{ \partial \tau_{h} \partial \tau_k }\nonumber \\
& &  =- \frac{1}{T}
\tr\left\{  \left[ \frac{1}{T} \sum_{t=1}^T                \e^{\tau'\psi_t(\theta)}
      \frac{\partial \psi_t(\theta)}{\partial \theta' }  \right]^{-1}
 \left[ \frac{1}{T} \sum_{t=1}^T                \e^{\tau'\psi_t(\theta)}
 \psi_{t,h}(\theta)
      \frac{\partial \psi_t(\theta)}{\partial \theta' }  \right]^{}
 \left[ \frac{1}{T} \sum_{t=1}^T                \e^{\tau'\psi_t(\theta)}
      \frac{\partial \psi_t(\theta)}{\partial \theta' }  \right]^{-1}
     \right.\nonumber
     \\
     &  &
     \hspace{.5in} \times
     \left.
      \left[
      \frac{1}{T} \sum_{t=1}^T                \e^{\tau'\psi_t(\theta)}
    \psi_{t,k}(\theta)  \frac{\partial \psi_t(\theta)}{\partial \theta' }
      \right]
\right\} \nonumber\\
& &  + \frac{1}{T}
\tr\left\{  \left[ \frac{1}{T} \sum_{t=1}^T                \e^{\tau'\psi_t(\theta)}
      \frac{\partial \psi_t(\theta)}{\partial \theta' }  \right]^{-1}
      \left[
      \frac{1}{T} \sum_{t=1}^T                \e^{\tau'\psi_t(\theta)}
      \psi_{t,k}(\theta)
    \psi_{t,h}(\theta)  \frac{\partial \psi_t(\theta)}{\partial \theta' }
      \right]
\right\}.\label{Eq:DerivativeTermDTauDTau}
\end{eqnarray}

\subsubsection{Derivatives of $
M_{3,T}(\theta, \tau)=-\frac{1}{2T } \ln \left[ \left|  \frac{1}{T} \sum_{t=1}^T                  \e^{\tau'\psi_t(\theta)}
    \psi_t(\theta)  \psi_t(\theta)  ' \right|_{\det}   \right]
$}

{\it First derivative

\noindent
 $\frac{\partial M_{3,T}(\theta, \tau)}{\partial \theta_j}$.}
If $F(.)$ is a differentiable matrix function s.t. $\vert F(x) \vert_{\det}> 0$, then
$D \ln[\vert F(x)\vert_{\det}] =\tr[ F(x)^{-1}DF(x)]$ \citep[e.g.,][ chap. 8 sec. 3]{1988MagnusNeudecker}.
Now, under Assumption \ref{Assp:ExistenceConsistency}(a)-(b),   (e) and (g)(h), by Lemma \ref{Lem:LogESPExistenceInS}iv (p. \pageref{Lem:LogESPExistenceInS}), $\P$-a.s. for $T$ big enough, for all $(\theta, \tau) \in \Sbf$,  $|  \frac{1}{T} \sum_{t=1}^T                  \e^{\tau'\psi_t(\theta)}
    \psi_t(\theta)  \psi_t(\theta)  ' |_{\det}>0$. In addition,  by  Assumption \ref{Assp:ExistenceConsistency}(b), $\P$-a.s.  $\theta \mapsto \psi(X_1, \theta)$ is continuously differentiable in $\T$, so that,  $\P$-a.s. for $T$ big enough,  $\theta\mapsto M_{3,T}(\theta, \tau) $ is differentiable in  $\T $, for all $(\theta, \tau) \in \Sbf$. Thus, under Assumption \ref{Assp:ExistenceConsistency}(a)-(b) and (e)(g)(h),   $\P$-a.s. for $T$ big enough,  for all $(\theta, \tau) \in \Sbf$, for all $j \in \ldsb 1,m \rdsb$,
\begin{eqnarray}
\frac{\partial M_{3,T}(\theta, \tau)}{\partial \theta_j} & = &
-\frac{1}{2T} \tr  \left\{
 \left[  \frac{1}{T} \sum_{t=1}^T \e^{\tau'\psi_t(\theta)}  \psi_t(\theta)  \psi_t(\theta)' \right]^{-1}  \right.\nonumber \\
 & & \hspace{.2in} \times
\left[  \frac{1}{T} \sum_{t=1}^T \e^{\tau'\psi_t(\theta)}
 \left\{ \frac{\partial \psi_t(\theta)}{\partial \theta_j }  \psi_t(\theta)'
 +  \psi_t(\theta)   \frac{\partial \psi_t(\theta)'}{\partial \theta_j }     \right\}  \right. \nonumber  \\
& &  \hspace{.25 in} \left . \left. +   \frac{1}{T} \sum_{t=1}^T \e^{\tau'\psi_t(\theta)}\negthickspace
\left(  \tau' \frac{\partial \psi_t(\theta)}{\partial \theta_j } \right) \negthickspace\psi_t(\theta)  \psi_t(\theta)'  \right] \right\} \label{Eq:VarianceTerm1stDerivative}
\end{eqnarray}

\textit{Second derivative $\frac{\partial^2 M_{3,T}(\theta, \tau)}{\partial \theta_\ell \partial \theta_j}$.} The trace of a derivative is the derivative of the trace because both the trace and differentiation operators are linear \cite[e.g.,][chap. 9 sec. 9]{1988MagnusNeudecker}. Moreover, if $F(.)$ is a differentiable matrix function s.t., for all $x$ in a neighborhood of $\dot{x}$, $\vert F(x) \vert_{\det}\neq 0$, then
$D \left[F(\dot{x})^{-1}\right] =-F(\dot{x})^{-1} [DF(\dot{x})]F(\dot{x})^{-1}$ \citep[e.g.,][ chap. 8 sec. 4]{1988MagnusNeudecker}. Now,  under Assumption \ref{Assp:ExistenceConsistency}(a)-(b), (e) and (g)-(h), by Lemma \ref{Lem:LogESPExistenceInS}iv (p. \pageref{Lem:LogESPExistenceInS}), $\P$-a.s. for $T$ big enough, for all $(\theta, \tau) \in \Sbf$,   $\frac{1}{T} \sum_{t=1}^T                \e^{\tau'\psi_t(\theta)}
     \psi_t(\theta)  \psi_t(\theta)' $ is invertible. In addition, by Assumption \ref{Assp:AsymptoticNormality}(a), $\P$-a.s. $\theta \mapsto \psi(X_1, \theta)$ is three times continuously differentiable in a neighborhood of $\theta_0$, so that, under Assumption \ref{Assp:ExistenceConsistency}(a)-(e) and (g)(h), by Lemma \ref{Lem:LogESPExistenceInS}ii (p. \pageref{Lem:LogESPExistenceInS}), $\theta \mapsto \frac{\partial M_{3,T}(\theta, \tau)}{\partial \theta_j} $ is differentiable in a neighborhood of $(\theta_0, \tau_0)$.  Thus, under Assumptions \ref{Assp:ExistenceConsistency}(a)(b), (e) and (g)(h), and   \ref{Assp:AsymptoticNormality}(a), by the above equation \eqref{Eq:VarianceTerm1stDerivative},   $\P$-a.s. for $T$ big enough, for all $(\theta, \tau)$ in a neighborhood of $(\theta_0, \tau_0)$, for all $(\ell, j)\in \ldsb 1,m\rdsb^2$,
\begin{eqnarray}
& & \frac{\partial^2 M_{3,T}(\theta, \tau)}{\partial \theta_\ell \partial \theta_j}  =   -\frac{1}{2T} \tr  \left\{
- \left[  \frac{1}{T} \sum_{t=1}^T \e^{\tau'\psi_t(\theta)}\psi_t(\theta)  \psi_t(\theta)' \right]^{-1}  \right. \nonumber\\
& & \times\negthickspace
 \left[  \frac{1}{T} \sum_{t=1}^T \e^{\tau'\psi_t(\theta)} \negthickspace \left( \frac{\partial \psi_t(\theta)}{\partial \theta_\ell }  \psi_t(\theta)'
  \negthickspace+   \negthickspace\psi_t(\theta)   \frac{\partial \psi_t(\theta)'}{\partial \theta_\ell }  \right)\nonumber % \right. \\
%& & \hspace{.2in}  \left.
\right. \\
& &  \hspace{.5 in} \left.
 + \frac{1}{T} \sum_{t=1}^T \e^{\tau'\psi_t(\theta)} \negthickspace\left(  \tau' \frac{\partial \psi_t(\theta)}{\partial \theta_\ell } \right) \psi_t(\theta)  \psi_t(\theta)'
 \right] \nonumber\\
& & \hspace{.1in}
\times  \left[  \frac{1}{T} \sum_{t=1}^T \e^{\tau'\psi_t(\theta)}\psi_t(\theta)  \psi_t(\theta)' \right]^{-1} \nonumber \\
 & & \hspace{.2in} \times
\left[  \frac{1}{T} \sum_{t=1}^T \e^{\tau'\psi_t(\theta)} \left( \frac{\partial \psi_t(\theta)}{\partial \theta_j }  \psi_t(\theta)'
 +  \psi_t(\theta)   \frac{\partial \psi_t(\theta)'}{\partial \theta_j }     \right)     \right. \nonumber\\
& &  \hspace{.5 in} \left.
+   \frac{1}{T} \sum_{t=1}^T \e^{\tau'\psi_t(\theta)}
\left(
\tau' \frac{\partial \psi_t(\theta)}{\partial \theta_j } \right) \psi_t(\theta)  \psi_t(\theta)'  \right]\nonumber \\
%
% now the second term
%
& &
+  \left[  \frac{1}{T} \sum_{t=1}^T \e^{\tau'\psi_t(\theta)} \psi_t(\theta)  \psi_t(\theta)' \right]^{-1}\nonumber  \\
 & & \hspace{.2in} \times
\left[  \frac{1}{T} \sum_{t=1}^T \e^{\tau'\psi_t(\theta)} \left(
  \frac{\partial^2 \psi_t(\theta)}{\partial \theta_{\ell} \partial \theta_j}  \psi_t(\theta)'
 + \frac{\partial \psi_t(\theta)}{\partial \theta_j }  \frac{\partial \psi_t(\theta)'}{\partial \theta_\ell}
    \right.
%    \right.   \\
%
%
%
% & & \hspace{2.5in} \left.
 +   \frac{\partial \psi_t(\theta)}{\partial \theta_\ell }  \frac{\partial \psi_t(\theta)'}{\partial \theta_j }
 +  \psi_t(\theta)  \frac{\partial^2 \psi_t(\theta)'}{\partial \theta_{\ell} \partial \theta_j}     \right)   \nonumber \\
 & & \hspace{.5in} + \frac{1}{T} \sum_{t=1}^T \e^{\tau'\psi_t(\theta)}
\left(
\tau' \frac{\partial \psi_t(\theta)}{\partial \theta_\ell } \right)
 \left(
  \frac{\partial \psi_t(\theta)}{\partial \theta_j }  \psi_t(\theta)'
  +
   \psi_t(\theta)  \frac{\partial \psi_t(\theta)'}{\partial \theta_j }
  \right)  \nonumber \\
%
% now the last term
%
& &  \hspace{.5 in}   +   \frac{1}{T} \sum_{t=1}^T \e^{\tau'\psi_t(\theta)}
\left(
\tau' \frac{\partial \psi_t(\theta)}{\partial \theta_j } \right)
\left( \frac{\partial \psi_t(\theta)}{\partial \theta_\ell }  \psi_t(\theta)' +
 \psi_t(\theta) \frac{\partial \psi_t(\theta)'}{\partial \theta_\ell }
\right)  \nonumber\\
& &  \hspace{.5 in}  +   \frac{1}{T} \sum_{t=1}^T \e^{\tau'\psi_t(\theta)}\negthickspace
\left(
 \tau' \frac{\partial^2 \psi_t(\theta)'}{\partial \theta_{\ell} \partial \theta_j}
\right)
\negthickspace\psi_t(\theta)  \psi_t(\theta)' \nonumber \\
& &  \hspace{.5 in}  +   \frac{1}{T} \sum_{t=1}^T \e^{\tau'\psi_t(\theta)}\negthickspace
\left(
\tau' \frac{\partial \psi_t(\theta)}{\partial \theta_\ell } \right)\negthickspace \negthickspace
%\\
%& & \hspace{1in} \times
\left. \left.
\left(  \tau' \frac{\partial \psi_t(\theta)}{\partial \theta_j } \right) \negthickspace\psi_t(\theta)  \psi_t(\theta)'  \right]
 \right\}\label{Eq:VarianceTerm2ndDerivative}
\end{eqnarray}

\textit{Second derivative $\frac{\partial^2 M_{3,T}(\theta, \tau)}{\partial \tau_k \partial \theta_j}$.} Follow a reasoning similar to the one for the derivative $\frac{\partial^2 M_{2,T}(\theta, \tau)}{\partial \theta_\ell \partial \theta_j}$.
 The trace of a derivative is the derivative of the trace because both the trace and differentiation operators are linear \cite[e.g.,][chap. 9 sec. 9]{1988MagnusNeudecker}. Moreover, if $F(.)$ is a differentiable matrix function s.t., for all $x$ in a neighborhood of $\dot{x}$, $\vert F(x) \vert_{\det}\neq 0$, then
$D \left[F(\dot{x})^{-1}\right] =-F(\dot{x})^{-1} [DF(\dot{x})]F(\dot{x})^{-1}$ \citep[e.g.,][ chap. 8 sec. 4]{1988MagnusNeudecker}. Now,  under Assumption \ref{Assp:ExistenceConsistency}(a)-(b), (e) and (g)-(h), by Lemma \ref{Lem:LogESPExistenceInS}iv (p. \pageref{Lem:LogESPExistenceInS}), $\P$-a.s. for $T$ big enough, for all $(\theta, \tau) \in \Sbf$,   $\frac{1}{T} \sum_{t=1}^T                \e^{\tau'\psi_t(\theta)}
     \psi_t(\theta)  \psi_t(\theta)' $ is invertible.  Thus, under Assumptions \ref{Assp:ExistenceConsistency} (a)-(b), (e), (g)(h), by the above equation \eqref{Eq:VarianceTerm1stDerivative},   $\P$-a.s. for $T$ big enough, for all $(\theta, \tau) \in\Sbf$, for all $(k, j)\in \ldsb 1,m\rdsb^2$
\begin{eqnarray}
& & \frac{\partial^{2} M_{3,T}(\theta, \tau)}{ \partial \tau_k \partial \theta_{j} } \nonumber
\\
& = & \frac{1}{2T }  \tr \left\{
\left[  \frac{1}{T} \sum_{t=1}^T                  \e^{\tau'\psi_t(\theta)}
    \psi_t(\theta)  \psi_t(\theta)' \right]^{-1}
    \left[  \frac{1}{T} \sum_{t=1}^T                  \e^{\tau'\psi_t(\theta)}
 \psi_{t,k}(\theta)
    \psi_t(\theta)  \psi_t(\theta)' \right]
    \left[  \frac{1}{T} \sum_{t=1}^T                  \e^{\tau'\psi_t(\theta)}
    \psi_t(\theta)  \psi_t(\theta)' \right]^{-1}
    \right.\nonumber
    \\
    & & \hspace{.5in}
    \left.
    \times \left[ \frac{1}{T} \sum_{t=1}^T                  \e^{\tau'\psi_t(\theta)}
   \left( \tau' \frac{\partial \psi_{t}(\theta)}{\partial \theta_{j}}  \psi_t(\theta)  \psi_t(\theta)' + \frac{\partial \psi_{t}(\theta)}{\partial \theta_{j}}   \psi_t(\theta)' + \psi_t(\theta)\frac{\partial \psi_{t}(\theta)'}{\partial \theta_{j}}    \right) \right]
\right\}\nonumber
\\
&  & -\frac{1}{2T }  \tr \left\{
\left[  \frac{1}{T} \sum_{t=1}^T                  \e^{\tau'\psi_t(\theta)}
    \psi_t(\theta)  \psi_t(\theta)' \right]^{-1}
    \right.\nonumber
    \\
    & & \hspace{.5in}
    \left.
    \times \left[ \frac{1}{T} \sum_{t=1}^T                  \e^{\tau'\psi_t(\theta)}
\psi_{t,k}(\theta)
   \left( \tau' \frac{\partial \psi_{t}(\theta)}{\partial \theta_{j}}  \psi_t(\theta)  \psi_t(\theta)' + \frac{\partial \psi_{t}(\theta)}{\partial \theta_{j}}   \psi_t(\theta)' + \psi_t(\theta)\frac{\partial \psi_{t}(\theta)'}{\partial \theta_{j}}    \right) \right]
\right\}\nonumber
\\
&  & -\frac{1}{2T }  \tr \left\{
\left[  \frac{1}{T} \sum_{t=1}^T                  \e^{\tau'\psi_t(\theta)}
    \psi_t(\theta)  \psi_t(\theta)' \right]^{-1}
    \left[ \frac{1}{T} \sum_{t=1}^T                  \e^{\tau'\psi_t(\theta)}
    \frac{\partial \psi_{t,k}(\theta)}{\partial \theta_{j}}  \psi_t(\theta)  \psi_t(\theta)'    \right]
\right\}.\label{Eq:VarianceTerm2ndDerivativeTau}
\end{eqnarray}

\textit{First derivative $\frac{\partial M_{3,T}(\theta, \tau)}{\partial \tau_{k}} $.} If $F(.)$ is a differentiable matrix function s.t. $\vert F(x) \vert_{\det}> 0$, then
$D \ln[\vert F(x)\vert_{\det}] =\tr[ F(x)^{-1}DF(x)]$ \citep[e.g.,][ chap. 8 sec. 3]{1988MagnusNeudecker}.
Now, under Assumption \ref{Assp:ExistenceConsistency}(a)-(b)(e)(g)(h), by Lemma \ref{Lem:LogESPExistenceInS}iv (p. \pageref{Lem:LogESPExistenceInS}), $\P$-a.s. for $T$ big enough, for all $(\theta, \tau) \in \Sbf$,  $|  \frac{1}{T} \sum_{t=1}^T                  \e^{\tau'\psi_t(\theta)}
    \psi_t(\theta)  \psi_t(\theta)  ' |_{\det}>0$. Thus, under Assumption \ref{Assp:ExistenceConsistency}(a)-(b)(e)(g)(h), by definition of $M_{3,T}(\theta, \tau)$ in Lemma \ref{Lem:LogESPDecomposition} (p. \pageref{Lem:LogESPDecomposition}),   $\P$-a.s. for $T$ big enough,  for all $(\theta, \tau) \in \Sbf$, for all $k\in \ldsb 1,m \rdsb$,
\begin{eqnarray}
& & \frac{\partial M_{3,T}(\theta, \tau)}{\partial \tau_{k} }  =  -\frac{1}{2T } \frac{1}{ \left|  \frac{1}{T} \sum_{i=1}^T                  \e^{\tau'\psi_i(\theta)}
    \psi_i(\theta)  \psi_i(\theta)  ' \right|_{\det} }\nonumber
\\
& & \times
\left|  \frac{1}{T} \sum_{t=1}^T                  \e^{\tau'\psi_t(\theta)}
    \psi_t(\theta)  \psi_t(\theta)  ' \right|_{\det} \nonumber
\\
& & \times \tr \left\{
\left[  \frac{1}{T} \sum_{t=1}^T                  \e^{\tau'\psi_t(\theta)}
    \psi_t(\theta)  \psi_t(\theta)' \right]^{-1}
     \left[ \frac{1}{T} \sum_{t=1}^T                  \e^{\tau'\psi_t(\theta)}
     \psi_{t, k} (\theta) \psi_t(\theta)  \psi_t(\theta)'     \right]
\right\} \nonumber
\\
& = & -  \frac{1}{2T} \tr \left\{
\left[  \frac{1}{T} \sum_{t=1}^T                  \e^{\tau'\psi_t(\theta)}
    \psi_t(\theta)  \psi_t(\theta)' \right]^{-1}
     \left[ \frac{1}{T} \sum_{t=1}^T                  \e^{\tau'\psi_t(\theta)}
     \psi_{t, k} (\theta) \psi_t(\theta)  \psi_t(\theta)'     \right]
\right\} \label{Eq:VarianceTermDTau}
\end{eqnarray}

\textit{Second derivative $ \frac{\partial^2 M_{3,T}(\theta, \tau)}{ \partial \tau_{h}\partial \tau_{k} } $.} The trace of a derivative is the derivative of the trace because both the trace and differentiation operators are linear \cite[e.g.,][chap. 9 sec. 9]{1988MagnusNeudecker}. Moreover, if $F(.)$ is a differentiable matrix function s.t., for all $x$ in a neighborhood of $\dot{x}$, $\vert F(x) \vert_{\det}\neq 0$, then
$D \left[F(\dot{x})^{-1}\right] =-F(\dot{x})^{-1} [DF(\dot{x})]F(\dot{x})^{-1}$ \citep[e.g.,][ chap. 8 sec. 4]{1988MagnusNeudecker}. Now,  under Assumption \ref{Assp:ExistenceConsistency}(a)-(b)(e)(g)-(h), by Lemma \ref{Lem:LogESPExistenceInS}iv (p. \pageref{Lem:LogESPExistenceInS}), $\P$-a.s. for $T$ big enough, for all $(\theta, \tau) \in \Sbf$,   $\frac{1}{T} \sum_{t=1}^T                \e^{\tau'\psi_t(\theta)}
     \psi_t(\theta)  \psi_t(\theta)' $ is invertible.  Thus, under Assumptions \ref{Assp:ExistenceConsistency}(a)-(b)(e)(g)(h), by the above equation \eqref{Eq:VarianceTermDTau},   $\P$-a.s. for $T$ big enough, for all $(\theta, \tau)$ in a neighborhood of $(\theta_0, \tau_0)$, for all $(h, k)\in \ldsb 1,m\rdsb^2$,
 \begin{eqnarray}
& & \frac{\partial^{2}M_{3,T}(\theta, \tau)}{\partial \tau_{h}\partial \tau_{k}   } \nonumber
\\
& = &  \frac{1}{2T } \tr \left\{
\left[  \frac{1}{T} \sum_{t=1}^T                  \e^{\tau'\psi_t(\theta)}
    \psi_t(\theta)  \psi_t(\theta)' \right]^{-1}
  \left[ \frac{1}{T} \sum_{t=1}^T                  \e^{\tau'\psi_t(\theta)}
     \psi_{t, k} (\theta) \psi_t(\theta)  \psi_t(\theta)'     \right]\nonumber
    \right.
    \\
    & & \hspace{.5in} \left.
    \times \left[  \frac{1}{T} \sum_{t=1}^T                  \e^{\tau'\psi_t(\theta)}
    \psi_t(\theta)  \psi_t(\theta)' \right]^{-1}
     \left[ \frac{1}{T} \sum_{t=1}^T                  \e^{\tau'\psi_t(\theta)}
     \psi_{t, h} (\theta) \psi_t(\theta)  \psi_t(\theta)'     \right]
\right\}\nonumber
\\
&  &   -\frac{1}{2T } \tr \left\{
\left[  \frac{1}{T} \sum_{t=1}^T                  \e^{\tau'\psi_t(\theta)}
    \psi_t(\theta)  \psi_t(\theta)' \right]^{-1}
     \left[ \frac{1}{T} \sum_{t=1}^T                  \e^{\tau'\psi_t(\theta)}
    \psi_{t,k} (\theta)  \psi_{t, h} (\theta) \psi_t(\theta)  \psi_t(\theta)'     \right]
\right\}.\label{Eq:VarianceTermDTauDtau}
\end{eqnarray}

\subsubsection{Derivatives of $\theta \mapsto L_T(\theta, \tau)$ } \textit{First derivative.}
Under Assumption    \ref{Assp:ExistenceConsistency}(a)-(e) and  (g)-(h) and  \ref{Assp:AsymptoticNormality}(a), by Lemma \ref{Lem:LogESPExistenceInS}ii (p. \pageref{Lem:LogESPExistenceInS}), $\Sbf$ contains an open neighborhood of $(\theta_0, \tau_0)$, so that the derivatives derived in $\Sbf$ also hold in a neighborhood of $(\theta_0, \tau_0)$. Thus,   by equations \eqref{Eq:LogETTerm1stDerivative}, \eqref{Eq:DerivativeTerm1stDerivative} and \eqref{Eq:VarianceTerm1stDerivative}
on pp. \pageref{Eq:LogETTerm1stDerivative}-\pageref{Eq:VarianceTerm1stDerivative}. Therefore,  under Assumptions \ref{Assp:ExistenceConsistency} and  \ref{Assp:AsymptoticNormality}(a),  $\P$-a.s. for $T$ big enough, for all $(\theta, \tau)$ in a neighborhood of $(\theta_0, \tau_0)$,

\begin{eqnarray*}
& & \frac{\partial  L_{T}(\theta, \tau ) }{\partial \theta_j} \nonumber\\
& = &
% the first term
\left( 1 - \frac{m}{2T} \right)
\frac{ \frac{1}{T} \sum_{t=1}^T \e^{\tau'\psi_t(\theta)}\tau' \frac{\partial \psi_t(\theta)}{\partial \theta_j}   }{    \frac{1}{T} \sum_{t=1}^T \e^{\tau'\psi_t(\theta)} }  \nonumber\\
%
% the second term
%
& & \hspace{.1in} + \frac{1}{T} \tr\left\{ \left[   \frac{1}{T} \sum_{t=1}^T \e^{\tau'\psi_t(\theta)}
\frac{\partial \psi_t(\theta)}{\partial \theta' }  \right]^{-1}  \right.\nonumber \\
& & \hspace{.2in}   \times   \left[  \frac{1}{T} \sum_{t=1}^T \e^{\tau'\psi_t(\theta)}
\frac{\partial^2 \psi_t(\theta)}{\partial \theta_j \partial \theta' }
%\right.  \\
%& & \hspace{.25in} \left.
 \left.
  +  \frac{1}{T} \sum_{t=1}^T \e^{\tau'\psi_t(\theta)}
\tau' \frac{\partial \psi_t(\theta)}{\partial \theta_j}
\frac{\partial \psi_t(\theta)}{\partial \theta' } \right]   \right\} \nonumber  \\
%
% the third term
%
& & \hspace{.1in} -\frac{1}{2T} \tr
\left\{
 \left[  \frac{1}{T} \sum_{t=1}^T \e^{\tau'\psi_t(\theta)}\psi_t(\theta)  \psi_t(\theta)' \right]^{-1}  \right.\nonumber \\
 & & \hspace{.2in} \negthickspace\times\negthickspace
\left[  \frac{1}{T} \sum_{t=1}^T \negthickspace\e^{\tau'\psi_t(\theta)}  \negthickspace\left\{ \frac{\partial \psi_t(\theta)}{\partial \theta_j}  \psi_t(\theta)'
 \negthickspace+\negthickspace  \psi_t(\theta)   \frac{\partial \psi_t(\theta)'}{\partial \theta_j}     \right\} %   \right. \\
%
%
%& &  \hspace{.25 in} \left.
\left. \negthickspace +  \frac{1}{T} \sum_{t=1}^T \negthickspace \e^{\tau'\psi_t(\theta)}
\tau' \frac{\partial \psi_t(\theta)}{\partial \theta_j}  \psi_t(\theta)  \psi_t(\theta)'  \right] \negthickspace\right\} %\label{Eq:LogESP1stDerivativeWRTTheta}
\end{eqnarray*}
Thus,  evaluated at $(\theta_0, \tau(\theta_0))$,
\begin{eqnarray}
& & \frac{\partial  L_{T}(\theta_{0}, \tau _{0}) }{\partial \theta_j} \nonumber\\
%
% the second term
%
& = &  \frac{1}{T} \tr\left\{ \left[   \frac{1}{T} \sum_{t=1}^T
\frac{\partial \psi_t(\theta_{0})}{\partial \theta' }  \right]^{-1} \left[  \frac{1}{T} \sum_{t=1}^T
\frac{\partial^2 \psi_t(\theta_{0})}{\partial \theta_j \partial \theta' }\right]   \right\}\nonumber \\
%
% the third term
%
& & \hspace{.1in} -\frac{1}{2T} \tr
\left\{
 \left[  \frac{1}{T} \sum_{t=1}^T \psi_t(\theta_{0})  \psi_t(\theta_{0})' \right]^{-1}\left[  \frac{1}{T} \sum_{t=1}^T \negthickspace  \left\{ \frac{\partial \psi_t(\theta_{0})}{\partial \theta_j}  \psi_t(\theta_{0})'
 \negthickspace+\negthickspace  \psi_t(\theta_{0})   \frac{\partial \psi_t(\theta_{0})'}{\partial \theta_j}     \right\}  \right]\right\}
 \label{Eq:LogESP1stDerivativeWRTThetaAtTheta0}
\end{eqnarray}
because $\tau(\theta_0)=0_{m \times 1}$ by  Lemma \ref{Lem:AsTiltingFct}iv (p. \pageref{Lem:AsTiltingFct}).

\subsection{Proof of Theorem \ref{theorem:ConsistencyAsymptoticNormality}(ii)\,: Asymptotic normality} \label{Ap:PfAsymptoticNormality}
The proof of Theorem \ref{theorem:ConsistencyAsymptoticNormality}(ii) (i.e., asymptotic normality) adapts the traditional approach
of expanding the FOCs (first order conditions). The two main differences  w.r.t. the proofs in the entropy literature are the following. Firstly,  instead of expanding the FOC $\left. \frac{\partial L_T(\theta, \tau_{T}(\theta))}{\partial \theta}\right\vert_{\theta=\hat{\theta}_T}$, we expand the approximate FOC $\left. \frac{\partial L_T(\theta, \tau)}{\partial \theta}\right\vert_{(\theta, \tau)=(\hat{\theta}_T, \tau_T(\hat{\theta}_T))}$ combined with the FOC \eqref{Eq:ESPTiltingEquation} for $\tau$  on p. \pageref{Eq:ESPTiltingEquation}. Secondly, we need to control the asymptotic behaviour of the derivatives that come from $ \ln \vert \Sigma_T(\theta)\vert_{\det}$.

\begin{proof}[Core of the proof of Theorem \ref{theorem:ConsistencyAsymptoticNormality}(ii)]
 We prove asymptotic normality  adapting the traditional approach of expanding the FOCs (first order conditions). \ Note that our approximate FOCs  are written as a functionof the $2m$ variables $\theta$ and $\tau$.  In other words, instead of using the implicit function $\tau_{T}(\theta),$  $\tau$ is an estimated parameter and hence the ET equation \eqref{Eq:ESPTiltingEquation} on p. \pageref{Eq:ESPTiltingEquation}  is also included in the expansion.

Under Assumptions \ref{Assp:ExistenceConsistency} and \ref{Assp:AsymptoticNormality}, by Proposition \ref{Prop:AsExpansionEstimator} (p. \pageref{Prop:AsExpansionEstimator}), $\P$-a.s. as $T \rightarrow \infty$,
\begin{eqnarray}
\sqrt{T} \left[ \begin{array}{c} \left( \hat{\theta}_T  - \theta_0 \right)  \\ \tau_T(\hat{\theta}_T) \end{array} \right]&=&-\begin{bmatrix}\E\left[  \frac{\partial \psi(X_1,\theta_{0})}{\partial \theta' }\right]^{-1} \\
0_{m \times m} \\
\end{bmatrix}
\frac{1}{\sqrt{T}}\sum_{t=1}^T\psi_t(\theta_0)+o_{\P}(1) \nonumber
 \\
& \underset{(a)}{\stackrel{D}{\rightarrow }} & -\begin{bmatrix}\E\left[  \frac{\partial \psi(X_1,\theta_{0})}{\partial \theta' }\right]^{-1} \\
0_{m \times m} \\
\end{bmatrix}\mathcal{N}(0,\E\left[\psi(X_1,\theta_{0}) \psi(X_1,\theta_{0})' \right]) \nonumber\\
& \underset{(b)}{\stackrel{D}{=}} & \mathcal{N}\left(0,\begin{bmatrix}\E\left[  \frac{\partial \psi(X_1,\theta_{0})}{\partial \theta' }\right]^{-1} \\
0_{m \times m} \\
\end{bmatrix}\E\left[\psi(X_1,\theta_{0}) \psi(X_1,\theta_{0})' \right] \begin{bmatrix}\E\left[  \frac{\partial \psi(X_1,\theta_{0})'}{\partial \theta }\right]^{-1} & 0_{m\times m} \\
\end{bmatrix}\right) \nonumber \\
& \underset{}{\stackrel{D}{=}}& \mathcal{N} \left(
0, \left( \begin{array}{c  c }
\Sigma(\theta_0)  & 0_{m \times m} \\
0_{m \times m} & 0_{m \times m}
  \end{array}\right)
\right) \text{  }\label{Eq:AsDistThetaTau}
\end{eqnarray}
 where $\Sigma(\theta_0)=\left[\E  \frac{\partial \psi(X_1,\theta_{0})}{\partial \theta' }\right]^{-1}\negthickspace\E\left[\psi(X_1,\theta_{0}) \psi(X_1,\theta_{0})' \right] \left[\E  \frac{\partial \psi(X_1,\theta_{0})'}{\partial \theta }\right]^{-1}$. \textit{(a)}  Under Assumption \ref{Assp:ExistenceConsistency}(a)-(c) and (g), by the Lindeberg-L{\'e}vy CLT theorem, 

\noindent
$ \frac{1}{\sqrt{T}}\sum_{t=1}^T\psi_t(\theta_0)\stackrel{D}{\rightarrow} \mathcal{N}(0,\E\left[\psi(X_1,\theta_{0}) \psi(X_1,\theta_{0})' ]\right )$, as $T \rightarrow \infty$. \textit{(b)} Firstly,  the
minus sign can be discarded because of the symmetry of the Gaussian distribution. Secondly, if $X$ is a random vector and $F$ is a (deterministic) matrix, then $\V(FX)=F\V(X)F'$.
\end{proof}

\begin{prop}[Asymptotic expansion of $\sqrt{T}(\hat{\theta}_T- \theta_0)$] \label{Prop:AsExpansionEstimator}
Under Assumptions \ref{Assp:ExistenceConsistency} and \ref{Assp:AsymptoticNormality}, $\P$-a.s. as $T\rightarrow \infty $,
\begin{eqnarray*}
 \sqrt{T} \left[ \begin{array}{c} \left( \hat{\theta}_T  - \theta_0 \right)  \\ \tau_T(\hat{\theta}_T) \end{array} \right]=-\begin{bmatrix}\E\left[  \frac{\partial \psi(X_1,\theta_{0})}{\partial \theta' }\right]^{-1} \\
0_{m \times m} \\
\end{bmatrix}\frac{1}{\sqrt{T}}\sum_{t=1}^T\psi_t(\theta_0)+o_{\P}(1)
\end{eqnarray*}
\end{prop}
\begin{proof}The function  $L_T(\theta, \tau)$ is well-defined and twice continuously differentiable in a neighborhood of $(\theta_0'\; \tau(\theta_0)')$   $\P$-a.s. for $T$ big enough by subsection \ref{Sec:LTAndDerivatives} (p. \pageref{Sec:LTAndDerivatives}),  under Assumptions \ref{Assp:ExistenceConsistency} and \ref{Assp:AsymptoticNormality}(a). Similarly, let  $S_T(\theta, \tau):= \frac{1}{T}\sum_{t=1}^T \e^{\tau'\psi_t(\theta)}\psi_t(\theta)$, which is  continuously differentiable in a neighborhood of $(\theta_0' \; \tau(\theta_0)')$ by Assumption \ref{Assp:ExistenceConsistency}(a)(b). Now, under Assumption \ref{Assp:ExistenceConsistency}, by Theorem \ref{theorem:ConsistencyAsymptoticNormality}i (p. \pageref{theorem:ConsistencyAsymptoticNormality}),  Lemma \ref{Lem:Schennachtheorem10PfFirstSteps}iii (p. \pageref{Lem:Schennachtheorem10PfFirstSteps}) and Lemma \ref{Lem:AsTiltingFct}iv (p. \pageref{Lem:AsTiltingFct}), $\P$-a.s., $\hat{\theta}_T \rightarrow \theta_0$ and $\tau_T(\hat{\theta}_T )\rightarrow \tau(\theta_0)$, where $\tau(\theta_0)=0_{m \times 1}$, so that $\P$-a.s. for $T$ big enough, $(\hat{\theta}_T ' \; \tau_T(\hat{\theta}_T)')$ is in any arbitrary small neighborhood of $(\theta_0'\; \tau(\theta_0)')$. Therefore,  under Assumption \ref{Assp:ExistenceConsistency} and \ref{Assp:AsymptoticNormality}(a), a stochastic first-order Taylor-Lagrange expansion \cite[Lemma 3]{1969Jen} around $(\theta_0, \tau(\theta_0))$ evaluated at $(\hat{\theta}_T, \tau_T(\hat{\theta}_T))$ yields, $\P$-a.s. for $T$ big enough
\begin{eqnarray}
%
% & & \left[ \begin{array}{c} \frac{\partial { L}_T(\theta, \tau_T(\theta)) }{\partial \theta}  \\ S_T(\hat{\theta}_T, \tau_T(\hat{\theta}_T)) \end{array} \right]
% %
% =
% %
% \left[ \begin{array}{c} \frac{\partial { L}_T(\theta_0, \tau_0) }{\partial \theta} \\ S_T(\theta_0, \tau_0 ) \end{array} \right]
% %
% +
% %
% \left[ \begin{array}{c c } \frac{\partial^2 { L}_T(\bar{\theta}_T(\theta), \bar{\tau}_T) }{\partial \theta' \partial \theta}
%                          & \frac{\partial^2 { L}_T(\bar{\theta}_T(\theta), \bar{\tau}_T) }{\partial \tau' \partial \theta} \\
%                          \frac{\partial  S_T(\bar{\theta}_T(\theta), \bar{\tau}_T) }{\partial \theta'}   & \frac{ \partial S_n(\bar{\theta}_T(\theta), \bar{\tau}_T) }{\partial \tau'} \end{array} \right]
% \left[ \begin{array}{c} \left( \theta  - \theta_0 \right)  \\ \tau \end{array} \right]\\
%
&  & \left[ \begin{array}{c} \frac{\partial { L}_T(\hat{\theta}_T, \tau_T(\hat{\theta}_T)) }{\partial \theta}  \\ S_T(\hat{\theta}_T, \tau_T(\hat{\theta}_T)) \end{array} \right]
=
\left[ \begin{array}{c} \frac{\partial { L}_T(\theta_0, \tau(\theta_0)) }{\partial \theta} \\ S_T(\theta_0, \tau(\theta_0) ) \end{array} \right]
+
\left[ \begin{array}{c c } \frac{\partial^2 { L}_T(\bar{\theta}_T, \bar{\tau}_T) }{\partial \theta' \partial \theta}
                         & \frac{\partial^2 { L}_T(\bar{\theta}_T, \bar{\tau}_T) }{\partial \tau' \partial \theta} \\
                         \frac{\partial  S_T(\bar{\theta}_T, \bar{\tau}_T) }{\partial \theta'}   & \frac{ \partial S_T(\bar{\theta}_T, \bar{\tau}_T) }{\partial \tau'} \end{array} \right]
\left[ \begin{array}{c} \left( \hat{\theta}_T  - \theta_0 \right)  \\ \tau_T(\hat{\theta}_T) \end{array} \right]\label{Eq:ExpansionAppFOC}
\end{eqnarray}
where $\bar{\theta}_T$ and $\bar{\tau}_T$ are between $\hat{\theta}_T$ and $\theta_0$, and between  $\tau_T(\hat{\theta}_T)$ and $\tau(\theta_0)$, respectively.  Under Assumptions \ref{Assp:ExistenceConsistency} and \ref{Assp:AsymptoticNormality},  by Lemma \ref{Lem:ApproximateFOC} (p. \pageref{Lem:ApproximateFOC}) and   by definition of $\tau_T(.)$ (equation \ref{Eq:ESPTiltingEquation} on p. \pageref{Eq:ESPTiltingEquation}), $\frac{\partial { L}_T(\hat{\theta}_T, \tau_T(\hat{\theta}_T)) }{\partial \theta}=O(T^{-1}) $ and  $S_T(\hat{\theta}_T, \tau_T(\hat{\theta}_T))=0$, respectively. Moreover, under Assumptions \ref{Assp:ExistenceConsistency} and \ref{Assp:AsymptoticNormality}, by Theorem \ref{theorem:ConsistencyAsymptoticNormality}i, Lemma \ref{Lem:Schennachtheorem10PfFirstSteps}iii (p. \pageref{Lem:Schennachtheorem10PfFirstSteps}) and Lemma \ref{Lem:PartialLAndS}ii (p. \pageref{Lem:PartialLAndS}), $\P$-a.s. for $T$ big enough, $\left[ \begin{array}{c c } \frac{\partial^2 { L}_T(\bar{\theta}_T, \bar{\tau}_T) }{\partial \theta' \partial \theta}
                         & \frac{\partial^2 { L}_T(\bar{\theta}_T, \bar{\tau}_T) }{\partial \tau' \partial \theta} \\
                         \frac{\partial  S_T(\bar{\theta}_T, \bar{\tau}_T) }{\partial \theta'}   & \frac{ \partial S_T(\bar{\theta}_T, \bar{\tau}_T) }{\partial \tau'} \end{array} \right]
$ is invertible. Thus, under Assumptions \ref{Assp:ExistenceConsistency} and \ref{Assp:AsymptoticNormality},
$\P$-a.s. for $T$ big enough,
\begin{eqnarray*}
& &  \sqrt{T} \left[ \begin{array}{c} \left( \hat{\theta}_T  - \theta_0 \right)  \\ \tau_T(\hat{\theta}_T) \end{array} \right]\\
& = &
-\left[ \begin{array}{c c } \frac{\partial^2 { L}_T(\bar{\theta}_T, \bar{\tau}_T) }{\partial \theta' \partial \theta}
                         & \frac{\partial^2 { L}_T(\bar{\theta}_T, \bar{\tau}_T) }{\partial \tau' \partial \theta} \\
                         \frac{\partial  S_T(\bar{\theta}_T, \bar{\tau}_T) }{\partial \theta'}   & \frac{ \partial S_T(\bar{\theta}_T, \bar{\tau}_T) }{\partial \tau'} \end{array} \right]^{-1}\sqrt{T}
\left[ \begin{array}{c} \frac{\partial { L}_T(\theta_0, \tau(\theta_0)) }{\partial \theta}+O(T^{-1}) \\ S_T(\theta_0, \tau(\theta_0) ) \end{array} \right]
 \\
&  \stackrel{(a)}{=} &
-\left[ \begin{array}{c c } \frac{\partial^2 { L}_T(\bar{\theta}_T, \bar{\tau}_T) }{\partial \theta' \partial \theta}
                         & \frac{\partial^2 { L}_T(\bar{\theta}_T, \bar{\tau}_T) }{\partial \tau' \partial \theta} \\
                         \frac{\partial  S_T(\bar{\theta}_T, \bar{\tau}_T) }{\partial \theta'}   & \frac{ \partial S_T(\bar{\theta}_T, \bar{\tau}_T) }{\partial \tau'} \end{array} \right]^{-1}
\left[ \begin{array}{c} O(T^{-\frac{1}{2}}) \\ \sqrt{T} \frac{1}{T}\sum_{t=1}^T\psi_t(\theta_0) \end{array} \right] \\
 &  \stackrel{(b)}{=} &
-\left[ \begin{array}{c c } -\Sigma(\theta_0)
                         & \E\left[  \frac{\partial \psi(X_1,\theta_{0})}{\partial \theta' }\right]^{-1} \\
                         \E\left[  \frac{\partial \psi(X_1,\theta_{0})'}{\partial \theta }\right]^{-1}   &\ 0_{m \times m} \end{array} \right]
\left[ \begin{array}{c} O(T^{-\frac{1}{2}}) \\ \sqrt{T} \frac{1}{T}\sum_{t=1}^T\psi_t(\theta_0) \end{array} \right] \\
&  &  -\left\{ \left[ \begin{array}{c c } \frac{\partial^2 { L}_T(\bar{\theta}_T, \bar{\tau}_T) }{\partial \theta' \partial \theta}
                         & \frac{\partial^2 { L}_T(\bar{\theta}_T, \bar{\tau}_T) }{\partial \tau' \partial \theta} \\
                         \frac{\partial  S_T(\bar{\theta}_T, \bar{\tau}_T) }{\partial \theta'}   & \frac{ \partial S_T(\bar{\theta}_T, \bar{\tau}_T) }{\partial \tau'} \end{array} \right]^{-1}
-\left[ \begin{array}{c c } -\Sigma(\theta_0)
                         & \E\left[  \frac{\partial \psi(X_1,\theta_{0})}{\partial \theta' }\right]^{-1} \\
                         \E\left[  \frac{\partial \psi(X_1,\theta_{0})'}{\partial \theta }\right]^{-1}   &\ 0_{m \times m} \end{array} \right]
\right\}\left[ \begin{array}{c} O(T^{-\frac{1}{2}}) \\ \sqrt{T} \frac{1}{T}\sum_{t=1}^T\psi_t(\theta_0) \end{array} \right]\\
%
% %&  = &
% %
% -\left[ \begin{array}{c}
%    M_{\psi 0}^{-1}
%  \sqrt{n} \Psi_n( \alpha_0) \\ 0 \end{array} \right]
% %
% + {  O}_p\left( n^{-1/2} \right).
% \end{eqnarray*}
% This implies the asymptotic distribution
% $$
% \sqrt{n} \left[  \begin{array}{c}  \hat{\alpha}_{espl} - \alpha_0 \\
% \hat{\tau}_{espl}  \end{array} \right]
%
& \underset{}{\stackrel{(c)}{=-}}  & \begin{bmatrix}\E\left[  \frac{\partial \psi(X_1,\theta_{0})}{\partial \theta' }\right]^{-1} \\
0_{m \times m} \\
\end{bmatrix}
\frac{1}{\sqrt{T}}\sum_{t=1}^T\psi_t(\theta_0)+o_{\P}(1)
\end{eqnarray*}
 where $\Sigma(\theta_0)=\left[\E  \frac{\partial \psi(X_1,\theta_{0})}{\partial \theta' }\right]^{-1}\negthickspace\E\left[\psi(X_1,\theta_{0}) \psi(X_1,\theta_{0})' \right] \left[\E  \frac{\partial \psi(X_1,\theta_{0})'}{\partial \theta }\right]^{-1}$. \textit{(a)} Firstly, under Assumptions \ref{Assp:ExistenceConsistency} and \ref{Assp:AsymptoticNormality}, by Lemma \ref{Lem:PartialLTheta0Tau0}i (p. \pageref{Lem:PartialLTheta0Tau0}), $\P$-a.s. as $T \rightarrow \infty$,  $\frac{\partial  L_{T}(\theta_0, \tau(\theta_0) ) }{\partial \theta_j}=O(T^{-1})$, so that $\sqrt{T}\left[\frac{\partial  L_{T}(\theta_0, \tau(\theta_0) ) }{\partial \theta_j}+O(T^{-1})\right]=O(T^{-\frac{1}{2}})$. Secondly, note that $S_T(\theta_0, \tau(\theta_0) )=\frac{1}{T}\sum_{t=1}^T\psi_t(\theta_0)$. \textit{(b)} Add and subtract the matrix $\left[ \begin{array}{c c } -\Sigma(\theta_0)
                         & \E\left[  \frac{\partial \psi(X_1,\theta_{0})}{\partial \theta' }\right]^{-1} \\
                         \E\left[  \frac{\partial \psi(X_1,\theta_{0})'}{\partial \theta }\right]^{-1}   &\ 0_{m \times m} \end{array} \right] $. \textit{(c)} Firstly, the first  column of the first square matrix
cancels out because the first element  of the vector is zero.
Secondly, under Assumptions \ref{Assp:ExistenceConsistency} and \ref{Assp:AsymptoticNormality},  by Lemma \ref{Lem:PartialLAndS}iii (p. \pageref{Lem:PartialLAndS}) and Theorem \ref{theorem:ConsistencyAsymptoticNormality}i (p. \pageref{theorem:ConsistencyAsymptoticNormality}), $\P$-a.s. as $T \rightarrow \infty$, the curly bracket is $o(1)$, and, under Assumption \ref{Assp:ExistenceConsistency}(a)-(c) and (g), by the Lindeberg-Lévy CLT,  $ \frac{1}{\sqrt{T}}\sum_{t=1}^T\psi_t(\theta_0)=O_{\P}(1)$, as $T \rightarrow \infty$.  
\end{proof}

\begin{rk}[Alternative approximate FOC] In the proof of Theorem \ref{theorem:ConsistencyAsymptoticNormality}ii, it is possible to use the approximate FOC $\frac{\partial M_{1,T}(\hat{\theta}_T, \tau_T(\hat{\theta}_T))}{\partial \theta}=O(T^{-1})$ instead of the approximate FOC $\frac{\partial { L}_T(\hat{\theta}_T, \tau_T(\hat{\theta}_T)) }{\partial \theta}=O(T^{-1})$. Under Assumption \ref{Assp:ExistenceConsistency} and \ref{Assp:AsymptoticNormality} (with $k_2 \in \ldsb 1,3\rdsb$ and $j \in \ldsb 0,2\rdsb$ in its part b), by Lemma \ref{Lem:LogESPDecomposition} (p. \pageref{Lem:LogESPDecomposition}) and \ref{Lem:Bound2ndPartialLPartialTheta}v-vii,xii-xiv (p. \pageref{Lem:Bound2ndPartialLPartialTheta}) and the ULLN \`a la Wald,  $\frac{\partial { L}_T(\hat{\theta}_T, \tau_T(\hat{\theta}_T)) }{\partial \theta}=\frac{\partial M_{1,T}(\hat{\theta}_T, \tau_T(\hat{\theta}_T))}{\partial \theta}+ \frac{\partial M_{2,T}(\hat{\theta}_T, \tau_T(\hat{\theta}_T))}{\partial \theta}+\frac{\partial M_{3,T}(\hat{\theta}_T, \tau_T(\hat{\theta}_T))}{\partial \theta}=\frac{\partial M_{1,T}(\hat{\theta}_T, \tau_T(\hat{\theta}_T))}{\partial \theta}+O(T^{-1})$. The approximate FOC  $\frac{\partial M_{1,T}(\hat{\theta}_T, \tau_T(\hat{\theta}_T))}{\partial \theta}=O(T^{-1})$ would lead to replace  expansion \eqref{Eq:ExpansionAppFOC} on p. \pageref{Eq:ExpansionAppFOC}  with the following expansion 
\begin{eqnarray*}
%
% & & \left[ \begin{array}{c} \frac{\partial { L}_T(\theta, \tau_T(\theta)) }{\partial \theta}  \\ S_T(\hat{\theta}_T, \tau_T(\hat{\theta}_T)) \end{array} \right]
% %
% =
% %
% \left[ \begin{array}{c} \frac{\partial { L}_T(\theta_0, \tau_0) }{\partial \theta} \\ S_T(\theta_0, \tau_0 ) \end{array} \right]
% %
% +
% %
% \left[ \begin{array}{c c } \frac{\partial^2 { L}_T(\bar{\theta}_T(\theta), \bar{\tau}_T) }{\partial \theta' \partial \theta}
%                          & \frac{\partial^2 { L}_T(\bar{\theta}_T(\theta), \bar{\tau}_T) }{\partial \tau' \partial \theta} \\
%                          \frac{\partial  S_T(\bar{\theta}_T(\theta), \bar{\tau}_T) }{\partial \theta'}   & \frac{ \partial S_n(\bar{\theta}_T(\theta), \bar{\tau}_T) }{\partial \tau'} \end{array} \right]
% \left[ \begin{array}{c} \left( \theta  - \theta_0 \right)  \\ \tau \end{array} \right]\\
%
&  & \left[ \begin{array}{c} \frac{\partial M_{1,T}(\hat{\theta}_T, \tau_T(\hat{\theta}_T)) }{\partial \theta}  \\ S_T(\hat{\theta}_T, \tau_T(\hat{\theta}_T)) \end{array} \right]
=
\left[ \begin{array}{c} \frac{\partial M_{1,T}(\theta_0, \tau(\theta_0)) }{\partial \theta} \\ S_T(\theta_0, \tau(\theta_0) ) \end{array} \right]
+
\left[ \begin{array}{c c } \frac{\partial^2 M_{1,T}(\bar{\theta}_T, \bar{\tau}_T) }{\partial \theta' \partial \theta}
                         & \frac{\partial^2 M_{1,T}(\bar{\theta}_T, \bar{\tau}_T) }{\partial \tau' \partial \theta} \\
                         \frac{\partial  S_T(\bar{\theta}_T, \bar{\tau}_T) }{\partial \theta'}   & \frac{ \partial S_T(\bar{\theta}_T, \bar{\tau}_T) }{\partial \tau'} \end{array} \right]
\left[ \begin{array}{c} \left( \hat{\theta}_T  - \theta_0 \right)  \\ \tau_T(\hat{\theta}_T) \end{array} \right]
\end{eqnarray*}
where $\frac{\partial^2 M_{1,T}(\bar{\theta}_T, \bar{\tau}_T) }{\partial \theta' \partial \theta}$ and $\frac{\partial^2 M_{1,T}(\bar{\theta}_T, \bar{\tau}_T) }{\partial \theta' \partial \theta} $ can easily be controlled by Lemma  \ref{Lem:Bound2ndPartialLPartialTheta}i-iv (p. \pageref{Lem:Bound2ndPartialLPartialTheta}),  Lemma \ref{Lem:Bound2ndPartialLPartialTau}i-v (p. \pageref{Lem:Bound2ndPartialLPartialTau}), Lemma \ref{Lem:dLdTaudTau}i-iii (p. \pageref{Lem:dLdTaudTau}) and  ULLN \`a la Wald under Assumptions \ref{Assp:ExistenceConsistency} and \ref{Assp:AsymptoticNormality} (with $k_2 \in \ldsb 1,3\rdsb$ and $j \in \ldsb 0,2\rdsb$ in its part b). The approximate FOC $\frac{\partial M_{1,T}(\hat{\theta}_T, \tau_T(\hat{\theta}_T))}{\partial \theta}=O(T^{-1})$ requires less assumptions than the approximate FOC $\frac{\partial { L}_T(\hat{\theta}_T, \tau_T(\hat{\theta}_T)) }{\partial \theta}=O(T^{-1})$ because it does not require to control the 2nd derivatives of  $M_{2,T}(\theta, \tau) $ and  $M_{3,T}(\theta, \tau) $. However, it would not save space and it would require to add one more block of assumptions because our proof of Theorem \ref{theorem:TrinityPlus1}   requires the full   Assumption  \ref{Assp:AsymptoticNormality}.
\hfill $\diamond $\end{rk}

\begin{lem}\label{Lem:PartialLAndS} Under Assumptions \ref{Assp:ExistenceConsistency} and \ref{Assp:AsymptoticNormality},  \begin{enumerate}
\item[(i)]for any sequence $(\theta_T, \tau_T)_{T \in \N}$ converging to $(\theta_0, \tau(\theta_0))$, $\P$-a.s. as $T \rightarrow \infty$, \\ $\left[ \begin{array}{c c } \frac{\partial^2 { L}_T({\theta}_T, {\tau}_T) }{\partial \theta' \partial \theta}
                         & \frac{\partial^2 { L}_T({\theta}_T, {\tau}_T) }{\partial \tau' \partial \theta} \\
                         \frac{\partial  S_T({\theta}_T, {\tau}_T) }{\partial \theta'}   & \frac{ \partial S_T({\theta}_T, {\tau}_T) }{\partial \tau'} \end{array} \right]\rightarrow \left[ \begin{array}{c c } 0_{m \times m}
                         & \E\left[  \frac{\partial \psi(X_1,\theta_{0})}{\partial \theta' }\right]' \\
                         \E\left[  \frac{\partial \psi(X_1,\theta_{0})}{\partial \theta' }\right]   &\ \E\left[\psi(X_1,\theta_{0}) \psi(X_1,\theta_{0})' \right] \end{array} \right]$;
\item[(ii)] $\left[ \begin{array}{c c } 0_{m \times m}
                         & \E\left[  \frac{\partial \psi(X_1,\theta_{0})}{\partial \theta' }\right]' \\
                         \E\left[  \frac{\partial \psi(X_1,\theta_{0})}{\partial \theta' }\right]   &\ \E\left[\psi(X_1,\theta_{0}) \psi(X_1,\theta_{0})' \right] \end{array} \right]$ is invertible,  so that, for any sequence $(\theta_T, \tau_T)_{T \in \N}$ converging to $(\theta_0, \tau(\theta_0))$, $\P$-a.s., for $T$ big enough, the matrix $\left[ \begin{array}{c c } \frac{\partial^2 { L}_T({\theta}_T, {\tau}_T) }{\partial \theta' \partial \theta}
                         & \frac{\partial^2 { L}_T({\theta}_T, {\tau}_T) }{\partial \tau' \partial \theta} \\
                         \frac{\partial  S_T({\theta}_T, {\tau}_T) }{\partial \theta'}   & \frac{ \partial S_T({\theta}_T, {\tau}_T) }{\partial \tau'} \end{array} \right]$ is invertible; and
\item[(iii)] for any sequence $(\theta_T, \tau_T)_{T \in \N}$ converging to $(\theta_0, \tau(\theta_0))$, $\P$-a.s. as $T \rightarrow \infty$,\\  $\left[ \begin{array}{c c } \frac{\partial^2 { L}_T(\theta_T, \tau_T) }{\partial \theta' \partial \theta}
                         & \frac{\partial^2 { L}_T(\theta_T, \tau_T) }{\partial \tau' \partial \theta} \\
                         \frac{\partial  S_T(\theta_T, \tau_T) }{\partial \theta'}   & \frac{ \partial S_T(\theta_T, \tau_T) }{\partial \tau'} \end{array} \right]^{-1}\rightarrow \left[ \begin{array}{c c } -\Sigma(\theta_0)
                         & \E\left[  \frac{\partial \psi(X_{1},\theta_{0})}{\partial \theta' }\right]^{-1} \\
                         \E\left[  \frac{\partial \psi(X_{1},\theta_{0})'}{\partial \theta }\right]^{-1}   &\ 0_{m \times m} \end{array} \right]$, where\\
$\displaystyle \left[ \begin{array}{c c } -\Sigma(\theta_0)
                         & \E\left[  \frac{\partial \psi(X_{1},\theta_{0})}{\partial \theta' }\right]^{-1} \\
                         \E\left[  \frac{\partial \psi(X_{1},\theta_{0})'}{\partial \theta }\right]^{-1}   &\ 0_{m \times m} \end{array} \right]=\left[ \begin{array}{c c } 0_{m \times m}
                         & \E\left[  \frac{\partial \psi(X_1,\theta_{0})}{\partial \theta' }\right]' \\
                         \E\left[  \frac{\partial \psi(X_{1},\theta_{0})}{\partial \theta' }\right]   &\ \E\left[\psi(X_1,\theta_{0}) \psi(X_1,\theta_{0})' \right] \end{array} \right]^{-1} $
.
\end{enumerate}
\end{lem}
\begin{proof}\textit{(i)} Under Assumptions \ref{Assp:ExistenceConsistency} and \ref{Assp:AsymptoticNormality}, it follows from  Lemma \ref{Lem:PartialLTheta0Tau0}ii and iii (p. \pageref{Lem:PartialLTheta0Tau0}) and Lemma \ref{Lem:ETEquationFirstDerivatives} (p. \pageref{Lem:ETEquationFirstDerivatives}), given that $\tau(\theta_0)=0_{m \times 1}$ by  Lemma \ref{Lem:AsTiltingFct}ii (p. \pageref{Lem:AsTiltingFct}) and Assumption \ref{Assp:ExistenceConsistency}(c), under  Assumption \ref{Assp:ExistenceConsistency}(a)(b)(d)(e)(g) and (h).

\textit{(ii)}
Assumption \ref{Assp:ExistenceConsistency}(h) implies the invertibility of

\noindent
$\E\left[\e^{\tau(\theta_{0})' \psi(X_{1},\theta_{0})}\psi(X_{1},\theta_{0}) \psi(X_{1},\theta_{0})' \right]=\E\left[\psi(X_1,\theta_{0}) \psi(X_1,\theta_{0})' \right]$ and $\E\left[\e^{\tau(\theta_0)' \psi(X_{1},\theta_{0})}  \frac{\partial \psi(X_{1},\theta_{0})}{\partial \theta' }\right]=\E\left[  \frac{\partial \psi(X_1,\theta_{0})}{\partial \theta' }\right]$  because    $\tau(\theta_0)=0_{m \times 1}$ by Lemma \ref{Lem:AsTiltingFct}iv (p. \pageref{Lem:AsTiltingFct}) under  Assumption \ref{Assp:ExistenceConsistency}(a)-(e)(g)-(h).
Thus, $\E\left[  \frac{\partial \psi(X_1,\theta_{0})'}{\partial \theta }\right]\E\left[\psi(X_1,\theta_{0}) \psi(X_1,\theta_{0})' \right]^{-1}\E\left[  \frac{\partial \psi(X_1,\theta_{0})}{\partial \theta' }\right]$ is also invertible, so that the first part of the statement (ii) follows from  Lemma \ref{Lem:BlockMatrixInverse}ii (p. \pageref{Lem:BlockMatrixInverse}) with $A=0_{m \times m}$, $B=\E\left[  \frac{\partial \psi(X_1,\theta_{0})'}{\partial \theta }\right]$, $C=\E\left[  \frac{\partial \psi(X_1,\theta_{0})}{\partial \theta' }\right]$ and  $D=\E\left[\psi(X_1,\theta_{0}) \psi(X_1,\theta_{0})' \right]$ . Then, the second
part of the statement follows from a trivial case of the Lemma \ref{Lem:UniFiniteSampleInvertibilityFromUniAsInvertibility} (p. \pageref{Lem:UniFiniteSampleInvertibilityFromUniAsInvertibility}).\\
\textit{(iii)} Under under  Assumption \ref{Assp:ExistenceConsistency}(a)(b)(c)(d)(e)(g)(h), by the statement (ii) of the present lemma, the limiting matrix is invertible. Thus, by the inverse formula for partitioned matrices \citep[e.g.,][Chap. 1 Sec. 11]{1988MagnusNeudecker},
\begin{eqnarray*}
\left[ \begin{array}{c c } 0_{m \times m}
                         & \E\left[  \frac{\partial \psi(X_1,\theta_{0})}{\partial \theta' }\right]' \\
                         \E\left[  \frac{\partial \psi(X_1,\theta_{0})}{\partial \theta' }\right]   &\ \E\left[\psi(X_1,\theta_{0}) \psi(X_1,\theta_{0})' \right] \end{array} \right]^{-1}=\left[ \begin{array}{c c } -\Sigma(\theta_0)
                         & \E\left[  \frac{\partial \psi(X_1,\theta_{0})}{\partial \theta' }\right]^{-1} \\
                         \E\left[  \frac{\partial \psi(X_1,\theta_{0})'}{\partial \theta }\right]^{-1}   &\ 0_{m \times 1} \end{array} \right]
\end{eqnarray*}
because $(-M'V^{-1}M)^{-1}= -M^{-1}V (M')^{-1}:=-\Sigma(\theta_0)$. Then, the result follows from the continuity of the inverse transformation  \citep[e.g.,][Theorem 9.8]{1953Rudin}.
\end{proof}

\begin{lem}\label{Lem:PartialLTheta0Tau0} Under Assumptions \ref{Assp:ExistenceConsistency} and \ref{Assp:AsymptoticNormality},
\begin{enumerate}
\item[(i)] $\P$-a.s. as $T \rightarrow \infty$,

\noindent
 $T\frac{\partial  L_{T}(\theta_0, \tau(\theta_0) ) }{\partial \theta_j} \rightarrow \tr\left\{  \left[\E \frac{\partial \psi(X_1, \theta_0)}{\partial \theta'}\right]^{-1} \E\left[ \frac{\partial^2 \psi_t(\theta_{0})}{\partial \theta_j \partial \theta' }\right]   \right\} -\frac{1}{2} \tr
\Big\{
 \left[\E \psi(X_1, \theta_0)\psi(X_1, \theta_0)' \right]^{-1} \\
 \times \Big[  \E \left[\frac{\partial \psi(X_1, \theta_0)}{\partial \theta_j} \psi(X_1, \theta_0)'\right]+ \E \left[\psi(X_1, \theta_0)\frac{\partial \psi(X_1, \theta_0)}{\partial \theta_j} '\right] \Big]\Big\}$, so that $\frac{\partial  L_{T}(\theta_0, \tau(\theta_0) ) }{\partial \theta_j}=O(T^{-1})$;
\item[(ii)] for any sequence $(\theta_T, \tau_T)_{T \in \N}$ converging to $(\theta_0, \tau(\theta_0))$, $\displaystyle\left \vert\frac{\partial^{2}  L_{T}(\theta_{T}, \tau_{T} ) }{\partial \theta_j \partial \theta_\ell}\right\vert=o(1) $, $\P$-a.s. as $T \rightarrow \infty$;
\item[(iii)] for any sequence $(\theta_T, \tau_T)_{T \in \N}$ converging to $(\theta_0, \tau(\theta_0))$, $\displaystyle\left \vert\frac{\partial^{2}  L_{T}(\theta_{T}, \tau_{T} ) }{\partial \theta' \partial \tau}-\E\left[  \frac{\partial \psi(X_{1},\theta_{0})}{\partial \theta' }\right]\right \vert=o(1) $, $\P$-a.s. as $T \rightarrow \infty$.
\end{enumerate}

\end{lem}
\begin{proof}\textit{(i)} By equation \eqref{Eq:LogESP1stDerivativeWRTThetaAtTheta0} on p. \pageref{Eq:LogESP1stDerivativeWRTThetaAtTheta0},  under Assumptions \ref{Assp:ExistenceConsistency} and  \ref{Assp:AsymptoticNormality}(a), for all $j \in \ldsb 1, m\rdsb $, $\P$-a.s. for $T$ big enough,  evaluating $\frac{\partial  L_{T}(\theta, \tau) }{\partial \theta_j}$ at $(\theta_0, \tau(\theta_0))$ yields

\begin{eqnarray}
& & \frac{\partial  L_{T}(\theta_{0}, \tau(\theta_0)) }{\partial \theta_j} \nonumber\\
%
% the second term
%
& = &  \frac{1}{T} \tr\left\{ \left[   \frac{1}{T} \sum_{t=1}^T
\frac{\partial \psi_t(\theta_{0})}{\partial \theta' }  \right]^{-1} \left[  \frac{1}{T} \sum_{t=1}^T
\frac{\partial^2 \psi_t(\theta_{0})}{\partial \theta_j \partial \theta' }\right]   \right\}\nonumber \\
%
% the third term
%
& & \hspace{.1in} -\frac{1}{2T} \tr
\left\{
 \left[  \frac{1}{T} \sum_{t=1}^T \psi_t(\theta_{0})  \psi_t(\theta_{0})' \right]^{-1}\left[  \frac{1}{T} \sum_{t=1}^T \negthickspace  \left\{ \frac{\partial \psi_t(\theta_{0})}{\partial \theta_j}  \psi_t(\theta_{0})'
 \negthickspace+\negthickspace  \psi_t(\theta_{0})   \frac{\partial \psi_t(\theta_{0})'}{\partial \theta_j}     \right\}  \right]\right\}.
\end{eqnarray}
Now, under Assumption \ref{Assp:ExistenceConsistency}(a)(b), \begin{itemize}
\item under additional Assumption \ref{Assp:ExistenceConsistency}(h), by the LLN and Lemma \ref{Lem:UniFiniteSampleInvertibilityFromUniAsInvertibility} (p. \pageref{Lem:UniFiniteSampleInvertibilityFromUniAsInvertibility}), $\P$-a.s. for $T$ big enough, $\frac{1}{T} \sum_{t=1}^T
\frac{\partial \psi_t(\theta_{0})}{\partial \theta' }$ is invertible, so that $\left[   \frac{1}{T} \sum_{t=1}^T
\frac{\partial \psi_t(\theta_{0})}{\partial \theta' }  \right]^{-1} \rightarrow \left[\E \frac{\partial \psi(X_1, \theta_0)}{\partial \theta'}\right]^{-1}$;
\item under additional Assumption \ref{Assp:AsymptoticNormality}(b), by the LLN, $\frac{1}{T} \sum_{t=1}^T
\frac{\partial^2 \psi_t(\theta_{0})}{\partial \theta_j \partial \theta' }\rightarrow \E\left[ \frac{\partial^2 \psi_t(\theta_{0})}{\partial \theta_j \partial \theta' }\right] $;
\item under additional Assumption  \ref{Assp:ExistenceConsistency}(h), by the LLN and Lemma \ref{Lem:UniFiniteSampleInvertibilityFromUniAsInvertibility} (p. \pageref{Lem:UniFiniteSampleInvertibilityFromUniAsInvertibility}), $\P$-a.s. for $T$ big enough,  $\frac{1}{T} \sum_{t=1}^T \psi_t(\theta_{0})  \psi_t(\theta_{0})'$ is invertible,  so that   $\left[  \frac{1}{T} \sum_{t=1}^T \psi_t(\theta_{0})  \psi_t(\theta_{0})' \right]^{-1}\rightarrow \left[\E \psi(X_1, \theta_0)\psi(X_1, \theta_0)' \right]^{-1}$; and
\item under additional Assumption \ref{Assp:ExistenceConsistency}(f)(g), by the Cauchy-Schwarz inequality and the monotonicity of integration, 

\noindent
$\E \left[\frac{\partial \psi(X_1, \theta_0)}{\partial \theta_j}\psi(X_1, \theta_0)\right]\leqslant \sqrt{\E\left[\sup_{\theta\in \T}\vert \frac{\partial \psi(X_1, \theta)}{\partial \theta_j}\vert^2\right] \E\left[\sup_{\theta\in \T}\vert \psi(X_1, \theta)\vert^2\right]}< \infty$, so that, by the LLN,

\noindent
 $   \frac{1}{T} \sum_{t=1}^T \negthickspace  \left\{ \frac{\partial \psi_t(\theta_{0})}{\partial \theta_j}  \psi_t(\theta_{0})'
 \negthickspace+\negthickspace  \psi_t(\theta_{0})   \frac{\partial \psi_t(\theta_{0})'}{\partial \theta_j}     \right\}  \rightarrow \E \left[\frac{\partial \psi(X_1, \theta_0)}{\partial \theta_j} \psi(X_1, \theta_0)'\right]+ \E \left[\psi(X_1, \theta_0)\frac{\partial \psi(X_1, \theta_0)}{\partial \theta_j} '\right]$ $\P$-a.s. as $T \rightarrow \infty$.
\end{itemize}
Thus, under Assumptions \ref{Assp:ExistenceConsistency} and \ref{Assp:AsymptoticNormality}, for all $j \in \ldsb 1,m \rdsb$, $\P$-a.s. as $T \rightarrow \infty$, $T\frac{\partial  L_{T}(\theta_0, \tau(\theta_0) ) }{\partial \theta_j} \rightarrow \tr\left\{  \left[\E \frac{\partial \psi(X_1, \theta_0)}{\partial \theta'}\right]^{-1} \E\left[ \frac{\partial^2 \psi_t(\theta_{0})}{\partial \theta_j \partial \theta' }\right]   \right\} -\frac{1}{2} \tr
\Big\{
 \left[\E \psi(X_1, \theta_0)\psi(X_1, \theta_0)' \right]^{-1}\Big[  \E \left[\frac{\partial \psi(X_1, \theta_0)}{\partial \theta_j} \psi(X_1, \theta_0)'\right]\\+ \E \left[\psi(X_1, \theta_0)\frac{\partial \psi(X_1, \theta_0)}{\partial \theta_j} '\right] \Big]\Big\}$, so that $\frac{\partial  L_{T}(\theta_{0}, \tau _{0}) }{\partial \theta_j}=\frac{1}{T}O(1)=O(T^{-1})$.

\textit{(ii)}  Under Assumptions \ref{Assp:ExistenceConsistency} and \ref{Assp:AsymptoticNormality}, by  Lemma \ref{Lem:UniformLimitLTPartialThetaPartialTheta} (p. \pageref{Lem:UniformLimitLTPartialThetaPartialTheta}) and Lemma \ref{Lem:LogESPDecomposition} (p. \pageref{Lem:LogESPDecomposition}), $\P$-a.s. as $T \rightarrow \infty$, uniformly over a closed ball around $(\theta_0, \tau(\theta_0))$  with strictly positive radius, $\frac{\partial^{2}  L_{T}(\theta, \tau ) }{\partial \theta_j \partial \theta_\ell} \rightarrow \frac{1   }{\left[ \E \e^{\tau'\psi(X_1,\theta)}\right]}\E\left\{ \e^{\tau'\psi(X_1,\theta)} \left[\tau' \frac{\partial \psi(X_1,\theta)}{\partial \theta_\ell}\right] \left[\tau' \frac{\partial \psi(X_1,\theta)}{\partial\theta_j}\right] +\e^{\tau'\psi(X_1,\theta)}\left[\tau'\frac{\partial^2 \psi(X_1,\theta)}{\partial \theta_j \partial \theta_\ell}\right] \right\} \nonumber   \\
-\frac{ 1 \frac{}{}  }{\left[ \E \e^{\tau'\psi(X_1,\theta)}\right]^2} \E\left[ \e^{\tau'\psi(X_1,\theta)}\tau'\frac{\partial \psi(X_1,\theta)}{\partial\theta_j } \right]  \times\E \left[ \e^{\tau'\psi(X_1,\theta)}\tau'\frac{\partial \psi(X_1,\theta)}{\partial \theta_\ell } \right] $. Now, under Assumption \ref{Assp:ExistenceConsistency}(a)(b)(d) (e)(g) and (h), by Lemma \ref{Lem:AsTiltingFct}ii (p. \pageref{Lem:AsTiltingFct}) and Assumption \ref{Assp:ExistenceConsistency}(c), put $\tau(\theta_0)=0_{m \times 1}$, so that the result follows.

\textit{(iii)}  Under Assumptions \ref{Assp:ExistenceConsistency} and \ref{Assp:AsymptoticNormality}, by  Lemma \ref{Lem:UniformLimitLTPartialTauPartialTheta} (p. \pageref{Lem:UniformLimitLTPartialTauPartialTheta}) and Lemma \ref{Lem:LogESPDecomposition} (p. \pageref{Lem:LogESPDecomposition}), $\P$-a.s. as $T \rightarrow \infty$, uniformly over a closed ball around $(\theta_0, \tau(\theta_0))$  with strictly positive radius, $\frac{\partial^{2}  L_{T}(\theta, \tau ) }{\partial \tau_k \partial \theta_\ell} \rightarrow \frac{1}{ \left[ \E \e^{\tau'\psi(X_1,\theta)}  \right]^2 }\nonumber
 \times
\left\{
  \E\left[\e^{\tau'\psi(X_1,\theta)} \right]\E\left\{\e^{\tau'\psi(X_1,\theta)} \tau' \frac{\partial \psi(X_1,\theta)}{\partial \theta_{\ell}}
 \psi_{k} (X_{1,}\theta) \nonumber
+
\e^{\tau'\psi(X_1,\theta)}  \frac{\partial \psi_{k} (X_{1,}\theta)}{\partial \theta_{\ell}} \right\}
\right.\nonumber
\\
 \left.
\hspace{.5in}
 -\E \left[ \e^{\tau'\psi(X_1,\theta)} \tau' \frac{\partial \psi(X_1,\theta)}{\partial \theta_{\ell}}\right]\E\left[  \e^{\tau'\psi(X_1,\theta)}   \psi_{k}(X_{1,}\theta) \right]
\right\} $. Now, under Assumption \ref{Assp:ExistenceConsistency}(a)(b)(d)(e)(g) and (h), by Lemma \ref{Lem:AsTiltingFct}ii (p. \pageref{Lem:AsTiltingFct}) and Assumption \ref{Assp:ExistenceConsistency}(c),  $\tau(\theta_0)=0_{m \times 1}$, so that $\P$-a.s. as $T \rightarrow \infty$, $\frac{\partial^{2}  L_{T}(\theta_{T}, \tau_{T} ) }{\partial \tau_k \partial \theta_\ell} \rightarrow \E\left[\frac{\partial \psi_{k} (X_{1,}\theta)}{\partial \theta_{\ell}}\right] $. Stack the components together in order to obtain the result.

\end{proof}

\begin{lem}[Uniform limit of $\frac{\partial^{2}  L_{T}(\theta, \tau ) }{\partial \theta_j \partial \theta_\ell}$ in a neighborhood  of $(\theta_0, \tau(\theta_0))$]\label{Lem:UniformLimitLTPartialThetaPartialTheta} Under Assumptions \ref{Assp:ExistenceConsistency} and \ref{Assp:AsymptoticNormality}, for all $(j, \ell)\in \ldsb 1,m\rdsb^2 $, $\P$-a.s. as $T \rightarrow \infty$, uniformly over a closed ball around $(\theta_0, \tau(\theta_0))$ with strictly positive radius,
 \begin{enumerate}
\item[(i)] $\frac{\partial^{2}  M_{1,T}(\theta, \tau ) }{\partial \theta_j \partial \theta_\ell} \rightarrow \frac{1   }{\left[ \E \e^{\tau'\psi(X_1,\theta)}\right]}\E\left\{ \e^{\tau'\psi(X_1,\theta)} \left[\tau' \frac{\partial \psi(X_1,\theta)}{\partial \theta_\ell}\right] \left[\tau' \frac{\partial \psi(X_1,\theta)}{\partial\theta_j}\right] +\e^{\tau'\psi(X_1,\theta)}\left[\tau'\frac{\partial^2 \psi(X_1,\theta)}{\partial \theta_j \partial \theta_\ell}\right] \right\} \nonumber   \\
-\frac{ 1 \frac{}{}  }{\left[ \E \e^{\tau'\psi(X_1,\theta)}\right]^2} \E\left[ \e^{\tau'\psi(X_1,\theta)}\tau'\frac{\partial \psi(X_1,\theta)}{\partial\theta_j } \right]  \times\E \left[ \e^{\tau'\psi(X_1,\theta)}\tau'\frac{\partial \psi(X_1,\theta)}{\partial \theta_\ell } \right] $;
\item[(ii)] $\frac{\partial^{2}  M_{2,T}(\theta, \tau ) }{\partial \theta_j \partial \theta_\ell} \rightarrow 0$;
\item[(iii)] $\frac{\partial^{2}  M_{3,T}(\theta, \tau ) }{\partial \theta_j \partial \theta_\ell} \rightarrow 0$.
\end{enumerate}
\end{lem}
\begin{proof}\textit{(i)} Under Assumptions \ref{Assp:ExistenceConsistency} and \ref{Assp:AsymptoticNormality}, by Lemma \ref{Lem:Bound2ndPartialLPartialTheta}i-iv (p. \pageref{Lem:Bound2ndPartialLPartialTheta}), Assumption \ref{Assp:ExistenceConsistency}(a) and (b), all the averages in $\frac{\partial^{2}  M_{1,T}(\theta, \tau ) }{\partial \theta_j \partial \theta_\ell}$ (equation \eqref{Eq:LogETTerm2ndDerivative} on p. \pageref{Eq:LogETTerm2ndDerivative}) satisfy the assumptions of the ULLN \`a la Wald. Moreover, under Assumption \ref{Assp:ExistenceConsistency}(a)-(b) (d)(e)(g) and (h), by Lemma  \ref{Lem:LogESPExistenceInS}i (p. \pageref{Lem:LogESPExistenceInS}) the averages in the denominators are bounded away from zero. Thus, the  result follows  from the ULLN \`a la Wald. Note that the coefficient $\frac{m}{2T}$ vanishes as it goes to zero, as $T \rightarrow \infty$.

\textit{(ii)} Under Assumptions \ref{Assp:ExistenceConsistency} and \ref{Assp:AsymptoticNormality}, by Lemma \ref{Lem:Bound2ndPartialLPartialTheta}v-xi (p. \pageref{Lem:Bound2ndPartialLPartialTheta}), Assumption \ref{Assp:ExistenceConsistency}(a) and (b), all the averages in $\frac{\partial^{2}  M_{2,T}(\theta, \tau ) }{\partial \theta_j \partial \theta_\ell}$ (equation \eqref{Eq:DerivativeTerm2ndDerivative} on p. \pageref{Eq:DerivativeTerm2ndDerivative}) satisfy the assumptions of the ULLN \`a la Wald. Moreover, under Assumption \ref{Assp:ExistenceConsistency}, by Lemma  \ref{Lem:LogESPExistenceInS}iii (p. \pageref{Lem:LogESPExistenceInS}) the averages in the inverted matrices are invertible in a neighborhood of $(\theta_0, \tau(\theta_0))$  $\P$-a.s. for $T$ big enough. Thus, the  result follows  from the ULLN \`a la Wald, the linearity of the trace operator and the scaling by $\frac{1}{T}$.

\textit{(iii)} Under Assumptions \ref{Assp:ExistenceConsistency} and \ref{Assp:AsymptoticNormality}, by Lemma \ref{Lem:Bound2ndPartialLPartialTheta}xii-xix (p. \pageref{Lem:Bound2ndPartialLPartialTheta}), Assumption \ref{Assp:ExistenceConsistency}(a) and (b), all the averages in $\frac{\partial^{2}  M_{3,T}(\theta, \tau ) }{\partial \theta_j \partial \theta_\ell}$ (equation \eqref{Eq:VarianceTerm2ndDerivative} on p. \pageref{Eq:VarianceTerm2ndDerivative}) satisfy the assumptions of the ULLN \`a la Wald. Moreover, under Assumption \ref{Assp:ExistenceConsistency}(a)(b)(e)(g) and (h), by Lemma  \ref{Lem:LogESPExistenceInS}iv (p. \pageref{Lem:LogESPExistenceInS}) the averages in the inverted matrices are invertible in a neighborhood of $(\theta_0, \tau(\theta_0))$,  $\P$-a.s. for $T$ big enough. Thus, the  result follows  from the ULLN \`a la Wald, the linearity of the trace operator and the scaling by $\frac{1}{T}$.
\end{proof}

\begin{lem}[Uniform limit of $\frac{\partial^{2}  L_{T}(\theta, \tau ) }{ \partial \tau_k\partial \theta_l}$ in a neighborhood of $(\theta_0, \tau(\theta_0))$] \label{Lem:UniformLimitLTPartialTauPartialTheta}Under Assumptions \ref{Assp:ExistenceConsistency} and \ref{Assp:AsymptoticNormality}, for all $(k, \ell)\in \ldsb 1,m\rdsb^2 $, $\P$-a.s. as $T \rightarrow \infty$, uniformly over a closed ball around $(\theta_0, \tau(\theta_0))$ with strictly positive radius,
 \begin{enumerate}
\item[(i)] $\frac{\partial^{2}  M_{1,T}(\theta, \tau ) }{\partial \tau_k \partial \theta_\ell} \rightarrow \frac{1}{ \left[ \E \e^{\tau'\psi(X_1,\theta)}  \right]^2 }\nonumber
 \times
\left\{
  \E\left[\e^{\tau'\psi(X_1,\theta)} \right]\E\left[\e^{\tau'\psi(X_1,\theta)} \tau' \frac{\partial \psi(X_1,\theta)}{\partial \theta_{\ell}}
 \psi_{k} (X_{1,}\theta) \nonumber
+
\e^{\tau'\psi(X_1,\theta)}  \frac{\partial \psi_{k} (X_{1,}\theta)}{\partial \theta_{\ell}} \right]
\right.\nonumber
\\
 \left.
\hspace{.5in}
 -\E \left[ \e^{\tau'\psi(X_1,\theta)} \tau' \frac{\partial \psi(X_1,\theta)}{\partial \theta_{\ell}}\right]\E\left[  \e^{\tau'\psi(X_1,\theta)}   \psi_{k}(X_{1,}\theta) \right]
\right\} $;
\item[(ii)] $\frac{\partial^{2}  M_{2,T}(\theta, \tau ) }{\partial\tau_k \partial \theta_\ell} \rightarrow 0$.
\item[(iii)] $\frac{\partial^{2}  M_{3,T}(\theta, \tau ) }{\partial \tau_k\partial \theta_\ell} \rightarrow 0$.
\end{enumerate}
\end{lem}
\begin{proof} The proof is similar to the one of Lemma \ref{Lem:UniformLimitLTPartialThetaPartialTheta} (p. \pageref{Lem:UniformLimitLTPartialThetaPartialTheta}).
\textit{(i)} Under Assumptions \ref{Assp:ExistenceConsistency} and \ref{Assp:AsymptoticNormality}, by Lemma \ref{Lem:Bound2ndPartialLPartialTau}i-v (p. \pageref{Lem:Bound2ndPartialLPartialTau}), Assumption \ref{Assp:ExistenceConsistency}(a) and (b), all the averages in $\frac{\partial^{2}  M_{1,T}(\theta, \tau ) }{\partial \tau_k \partial \theta_\ell}$ (equation \eqref{Eq:LogETTerm2ndDerivativeThetaTau} on p. \pageref{Eq:LogETTerm2ndDerivativeThetaTau}) satisfy the assumptions of the ULLN \`a la Wald. Moreover, under Assumption \ref{Assp:ExistenceConsistency}(a)-(b) (d)(e)(g) and (h), by Lemma  \ref{Lem:LogESPExistenceInS}i (p. \pageref{Lem:LogESPExistenceInS}) the averages in the denominators are bounded away from zero. Thus, the  result follows  from the ULLN \`a la Wald. Note that the coefficient $\frac{m}{2T}$ vanishes as it goes to zero, as $T \rightarrow \infty$.

\textit{(ii)} Under Assumptions \ref{Assp:ExistenceConsistency} and \ref{Assp:AsymptoticNormality}, by Lemma \ref{Lem:Bound2ndPartialLPartialTau}vi-xii (p. \pageref{Lem:Bound2ndPartialLPartialTau}), Assumption \ref{Assp:ExistenceConsistency}(a) and (b), all the averages in $\frac{\partial^{2}  M_{2,T}(\theta, \tau ) }{\partial \tau_k \partial \theta_\ell}$ (equation \eqref{Eq:DerivativeTerm2ndDerivativeTauTheta} on p. \pageref{Eq:DerivativeTerm2ndDerivativeTauTheta}) satisfy the assumptions of the ULLN \`a la Wald. Moreover, under Assumption \ref{Assp:ExistenceConsistency}, by Lemma  \ref{Lem:LogESPExistenceInS}iii (p. \pageref{Lem:LogESPExistenceInS}) the averages in the inverted matrices are invertible in a neighborhood of $(\theta_0, \tau(\theta_0))$  $\P$-a.s. for $T$ big enough. Thus, the  result follows  from the ULLN \`a la Wald, the linearity of the trace operator and the scaling by $\frac{1}{T}$.

\textit{(iii)} Under Assumptions \ref{Assp:ExistenceConsistency} and \ref{Assp:AsymptoticNormality}, by Lemma \ref{Lem:Bound2ndPartialLPartialTau}xiii-xix (p. \pageref{Lem:Bound2ndPartialLPartialTau}), Assumption \ref{Assp:ExistenceConsistency}(a) and (b), all the averages in $\frac{\partial^{2}  M_{3,T}(\theta, \tau ) }{\partial \tau_k \partial \theta_\ell}$ (equation \eqref{Eq:DerivativeTermDTauDTau}on p. \pageref{Eq:VarianceTerm2ndDerivativeTau}) satisfy the assumptions of the ULLN \`a la Wald. Moreover, under Assumption \ref{Assp:ExistenceConsistency}(a)(b)(e)(g) and (h), by Lemma  \ref{Lem:LogESPExistenceInS}iv (p. \pageref{Lem:LogESPExistenceInS}) the averages in the inverted matrices are invertible in a neighborhood of $(\theta_0, \tau(\theta_0))$,  $\P$-a.s. for $T$ big enough. Thus, the  result follows  from the ULLN \`a la Wald, the linearity of the trace operator and the scaling by $\frac{1}{T}$.
\end{proof}

\begin{lem}\label{Lem:ETEquationFirstDerivatives}  Put $S_T(\theta,\tau):=\frac{1}{T}\sum_{t=1}^T\e^{\tau'\psi_t(\theta)} \psi_t(\theta)$. Under Assumptions  \ref{Assp:ExistenceConsistency} and \ref{Assp:AsymptoticNormality},  there exists a closed ball centered at $(\theta_0' \, \tau(\theta_0)')$ with strictly positive radius s.t., $\P$-a.s. as $T \rightarrow \infty$,
\begin{enumerate}
\item[(i)] $\sup_{(\theta, \tau)\in \overline{B_{r_L}((\theta_0, \tau_{0}))}}\left\vert\frac{\partial S_T(\theta, \tau)}{\partial \theta'} -\E\left[\e^{\tau' \psi(X_1, \theta)}\tau'\frac{\partial \psi(X_1, \theta)}{\partial \theta'} \right]-\E\left[\e^{\tau' \psi(X_1, \theta)}\frac{\partial \psi(X_1, \theta)}{\partial \theta'} \right]\right\vert =o(1)$;
\item[(ii)]  $\sup_{(\theta, \tau)\in \overline{B_{r_L}((\theta_0, \tau_{0}))}}\left\vert\frac{\partial S_T(\theta, \tau)}{\partial \tau'} -\E\left[\e^{\tau' \psi(X_1, \theta)} \psi(X_1, \theta)\psi(X_1, \theta)' \right]\right\vert =o(1)$.
\end{enumerate}
\end{lem}
\begin{proof}
\textit{(i)} By definition of $S_T(\theta, \tau)$,
\begin{eqnarray*}
\frac{\partial S(\theta, \tau)}{\partial \theta'}
& = &
\frac{1}{T} \sum_{t=1}^T  \e^{\tau ' \psi_t(\theta) } \psi_{t}(\theta)
\tau'
 \frac{\partial \psi_{t}(\theta)}{
 \partial \theta'}
 + \frac{1}{T} \sum_{t=1}^T  \e^{\tau ' \psi_t(\theta) } \frac{\partial \psi_{t}(\theta)}{
 \partial \theta'}.
\end{eqnarray*}
Thus, by the triangle inequality,
\begin{eqnarray*}
& &\sup_{(\theta, \tau)\in \overline{B_{r_L}((\theta_0, \tau_{0}))}}\left\vert\frac{\partial S_T(\theta, \tau)}{\partial \theta'} -\E\left[\e^{\tau' \psi(X_1, \theta)}\tau'\frac{\partial \psi(X_1, \theta)}{\partial \theta'} \right]-\E\left[\e^{\tau' \psi(X_1, \theta)}\frac{\partial \psi(X_1, \theta)}{\partial \theta'} \right]\right\vert\\
& \leqslant & \sup_{(\theta, \tau)\in \overline{B_{r_L}((\theta_0, \tau_{0}))}}\left\vert\frac{1}{T} \sum_{t=1}^T  \e^{\tau ' \psi_t(\theta) } \psi_{t}(\theta)
\tau'
 \frac{\partial \psi_{t}(\theta)}{
 \partial \theta'}-\E\left[\e^{\tau' \psi(X_1, \theta)}\tau'\frac{\partial \psi(X_1, \theta)}{\partial \theta'} \right]\right\vert\\
& & +\sup_{(\theta, \tau)\in \overline{B_{r_L}((\theta_0, \tau_{0}))}}\left\vert\frac{1}{T} \sum_{t=1}^T  \e^{\tau ' \psi_t(\theta) } \frac{\partial \psi_{t}(\theta)}{
 \partial \theta'}-\E\left[\e^{\tau' \psi(X_1, \theta)}\frac{\partial \psi(X_1, \theta)}{\partial \theta'} \right]\right\vert\\
& = & o(1) \text{ $\P$-a.s. as $T \rightarrow \infty$}
\end{eqnarray*}
 where the last equality follows from the ULLN \`a la Wald by Assumption \ref{Assp:ExistenceConsistency}(a)(b) and   Lemma \ref{Lem:Bound2ndPartialLPartialTheta}iv-v (p. \pageref{Lem:Bound2ndPartialLPartialTheta}), under Assumptions \ref{Assp:ExistenceConsistency} and \ref{Assp:AsymptoticNormality}.

\textit{(ii)}  By definition of $S_T(\theta, \tau)$, $\frac{\partial S_T(\theta, \tau)}{\partial \tau'}=\frac{1}{T}\sum_{t=1}^T\e^{\tau' \psi_t(\theta)}\psi_t(\theta)\psi_t(\theta)'  $. Now, under Assumption \ref{Assp:ExistenceConsistency}(a)-(b)(e) and (g), by Lemma \ref{Lem:TiltedVariance}i (p. \pageref{Lem:TiltedVariance}) and Assumption \ref{Assp:ExistenceConsistency}(a)(b), the assumptions of the ULLN \`a la Wald are satisfied, so that the result follows from the latter.
\end{proof}

\begin{lem}[Finiteness of the expectations of the supremum of the terms from  $\frac{\partial ^{2}L_T(\theta, \tau)}{\partial \theta_{\ell}\partial \theta_j}$]\label{Lem:Bound2ndPartialLPartialTheta} Under Assumptions \ref{Assp:ExistenceConsistency} and \ref{Assp:AsymptoticNormality}, there exists a closed ball $\overline{B_L}\subset \Sbf$ centered at $(\theta_0, \tau(\theta_0))$ with strictly positive radius s.t., for all $(\ell, j)\in \ldsb 1,m\rdsb^2$,

\begin{enumerate}
\item[(i)] $\E\left[\sup_{(\theta, \tau)\in \overline{B_L}}  \e^{\tau'\psi(X_{1,}\theta)} \right] < \infty  $;
\item[(ii)]  $\E\left[\sup_{(\theta, \tau)\in \overline{B_L}} \vert \e^{\tau'\psi(X_1,\theta)}\tau' \frac{\partial \psi(X_1, \theta)}{\partial \theta_{\ell}}\tau' \frac{\partial \psi(X_1, \theta)}{\partial \theta_{j}}\vert \right] < \infty$;
\item[(iii)] $\E\left[\sup_{(\theta, \tau)\in \overline{B_L}} \vert \e^{\tau'\psi(X_1,\theta)}\tau' \frac{\partial ^{2}\psi(X_1, \theta)}{\partial \theta_{j}\partial \theta_{\ell}}\vert \right] < \infty$;
\item[(iv)] $\E\left[\sup_{(\theta, \tau)\in \overline{B_L}} \vert \e^{\tau'\psi(X_1,\theta)}\tau' \frac{\partial \psi(X_1, \theta)}{\partial \theta_{\ell}}\vert \right] < \infty$;
\item[(v)] $\E\left[\sup_{(\theta, \tau)\in \overline{B_L}} \vert \e^{\tau'\psi(X_1,\theta)} \frac{\partial \psi(X_1, \theta)}{\partial \theta'}\vert \right] < \infty$;
\item[(vi)] $\E\left[\sup_{(\theta, \tau)\in \overline{B_L}} \vert \e^{\tau'\psi(X_1,\theta)} \frac{\partial ^{2}\psi(X_1, \theta)}{\partial \theta_{l}\partial \theta'}\vert \right] < \infty$;
\item[(vii)] $\E\left[\sup_{(\theta, \tau)\in \overline{B_L}} \vert \e^{\tau'\psi(X_1,\theta)}\tau' \frac{\partial \psi(X_1, \theta)}{\partial \theta_{\ell}} \frac{\partial \psi(X_1, \theta)}{\partial \theta'}\vert \right] < \infty$;
\item[(viii)]$\E\left[\sup_{(\theta, \tau)\in \overline{B_L}} \vert \e^{\tau'\psi(X_1,\theta)} \frac{\partial ^{3}\psi(X_1, \theta)}{\partial \theta_{l}\partial \theta_{j}\partial \theta'}\vert \right] < \infty$;
\item[(ix)]$\E\left[\sup_{(\theta, \tau)\in \overline{B_L}} \vert \e^{\tau'\psi(X_1,\theta)}\tau' \frac{\partial \psi(X_1, \theta)}{\partial \theta_{\ell}} \frac{\partial^{2} \psi(X_1, \theta)}{\partial \theta_j\partial \theta'}\vert \right] < \infty$;
\item[(x)]$\E\left[\sup_{(\theta, \tau)\in \overline{B_L}} \vert \e^{\tau'\psi(X_1,\theta)}\tau' \frac{\partial^{2} \psi(X_1, \theta)}{\partial \theta_\ell\partial \theta_j} \frac{\partial \psi(X_1, \theta)}{\partial \theta'}\vert \right] < \infty$;
\item[(xi)]$\E\left[\sup_{(\theta, \tau)\in \overline{B_L}} \vert \e^{\tau'\psi(X_1,\theta)}\tau' \frac{\partial \psi(X_1, \theta)}{\partial \theta_\ell} \tau' \frac{\partial \psi(X_1, \theta)}{\partial \theta_j}\frac{\partial \psi(X_1, \theta)}{\partial \theta'}\vert \right] < \infty$;
\item[(xii)]$\E\left[\sup_{(\theta, \tau)\in \overline{B_L}} \vert \e^{\tau'\psi(X_1,\theta)}
\psi(X_1, \theta)\psi(X_1, \theta)'\vert \right] < \infty$;
\item[(xiii)]$\E\left[\sup_{(\theta, \tau)\in \overline{B_L}} \vert \e^{\tau'\psi(X_1,\theta)}
\frac{\partial \psi(X_1, \theta)}{\partial \theta_{\ell}}\psi(X_1, \theta)'\vert \right] < \infty$;
\item[(xiv)] $\E\left[\sup_{(\theta, \tau)\in \overline{B_L}} \vert \e^{\tau'\psi(X_1,\theta)}\tau' \frac{\partial \psi(X_1, \theta)}{\partial \theta_{\ell}}
\psi(X_1, \theta)\psi(X_1, \theta)'\vert \right] < \infty$;
\item[(xv)] $\E\left[\sup_{(\theta, \tau)\in \overline{B_L}} \vert \e^{\tau'\psi(X_1,\theta)} \frac{\partial ^{2}\psi(X_1, \theta)}{\partial \theta_{l}\partial \theta_{j}}\psi(X_1, \theta)'\vert \right] < \infty$;
\item[(xvi)]  $\E\left[\sup_{(\theta, \tau)\in \overline{B_L}} \vert \e^{\tau'\psi(X_1,\theta)} \frac{\partial \psi(X_1, \theta)}{\partial \theta_{\ell}} \frac{\partial \psi(X_1, \theta)}{\partial \theta_{j}}\vert \right] < \infty$;
\item[(xvii)] $\E\left[\sup_{(\theta, \tau)\in \overline{B_L}} \vert \e^{\tau'\psi(X_1,\theta)}(\tau'\frac{\partial \psi(X_1, \theta)}{\partial \theta_{\ell}}) \frac{\partial \psi(X_1, \theta)}{\partial \theta_{j}}\psi(X_1, \theta) '\vert \right] < \infty$;

\item[(xviii)] $\E\left[\sup_{(\theta, \tau)\in \overline{B_L}} \vert \e^{\tau'\psi(X_1,\theta)}(\tau'\frac{\partial ^{2}\psi(X_1, \theta)}{\partial \theta_{\ell}\partial \theta_{j}} )\psi(X_1, \theta)\psi(X_1, \theta) '\vert \right] < \infty$; and
\item[(xix)] $\E\left[\sup_{(\theta, \tau)\in \overline{B_L}} \vert \e^{\tau'\psi(X_1,\theta)}(\tau'\frac{\partial \psi(X_1, \theta)}{\partial \theta_{\ell}}) (\tau'\frac{\partial \psi(X_1, \theta)}{\partial \theta_{j}})\psi(X_1, \theta)\psi(X_1, \theta) '\vert \right] < \infty$.
\end{enumerate}
\end{lem}
\begin{proof}
\textit{(i)} Under Assumption \ref{Assp:ExistenceConsistency}(a)-(e) and (g)-(h), by Lemma \ref{Lem:LogESPExistenceInS}ii (p. \pageref{Lem:LogESPExistenceInS}), $\Sbf$ contains an open ball centered at $(\theta_0, \tau(\theta_0))$, so that, by the Cauchy-Schwarz inequality, for $\overline{B_L}$ of sufficiently small radius, $\E\left[\sup_{(\theta, \tau)\in \overline{B_L}}  \e^{\tau'\psi(X_{1,}\theta)} \right] \leqslant \sqrt{\E\left[(\sup_{(\theta, \tau)\in \overline{B_L}}  \e^{\tau'\psi(X_{1,}\theta)})^2 \right]}=  \sqrt{\E\left[\sup_{(\theta, \tau)\in \Sbf}  \e^{2\tau'\psi(X_{1,}\theta)} \right]}<  \infty  $  where the equality follows from the fact that supremum of the square of a positive function is the square of the supremum of the function, and the last inequality  from Assumption \ref{Assp:ExistenceConsistency}(e).

\textit{(ii)}  The norm of a product of matrices is smaller than the product of the norms \citep[e.g.,][Theorem 9.7 and note that all norms are equivalent on finite dimensional spaces]{1953Rudin}. Thus,  \\for $\oBL$ of sufficiently small radius,  for all $(\ell, j)\in \ldsb 1,m\rdsb^2$,
\begin{eqnarray*}
&&\ \E\left[\sup_{(\theta, \tau)\in \overline{B_L}} \vert \e^{\tau'\psi(X_1,\theta)}\tau' \frac{\partial \psi(X_1, \theta)}{\partial \theta_{\ell}}\tau' \frac{\partial \psi(X_1, \theta)}{\partial \theta_{j}}\vert\right]\\
& \leqslant & (\sup_{(\theta,\tau) \in \overline{B_L} }\vert \tau \vert^{2})  \E\left[\sup_{(\theta, \tau)\in \overline{B_L}}  \e^{\tau'\psi(X_1,\theta)}\vert \frac{\partial \psi(X_1, \theta)}{\partial \theta_{\ell}}\vert \vert \frac{\partial \psi(X_1, \theta)}{\partial \theta_{j}}\vert\right]\\
& \stackrel{(a)}{\leqslant} & (\sup_{(\theta,\tau) \in \overline{B_L} }\vert \tau \vert^{2}) \E\left[ \sup_{\theta\in \mathcal{N} } \sup_{\tau \in \Tbf(\theta)}\e^{ \tau'\psi(X_1, \theta)} b(X_1)^{2}\right] \stackrel{(b)}{<} \infty.
\end{eqnarray*}
 \textit{(a)}   Firstly, under Assumption \ref{Assp:ExistenceConsistency}(a)-(e) and (g)-(h), by Lemma \ref{Lem:LogESPExistenceInS}ii (p. \pageref{Lem:LogESPExistenceInS}), $\Sbf$ contains an open ball centered at $(\theta_0, \tau(\theta_0))$ Thus, under Assumption \ref{Assp:ExistenceConsistency}(a)-(e) and (g)-(h),  for $\oBL$ of sufficiently small radius, by definition of $\Sbf$, $\oBL \subset\{(\theta, \tau): \theta \in \mathcal{N}\wedge \tau \in \Tbf(\theta) \}\subset \Sbf$, because $ \mathcal{N}\subset \T$  by Assumption \ref{Assp:AsymptoticNormality}(a). Secondly, by Assumption \ref{Assp:AsymptoticNormality}(b), $\sup_{\theta\in \mathcal{N} }\vert \frac{\partial \psi(X_1, \theta)}{\partial \theta_{\ell}}\vert \leqslant b(X)$ and  $\sup_{\theta\in \mathcal{N} }\vert \frac{\partial \psi(X_1, \theta)}{\partial \theta_{j}}\vert\leqslant b(X)$.   \textit{(b)}   Firstly,  $\sup_{(\theta,\tau) \in \overline{B_L} }\vert \tau \vert^{2}< \infty$ because    $ \overline{B_L}$ is bounded. Secondly, by Assumption \ref{Assp:AsymptoticNormality}(b), $\E\left[ \sup_{\theta\in \mathcal{N} } \sup_{\tau \in \Tbf(\theta)}\e^{ \tau'\psi(X_1, \theta)} b(X_1)^{2}\right]< \infty$. \\
\textit{(iii)} Similarly to the proof of statement (ii), under Assumption \ref{Assp:ExistenceConsistency}(a)-(e) and (g)-(h),  for $\oBL$ of sufficiently small radius,  for all $(\ell, j)\in \ldsb 1,m\rdsb^2$,
  $\E\left[\sup_{(\theta, \tau)\in \overline{B_L}} \vert \e^{\tau'\psi(X_1,\theta)}\tau' \frac{\partial ^{2}\psi(X_1, \theta)}{\partial \theta_{j}\partial \theta_{\ell}}\vert \right] \leqslant (\sup_{(\theta,\tau) \in \overline{B_L} }\vert \tau \vert)\\\E\left[\sup_{(\theta, \tau)\in \overline{B_L}}  \e^{\tau'\psi(X_1,\theta)} \vert\frac{\partial ^{2}\psi(X_1, \theta)}{\partial \theta_{j}\partial \theta_{\ell}}\vert \right]  \leqslant (\sup_{(\theta,\tau) \in \overline{B_L} }\vert \tau \vert) \E\left[ \sup_{\theta\in \mathcal{N} } \sup_{\tau \in \Tbf(\theta)}\e^{ \tau'\psi(X_1, \theta)} b(X_1)\right] < \infty$, where the two last inequalities follow from Assumption \ref{Assp:AsymptoticNormality}(b) and the boundedness of $ \overline{B_L}$.

\textit{(iv)} Similarly to the proof of statement (ii), under Assumption \ref{Assp:ExistenceConsistency}(a)-(e) and (g)-(h),  for $\oBL$ of sufficiently small radius,  for all $\ell\in \ldsb 1,m\rdsb$,
$\E\left[\sup_{(\theta, \tau)\in \overline{B_L}} \vert \e^{\tau'\psi(X_1,\theta)}\tau' \frac{\partial \psi(X_1, \theta)}{\partial \theta_{\ell}}\vert \right]\leqslant (\sup_{(\theta,\tau) \in \overline{B_L} }\vert \tau \vert)\\\E\left[\sup_{(\theta, \tau)\in \overline{B_L}}  \e^{\tau'\psi(X_1,\theta)}\vert \frac{\partial \psi(X_1, \theta)}{\partial \theta_{\ell}}\vert \right]  \leqslant (\sup_{(\theta,\tau) \in \overline{B_L} }\vert \tau \vert) \E\left[ \sup_{\theta\in \mathcal{N} } \sup_{\tau \in \Tbf(\theta)}\e^{ \tau'\psi(X_1, \theta)} b(X_1)\right] < \infty$, where the two last inequalities follow from Assumption \ref{Assp:AsymptoticNormality}(b) and the boundedness of $ \overline{B_L}$.

\textit{(v)} Similarly to the proof of statement (ii), under Assumption \ref{Assp:ExistenceConsistency}(a)-(e) and (g)-(h),  for $\oBL$ of sufficiently small radius, $\E\left[\sup_{(\theta, \tau)\in \overline{B_L}} \vert \e^{\tau'\psi(X_1,\theta)} \frac{\partial \psi(X_1, \theta)}{\partial \theta'}\vert \right]\leqslant  \E\left[ \sup_{\theta\in \mathcal{N} } \sup_{\tau \in \Tbf(\theta)}\e^{ \tau'\psi(X_1, \theta)} b(X_1)\right] < \infty$, where the last inequality follows from Assumption \ref{Assp:AsymptoticNormality}(b).

\textit{(vi)} Similarly to the proof of statement (ii), under Assumption \ref{Assp:ExistenceConsistency}(a)-(e) and (g)-(h),  for $\oBL$ of sufficiently small radius, for all $\ell\in \ldsb 1,m\rdsb$, $\E\left[\sup_{(\theta, \tau)\in \overline{B_L}} \vert \e^{\tau'\psi(X_1,\theta)} \frac{\partial ^{2}\psi(X_1, \theta)}{\partial \theta_{\ell}\partial \theta'}\vert \right]\\ \leqslant  \E\left[ \sup_{\theta\in \mathcal{N} } \sup_{\tau \in \Tbf(\theta)}\e^{ \tau'\psi(X_1, \theta)} b(X_1)\right] < \infty$, where the  last inequality follows from Assumption \ref{Assp:AsymptoticNormality}(b).

\textit{(vii)}   Similarly to the proof of statement (ii), under Assumption \ref{Assp:ExistenceConsistency}(a)-(e) and (g)-(h),  for $\oBL$ of sufficiently small radius, for all $\ell\in \ldsb 1,m\rdsb$, $\E\left[\sup_{(\theta, \tau)\in \overline{B_L}} \vert \e^{\tau'\psi(X_1,\theta)}\tau' \frac{\partial \psi(X_1, \theta)}{\partial \theta_{\ell}} \frac{\partial \psi(X_1, \theta)}{\partial \theta'}\vert\right]<(\sup_{(\theta,\tau) \in \overline{B_L} }\vert \tau \vert)  \E\left[\sup_{(\theta, \tau)\in \overline{B_L}}  \e^{\tau'\psi(X_1,\theta)}\vert \frac{\partial \psi(X_1, \theta)}{\partial \theta_{\ell}}\vert \vert \frac{\partial \psi(X_1, \theta)}{\partial \theta'}\vert\right] \leqslant (\sup_{(\theta,\tau) \in \overline{B_L} }\vert \tau \vert)\\ \E\left[ \sup_{\theta\in \mathcal{N} } \sup_{\tau \in \Tbf(\theta)}\e^{ \tau'\psi(X_1, \theta)} b(X_1)^{2}\right]< \infty $ where the two last inequalities follow from Assumption \ref{Assp:AsymptoticNormality}(b) and the boundedness of $ \overline{B_L}$.

\textit{(viii)}  Similarly to the proof of statement (ii), under Assumption \ref{Assp:ExistenceConsistency}(a)-(e) and (g)-(h),  for $\oBL$ of sufficiently small radius, for all $(\ell, j)\in \ldsb 1,m\rdsb^2 $, $\E\left[\sup_{(\theta, \tau)\in \overline{B_L}} \vert \e^{\tau'\psi(X_1,\theta)} \frac{\partial ^{3}\psi(X_1, \theta)}{\partial \theta_{l}\partial \theta_{j}\partial \theta'}\vert \right]\leqslant \E\left[ \sup_{\theta\in \mathcal{N} } \sup_{\tau \in \Tbf(\theta)}\e^{ \tau'\psi(X_1, \theta)} b(X_1)\right]  < \infty$, where the last inequality follows from Assumption \ref{Assp:AsymptoticNormality}(b).

\textit{(ix)} Similarly to the proof of statement (ii), under Assumption \ref{Assp:ExistenceConsistency}(a)-(e) and (g)-(h),  for $\oBL$ of sufficiently small radius,  for all $(\ell, j)\in \ldsb 1,m\rdsb^2$,
 $\E\left[\sup_{(\theta, \tau)\in \overline{B_L}} \vert \e^{\tau'\psi(X_1,\theta)}\tau' \frac{\partial \psi(X_1, \theta)}{\partial \theta_{\ell}}\tau' \frac{\partial ^{2}\psi(X_1, \theta)}{\partial \theta_{j}\partial \theta'}\vert\right]<(\sup_{(\theta,\tau) \in \overline{B_L} }\vert \tau \vert^{2})  \E\left[\sup_{(\theta, \tau)\in \overline{B_L}}  \e^{\tau'\psi(X_1,\theta)}\vert \frac{\partial \psi(X_1, \theta)}{\partial \theta_{\ell}}\vert \vert \frac{\partial^{2} \psi(X_1, \theta)}{\partial \theta_{j}\partial \theta'}\vert\right] \leqslant (\sup_{(\theta,\tau) \in \overline{B_L} }\vert \tau \vert)^{2} \\ \E\left[ \sup_{\theta\in \mathcal{N} } \sup_{\tau \in \Tbf(\theta)}\e^{ \tau'\psi(X_1, \theta)} b(X_1)^{2}\right]< \infty $ where the two last inequalities follow from Assumption \ref{Assp:AsymptoticNormality}(b) and the boundedness of $ \overline{B_L}$.

\textit{(x)} Similarly to the proof of statement (ii), under Assumption \ref{Assp:ExistenceConsistency}(a)-(e) and (g)-(h),  for $\oBL$ of sufficiently small radius,  for all $(\ell, j)\in \ldsb 1,m\rdsb^2$,
  $\E\left[\sup_{(\theta, \tau)\in \overline{B_L}} \vert \e^{\tau'\psi(X_1,\theta)}\tau' \frac{\partial^{2} \psi(X_1, \theta)}{\partial \theta_\ell\partial \theta_j} \frac{\partial \psi(X_1, \theta)}{\partial \theta'}\vert \right] <(\sup_{(\theta,\tau) \in \overline{B_L} }\vert \tau \vert)  \E\left[\sup_{(\theta, \tau)\in \overline{B_L}}  \e^{\tau'\psi(X_1,\theta)}\vert \frac{\partial^{2} \psi(X_1, \theta)}{\partial \theta_\ell\partial \theta_j}\vert \vert \frac{\partial \psi(X_1, \theta)}{\partial \theta'}\vert\right] \\\leqslant (\sup_{(\theta,\tau) \in \overline{B_L} }\vert \tau \vert) \E\left[ \sup_{\theta\in \mathcal{N} } \sup_{\tau \in \Tbf(\theta)}\e^{ \tau'\psi(X_1, \theta)} b(X_1)^{2}\right]< \infty $ where the two last inequalities follow from Assumption \ref{Assp:AsymptoticNormality}(b) and the boundedness of $ \overline{B_L}$.

\textit{(xi)} Similarly to the proof of statement (ii), under Assumption \ref{Assp:ExistenceConsistency}(a)-(e) and (g)-(h),  for $\oBL$ of sufficiently small radius,  for all $(\ell, j)\in \ldsb 1,m\rdsb^2$,
 $\E\left[\sup_{(\theta, \tau)\in \overline{B_L}} \vert \e^{\tau'\psi(X_1,\theta)}\tau' \frac{\partial \psi(X_1, \theta)}{\partial \theta_\ell} \tau' \frac{\partial \psi(X_1, \theta)}{\partial \theta_j}\frac{\partial \psi(X_1, \theta)}{\partial \theta'}\vert \right] <(\sup_{(\theta,\tau) \in \overline{B_L} }\vert \tau \vert^{2}) \E\left[\sup_{(\theta, \tau)\in \overline{B_L}}  \e^{\tau'\psi(X_1,\theta)}\vert \frac{\partial \psi(X_1, \theta)}{\partial \theta_\ell} \vert \vert \frac{\partial \psi(X_1, \theta)}{\partial \theta_j} \vert \vert \frac{\partial \psi(X_1, \theta)}{\partial \theta'}\vert\right] \leqslant (\sup_{(\theta,\tau) \in \overline{B_L} }\vert \tau \vert^{2})\\ \E\left[ \sup_{\theta\in \mathcal{N} } \sup_{\tau \in \Tbf(\theta)}\e^{ \tau'\psi(X_1, \theta)} b(X_1)^{3}\right] < \infty $ where the two last inequalities follow from Assumption \ref{Assp:AsymptoticNormality}(b) and the boundedness of $ \overline{B_L}$.

\textit{(xii)} Under Assumption \ref{Assp:ExistenceConsistency}(a)-(e) and (g)-(h), by Lemma \ref{Lem:LogESPExistenceInS}ii (p. \pageref{Lem:LogESPExistenceInS}), $\Sbf$ contains an open ball centered at $(\theta_0, \tau(\theta_0))$, so that,    for $\overline{B_L}$ of sufficiently small radius, 

\noindent
$\E\left[\sup_{(\theta, \tau)\in \overline{B_L}} \vert \e^{\tau'\psi(X_1,\theta)}
\psi(X_1, \theta)\psi(X_1, \theta)'\vert \right] \leqslant \E\left[\sup_{(\theta, \tau)\in \Sbf^{\epsilon}} \vert \e^{\tau'\psi(X_1,\theta)}
\psi(X_1, \theta)\psi(X_1, \theta)'\vert \right] < \infty $  where the last inequality  follows from Lemma \ref{Lem:TiltedVariance}i (p. \pageref{Lem:TiltedVariance}) under Assumption \ref{Assp:ExistenceConsistency}(a)-(b)(e)(g).

\textit{(xiii)}
The supremum of the absolute value of the product is smaller than the product of the
suprema of the absolute values. Thus,
 under Assumption \ref{Assp:ExistenceConsistency}(a)(b),   for $\overline{B_L}$ of sufficiently small radius, for all $\ell\in \ldsb 1,m\rdsb$,\begin{eqnarray*}
& & \E\left[\sup_{(\theta, \tau)\in \overline{B_L}} \vert \e^{\tau'\psi(X_1,\theta)}
\frac{\partial \psi(X_1, \theta)}{\partial \theta_{\ell}}\psi(X_1, \theta)'\vert \right] \\
& \leqslant & \E\left[\sup_{(\theta, \tau)\in \overline{B_L}} \vert \e^{\tau'\psi(X_1,\theta)}
\frac{\partial \psi(X_1, \theta)}{\partial \theta_{\ell}}\vert\sup_{(\theta, \tau)\in \overline{B_L}}  \vert\psi(X_1, \theta)'\vert \right] \\
& \stackrel{(a)}{\leqslant} & \sqrt{ \E\left[\sup_{(\theta, \tau)\in \overline{B_L}} \vert \e^{\tau'\psi(X_1,\theta)}
\frac{\partial \psi(X_1, \theta)}{\partial \theta_{\ell}}\vert^{2}\right] } \sqrt{\E\left[\sup_{(\theta, \tau)\in \overline{B_L}}  \vert\psi(X_1, \theta)'\vert^2 \right]} \\
& \stackrel{(b)}{\leqslant} & \sqrt{\E\left[ \sup_{\theta\in \mathcal{N} } \sup_{\tau \in \Tbf(\theta)}\e^{ 2\tau'\psi(X_1, \theta)} b(X_1)^{2}\right]} \sqrt{\E\left[\sup_{\theta \in \T^{\epsilon}}  \vert\psi(X_1, \theta)'\vert^2 \right]}  \stackrel{(c)}{<}\infty.
\end{eqnarray*}
\textit{(a)} Apply the Cauchy-Schwarz inequality, and note that the supremum of the square of a positive function is the square of the supremum of the function. \textit{(b)}  Firstly, under Assumption \ref{Assp:ExistenceConsistency}(a)-(e) and (g)-(h),  by Lemma \ref{Lem:LogESPExistenceInS}ii (p. \pageref{Lem:LogESPExistenceInS}), $\Sbf$ contains an open ball centered at $(\theta_0, \tau(\theta_0))$, so that,    for $\overline{B_L}$ of sufficiently small radius, $\oBL \subset\{(\theta, \tau): \theta \in \mathcal{N}\wedge \tau \in \Tbf(\theta) \}\subset \Sbf \subset \Sbf^\epsilon $ because $\mathcal{N}\subset \T$ and $\Sbf=\{(\theta, \tau): \theta \in \T\wedge \tau \in \Tbf(\theta) \}$. Secondly, as the second supremum does not depend on $\tau$, $\sup_{(\theta, \tau)\in \overline{B_L}}  \vert\psi(X_1, \theta)'\vert^2\leqslant \sup_{\theta\in \T^\epsilon}  \vert\psi(X_1, \theta)'\vert^2 $ because $\oBL \subset \Sbf $, for $\oBL$ of radius small enough.  \textit{(c)} By Assumption \ref{Assp:AsymptoticNormality}(b), the  first expectation is bounded. Under Assumption \ref{Assp:ExistenceConsistency}(a)(b)(g), by Lemma \ref{Lem:BoundAsTiltingEquation}i (p. \pageref{Lem:BoundAsTiltingEquation}), the second expectation is also bounded.

\textit{(xiv)}
Proof similar to the one of statement (xiii). The supremum of the absolute value of the product is smaller than the product of the
suprema of the absolute values. Thus,
 under Assumption \ref{Assp:ExistenceConsistency}(a)(b),   for $\overline{B_L}$ of sufficiently small radius, for all $\ell\in \ldsb 1,m\rdsb$,\begin{eqnarray*}
& & \E\left[\sup_{(\theta, \tau)\in \overline{B_L}} \vert \e^{\tau'\psi(X_1,\theta)}\tau' \frac{\partial \psi(X_1, \theta)}{\partial \theta_{\ell}}
\psi(X_1, \theta)\psi(X_1, \theta)'\vert \right]  \\
& \leqslant & (\sup_{(\theta,\tau) \in \overline{B_L} }\vert \tau \vert)\E\left[\sup_{(\theta, \tau)\in \overline{B_L}} \vert \e^{\tau'\psi(X_1,\theta)}
\frac{\partial \psi(X_1, \theta)}{\partial \theta_{\ell}}\vert\sup_{(\theta, \tau)\in \overline{B_L}}  \vert\psi(X_1, \theta)\psi(X_1, \theta)'\vert \right] \\
& \stackrel{(a)}{\leqslant} & (\sup_{(\theta,\tau) \in \overline{B_L} }\vert \tau \vert) \sqrt{ \E\left[\sup_{(\theta, \tau)\in \overline{B_L}} \vert \e^{\tau'\psi(X_1,\theta)}
\frac{\partial \psi(X_1, \theta)}{\partial \theta_{\ell}}\vert^{2}\right] } \sqrt{\E\left[\sup_{(\theta, \tau)\in \overline{B_L}}  \vert\psi(X_1, \theta)\psi(X_1, \theta)'\vert^2 \right]} \\
& \stackrel{(b)}{\leqslant} & (\sup_{(\theta,\tau) \in \overline{B_L} }\vert \tau \vert) \sqrt{\E\left[ \sup_{\theta\in \mathcal{N} } \sup_{\tau \in \Tbf(\theta)}\e^{ 2\tau'\psi(X_1, \theta)} b(X_1)^{2}\right]} \sqrt{\E\left[\sup_{\theta \in \T^{\epsilon}}  \vert\psi(X_1, \theta)\psi(X_1, \theta)'\vert^2 \right]}  \stackrel{(c)}{<}\infty.
\end{eqnarray*}
\textit{(a)} Apply the Cauchy-Schwarz inequality, and note that the supremum of the square of a positive function is the square of the supremum of the function. \textit{(b)}  Firstly, under Assumption \ref{Assp:ExistenceConsistency}(a)-(e) and (g)-(h),  by Lemma \ref{Lem:LogESPExistenceInS}ii (p. \pageref{Lem:LogESPExistenceInS}), $\Sbf$ contains an open ball centered at $(\theta_0, \tau(\theta_0))$, so that,    for $\overline{B_L}$ of sufficiently small radius, $\oBL \subset\{(\theta, \tau): \theta \in \mathcal{N}\wedge \tau \in \Tbf(\theta) \}\subset \Sbf \subset \Sbf^\epsilon $. Secondly, as the second supremum does not depend on $\tau$, $\sup_{(\theta, \tau)\in \overline{B_L}}  \vert \psi(X_1, \theta)\psi(X_1, \theta)'\vert^2 \leqslant \sup_{\theta\in \T^\epsilon}  \vert \psi(X_1, \theta)\psi(X_1, \theta)'\vert^2 $ because $\oBL \subset \Sbf $, for $\oBL$ of radius small enough.  \textit{(c)} Firstly, because $\oBL$ is bounded, $ (\sup_{(\theta,\tau) \in \overline{B_L} }\vert \tau \vert)< \infty$. Secondly, by Assumption \ref{Assp:AsymptoticNormality}(b), the  first expectation is bounded. Thirdly, by Assumption \ref{Assp:ExistenceConsistency}(g), the second expectation is also bounded.

\textit{(xv)}
The proof is the same as for statement (xiii) with $\frac{\partial ^{2}\psi(X_1, \theta)}{\partial \theta_{l}\partial \theta_{j}}$ instead of $\frac{\partial \psi(X_1, \theta)}{\partial \theta_{\ell}}$. The supremum of the absolute value of the product is smaller than the product of the
suprema of the absolute values. Thus,
 under Assumption \ref{Assp:ExistenceConsistency}(a)(b),   for $\overline{B_L}$ of sufficiently small radius, for all $(\ell,j)\in \ldsb 1,m\rdsb^2$,\begin{eqnarray*}
& & \E\left[\sup_{(\theta, \tau)\in \overline{B_L}} \vert \e^{\tau'\psi(X_1,\theta)} \frac{\partial ^{2}\psi(X_1, \theta)}{\partial \theta_{l}\partial \theta_{j}}\psi(X_1, \theta)'\vert \right] \\
& \leqslant & \E\left[\sup_{(\theta, \tau)\in \overline{B_L}} \vert \e^{\tau'\psi(X_1,\theta)}
\frac{\partial ^{2}\psi(X_1, \theta)}{\partial \theta_{l}\partial \theta_{j}}\vert\sup_{(\theta, \tau)\in \overline{B_L}}  \vert\psi(X_1, \theta)'\vert \right] \\
& \stackrel{(a)}{\leqslant} & \sqrt{ \E\left[\sup_{(\theta, \tau)\in \overline{B_L}} \vert \e^{\tau'\psi(X_1,\theta)}
\frac{\partial ^{2}\psi(X_1, \theta)}{\partial \theta_{l}\partial \theta_{j}}\vert^{2}\right] } \sqrt{\E\left[\sup_{(\theta, \tau)\in \overline{B_L}}  \vert\psi(X_1, \theta)'\vert^2 \right]} \\
& \stackrel{(b)}{\leqslant} & \sqrt{\E\left[ \sup_{\theta\in \mathcal{N} } \sup_{\tau \in \Tbf(\theta)}\e^{ 2\tau'\psi(X_1, \theta)} b(X_1)^{2}\right]} \sqrt{\E\left[\sup_{\theta \in \T^{\epsilon}}  \vert\psi(X_1, \theta)'\vert^2 \right]}  \stackrel{(c)}{<}\infty.
\end{eqnarray*}
\textit{(a)} Apply the Cauchy-Schwarz inequality, and note that the supremum of the square of a positive function is the square of the supremum of the function. \textit{(b)}  Firstly, under Assumption \ref{Assp:ExistenceConsistency}(a)-(e) and (g)-(h),  by Lemma \ref{Lem:LogESPExistenceInS}ii (p. \pageref{Lem:LogESPExistenceInS}), $\Sbf$ contains an open ball centered at $(\theta_0, \tau(\theta_0))$, so that,    for $\overline{B_L}$ of sufficiently small radius, $\oBL \subset\{(\theta, \tau): \theta \in \mathcal{N}\wedge \tau \in \Tbf(\theta) \}\subset \Sbf \subset \Sbf^\epsilon $. Secondly, as the second supremum does not depend on $\tau$, $\sup_{(\theta, \tau)\in \overline{B_L}}  \vert\psi(X_1, \theta)'\vert^2\leqslant\sup_{\theta\in \T^\epsilon}  \vert\psi(X_1, \theta)'\vert^2 $ because $\oBL \subset \Sbf $, for $\oBL$ of radius small enough.  \textit{(c)} By Assumption \ref{Assp:AsymptoticNormality}(b), the  first expectation is bounded. Under Assumption \ref{Assp:ExistenceConsistency}(a)(b)(g), by Lemma \ref{Lem:BoundAsTiltingEquation}i (p. \pageref{Lem:BoundAsTiltingEquation}), the second expectation is also bounded.

\textit{(xvi)}  Similarly to the proof of statement (ii), under Assumption \ref{Assp:ExistenceConsistency}(a)-(e) and (g)-(h),  for $\oBL$ of sufficiently small radius,  for all $(\ell, j)\in \ldsb 1,m\rdsb^2$,    $\E\left[\sup_{(\theta, \tau)\in \overline{B_L}} \vert \e^{\tau'\psi(X_1,\theta)} \frac{\partial \psi(X_1, \theta)}{\partial \theta_{\ell}} \frac{\partial \psi(X_1, \theta)}{\partial \theta_{j}}\vert\right]<  \E\left[\sup_{(\theta, \tau)\in \overline{B_L}}  \e^{\tau'\psi(X_1,\theta)}\vert \frac{\partial \psi(X_1, \theta)}{\partial \theta_{\ell}}\vert \vert \frac{\partial \psi(X_1, \theta)}{\partial \theta_{j}}\vert\right]\leqslant  \E\left[ \sup_{\theta\in \mathcal{N} } \sup_{\tau \in \Tbf(\theta)}\e^{ \tau'\psi(X_1, \theta)} b(X_1)^{2}\right]< \infty $ where the  \ last inequality follows from Assumption \ref{Assp:AsymptoticNormality}(b).

\textit{(xvii)} Proof similar to the one of statement (xiii). The norm of a product of matrices is smaller than the product of the norms \citep[e.g.,][Theorem 9.7 and note that all norms are equivalent on finite dimensional spaces]{1953Rudin}. Moreover, the supremum of the absolute value of the product is smaller than the product of the
suprema of the absolute values. Thus,
 under Assumption \ref{Assp:ExistenceConsistency}(a)(b),   for $\overline{B_L}$ of sufficiently small radius, for all $(\ell, j)\in \ldsb 1,m\rdsb^2$,\begin{eqnarray*}
& & \E\left[\sup_{(\theta, \tau)\in \overline{B_L}} \vert \e^{\tau'\psi(X_1,\theta)}(\tau'\frac{\partial \psi(X_1, \theta)}{\partial \theta_{\ell}}) \frac{\partial \psi(X_1, \theta)}{\partial \theta_{j}}\psi(X_1, \theta) '\vert \right] < \infty  \\
& \leqslant & (\sup_{(\theta,\tau) \in \overline{B_L} }\vert \tau \vert)\E\left[\sup_{(\theta, \tau)\in \overline{B_L}} \left( \e^{\tau'\psi(X_1,\theta)}
\vert\frac{\partial \psi(X_1, \theta)}{\partial \theta_{\ell}}\vert\vert\frac{\partial \psi(X_1, \theta)}{\partial \theta_{j}}\vert\right)\sup_{(\theta, \tau)\in \overline{B_L}}  \vert\psi(X_1, \theta)\vert \right] \\
& \stackrel{(a)}{\leqslant} & (\sup_{(\theta,\tau) \in \overline{B_L} }\vert \tau \vert) \sqrt{ \E\left[\sup_{(\theta, \tau)\in \overline{B_L}} \left( \e^{\tau'\psi(X_1,\theta)}
\vert\frac{\partial \psi(X_1, \theta)}{\partial \theta_{\ell}}\vert\vert\frac{\partial \psi(X_1, \theta)}{\partial \theta_{j}}\vert\right)^{2}\right] } \sqrt{\E\left[\sup_{(\theta, \tau)\in \overline{B_L}}  \vert\psi(X_1, \theta)\vert^2 \right]} \\
& \stackrel{(b)}{\leqslant} & (\sup_{(\theta,\tau) \in \overline{B_L} }\vert \tau \vert) \sqrt{\E\left[ \sup_{\theta\in \mathcal{N} } \sup_{\tau \in \Tbf(\theta)}\e^{ 2\tau'\psi(X_1, \theta)} b(X_1)^{4}\right]} \sqrt{\E\left[\sup_{\theta \in \T^{\epsilon}}  \vert\psi(X_1, \theta)\vert^2 \right]}  \stackrel{(c)}{<}\infty.
\end{eqnarray*}
\textit{(a)} Apply the Cauchy-Schwarz inequality, and note that the supremum of the square of a positive function is the square of the supremum of the function. \textit{(b)}  Firstly, under Assumption \ref{Assp:ExistenceConsistency}(a)-(e) and (g)-(h),  by Lemma \ref{Lem:LogESPExistenceInS}ii (p. \pageref{Lem:LogESPExistenceInS}), $\Sbf$ contains an open ball centered at $(\theta_0, \tau(\theta_0))$, so that,    for $\overline{B_L}$ of sufficiently small radius, $\oBL \subset\{(\theta, \tau): \theta \in \mathcal{N}\wedge \tau \in \Tbf(\theta) \}\subset \Sbf \subset \Sbf^\epsilon $. Secondly, as the second supremum does not depend on $\tau$, $\sup_{(\theta, \tau)\in \overline{B_L}}  \vert\psi(X_1, \theta)'\vert^2 \leqslant \sup_{\theta\in \T^\epsilon}  \vert\psi(X_1, \theta)'\vert^2 $ because $\oBL \subset \Sbf $, for $\oBL$ of radius small enough.  \textit{(c)} Firstly, because $\oBL$ is bounded, $ (\sup_{(\theta,\tau) \in \overline{B_L} }\vert \tau \vert)< \infty$. Secondly, by Assumption \ref{Assp:AsymptoticNormality}(b), the  first expectation is bounded. Thirdly, under  Assumption \ref{Assp:ExistenceConsistency}(a)(b)(g), by Lemma \ref{Lem:BoundAsTiltingEquation} (p. \pageref{Lem:BoundAsTiltingEquation}), the second expectation is also bounded.

\textit{(xviii)}
Proof similar to the one of statement (xiii). The supremum of the absolute value of the product is smaller than the product of the
suprema of the absolute values. Thus, under Assumption \ref{Assp:ExistenceConsistency}(a)(b),   for $\overline{B_L}$ of sufficiently small radius, for all $(j,\ell)\in \ldsb 1,m\rdsb^{2}$,\begin{eqnarray*}
& & \E\left[\sup_{(\theta, \tau)\in \overline{B_L}} \vert \e^{\tau'\psi(X_1,\theta)}(\tau'\frac{\partial ^{2}\psi(X_1, \theta)}{\partial \theta_{\ell}\partial \theta_{j}} )\psi(X_1, \theta)\psi(X_1, \theta) '\vert \right]  \\
& \leqslant & (\sup_{(\theta,\tau) \in \overline{B_L} }\vert \tau \vert)\E\left[\sup_{(\theta, \tau)\in \overline{B_L}} \vert \e^{\tau'\psi(X_1,\theta)}
\frac{\partial ^{2}\psi(X_1, \theta)}{\partial \theta_{\ell}\partial \theta_{j}}\vert\sup_{(\theta, \tau)\in \overline{B_L}}  \vert\psi(X_1, \theta)\psi(X_1, \theta)'\vert \right] \\
& \stackrel{(a)}{\leqslant} & (\sup_{(\theta,\tau) \in \overline{B_L} }\vert \tau \vert) \sqrt{ \E\left[\sup_{(\theta, \tau)\in \overline{B_L}} \vert \e^{\tau'\psi(X_1,\theta)}
\frac{\partial ^{2}\psi(X_1, \theta)}{\partial \theta_{\ell}\partial \theta_{j}}\vert^{2}\right] } \sqrt{\E\left[\sup_{(\theta, \tau)\in \overline{B_L}}  \vert\psi(X_1, \theta)\psi(X_1, \theta)'\vert^2 \right]} \\
& \stackrel{(b)}{\leqslant} & (\sup_{(\theta,\tau) \in \overline{B_L} }\vert \tau \vert) \sqrt{\E\left[ \sup_{\theta\in \mathcal{N} } \sup_{\tau \in \Tbf(\theta)}\e^{ 2\tau'\psi(X_1, \theta)} b(X_1)^{2}\right]} \sqrt{\E\left[\sup_{\theta \in \T^{\epsilon}}  \vert\psi(X_1, \theta)\psi(X_1, \theta)'\vert^2 \right]}  \stackrel{(c)}{<}\infty.
\end{eqnarray*}
\textit{(a)} Apply the Cauchy-Schwarz inequality, and note that the supremum of the square of a positive function is the square of the supremum of the function. \textit{(b)}  Firstly, under Assumption \ref{Assp:ExistenceConsistency}(a)-(e) and (g)-(h),  by Lemma \ref{Lem:LogESPExistenceInS}ii (p. \pageref{Lem:LogESPExistenceInS}), $\Sbf$ contains an open ball centered at $(\theta_0, \tau(\theta_0))$, so that,    for $\overline{B_L}$ of sufficiently small radius, $\oBL \subset\{(\theta, \tau): \theta \in \mathcal{N}\wedge \tau \in \Tbf(\theta) \}\subset \Sbf \subset \Sbf^\epsilon $. Secondly, as the second supremum does not depend on $\tau$, $\sup_{(\theta, \tau)\in \overline{B_L}}  \vert\psi(X_1, \theta)'\vert^2 \leqslant \sup_{\theta\in \T^\epsilon}  \vert\psi(X_1, \theta)'\vert^2 $ because $\oBL \subset \Sbf $, for $\oBL$ of radius small enough.  \textit{(c)} Firstly, because $\oBL$ is bounded, $ (\sup_{(\theta,\tau) \in \overline{B_L} }\vert \tau \vert)< \infty$. Secondly, by Assumption \ref{Assp:AsymptoticNormality}(b), the  first expectation is bounded. Thirdly, by Assumption \ref{Assp:ExistenceConsistency}(g), the second expectation is also bounded.

 \textit{(xix)} Proof similar to the one of statement (xiii). The supremum of the absolute value of the product is smaller than the product of the
suprema of the absolute values. Thus,
 under Assumption \ref{Assp:ExistenceConsistency}(a)(b),   for $\overline{B_L}$ of sufficiently small radius, for all $(j,\ell)\in \ldsb 1,m\rdsb^{2}$,\begin{eqnarray*}
& & \E\left[\sup_{(\theta, \tau)\in \overline{B_L}} \vert \e^{\tau'\psi(X_1,\theta)}(\tau'\frac{\partial \psi(X_1, \theta)}{\partial \theta_{\ell}} )(\tau'\frac{\partial \psi(X_1, \theta)}{\partial \theta_{j}} )\psi(X_1, \theta)\psi(X_1, \theta) '\vert \right]  \\
& \leqslant & (\sup_{(\theta,\tau) \in \overline{B_L} }\vert \tau \vert^{2})\E\left[\sup_{(\theta, \tau)\in \overline{B_L}} \left( \e^{\tau'\psi(X_1,\theta)}
\vert\frac{\partial \psi(X_1, \theta)}{\partial \theta_{\ell}}\vert\vert\frac{\partial \psi(X_1, \theta)}{\partial \theta_{j}}\vert\right)\sup_{(\theta, \tau)\in \overline{B_L}}  \vert\psi(X_1, \theta)\psi(X_1, \theta)'\vert \right] \\
& \stackrel{(a)}{\leqslant} & (\sup_{(\theta,\tau) \in \overline{B_L} }\vert \tau \vert^{2}) \sqrt{ \E\left[\sup_{(\theta, \tau)\in \overline{B_L}} \left( \e^{\tau'\psi(X_1,\theta)}
\vert\frac{\partial \psi(X_1, \theta)}{\partial \theta_{\ell}}\vert\vert\frac{\partial \psi(X_1, \theta)}{\partial \theta_{j}}\vert\right)^{2}\right] } \sqrt{\E\left[\sup_{(\theta, \tau)\in \overline{B_L}}  \vert\psi(X_1, \theta)\psi(X_1, \theta)'\vert^2 \right]} \\
& \stackrel{(b)}{\leqslant} & (\sup_{(\theta,\tau) \in \overline{B_L} }\vert \tau \vert^{2}) \sqrt{\E\left[ \sup_{\theta\in \mathcal{N} } \sup_{\tau \in \Tbf(\theta)}\e^{ 2\tau'\psi(X_1, \theta)} b(X_1)^{4}\right]} \sqrt{\E\left[\sup_{\theta \in \T^{\epsilon}}  \vert\psi(X_1, \theta)\psi(X_1, \theta)'\vert^2 \right]}  \stackrel{(c)}{<}\infty.
\end{eqnarray*}
\textit{(a)} Apply the Cauchy-Schwarz inequality, and note that the supremum of the square of a positive function is the square of the supremum of the function. \textit{(b)}  Firstly, under Assumption \ref{Assp:ExistenceConsistency}(a)-(e) and (g)-(h),  by Lemma \ref{Lem:LogESPExistenceInS}ii (p. \pageref{Lem:LogESPExistenceInS}), $\Sbf$ contains an open ball centered at $(\theta_0, \tau(\theta_0))$, so that,    for $\overline{B_L}$ of sufficiently small radius, $\oBL \subset\{(\theta, \tau): \theta \in \mathcal{N}\wedge \tau \in \Tbf(\theta) \}\subset \Sbf \subset \Sbf^\epsilon $. Secondly, as the second supremum does not depend on $\tau$, $\sup_{(\theta, \tau)\in \overline{B_L}}  \vert\psi(X_1, \theta)\psi(X_1, \theta)'\vert^2 \leqslant \sup_{\theta\in \T^\epsilon}  \vert\psi(X_1, \theta)\psi(X_1, \theta)'\vert^2 $ because $\oBL \subset \Sbf $, for $\oBL$ of radius small enough.  \textit{(c)} Firstly, because $\oBL$ is bounded, $ (\sup_{(\theta,\tau) \in \overline{B_L} }\vert \tau \vert^{2})< \infty$. Secondly, by Assumption \ref{Assp:AsymptoticNormality}(b), the  first expectation is bounded. Thirdly, by Assumption \ref{Assp:ExistenceConsistency}(g), the second expectation is also bounded.
 \end{proof}

\begin{lem}[Finiteness of the expectations of the supremum of the terms from $\frac{\partial ^{2}L_T(\theta, \tau)}{\partial \tau_{k}\partial \theta_j}$]\label{Lem:Bound2ndPartialLPartialTau} Under Assumptions \ref{Assp:ExistenceConsistency} and \ref{Assp:AsymptoticNormality}, there exists a closed ball $\overline{B_L}$ centered at $(\theta_0, \tau(\theta_0))$ with strictly positive radius s.t., for all $(k, j)\in \ldsb 1,m\rdsb^2$,

\begin{enumerate}
\item[(i)] $\E\left[\sup_{(\theta, \tau)\in \overline{B_L}}  \e^{\tau'\psi(X_{1,}\theta)} \right] < \infty  $;
\item[(ii)]  $\E\left[\sup_{(\theta, \tau)\in \overline{B_L}} \vert \e^{\tau'\psi(X_1,\theta)}\tau'\frac{\partial  \psi(X_1,\theta)}{\partial \theta_{\ell}}
  \psi_k(X_1,\theta)\vert \right] < \infty$;
\item[(iii)] $\E\left[\sup_{(\theta, \tau)\in \overline{B_L}} \vert \e^{\tau'\psi(X_1,\theta)}\frac{\partial  \psi_k(X_1,\theta)}{\partial \theta_{\ell}}\vert \right] < \infty$;
\item[(iv)] $\E\left[\sup_{(\theta, \tau)\in \overline{B_L}} \vert \e^{\tau'\psi(X_1,\theta)}\tau' \frac{\partial  \psi(X_1,\theta)}{\partial \theta_{\ell}}\vert \right] < \infty$;
\item[(v)] $\E\left[\sup_{(\theta, \tau)\in \overline{B_L}} \vert \e^{\tau'\psi(X_1,\theta)}  \psi_k(X_1,\theta)\vert \right] < \infty$;
\item[(vi)] $\E\left[\sup_{(\theta, \tau)\in \overline{B_L}} \vert \e^{\tau'\psi(X_1,\theta)} \frac{\partial  \psi(X_1,\theta)}{\partial \theta' }\vert \right] < \infty$;
\item[(vii)] $\E\left[\sup_{(\theta, \tau)\in \overline{B_L}} \vert \e^{\tau'\psi(X_1,\theta)} \psi_k(X_1,\theta)
      \frac{\partial  \psi(X_1,\theta)}{\partial \theta' }\vert \right] < \infty$;
\item[(viii)]$\E\left[\sup_{(\theta, \tau)\in \overline{B_L}} \vert \e^{\tau'\psi(X_1,\theta)} \frac{\partial  \psi(X_1,\theta)}{\partial \theta' }
\tau' \frac{\partial  \psi(X_1,\theta)}{\partial \theta_{j} }\vert \right] < \infty$;
\item[(ix)]$\E\left[\sup_{(\theta, \tau)\in \overline{B_L}} \vert \e^{\tau'\psi(X_1,\theta)}\frac{\partial^2  \psi(X_1,\theta)}{ \partial \theta_{j} \partial \theta' }\vert \right] < \infty$;
\item[(x)]$\E\left[\sup_{(\theta, \tau)\in \overline{B_L}} \vert \e^{\tau'\psi(X_1,\theta)} \psi_k(X_1,\theta)
      \frac{\partial  \psi(X_1,\theta)}{\partial \theta' }
\tau' \frac{\partial  \psi(X_1,\theta)}{\partial \theta_{j} } \vert \right] < \infty$;
\item[(xi)]$\E\left[\sup_{(\theta, \tau)\in \overline{B_L}} \vert \e^{\tau'\psi(X_1,\theta)} \psi_k(X_1,\theta)\frac{\partial^2  \psi(X_1,\theta)}{ \partial \theta_{j} \partial \theta' }\vert \right] < \infty$;
\item[(xii)]$\E\left[\sup_{(\theta, \tau)\in \overline{B_L}} \vert \e^{\tau'\psi(X_1,\theta)}
\frac{\partial  \psi(X_1,\theta)}{\partial \theta' }
 \frac{\partial  \psi_k(X_1,\theta)}{\partial \theta_{j} } \vert \right] < \infty$;
\item[(xiii)]$\E\left[\sup_{(\theta, \tau)\in \overline{B_L}} \vert \e^{\tau'\psi(X_1,\theta)}
 \psi(X_1,\theta)   \psi(X_1,\theta)'\vert \right] < \infty$;
\item[(xiv)] $\E\left[\sup_{(\theta, \tau)\in \overline{B_L}} \vert \e^{\tau'\psi(X_1,\theta)} \psi_k(X_1,\theta)
     \psi(X_1,\theta)   \psi(X_1,\theta)'\vert \right] < \infty$;
\item[(xv)] $\E\left[\sup_{(\theta, \tau)\in \overline{B_L}} \vert \e^{\tau'\psi(X_1,\theta)} \tau' \frac{\partial \psi(X_1,\theta)}{\partial \theta_{j}}   \psi(X_1,\theta)   \psi(X_1,\theta)'\vert \right] < \infty$;
\item[(xvi)]  $\E\left[\sup_{(\theta, \tau)\in \overline{B_L}} \vert \e^{\tau'\psi(X_1,\theta)} \frac{\partial \psi(X_1,\theta)}{\partial \theta_{j}}    \psi(X_1,\theta)'\vert \right] < \infty$;
\item[(xvii)] $\E\left[\sup_{(\theta, \tau)\in \overline{B_L}} \vert \e^{\tau'\psi(X_1,\theta)} \psi_k(X_1,\theta)\tau' \frac{\partial \psi(X_1,\theta)}{\partial \theta_{j}}   \psi(X_1,\theta)   \psi(X_1,\theta)'\vert \right] < \infty$;

\item[(xviii)] $\E\left[\sup_{(\theta, \tau)\in \overline{B_L}} \vert \e^{\tau'\psi(X_1,\theta)} \psi_k(X_1,\theta)\frac{\partial \psi(X_1,\theta)}{\partial \theta_{j}}    \psi(X_1,\theta)'\vert \right] < \infty$; and
\item[(xix)] $\E\left[\sup_{(\theta, \tau)\in \overline{B_L}} \vert \e^{\tau'\psi(X_1,\theta)} \frac{\partial  \psi_k(X_1,\theta)}{\partial \theta_{j}}   \psi(X_1,\theta)   \psi(X_1,\theta)'\vert \right] < \infty$.
\end{enumerate}
\end{lem}
\begin{proof} The proofs are similar to the ones of  Lemma \ref{Lem:Bound2ndPartialLPartialTheta} (p. \pageref{Lem:Bound2ndPartialLPartialTheta}): We only use more often the  inequality  that states that the norm of a component of a vector is smaller than the norm of the vector (e.g., $\vert\psi_k(X_1, \theta)\vert\leqslant \sqrt{\sum_{l=1}^m \psi_l(X_1, \theta)^2}=\vert \psi(X_1, \theta) \vert$). Thus, we only provide proof sketches.

\textit{(i)} See Lemma \ref{Lem:Bound2ndPartialLPartialTheta}i p. \pageref{Lem:Bound2ndPartialLPartialTheta}.

\textit{(ii)}  For $\overline{B_L}$ of sufficiently small radius, for all $(k,j)\in \ldsb 1,m\rdsb^{2}$,\begin{eqnarray*}
& & \E\left[\sup_{(\theta, \tau)\in \overline{B_L}} \vert \e^{\tau'\psi(X_1,\theta)}\tau'\frac{\partial  \psi(X_1,\theta)}{\partial \theta_{j}}
  \psi_k(X_1,\theta)\vert \right]  \\
& \leqslant & (\sup_{(\theta,\tau) \in \overline{B_L} }\vert \tau \vert)\E\left[\sup_{(\theta, \tau)\in \overline{B_L}} \vert \e^{\tau'\psi(X_1,\theta)}
\frac{\partial \psi(X_1, \theta)}{\partial \theta_{j}}\vert\sup_{(\theta, \tau)\in \overline{B_L}}  \vert\psi_{k}(X_1, \theta)'\vert \right] \\
& \stackrel{}{\leqslant} & (\sup_{(\theta,\tau) \in \overline{B_L} }\vert \tau \vert)\sqrt{ \E\left[\sup_{(\theta, \tau)\in \overline{B_L}} \vert \e^{\tau'\psi(X_1,\theta)}
\frac{\partial \psi(X_1, \theta)}{\partial \theta_{j}}\vert^{2}\right] } \sqrt{\E\left[\sup_{(\theta, \tau)\in \overline{B_L}}  \vert\psi_{k}(X_1, \theta)'\vert^2 \right]} \\
& \stackrel{}{\leqslant} & (\sup_{(\theta,\tau) \in \overline{B_L} }\vert \tau \vert) \sqrt{\E\left[ \sup_{\theta\in \mathcal{N} } \sup_{\tau \in \Tbf(\theta)}\e^{ 2\tau'\psi(X_1, \theta)} b(X_1)^{2}\right]} \sqrt{\E\left[\sup_{\theta \in \T^{\epsilon}}  \vert\psi(X_1, \theta)'\vert^2 \right]}  \stackrel{}{<}\infty,
\end{eqnarray*}
where the last inequality follows from  Assumption \ref{Assp:AsymptoticNormality}(b),  and    Lemma \ref{Lem:BoundAsTiltingEquation}i (p. \pageref{Lem:BoundAsTiltingEquation}), under Assumption \ref{Assp:ExistenceConsistency}(a)(b)(g).

\textit{(iii)} For $\oBL$ of sufficiently small radius,  for all $(k,j)\in \ldsb 1,m\rdsb^{2}$, $\E\left[\sup_{(\theta, \tau)\in \overline{B_L}} \vert \e^{\tau'\psi(X_1,\theta)}\frac{\partial  \psi_k(X_1,\theta)}{\partial \theta_{j}}\vert \right]\leqslant \E\left[\sup_{(\theta, \tau)\in \overline{B_L}}  \e^{\tau'\psi(X_1,\theta)}\vert \frac{\partial \psi_{k}(X_1, \theta)}{\partial \theta'}\vert \right]  \leqslant  \E\left[ \sup_{\theta\in \mathcal{N} } \sup_{\tau \in \Tbf(\theta)}\e^{ \tau'\psi(X_1, \theta)} b(X_1)\right] < \infty$, where the last inequality follows from Assumption \ref{Assp:AsymptoticNormality}(b).

\textit{(iv)} See Lemma \ref{Lem:Bound2ndPartialLPartialTheta}iv p. \pageref{Lem:Bound2ndPartialLPartialTheta}.

\textit{(v)} For $\oBL$ of sufficiently small radius,  for all $k\in \ldsb 1,m\rdsb$, $\E\left[\sup_{(\theta, \tau)\in \overline{B_L}} \vert \e^{\tau'\psi(X_1,\theta)}  \psi_k(X_1,\theta)\vert \right] \leqslant \E\left[\sup_{(\theta, \tau)\in \overline{B_L}} \vert \e^{\tau'\psi(X_1,\theta)}  \psi(X_1,\theta)\vert \right]\leqslant \E\left[\sup_{(\theta, \tau)\in \Sbf^\epsilon} \vert \e^{\tau'\psi(X_1,\theta)}  \psi(X_1,\theta)\vert \right]< \infty$ where the last inequality follows from Lemma \ref{Lem:BoundAsTiltingEquation}ii (p. \pageref{Lem:BoundAsTiltingEquation}) under Assumption \ref{Assp:ExistenceConsistency}(a)(b)(e)(g).

\textit{(vi)} See Lemma \ref{Lem:Bound2ndPartialLPartialTheta}v p. \pageref{Lem:Bound2ndPartialLPartialTheta}.

\textit{(vii)} For $\overline{B_L}$ of sufficiently small radius, for all $k\in \ldsb 1,m\rdsb$,\begin{eqnarray*}
& & \E\left[\sup_{(\theta, \tau)\in \overline{B_L}} \vert \e^{\tau'\psi(X_1,\theta)} \psi_k(X_1,\theta)
      \frac{\partial  \psi(X_1,\theta)}{\partial \theta' }\vert \right]   \\
& \leqslant & \E\left[\sup_{(\theta, \tau)\in \overline{B_L}} \vert \e^{\tau'\psi(X_1,\theta)}
\frac{\partial  \psi(X_1,\theta)}{\partial \theta' }\vert\sup_{(\theta, \tau)\in \overline{B_L}}  \vert\psi_{k}(X_1, \theta)\vert \right] \\
& \stackrel{}{\leqslant} & \sqrt{ \E\left[\sup_{(\theta, \tau)\in \overline{B_L}} \vert \e^{\tau'\psi(X_1,\theta)}
\frac{\partial  \psi(X_1,\theta)}{\partial \theta' }\vert^{2}\right] } \sqrt{\E\left[\sup_{(\theta, \tau)\in \overline{B_L}}  \vert\psi_{k}(X_1, \theta)\vert^2 \right]} \\
& \stackrel{}{\leqslant} &  \sqrt{\E\left[ \sup_{\theta\in \mathcal{N} } \sup_{\tau \in \Tbf(\theta)}\e^{ 2\tau'\psi(X_1, \theta)} b(X_1)^{2}\right]} \sqrt{\E\left[\sup_{\theta \in \T^{\epsilon}}  \vert\psi(X_1, \theta)\vert^2 \right]}  \stackrel{}{<}\infty,
\end{eqnarray*}
where the last inequality follows from Assumption \ref{Assp:AsymptoticNormality}(b) and Lemma \ref{Lem:BoundAsTiltingEquation}i (p. \pageref{Lem:BoundAsTiltingEquation}) under Assumption \ref{Assp:ExistenceConsistency}(a)(b)(g).

\textit{(viii)} For all $j\in \ldsb 1,m\rdsb$, $\E\left[\sup_{(\theta, \tau)\in \overline{B_L}} \vert \e^{\tau'\psi(X_1,\theta)} \frac{\partial  \psi(X_1,\theta)}{\partial \theta' }
\tau' \frac{\partial  \psi(X_1,\theta)}{\partial \theta_{j} }\vert \right]<(\sup_{(\theta,\tau) \in \overline{B_L} }\vert \tau \vert)\\  \E\left[\sup_{(\theta, \tau)\in \overline{B_L}}  \e^{\tau'\psi(X_1,\theta)}\vert \frac{\partial  \psi(X_1,\theta)}{\partial \theta' }\vert \vert \frac{\partial \psi(X_1, \theta)}{\partial \theta'}\vert\right] \leqslant (\sup_{(\theta,\tau) \in \overline{B_L} }\vert \tau \vert) \E\left[ \sup_{\theta\in \mathcal{N} } \sup_{\tau \in \Tbf(\theta)}\e^{ \tau'\psi(X_1, \theta)} b(X_1)^{2}\right]< \infty $ where the two last inequalities follow from Assumption \ref{Assp:AsymptoticNormality}(b) and the boundedness of $ \overline{B_L}$.

\textit{(ix)} See Lemma  \ref{Lem:Bound2ndPartialLPartialTheta}vi p. \pageref{Lem:Bound2ndPartialLPartialTheta}.

\textit{(x)} Under Assumption \ref{Assp:ExistenceConsistency}(a)(b),   for $\overline{B_L}$ of sufficiently small radius, for all $( j,k)\in \ldsb 1,m\rdsb^2$,\begin{eqnarray*}
& & \E\left[\sup_{(\theta, \tau)\in \overline{B_L}} \vert \e^{\tau'\psi(X_1,\theta)} \psi_k(X_1,\theta)
      \frac{\partial  \psi(X_1,\theta)}{\partial \theta' }
\tau' \frac{\partial  \psi(X_1,\theta)}{\partial \theta_{j} } \vert \right] < \infty  \\
& \leqslant & (\sup_{(\theta,\tau) \in \overline{B_L} }\vert \tau \vert)\E\left[\sup_{(\theta, \tau)\in \overline{B_L}} \left( \e^{\tau'\psi(X_1,\theta)}
\vert\frac{\partial  \psi(X_1,\theta)}{\partial \theta' }\vert\vert\frac{\partial \psi(X_1, \theta)}{\partial \theta_{j}}\vert\right)\sup_{(\theta, \tau)\in \overline{B_L}}  \vert\psi_{k}(X_1, \theta)\vert \right] \\
& \stackrel{}{\leqslant} & (\sup_{(\theta,\tau) \in \overline{B_L} }\vert \tau \vert) \sqrt{ \E\left[\sup_{(\theta, \tau)\in \overline{B_L}} \left( \e^{\tau'\psi(X_1,\theta)}
\vert\frac{\partial  \psi(X_1,\theta)}{\partial \theta' }\vert\vert\frac{\partial \psi(X_1, \theta)}{\partial \theta_{j}}\vert\right)^{2}\right] } \sqrt{\E\left[\sup_{(\theta, \tau)\in \overline{B_L}}  \vert\psi_{k}(X_1, \theta)\vert^2 \right]} \\
& \stackrel{}{\leqslant} & (\sup_{(\theta,\tau) \in \overline{B_L} }\vert \tau \vert) \sqrt{\E\left[ \sup_{\theta\in \mathcal{N} } \sup_{\tau \in \Tbf(\theta)}\e^{ 2\tau'\psi(X_1, \theta)} b(X_1)^{4}\right]} \sqrt{\E\left[\sup_{\theta \in \T^{\epsilon}}  \vert\psi(X_1, \theta)\vert^2 \right]}  \stackrel{}{<}\infty,
\end{eqnarray*}
where the last inequality follows from the boundedness of $\oBL$, Assumption \ref{Assp:AsymptoticNormality}(b) and Lemma \ref{Lem:BoundAsTiltingEquation}i (p. \pageref{Lem:BoundAsTiltingEquation}) under Assumption \ref{Assp:ExistenceConsistency}(a)(b)(g) and (e).

\textit{(xi)} Under Assumption \ref{Assp:ExistenceConsistency}(a)(b),   for $\overline{B_L}$ of sufficiently small radius, for all $(j,k)\in \ldsb 1,m\rdsb^2$,\begin{eqnarray*}
& & \E\left[\sup_{(\theta, \tau)\in \overline{B_L}} \vert \e^{\tau'\psi(X_1,\theta)} \psi_k(X_1,\theta)\frac{\partial^2  \psi(X_1,\theta)}{ \partial \theta_{j} \partial \theta' }\vert \right] \\
& \leqslant & \E\left[\sup_{(\theta, \tau)\in \overline{B_L}} \vert \e^{\tau'\psi(X_1,\theta)}
\frac{\partial^2  \psi(X_1,\theta)}{ \partial \theta_{j} \partial \theta' }\vert\sup_{(\theta, \tau)\in \overline{B_L}}  \vert\psi_{k}(X_1, \theta)'\vert \right] \\
& \stackrel{}{\leqslant} & \sqrt{ \E\left[\sup_{(\theta, \tau)\in \overline{B_L}} \vert \e^{\tau'\psi(X_1,\theta)}
\frac{\partial^2  \psi(X_1,\theta)}{ \partial \theta_{j} \partial \theta' }\vert^{2}\right] } \sqrt{\E\left[\sup_{(\theta, \tau)\in \overline{B_L}}  \vert\psi_{k}(X_1, \theta)'\vert^2 \right]} \\
& \stackrel{}{\leqslant} & \sqrt{\E\left[ \sup_{\theta\in \mathcal{N} } \sup_{\tau \in \Tbf(\theta)}\e^{ 2\tau'\psi(X_1, \theta)} b(X_1)^{2}\right]} \sqrt{\E\left[\sup_{\theta \in \T^{\epsilon}}  \vert\psi(X_1, \theta)'\vert^2 \right]}  \stackrel{}{<}\infty,
\end{eqnarray*}
where the last inequality follows from Assumption \ref{Assp:AsymptoticNormality}(b) and Lemma \ref{Lem:BoundAsTiltingEquation} (p. \pageref{Lem:BoundAsTiltingEquation}), under Assumption \ref{Assp:ExistenceConsistency}(a)(b)(g) and (e).

\textit{(xii)} Under Assumption \ref{Assp:ExistenceConsistency}(a)-(e) and (g)-(h),  for $\oBL$ of sufficiently small radius,  for all $ (j,k)\in \ldsb 1,m\rdsb^{2}$,    $\E\left[\sup_{(\theta, \tau)\in \overline{B_L}} \vert \e^{\tau'\psi(X_1,\theta)}
\frac{\partial  \psi(X_1,\theta)}{\partial \theta' }
 \frac{\partial  \psi_k(X_1,\theta)}{\partial \theta_{j} } \vert \right]<  \E\left[\sup_{(\theta, \tau)\in \overline{B_L}}  \e^{\tau'\psi(X_1,\theta)}\vert \frac{\partial  \psi(X_1,\theta)}{\partial \theta' }\vert \vert \frac{\partial \psi(X_1, \theta)}{\partial \theta_{j}}\vert\right]\leqslant  \E\left[ \sup_{\theta\in \mathcal{N} } \sup_{\tau \in \Tbf(\theta)}\e^{ \tau'\psi(X_1, \theta)} b(X_1)^{2}\right]< \infty $ where the two last inequalities follow from Assumption \ref{Assp:AsymptoticNormality}(b).

\textit{(xiii)} See Lemma  \ref{Lem:Bound2ndPartialLPartialTheta}xii p. \pageref{Lem:Bound2ndPartialLPartialTheta}.

\textit{(xiv)} Under Assumption \ref{Assp:ExistenceConsistency}(a)-(e) and (g)-(h),    for $\overline{B_L}$ of sufficiently small radius, for all $k\in \ldsb 1,m\rdsb$,\begin{eqnarray*}
& & \E\left[\sup_{(\theta, \tau)\in \overline{B_L}} \vert \e^{\tau'\psi(X_1,\theta)} \psi_k(X_1,\theta)
     \psi(X_1,\theta)   \psi(X_1,\theta)'\vert \right]   \\
& \leqslant & (\sup_{(\theta,\tau) \in \overline{B_L} }\vert \tau \vert^{2})\E\left[\sup_{(\theta, \tau)\in \overline{B_L}} \left( \e^{\tau'\psi(X_1,\theta)}
\vert\psi_k(X_1,\theta)\vert \right)\sup_{(\theta, \tau)\in \overline{B_L}}  \vert\psi(X_1, \theta)\psi(X_1, \theta)'\vert \right] \\
& \stackrel{}{\leqslant} & (\sup_{(\theta,\tau) \in \overline{B_L} }\vert \tau \vert^{2}) \sqrt{ \E\left[\sup_{(\theta, \tau)\in \overline{B_L}} \left( \e^{\tau'\psi(X_1,\theta)}
\vert\psi_k(X_1,\theta)\vert\right)^{2}\right] } \sqrt{\E\left[\sup_{(\theta, \tau)\in \overline{B_L}}  \vert\psi(X_1, \theta)\psi(X_1, \theta)'\vert^2 \right]} \\
& \stackrel{}{\leqslant} & (\sup_{(\theta,\tau) \in \overline{B_L} }\vert \tau \vert^{2}) \sqrt{\E\left[ \sup_{\theta\in \mathcal{N} } \sup_{\tau \in \Tbf(\theta)}\e^{ 2\tau'\psi(X_1, \theta)} b(X_1)^{2}\right]} \sqrt{\E\left[\sup_{\theta \in \T^{\epsilon}}  \vert\psi(X_1, \theta)\psi(X_1, \theta)'\vert^2 \right]}  \stackrel{}{<}\infty,
\end{eqnarray*}
where the last inequality follows from the boundedness of  $\oBL$, Assumptions \ref{Assp:ExistenceConsistency}(g) and \ref{Assp:AsymptoticNormality}(b).

\textit{(xv)} See Lemma  \ref{Lem:Bound2ndPartialLPartialTheta}xiv p. \pageref{Lem:Bound2ndPartialLPartialTheta}.

\textit{(xvi)} See Lemma  \ref{Lem:Bound2ndPartialLPartialTheta}xiii p. \pageref{Lem:Bound2ndPartialLPartialTheta}.

\textit{(xvii)} Under Assumption \ref{Assp:ExistenceConsistency}(a)(b),   for $\overline{B_L}$ of sufficiently small radius, for all $(k,j)\in \ldsb 1,m\rdsb^{2}$,\begin{eqnarray*}
& & \E\left[\sup_{(\theta, \tau)\in \overline{B_L}} \vert \e^{\tau'\psi(X_1,\theta)} \psi_k(X_1,\theta)\tau' \frac{\partial \psi(X_1,\theta)}{\partial \theta_{j}}   \psi(X_1,\theta)   \psi(X_1,\theta)'\vert \right]   \\
& \leqslant & (\sup_{(\theta,\tau) \in \overline{B_L} }\vert \tau \vert)\E\left[\sup_{(\theta, \tau)\in \overline{B_L}} \left( \e^{\tau'\psi(X_1,\theta)}
\vert\psi_k(X_1,\theta)\vert\vert\frac{\partial \psi(X_1, \theta)}{\partial \theta_{j}}\vert\right)\sup_{(\theta, \tau)\in \overline{B_L}}  \vert\psi(X_1, \theta)\psi(X_1, \theta)'\vert \right] \\
& \stackrel{}{\leqslant} & (\sup_{(\theta,\tau) \in \overline{B_L} }\vert \tau \vert) \sqrt{ \E\left[\sup_{(\theta, \tau)\in \overline{B_L}} \left( \e^{\tau'\psi(X_1,\theta)}
\vert\psi_k(X_1,\theta)\vert\vert\frac{\partial \psi(X_1, \theta)}{\partial \theta_{j}}\vert\right)^{2}\right] } \sqrt{\E\left[\sup_{(\theta, \tau)\in \overline{B_L}}  \vert\psi(X_1, \theta)\psi(X_1, \theta)'\vert^2 \right]} \\
& \stackrel{}{\leqslant} & \negthickspace(\negthickspace\negthickspace\sup_{(\theta,\tau) \in \overline{B_L} }\negthickspace\negthickspace\negthickspace\vert \tau \vert) \sqrt{ \E\left[\negthickspace\sup_{\theta\in \mathcal{N} } \sup_{\tau \in \Tbf(\theta)}\negthickspace  \e^{2\tau'\psi(X_1,\theta)}
b(X)^{4}\right] }  \negthickspace\sqrt{\E\left[\negthickspace\sup_{(\theta, \tau)\in \overline{B_L}}  \negthickspace\vert\psi(X_1, \theta)\psi(X_1, \theta)'\vert^2 \right]} \\
& \stackrel{}{<}&\infty,
\end{eqnarray*}
where the last inequality follows from the boundedness of  $\oBL$, Assumption \ref{Assp:AsymptoticNormality}(b) and Assumption \ref{Assp:ExistenceConsistency}(g).

\textit{(xviii)} Under Assumption \ref{Assp:ExistenceConsistency}(a)(b),   for $\overline{B_L}$ of sufficiently small radius, for all $(k,j)\in \ldsb 1,m\rdsb^{2}$,\begin{eqnarray*}
& & \E\left[\sup_{(\theta, \tau)\in \overline{B_L}} \vert \e^{\tau'\psi(X_1,\theta)} \psi_k(X_1,\theta)\frac{\partial \psi(X_1,\theta)}{\partial \theta_{j}}    \psi(X_1,\theta)'\vert \right]  \\
& \leqslant & \E\left[\sup_{(\theta, \tau)\in \overline{B_L}} \vert \e^{\tau'\psi(X_1,\theta)}
\psi_k(X_1,\theta)\frac{\partial \psi(X_1, \theta)}{\partial \theta_{j}}\vert\sup_{(\theta, \tau)\in \overline{B_L}}  \vert\psi(X_1, \theta)'\vert \right] \\
& \stackrel{}{\leqslant} & \sqrt{ \E\left[\sup_{(\theta, \tau)\in \overline{B_L}} \vert \e^{\tau'\psi(X_1,\theta)}\psi_k(X_1,\theta)
\frac{\partial \psi(X_1, \theta)}{\partial \theta_{j}}\vert^{2}\right] } \sqrt{\E\left[\sup_{(\theta, \tau)\in \overline{B_L}}  \vert\psi(X_1, \theta)\vert^2 \right]} \\
& \stackrel{}{\leqslant} & \sqrt{\E\left[ \sup_{\theta\in \mathcal{N} } \sup_{\tau \in \Tbf(\theta)}\e^{ 2\tau'\psi(X_1, \theta)} b(X_1)^{4}\right]} \sqrt{\E\left[\sup_{\theta \in \T^{\epsilon}}  \vert\psi(X_1, \theta)\vert^2 \right]}  \stackrel{}{<}\infty,
\end{eqnarray*}
where the last inequality follows from  Assumption \ref{Assp:AsymptoticNormality}(b), and Lemma \ref{Lem:BoundAsTiltingEquation}i (p. \pageref{Lem:BoundAsTiltingEquation}), under Assumption \ref{Assp:ExistenceConsistency}(a)(b)(g).

\textit{(xix)}  Under Assumption \ref{Assp:ExistenceConsistency}(a)(b),   for $\overline{B_L}$ of sufficiently small radius, for all $(k,j)\in \ldsb 1,m\rdsb^{2}$,\begin{eqnarray*}
& & \E\left[\sup_{(\theta, \tau)\in \overline{B_L}} \vert \e^{\tau'\psi(X_1,\theta)} \frac{\partial  \psi_k(X_1,\theta)}{\partial \theta_{j}}   \psi(X_1,\theta)   \psi(X_1,\theta)'\vert \right] \\
& \leqslant & \E\left[\sup_{(\theta, \tau)\in \overline{B_L}} \vert \e^{\tau'\psi(X_1,\theta)}
\frac{\partial \psi_{k}(X_1, \theta)}{\partial \theta_{j}}\vert\sup_{(\theta, \tau)\in \overline{B_L}}  \vert\psi(X_1,\theta)\psi(X_1, \theta)'\vert \right] \\
& \stackrel{}{\leqslant} & \sqrt{ \E\left[\sup_{(\theta, \tau)\in \overline{B_L}} \vert \e^{\tau'\psi(X_1,\theta)}
\frac{\partial \psi_{k}(X_1, \theta)}{\partial \theta_{j}}\vert^{2}\right] } \sqrt{\E\left[\sup_{(\theta, \tau)\in \overline{B_L}}  \vert\psi(X_1,\theta)\psi(X_1, \theta)'\vert^2 \right]} \\
& \stackrel{}{\leqslant} & \sqrt{\E\left[ \sup_{\theta\in \mathcal{N} } \sup_{\tau \in \Tbf(\theta)}\e^{ 2\tau'\psi(X_1, \theta)} b(X_1)^{2}\right]} \sqrt{\E\left[\sup_{\theta \in \T^{\epsilon}}  \vert\psi(X_1,\theta)\psi(X_1, \theta)'\vert^2 \right]}  \stackrel{}{<}\infty,
\end{eqnarray*}
where the last inequality follows from Assumption \ref{Assp:AsymptoticNormality}(b) and Assumption \ref{Assp:ExistenceConsistency}(g).
\end{proof}

\begin{lem} \label{Lem:ApproximateFOC} Under Assumptions \ref{Assp:ExistenceConsistency} and \ref{Assp:AsymptoticNormality}, $\P$-a.s. as $T \rightarrow \infty$, $ \frac{\partial L_T(\hat{\theta}_T,\tau_T(\hat{\theta}_T))}{\partial \theta }=O(T^{-1}) $.

\end{lem}
\begin{proof}
Unlike in most of the rest of the paper, for clarity, in this proof we do not use the potentially ambiguous notation that denotes $\left.\frac{\partial L_T(\theta, \tau)}{\partial \theta }\right\vert_{(\theta, \tau)=(\hat{\theta}_T, \tau_T(\hat{\theta}_T))}$ with $ \frac{\partial L_T(\hat{\theta}_T,\tau_T(\hat{\theta}_T))}{\partial \theta }$.\footnote{\label{FootN:AmbiguousNotation}This is a potentially ambiguous  notation in the sense that $\frac{\partial L_T(\hat{\theta}_T,\tau_T(\hat{\theta}_T))}{\partial \theta }$ could also denote $\left.\frac{\partial L_T(\theta, \tau_T(\theta))}{\partial \theta }\right\vert_{\theta=\hat{\theta}_T}$. Except when indicated otherwise, such an ambiguity cannot occur because we never use  derivatives of $\theta \mapsto L_T(\theta, \tau_T(\theta))$.}

Under Assumptions \ref{Assp:ExistenceConsistency} and \ref{Assp:AsymptoticNormality}(a), by subsection \ref{Sec:LTAndDerivatives} (p. \pageref{Sec:LTAndDerivatives}), the function  $L_T(\theta, \tau)$ is well-defined and twice continuously differentiable in a neighborhood of $(\theta_0'\; \tau(\theta_0)')$   $\P$-a.s. for $T$ big enough. Moreover, under Assumption \ref{Assp:ExistenceConsistency}(a)(b) and (d)-(h), by Lemma \ref{Lem:DerivativeImplicitFunctionImplicit}i (p. \pageref{Lem:DerivativeImplicitFunctionImplicit}),  $\tau_T(.)$ is continuously differentiable in $\T$. Now,  under Assumption \ref{Assp:ExistenceConsistency}, by Theorem \ref{theorem:ConsistencyAsymptoticNormality}i (p. \pageref{theorem:ConsistencyAsymptoticNormality}) and  Lemma \ref{Lem:Schennachtheorem10PfFirstSteps}iii (p. \pageref{Lem:Schennachtheorem10PfFirstSteps}), $\P$-a.s., $\hat{\theta}_T \rightarrow \theta_0$ and $\tau_T(\hat{\theta}_T )\rightarrow \tau(\theta_0)$, so that $\P$-a.s. for $T$ big enough, $(\hat{\theta}_T ' \; \tau_T(\hat{\theta}_T)')$ is in any arbitrary small neighborhood of $(\theta_0'\; \tau(\theta_0)')$. Therefore,  under Assumption \ref{Assp:ExistenceConsistency} and \ref{Assp:AsymptoticNormality}(a),  by the chain rule theorem \citep[e.g.,][Chap. 5 sec. 11]{1988MagnusNeudecker}, $\P$-a.s. for $T$ big enough, $\theta \mapsto L_T(\theta, \tau_T(\theta))$ is continuously differentiable in a neighborhood of $\hat{\theta}_T$, and, for all $j \in \ldsb 1,m\rdsb$,
\begin{eqnarray*}
& & 0  =  \left.\frac{\partial L_T(\theta, \tau_{T}(\theta))}{\partial \theta_{j} }\right\vert_{\theta=\hat{\theta}_T}\\
& \stackrel{(a)}{\Leftrightarrow} & 0= \left.\frac{\partial L_T(\theta, \tau)}{\partial \theta_{j} }\right\vert_{(\theta, \tau)=(\hat{\theta}_T, \tau_T(\hat{\theta}_T))}+ \left.\frac{\partial L_T(\theta, \tau)}{\partial \tau' }\right\vert_{(\theta, \tau)=(\hat{\theta}_T, \tau_T(\hat{\theta}_T))}\left.\frac{\partial \tau(\theta)}{\partial \theta_{j}}\right\vert_{\theta=\hat{\theta}_T}\\
& \stackrel{}{\Leftrightarrow} & \left.\frac{\partial L_T(\theta, \tau)}{\partial \theta_{j} }\right\vert_{(\theta, \tau)=(\hat{\theta}_T, \tau_T(\hat{\theta}_T))}=-\left.\frac{\partial L_T(\theta, \tau)}{\partial \tau' }\right\vert_{(\theta, \tau)=(\hat{\theta}_T, \tau_T(\hat{\theta}_T))}\left.\frac{\partial \tau(\theta)}{\partial \theta_{j}}\right\vert_{\theta =\hat{\theta}_T }\\
& \stackrel{(b)}{\Leftrightarrow} & \left.\frac{\partial L_T(\theta, \tau)}{\partial \theta_{j} }\right\vert_{(\theta, \tau)=(\hat{\theta}_T, \tau_T(\hat{\theta}_T))}= O(T^{-1})O(1)=O(T^{-1}).
\end{eqnarray*}
\textit{(a)} It is an immediate and standard implication  of the chain rule \citep[e.g.,][chap. 5, sec. 12, exercise 3]{1988MagnusNeudecker}. \textit{(b)} Firstly, under Assumptions \ref{Assp:ExistenceConsistency} and \ref{Assp:AsymptoticNormality}, by Lemma \ref{Lem:DLDTau}iv  (p. \pageref{Lem:DLDTau}), $\P$-a.s. as $T \rightarrow \infty$, $\left.\frac{\partial L_T(\theta, \tau)}{\partial \tau' }\right\vert_{(\theta, \tau)=(\hat{\theta}_T, \tau_T(\hat{\theta}_T))}=O(T^{-1})$ because $(\hat{\theta}_T, \tau_T(\hat{\theta}_T))\rightarrow (\theta_0, \tau(\theta_0))$, $\P$-a.s. as $T \rightarrow \infty$, by Theorem \ref{theorem:ConsistencyAsymptoticNormality}i (p. \pageref{theorem:ConsistencyAsymptoticNormality}) and Lemma \ref{Lem:Schennachtheorem10PfFirstSteps}iii (p. \pageref{Lem:Schennachtheorem10PfFirstSteps}). Secondly, under Assumptions \ref{Assp:ExistenceConsistency} and \ref{Assp:AsymptoticNormality}, by Theorem \ref{theorem:ConsistencyAsymptoticNormality}i (p. \pageref{theorem:ConsistencyAsymptoticNormality}) and Lemma  \ref{Lem:DerivativeImplicitFunctionImplicit}iii (p. \pageref{Lem:DerivativeImplicitFunctionImplicit}), $\P$-a.s. as $T \rightarrow \infty$,$\left.\frac{\partial \tau(\theta)}{\partial \theta_{j}}\right\vert_{(\theta, \tau)=(\hat{\theta}_T, \tau_T(\hat{\theta}_T))}=O(1)$.
\end{proof}

\begin{lem}[First Derivative of the implicit function $\tau_T(.)$]\label{Lem:DerivativeImplicitFunctionImplicit} Under Assumption \ref{Assp:ExistenceConsistency}(a)(b) and (d)-(h),
\begin{enumerate}
\item[(i)]  $\P$-a.s. for $T$ big enough, the function  $\tau_T: \T \rightarrow \R^m$ is continuously differentiable in $\T$ and its  first derivative is
\begin{eqnarray*}
 \frac{\partial \tau_T(\theta)}{\partial \theta'}
 =  \negthickspace-\negthickspace\left[
\frac{1}{T}\sum_{t=1}^T \e^{\tau_T(\theta)'\psi_t(\theta)}\psi_t(\theta)\psi_t(\theta)   '\right]^{-1}  \negthickspace
\left[
\frac{1}{T} \sum_{t=1}^T
\e^{\tau_T(\theta)'\psi_t(\theta)}\negthickspace
\left(
\frac{\partial \psi_{t}(\theta)}{\partial \theta'}
 \negthickspace+\negthickspace \psi_{t}(\theta) \tau_T(\theta)'  \frac{\partial \psi_{t}(\theta)}{\partial \theta'}
\right) \right];
\end{eqnarray*}
\item[(ii)] for any  sequence $(\theta_T)_{T \in \N}\in \T^\N$ converging to $ \theta_0$, $\P$-a.s. for $T$ big enough, there exists $\bar{\theta}_T$ between $\theta_T$ and $\theta_0$ s.t.  $\sqrt{T} [\tau_T(\theta_T)- \tau_T(\theta_{0})]=\frac{\partial \tau_T(\bar{\theta}_T)}{\partial \theta'}\sqrt{T}(\theta_T- \theta_0) $;

\item[(iii)] under additional Assumptions \ref{Assp:ExistenceConsistency}(c) and \ref{Assp:AsymptoticNormality}(b), for any  sequence $(\theta_T)_{T \in \N}\in \T^\N$ converging to $ \theta_0$, $\P$-a.s. as $T \rightarrow \infty$,   $\frac{\partial \tau_T(\theta_T)}{\partial \theta'}\rightarrow -\E[\psi(X_1, \theta_0)\psi(X_1, \theta_0)']^{-1} \E\left[\frac{\partial \psi(X_1, \theta_0)}{\partial \theta}\right]$; and

\item[(iv)] under additional Assumptions \ref{Assp:ExistenceConsistency}(c) and \ref{Assp:AsymptoticNormality}(b), for any  sequence $(\theta_T)_{T \in \N}\in \T^\N$ converging to $ \theta_0$ s.t., as $T \rightarrow \infty$, $\sqrt{T}( \theta_T- \theta_0)=O_{\P}(1)$ $\P$-a.s. as $T \rightarrow \infty$,  $\tau_T(\theta_T)- \tau_T(\theta_{0})=-V^{-1} M(\theta_T- \theta_0)+o_{\P}(T^{-1/2}) $, where $V:=\E[\psi(X_1, \theta_0)\psi(X_1, \theta_0)']$ and $M:=\E\left[\frac{\partial \psi(X_1, \theta_0)}{\partial \theta}\right]$.
\end{enumerate}

\end{lem}
\begin{proof} \textit{(i)}
 Under  Assumption \ref{Assp:ExistenceConsistency}(a)(b) and (d)-(h), by  Lemma \ref{Lem:ESPExistence}ii (p. \pageref{Lem:ESPExistence}) and its  proof,  $\P$-a.s. for $T$ big enough, the assumptions of the standard implicit function theorem hold and $\tau_T(.)$ is continuously differentiable. Thus, under  Assumption \ref{Assp:ExistenceConsistency}(a)(b) and (d)-(h),  $\P$-a.s. for $T$ big enough, application of the implicit function theorem yields
\begin{eqnarray*}
& & \frac{\partial \tau_T(\theta)}{\partial \theta'} \nonumber\\
 & = &
\left.-\left[ \frac{\partial \left[ \frac{1}{T}\sum_{t=1}^T \e^{\tau'\psi_t(\theta)}\psi_t(\theta) \right]}{\partial \tau'}\right]^{-1} \left[ \frac{\partial \left[ \frac{1}{T}\sum_{t=1}^T \e^{\tau'\psi_t(\theta)}\psi_t(\theta) \right]}{\partial \theta'}\right]\right|_{\tau=\tau_T(\theta)} \nonumber\\
& = &
 \left.-\left[
\frac{1}{T}\sum_{t=1}^T \e^{\tau'\psi_t(\theta)}\psi_t(\theta)\psi_t(\theta)   '\right]^{-1}
\left[
\frac{1}{T} \sum_{t=1}^T
\e^{\tau'\psi_t(\theta)}
\left(
\frac{\partial \psi_{t}(\theta)}{\partial \theta'}
 + \psi_{t}(\theta) \tau'  \frac{\partial \psi_{t}(\theta)}{\partial \theta'}
\right) \right]\right|_{\tau=\tau_T(\theta)}\nonumber\\
& = & \negthickspace-\negthickspace\left[
\frac{1}{T}\sum_{t=1}^T \e^{\tau_T(\theta)'\psi_t(\theta)}\psi_t(\theta)\psi_t(\theta)   '\right]^{-1}  \negthickspace
\left[
\frac{1}{T} \sum_{t=1}^T
\e^{\tau_T(\theta)'\psi_t(\theta)}\negthickspace
\left(
\frac{\partial \psi_{t}(\theta)}{\partial \theta'}
 \negthickspace+\negthickspace \psi_{t}(\theta) \tau_T(\theta)'  \frac{\partial \psi_{t}(\theta)}{\partial \theta'}
\right) \right].
\end{eqnarray*}

\textit{(ii)} Again, under  Assumption \ref{Assp:ExistenceConsistency}(a)(b) and (d)-(h), by  Lemma \ref{Lem:ESPExistence}ii (p. \pageref{Lem:ESPExistence}),  $\P$-a.s. for $T$ big enough, $\tau_T(.)$ is continuously differentiable, so that the result follows from a first-order stochastic Taylor-Lagrange expansion  \citep[Lemma 3]{1969Jen}.

\textit{(iii)}
Firstly, under Assumption \ref{Assp:ExistenceConsistency}(a)(b)(d)(e)(g)(h), by Lemma \ref{Lem:Schennachtheorem10PfFirstSteps}iii (p. \pageref{Lem:Schennachtheorem10PfFirstSteps}), $\P$-a.s. as $T \rightarrow \infty$, $\sup_{\theta\in \T}\vert\tau_T(\theta)-\tau(\theta) \vert=o(1)$, so that $\tau_T(\theta_T)\rightarrow \tau(\theta_0)$.  Secondly, under Assumptions \ref{Assp:ExistenceConsistency} and \ref{Assp:AsymptoticNormality}, by Lemma \ref{Lem:dLdTaudTau}iv, vii and x (p. \pageref{Lem:dLdTaudTau}), for $\overline{B_L}$ a ball around $(\theta_0, \tau(\theta_0))$ of sufficiently small radius, $\E\left[\sup_{(\theta, \tau)\in \overline{B_L}} \vert \e^{\tau'\psi(X_1,\theta)} \frac{\partial \psi(X_1, \theta)}{\partial \theta'}\vert \right] < \infty$, $\E\left[\sup_{(\theta, \tau)\in \overline{B_L}} \vert \e^{\tau'\psi(X_1,\theta)} \psi(X_1,\theta)\psi(X_1,\theta)'
     \vert \right]< \infty$, and $\E\left[\sup_{(\theta, \tau)\in \overline{B_L}} \vert \e^{\tau'\psi(X_1,\theta)}\psi(X_1, \theta)\tau' \frac{\partial \psi(X_1, \theta)}{\partial \theta'}\vert\right]<\infty$.
Thus, by Assumptions \ref{Assp:ExistenceConsistency}(a)(b) and (d),   the ULLN (uniform law of large numbers) \`a la  Wald  \citep[e.g.,][pp. 24-25, Theorem 1.3.3]{2003GhoRam},  implies that, for all $k\in \ldsb 1,m\rdsb$, $\P$-a.s. as $T \rightarrow \infty$,
\begin{eqnarray*}
& &\ \frac{\partial \tau_T(\theta_T)}{\partial \theta'}\\
& \rightarrow & -\E[\e^{\tau(\theta_0)'\psi(X_{1,}\theta_{0})}\psi(X_1, \theta_0)\psi(X_1, \theta_0)']^{-1} \left\{ \E\left[\e^{\tau(\theta_0)'\psi(X_{1,}\theta_{0})}\frac{\partial \psi(X_1, \theta_0)}{\partial \theta}\right]\right. \\
& &\ \left.+\E\left[\e^{\tau(\theta_0)'\psi(X_{1,}\theta_{0})}\psi(X_1, \theta)\tau(\theta_0)'\frac{\partial \psi(X_1, \theta_0)}{\partial \theta}\right]\right\}\\
& = & -\E[\psi(X_1, \theta_0)\psi(X_1, \theta_0)']^{-1} \E\left[\frac{\partial \psi(X_1, \theta_0)}{\partial \theta}\right]
\end{eqnarray*}
because $\tau(\theta_0)=0_{m \times 1}$ by Lemma \ref{Lem:AsTiltingFct}iv (p. \pageref{Lem:AsTiltingFct}) under Assumption \ref{Assp:ExistenceConsistency}(a)-(e) and (g)-(h).

\textit{(iv)} Under Assumption \ref{Assp:ExistenceConsistency} and \ref{Assp:AsymptoticNormality}(b), by the statement (ii) of the present lemma, $\P$-a.s. as $T \rightarrow \infty$, there exists  $\bar{\theta}_T$ between $\theta_T$ and $\theta_0$ s.t.
\begin{eqnarray*}
 \tau_T(\theta_T)- \tau_T(\theta_{0})&=&\frac{\partial \tau_T(\bar{\theta}_T)}{\partial \theta'}(\theta_T- \theta_0)\\
& \stackrel{(a)}{=} & -V^{-1} M(\theta_T- \theta_0)+\left[  \frac{\partial \tau_T(\bar{\theta}_T)}{\partial \theta'}+V^{-1} M\right](\theta_T- \theta_0)\\
& \stackrel{(b)}{=} & -V^{-1} M(\theta_T- \theta_0)+o_{\P}(T^{-1/2})
\end{eqnarray*}
\textit{(a)} Add and subtract $ V^{-1} M(\theta_T- \theta_0)$. \textit{(b)} Under Assumption \ref{Assp:ExistenceConsistency} and \ref{Assp:AsymptoticNormality}(b), by the statement (ii) of the present lemma, $\P$-a.s. as $T \rightarrow \infty$, $\frac{\partial \tau_T(\bar{\theta}_T)}{\partial \theta'}+V^{-1} M=o(1)$. Moreover, by assumption, as $T \rightarrow \infty$, $\theta_T- \theta_0=O_{\P}(T^{-1/2})$, so that $ \left[  \frac{\partial \tau_T(\bar{\theta}_T)}{\partial \theta'}+V^{-1} M\right](\theta_T- \theta_0)=o_{\P}(T^{-1/2})$.
\end{proof}
\begin{rk}  As notation indicates, $\frac{\partial \tau_T(.)}{\partial \theta'}$ corresponds to a partial derivative as $\tau_T(.)$ is also a function of the data.\hfill $\diamond$
\end{rk}

\begin{lem}[Asymptotic limit of $\frac{\partial L_T(\theta_T, \tau_T(\theta_T)) }{\partial \tau'}$]\label{Lem:DLDTau} Under Assumptions \ref{Assp:ExistenceConsistency} and \ref{Assp:AsymptoticNormality}, for any  sequence $(\theta_T)_{T \in \N}\in \T^\N$ converging to $ \theta_0$, for all $k \in \ldsb 1,m\rdsb$, $\P$-a.s. as $T \rightarrow \infty$,
\begin{enumerate}
\item[(i)] $\frac{\partial M_{1,T}(\theta_T, \tau_T(\theta_T))}{\partial \tau_{k}}=0$;
\item[(ii)] $\frac{\partial M_{2,T}(\theta_T, \tau_T(\theta_T))}{\partial \tau_{k}}=O(T^{-1})$;
\item[(iii)] $\frac{\partial M_{3,T}(\theta_T, \tau_T(\theta_T))}{\partial \tau_{k}}=O(T^{-1})$; and
\item[(iv)] $\frac{\partial L_T(\theta_T, \tau_T(\theta_T)) }{\partial \tau'}=O(T^{-1})$.
\end{enumerate}

\end{lem}
\begin{proof} \textit{(i)} Under Assumption \ref{Assp:ExistenceConsistency}(a)(b) and (d)-(h), by Lemma \ref{Lem:Schennachtheorem10PfFirstSteps}ii (p. \pageref{Lem:Schennachtheorem10PfFirstSteps}), $\P$-a.s. for $T$ big enough,  $\tau_T(\theta_T)$ exists, so that, by equation \eqref{Eq:LogETTermDTau} on p. \pageref{Eq:LogETTermDTau}, $\P$-a.s. for $T$ big enough, for all $k \in \ldsb 1,m\rdsb$,
\begin{eqnarray*}
\frac{\partial M_{1,T}(\theta_T, \tau_T(\theta_T))}{\partial \tau_{k}}
& = &
\left(1 - \frac{m}{2T}  \right) \frac{ \frac{1}{T} \sum_{t=1}^T \e^{ \tau_T(\theta_T)'\psi_t(\theta_T)}  \psi_{t, k} (\theta_T) }
{\frac{1}{T} \sum_{i=1}^T \e^{ \tau_T(\theta_T)'\psi_i(\theta_T)} }\\
& = &0
\end{eqnarray*}
because, by definition of $\tau_T(\theta)$ in equation \eqref{Eq:ESPTiltingEquation} on p. \pageref{Eq:ESPTiltingEquation}, $\frac{1}{T} \sum_{t=1}^T \e^{ \tau_T(\theta_T)'\psi_t(\theta_T)}  \psi_{t, k} (\theta_T)=0$.

\textit{(ii)} Similarly, under Assumption \ref{Assp:ExistenceConsistency}, by equation \eqref{Eq:DerivativeTermDTau} on p. \pageref{Eq:DerivativeTermDTau}, $\P$-a.s. for $T$ big enough, for all $k \in \ldsb 1,m\rdsb$,
 \begin{eqnarray*}
 \frac{\partial M_{2,T}(\theta_T, \tau_T(\theta_T))}{\partial \tau_{k}}  = \frac{1}{T}
{\rm tr}\left\{  \left[ \frac{1}{T} \sum_{t=1}^T                \e^{\tau_T(\theta_T)'\psi_t(\theta_T)}
      \frac{\partial \psi_t(\theta_T)}{\partial \theta' }  \right]^{-1}
      \left[
      \frac{1}{T} \sum_{t=1}^T                \e^{\tau_T(\theta_T)'\psi_t(\theta_T)}
    \psi_{t,k}(\theta_T)  \frac{\partial \psi_t(\theta_T)}{\partial \theta' }
      \right]
\right\}
\end{eqnarray*}
where $\P$-a.s. as $T \rightarrow \infty$, $(\theta_T' \quad\tau_T(\theta_T)') \rightarrow (\theta_0' \quad \tau(\theta_0)')$ by the lemma's assumption and Lemma \ref{Lem:Schennachtheorem10PfFirstSteps}iii (p. \pageref{Lem:Schennachtheorem10PfFirstSteps}).
Now, under Assumptions \ref{Assp:ExistenceConsistency} and \ref{Assp:AsymptoticNormality}, by Lemma \ref{Lem:dLdTaudTau}iv and v (p. \pageref{Lem:dLdTaudTau}), for $\overline{B_L}$ a ball around $(\theta_0, \tau(\theta_0))$ of sufficiently small radius, $\E\left[\sup_{(\theta, \tau)\in \overline{B_L}} \vert \e^{\tau'\psi(X_1,\theta)} \frac{\partial \psi(X_1, \theta)}{\partial \theta'}\vert \right] < \infty$, and,  for all $k\in \ldsb 1,m\rdsb$, $\E\left[\sup_{(\theta, \tau)\in \overline{B_L}} \vert \e^{\tau'\psi(X_1,\theta)} \psi_{k}(X_1,\theta)
      \frac{\partial  \psi(X_1,\theta)}{\partial \theta' }\vert \right]< \infty$.
Thus, by Assumptions \ref{Assp:ExistenceConsistency}(a)(b) and (d),   the ULLN (uniform law of large numbers) \`a la  Wald  \citep[e.g.,][pp. 24-25, Theorem 1.3.3]{2003GhoRam},  implies that, for all $k\in \ldsb 1,m\rdsb$, $\P$-a.s. as $T \rightarrow \infty$,
\begin{eqnarray*}
& & T\frac{\partial M_{2,T}(\theta_T, \tau_T(\theta_T))}{\partial \tau_{k}} \\
& \rightarrow & \tr\left\{   \E \left[\e^{\tau(\theta_0)'\psi(X_{1,}\theta_{0})} \frac{\partial \psi(X_1, \theta_0)}{\partial \theta'}\right]^{-1} \E \left[\e^{\tau(\theta_0)'\psi(X_{1,}\theta_{0})} \psi_{k}(X_1,\theta_{0})\frac{\partial \psi(X_1, \theta_0)}{\partial \theta'}\right]\right\}\\
& = & \tr\left\{   \E \left[\frac{\partial \psi(X_1, \theta_0)}{\partial \theta'}\right]^{-1} \E \left[\psi_{k}(X_1,\theta_{0})\frac{\partial \psi(X_1, \theta_0)}{\partial \theta'}\right]\right\},
\end{eqnarray*}
because $\tau(\theta_0)=0_{m \times 1}$ by Lemma \ref{Lem:AsTiltingFct}iv (p. \pageref{Lem:AsTiltingFct}) under Assumption \ref{Assp:ExistenceConsistency}(a)-(e) and (g)-(h).
    Therefore, $\P$-a.s. as $T \rightarrow \infty$, $\frac{\partial M_{2,T}(\theta_T, \tau_T(\theta_T))}{\partial \tau_{k}}=O(T^{-1})$.

\textit{(iii)} Under Assumption \ref{Assp:ExistenceConsistency},  by  equation \eqref{Eq:VarianceTermDTau}  (p. \pageref{Eq:VarianceTermDTau}), for all $k \in \ldsb 1,m\rdsb$,
\begin{eqnarray*}
& & \frac{\partial M_{3,T}(\theta_T, \tau_T(\theta_T))}{\partial \tau_{k} }\\
& = & -  \frac{1}{2T} {\rm tr} \left\{
\left[  \frac{1}{T} \sum_{t=1}^T                  \e^{\tau_T(\theta_T)'\psi_t(\theta_T)}
    \psi_t(\theta_T)  \psi_t(\theta_T)' \right]^{-1}
     \left[ \frac{1}{T} \sum_{t=1}^T                  \e^{\tau_T(\theta_T)'\psi_t(\theta)}
     \psi_{t, k} (\theta_T) \psi_t(\theta_T)  \psi_t(\theta_T)'     \right]
\right\}% \\
%& = & -  \frac{1}{2T} {\rm tr} \left\{
% \left[  \frac{1}{T} \sum_{t=1}^T                  \psi_t(\hat{\theta}_T)  \psi_t(\hat{\theta}_T)' \right]^{-1}
%      \left[ \frac{1}{T} \sum_{t=1}^T
%      \psi_{t, k} (\hat{\theta}_T) \psi_t(\hat{\theta}_T)  \psi_t(\hat{\theta}_T)'     \right]
% \right\}
\end{eqnarray*}
where $\P$-a.s. as $T \rightarrow \infty$, $(\theta_T' \quad\tau_T(\theta_T)') \rightarrow (\theta_0' \quad \tau(\theta_0)')$ by Theorem \ref{theorem:ConsistencyAsymptoticNormality}i (p. \pageref{theorem:ConsistencyAsymptoticNormality}). Now,
under Assumptions \ref{Assp:ExistenceConsistency} and \ref{Assp:AsymptoticNormality}, by Lemma \ref{Lem:dLdTaudTau}vii and viii (p. \pageref{Lem:dLdTaudTau}), there exists a closed ball $\overline{B_L}\subset \Sbf$ centered at $(\theta_0, \tau(\theta_0))$ with strictly positive radius s.t.,   for all $k \in \ldsb 1,m\rdsb$, 

\noindent
$\E\left[\sup_{(\theta, \tau)\in \overline{B_L}}  \vert\e^{\tau'\psi(X_{1,}\theta)} \psi(X_1, \theta)\psi(X_1, \theta)'\vert \right] < \infty$ and

\noindent
$\E\left[\sup_{(\theta, \tau)\in \overline{B_L}}  \vert\e^{\tau'\psi(X_{1,}\theta)} \psi_{k}(X_1, \theta)\psi(X_1, \theta)\psi(X_1, \theta)'\vert \right] < \infty$.
Thus,
under Assumptions \ref{Assp:ExistenceConsistency} and \ref{Assp:AsymptoticNormality},  by ULLN (uniform law of large numbers) \`a la  Wald  \citep[e.g.,][pp. 24-25, Theorem 1.3.3]{2003GhoRam},  for all $k\in \ldsb 1,m\rdsb$, $\P$-a.s. as $T \rightarrow \infty$,
\begin{eqnarray*}
& & T \frac{\partial M_{3,T}(\theta_T, \tau_T(\theta_T))}{\partial \tau_{k} }\\
 & \rightarrow &  -  \frac{1}{2} {\rm tr} \left\{
\E \left[ \e^{\tau(\theta_0)'\psi(X_{1,}\theta_{0})} \psi(X_1, \theta_{0})\psi(X_1, \theta_{0})'\right]^{-1}
    \E \left[  \e^{\tau(\theta_0)'\psi(X_{1,}\theta_{0})} \psi_{k}(X_1, \theta_{0})\psi(X_1, \theta_{0})\psi(X_1, \theta_{0})'\right]
\right\} \\
& = &  -  \frac{1}{2} {\rm tr} \left\{
\E \left[ \psi(X_1, \theta_{0})\psi(X_1, \theta_{0})'\right]^{-1}
    \E \left[  \psi_{k}(X_1, \theta_{0})\psi(X_1, \theta_{0})\psi(X_1, \theta_{0})'\right]
\right\},
\end{eqnarray*}
because $\tau(\theta_0)=0_{m \times 1}$ by Lemma \ref{Lem:AsTiltingFct}iv (p. \pageref{Lem:AsTiltingFct}) under Assumption \ref{Assp:ExistenceConsistency}(a)-(e) and (g)-(h).
 Therefore, $\P$-a.s. as $T \rightarrow \infty$,  $\frac{\partial M_{3,T}(\theta_T, \tau_T(\theta_T))}{\partial \tau_{k}}=O(T^{-1})$.

\textit{(iv)} Under Assumption \ref{Assp:ExistenceConsistency}(a)-(b) and (d)-(h),  by Lemma \ref{Lem:LogESPDecomposition} (p. \pageref{Lem:LogESPDecomposition}), $L_T(\theta, \tau)=M_{1,T}(\theta, \tau)+M_{2,T}(\theta, \tau)+M_{3,T}(\theta, \tau)$, so that the result follows from the statement (i)-(iii) of the present lemma.
\end{proof}

\begin{rk} In the case in which $\theta_T=\hat{\theta}_T$, there exist at least one other way to prove  Lemma \ref{Lem:DLDTau} that do  not require Assumption \ref{Assp:AsymptoticNormality}. This way follows an approach \`a la  \cite{2004NewSmi}, which  relies on ULLN with $\Tbf_{T}(\theta)=\{\tau \in \R^m: \vert \tau\vert\leqslant T^{-\zeta} \}$ and $\zeta >0$.  We do not follow this ways because (i) Other parts of the proof of Theorem \ref{theorem:ConsistencyAsymptoticNormality}ii (p. \pageref{theorem:ConsistencyAsymptoticNormality}) require the asymptotic normality of $\hat{\theta}_T$ and thus Assumption \ref{Assp:AsymptoticNormality}; (ii) It would lengthen the proofs and complicate their logic; (iii) We later use   Lemma \ref{Lem:DLDTau} with $\theta_T=\check{\theta}_T$, where $\check{\theta}_T$ is a constrained estimator. \hfill $\diamond$
\end{rk}

\begin{lem}[Finiteness of the expectations of supremum of the terms from $\frac{\partial L_T(\theta, \tau)}{\partial \tau}$ and $\frac{\partial ^{2}L_T(\theta, \tau)}{\partial \tau' \partial \tau}$]\label{Lem:dLdTaudTau} Under Assumptions \ref{Assp:ExistenceConsistency} and \ref{Assp:AsymptoticNormality}, there exists a closed ball $\overline{B_L}\subset \Sbf$ centered at $(\theta_0, \tau(\theta_0))$ with strictly positive radius s.t., for all $(h,k)\in \ldsb 1,m\rdsb^2$, \begin{itemize}
\item[(i)] $\E\left[\sup_{(\theta, \tau)\in \overline{B_L}}  \e^{\tau'\psi(X_{1,}\theta)} \right] < \infty  $;
\item[(ii)]  $\E\left[\sup_{(\theta, \tau)\in \overline{B_L}}  \vert\e^{\tau'\psi(X_{1,}\theta)} \psi_{k}(X_1, \theta)\psi_{h}(X_1, \theta)\vert \right] < \infty  $;
\item[(iii)] $\E\left[\sup_{(\theta, \tau)\in \overline{B_L}}  \vert\e^{\tau'\psi(X_{1,}\theta)} \psi_{k}(X_1, \theta)\vert \right] < \infty$;
\item[(iv)] $\E\left[\sup_{(\theta, \tau)\in \overline{B_L}}  \vert\e^{\tau'\psi(X_{1,}\theta)} \frac{\partial\psi(X_1, \theta)}{\partial \theta'}\vert \right] < \infty$;
\item[(v)] $\E\left[\sup_{(\theta, \tau)\in \overline{B_L}}  \vert\e^{\tau'\psi(X_{1,}\theta)} \psi_{k}(X_1, \theta)\frac{\partial\psi(X_1, \theta)}{\partial \theta'}\vert \right]< \infty $;
\item[(vi)] $\E\left[\sup_{(\theta, \tau)\in \overline{B_L}}  \vert\e^{\tau'\psi(X_{1,}\theta)} \psi_{k}(X_1, \theta)\psi_{h}(X_1, \theta)\frac{\partial\psi(X_1, \theta)}{\partial \theta'}\vert \right]< \infty $;
\item[(vii)] $\E\left[\sup_{(\theta, \tau)\in \overline{B_L}}  \vert\e^{\tau'\psi(X_{1,}\theta)} \psi(X_1, \theta)\psi(X_1, \theta)'\vert \right] < \infty  $;
\item[(viii)] $\E\left[\sup_{(\theta, \tau)\in \overline{B_L}}  \vert\e^{\tau'\psi(X_{1,}\theta)} \psi_{k}(X_1, \theta)\psi(X_1, \theta)\psi(X_1, \theta)'\vert \right] < \infty  $;
 \item[(ix)] $\E\left[\sup_{(\theta, \tau)\in \overline{B_L}}  \vert\e^{\tau'\psi(X_{1,}\theta)} \psi_{k}(X_1, \theta)\psi_{h}(X_1, \theta)\psi(X_1, \theta)\psi(X_1, \theta)'\vert \right] < \infty  $; and
\item[(x)] $\E\left[\sup_{(\theta, \tau)\in \overline{B_L}} \vert \e^{\tau'\psi(X_1,\theta)}\psi(X_1, \theta)\tau' \frac{\partial \psi(X_1, \theta)}{\partial \theta'}\vert\right]<\infty$.
\end{itemize}

\end{lem}
\begin{proof} \textit{(i)} Apply Lemma \ref{Lem:Bound2ndPartialLPartialTheta}i (p. \pageref{Lem:Bound2ndPartialLPartialTheta}) under Assumptions \ref{Assp:ExistenceConsistency} and \ref{Assp:AsymptoticNormality}.
Note that
it does not immediately follow from Assumption \ref{Assp:ExistenceConsistency}(e) and the Cauchy-Schwarz inequality because we need additional assumptions to ensures that there exists $\overline{B_L} \subset \Sbf$: See Lemma \ref{Lem:LogESPExistenceInS}(ii) on p. \pageref{Lem:LogESPExistenceInS}.

\textit{(ii)} For all $( h,k)\in \ldsb 1,m\rdsb^2$,  for all $(\theta, \tau)\in \overline{B_L}$, $\e^{\tau'\psi(X_{1,}\theta)}\vert \psi_{k}(X_1, \theta)\psi_{h}(X_1, \theta)\vert=\e^{\tau'\psi(X_{1,}\theta)} \\\sqrt{[\psi_{k}(X_1, \theta)\psi_{h}(X_1, \theta)]^{2}}\leqslant  \e^{\tau'\psi(X_{1,}\theta)}\sqrt{\sum_{(i,j)\in \ldsb 1,m\rdsb^2} [\psi_{i}(X_1, \theta)\psi_j(X_1, \theta)]^2}=  \  \e^{\tau'\psi(X_{1,}\theta)} \vert\psi(X_1, \theta)\psi(X_1, \theta)' \vert $, so that $ \E\left[\sup_{(\theta, \tau)\in \overline{B_L}}  \vert\e^{\tau'\psi(X_{1,}\theta)} \psi_{k}(X_1, \theta)\psi_{h}(X_1, \theta)\vert \right] \leqslant \E\left[\sup_{(\theta, \tau)\in \overline{B_L}}  \vert\e^{\tau'\psi(X_{1,}\theta)} \psi(X_1, \theta)\psi(X_1, \theta)\vert \right]< \infty $, where the last inequality follows from Lemma \ref{Lem:Bound2ndPartialLPartialTheta}xii (p. \pageref{Lem:Bound2ndPartialLPartialTheta}) under Assumptions \ref{Assp:ExistenceConsistency} and \ref{Assp:AsymptoticNormality}.

\textit{(iii)} Apply Lemma \ref{Lem:Bound2ndPartialLPartialTau}v (p. \pageref{Lem:Bound2ndPartialLPartialTau}) under Assumptions \ref{Assp:ExistenceConsistency} and \ref{Assp:AsymptoticNormality}.

\textit{(iv)} Apply Lemma \ref{Lem:Bound2ndPartialLPartialTheta}v (p. \pageref{Lem:Bound2ndPartialLPartialTheta}) under Assumptions \ref{Assp:ExistenceConsistency} and \ref{Assp:AsymptoticNormality}.

\textit{(v)} Apply Lemma  \ref{Lem:Bound2ndPartialLPartialTau}vii (p. \pageref{Lem:Bound2ndPartialLPartialTau}) under Assumptions \ref{Assp:ExistenceConsistency} and \ref{Assp:AsymptoticNormality}.

\textit{(vi)} Proof similar to the one of Lemma   \ref{Lem:Bound2ndPartialLPartialTheta}xiii (p. \pageref{Lem:Bound2ndPartialLPartialTheta}). The supremum of the absolute value of the product is smaller than the product of the
suprema of the absolute values. Thus, under Assumption \ref{Assp:ExistenceConsistency}(a)(b),   for $\overline{B_L}$ of sufficiently small radius, for all $(h,k)\in \ldsb 1,m\rdsb^{2}$,\begin{eqnarray*}
& & \E\left[\sup_{(\theta, \tau)\in \overline{B_L}} \vert \e^{\tau'\psi(X_{1,}\theta)} \psi_{k}(X_1, \theta)\psi_{h}(X_1, \theta)\frac{\partial\psi(X_1, \theta)}{\partial \theta'}\vert \right]  \\
& \leqslant & \E\left[\sup_{(\theta, \tau)\in \overline{B_L}} \vert \e^{\tau'\psi(X_1,\theta)}
\frac{\partial\psi(X_1, \theta)}{\partial \theta'}\vert\sup_{(\theta, \tau)\in \overline{B_L}}  \vert\psi_{k}(X_1, \theta)\psi_{h}(X_1, \theta)\vert \right] \\
& \stackrel{(a)}{\leqslant} & \E\left[\sup_{(\theta, \tau)\in \overline{B_L}} \vert \e^{\tau'\psi(X_1,\theta)}
\frac{\partial\psi(X_1, \theta)}{\partial \theta'}\vert\sup_{(\theta, \tau)\in \overline{B_L}}  \vert\psi(X_1, \theta)\psi(X_1, \theta)'\vert \right] \\
& \stackrel{(b)}{\leqslant} &  \sqrt{ \E\left[\sup_{(\theta, \tau)\in \overline{B_L}} \vert \e^{\tau'\psi(X_1,\theta)}
\frac{\partial\psi(X_1, \theta)}{\partial \theta'}\vert^{2}\right] } \sqrt{\E\left[\sup_{(\theta, \tau)\in \overline{B_L}}  \vert\psi(X_1, \theta)\psi(X_1, \theta)'\vert^2 \right]} \\
& \stackrel{(c)}{\leqslant} &  \sqrt{\E\left[ \sup_{\theta\in \mathcal{N} } \sup_{\tau \in \Tbf(\theta)}\e^{ 2\tau'\psi(X_1, \theta)} b(X_1)^{2}\right]} \sqrt{\E\left[\sup_{\theta \in \T^{\epsilon}}  \vert\psi(X_1, \theta)\psi(X_1, \theta)'\vert^2 \right]}  \stackrel{(d)}{<}\infty.
\end{eqnarray*}
\textit{(a)} As in the proof of statement (ii), for all $(h,k)\in \ldsb 1,m\rdsb^2$,  for all 
 
 \noindent
 $(\theta, \tau)\in \overline{B_L}$, $\vert \psi_{k}(X_1, \theta)\psi_{h}(X_1, \theta)\vert\leqslant \vert \psi(X_1, \theta)\psi(X_1, \theta)' \vert $.
\textit{(b)} Apply the Cauchy-Schwarz inequality, and note that the supremum of the square of a positive function is the square of the supremum of the function. \textit{(c)}  Firstly, under Assumption \ref{Assp:ExistenceConsistency}(a)-(e) and (g)-(h),  by Lemma \ref{Lem:LogESPExistenceInS}ii (p. \pageref{Lem:LogESPExistenceInS}), $\Sbf$ contains an open ball centered at $(\theta_0, \tau(\theta_0))$, so that,    for $\overline{B_L}$ of sufficiently small radius, $\oBL \subset\{(\theta, \tau): \theta \in \mathcal{N}\wedge \tau \in \Tbf(\theta) \}\subset \Sbf \subset \Sbf^\epsilon $. Secondly, as the second supremum does not depend on $\tau$, $\sup_{(\theta, \tau)\in \overline{B_L}}  \vert \psi(X_1, \theta)\psi(X_1, \theta)'\vert^2 \leqslant \sup_{\theta\in \T^\epsilon}  \vert\psi(X_1, \theta)\psi(X_1, \theta)'\vert^2 $ because $\oBL \subset \Sbf $, for $\oBL$ of radius small enough.  \textit{(d)} Firstly,   by Assumption \ref{Assp:AsymptoticNormality}(b), the  first expectation is bounded. Secondly, by Assumption \ref{Assp:ExistenceConsistency}(g), the second expectation is also bounded.

\textit{(vii)} Apply Lemma \ref{Lem:Bound2ndPartialLPartialTheta}xii (p. \pageref{Lem:Bound2ndPartialLPartialTheta}) under Assumptions \ref{Assp:ExistenceConsistency} and \ref{Assp:AsymptoticNormality}.

\textit{(viii)} Proof similar to the one of Lemma   \ref{Lem:Bound2ndPartialLPartialTheta}xiii (p. \pageref{Lem:Bound2ndPartialLPartialTheta}) and to the statement (vi) of the present lemma. The supremum of the absolute value of the product is smaller than the product of the
suprema of the absolute values. Thus, under Assumption \ref{Assp:ExistenceConsistency}(a)(b),   for $\overline{B_L}$ of sufficiently small radius, for all $(h,k)\in \ldsb 1,m\rdsb^{2}$,\begin{eqnarray*}
& & \E\left[\sup_{(\theta, \tau)\in \overline{B_L}} \vert\e^{\tau'\psi(X_{1,}\theta)} \psi_{k}(X_1, \theta)\psi(X_1, \theta)\psi(X_1, \theta)\vert \right]  \\
& \stackrel{}{\leqslant} & \E\left[\sup_{(\theta, \tau)\in \overline{B_L}} \vert \e^{\tau'\psi(X_1,\theta)}
\psi_{k}(X_1, \theta)\vert\sup_{(\theta, \tau)\in \overline{B_L}}  \vert\psi(X_1, \theta)\psi(X_1, \theta)'\vert \right] \\
& \stackrel{(a)}{\leqslant} &  \sqrt{ \E\left[\sup_{(\theta, \tau)\in \overline{B_L}} \vert \e^{\tau'\psi(X_1,\theta)}
\psi_{k}(X_1, \theta)\vert^{2}\right] } \sqrt{\E\left[\sup_{(\theta, \tau)\in \overline{B_L}}  \vert\psi(X_1, \theta)\psi(X_1, \theta)'\vert^2 \right]} \\
& \stackrel{(b)}{\leqslant} &  \sqrt{\E\left[ \sup_{\theta\in \mathcal{N} } \sup_{\tau \in \Tbf(\theta)}\e^{ 2\tau'\psi(X_1, \theta)} b(X_1)^{2}\right]} \sqrt{\E\left[\sup_{\theta \in \T^{\epsilon}}  \vert\psi(X_1, \theta)\psi(X_1, \theta)'\vert^2 \right]}  \stackrel{(c)}{<}\infty.
\end{eqnarray*}
\textit{(a)} Apply the Cauchy-Schwarz inequality, and note that the supremum of the square of a positive function is the square of the supremum of the function. \textit{(b)}  Firstly, under Assumption \ref{Assp:ExistenceConsistency}(a)-(e) and (g)-(h),  by Lemma \ref{Lem:LogESPExistenceInS}ii (p. \pageref{Lem:LogESPExistenceInS}), $\Sbf$ contains an open ball centered at $(\theta_0, \tau(\theta_0))$, so that,    for $\overline{B_L}$ of sufficiently small radius, $\oBL \subset\{(\theta, \tau): \theta \in \mathcal{N}\wedge \tau \in \Tbf(\theta) \}\subset \Sbf \subset \Sbf^\epsilon $. Moreover,   for all $k \in \ldsb 1,m\rdsb $, for all $\theta \in \T$, $\vert \psi_{k}(X_1, \theta)\vert \leqslant\vert \psi(X_1, \theta)\vert\leqslant b(X) $, where the last inequality follows from Assumption \ref{Assp:AsymptoticNormality}(b).  Secondly, as the second supremum does not depend on $\tau$, $\sup_{(\theta, \tau)\in \overline{B_L}}  \vert\psi(X_1, \theta)\psi(X_1, \theta)'\vert^2 \leqslant \sup_{\theta\in \T^\epsilon}  \vert\psi(X_1, \theta)\psi(X_1, \theta)'\vert^2 $ because $\oBL \subset \Sbf $, for $\oBL$ of radius small enough.  \textit{(c)} Firstly,   by Assumption \ref{Assp:AsymptoticNormality}(b), the  first expectation is bounded. Secondly, by Assumption \ref{Assp:ExistenceConsistency}(g), the second expectation is also bounded.

\textit{(ix)} Proof similar to the one of Lemma   \ref{Lem:Bound2ndPartialLPartialTheta}xiii (p. \pageref{Lem:Bound2ndPartialLPartialTheta}) and to the statement (vi) of the present lemma. The supremum of the absolute value of the product is smaller than the product of the
suprema of the absolute values. Thus, under Assumption \ref{Assp:ExistenceConsistency}(a)(b),   for $\overline{B_L}$ of sufficiently small radius, for all $(h,k)\in \ldsb 1,m\rdsb^{2}$,\begin{eqnarray*}
& & \E\left[\sup_{(\theta, \tau)\in \overline{B_L}} \vert\e^{\tau'\psi(X_{1,}\theta)} \psi_{k}(X_1, \theta)\psi(X_1, \theta)\psi(X_1, \theta)\vert \right]  \\
& \stackrel{}{\leqslant} & \E\left[\sup_{(\theta, \tau)\in \overline{B_L}} \vert \e^{\tau'\psi(X_1,\theta)}
\psi_{k}(X_1, \theta)\psi_{h}(X_1, \theta)\vert\sup_{(\theta, \tau)\in \overline{B_L}}  \vert\psi(X_1, \theta)\psi(X_1, \theta)'\vert \right] \\
& \stackrel{(a)}{\leqslant} &  \sqrt{ \E\left[\sup_{(\theta, \tau)\in \overline{B_L}} \vert \e^{\tau'\psi(X_1,\theta)}
\psi_{k}(X_1, \theta)\psi_{h}(X_1, \theta)\vert^{2}\right] } \sqrt{\E\left[\sup_{(\theta, \tau)\in \overline{B_L}}  \vert\psi(X_1, \theta)\psi(X_1, \theta)'\vert^2 \right]} \\
& \stackrel{(b)}{\leqslant} &  \sqrt{\E\left[ \sup_{\theta\in \mathcal{N} } \sup_{\tau \in \Tbf(\theta)}\e^{ 2\tau'\psi(X_1, \theta)} b(X_1)^{4}\right]} \sqrt{\E\left[\sup_{\theta \in \T^{\epsilon}}  \vert\psi(X_1, \theta)\psi(X_1, \theta)'\vert^2 \right]}  \stackrel{(c)}{<}\infty.
\end{eqnarray*}
\textit{(a)} Apply the Cauchy-Schwarz inequality, and note that the supremum of the square of a positive function is the square of the supremum of the function. \textit{(b)}  Firstly, under Assumption \ref{Assp:ExistenceConsistency}(a)-(e) and (g)-(h),  by Lemma \ref{Lem:LogESPExistenceInS}ii (p. \pageref{Lem:LogESPExistenceInS}), $\Sbf$ contains an open ball centered at $(\theta_0, \tau(\theta_0))$, so that,    for $\overline{B_L}$ of sufficiently small radius, $\oBL \subset\{(\theta, \tau): \theta \in \mathcal{N}\wedge \tau \in \Tbf(\theta) \}\subset \Sbf \subset \Sbf^\epsilon $. Moreover,  for all $k \in \ldsb 1,m\rdsb $, for all $\theta \in \T$, $\vert \psi_{k}(X_1, \theta)\vert \leqslant\vert \psi(X_1, \theta)\vert\leqslant b(X) $ where the last inequality follows from Assumption \ref{Assp:AsymptoticNormality}(b).  Secondly, as the second supremum does not depend on $\tau$, $\sup_{(\theta, \tau)\in \overline{B_L}}  \vert\psi(X_1, \theta)\psi(X_1, \theta)'\vert^2 \leqslant \sup_{\theta\in \T^\epsilon}  \vert\psi(X_1, \theta)\psi(X_1, \theta)'\vert^2 $ because $\oBL \subset \Sbf $, for $\oBL$ of radius small enough.  \textit{(c)} Firstly,   by Assumption \ref{Assp:AsymptoticNormality}(b), the  first expectation is bounded. Secondly, by Assumption \ref{Assp:ExistenceConsistency}(g), the second expectation is also bounded.

\textit{(x)} The norm of a product of matrices is smaller than the product of the norms \citep[e.g.,][Theorem 9.7 and note that all norms are equivalent on finite dimensional spaces]{1953Rudin}. Thus, for $\oBL$ of sufficiently small radius,  for all $(\ell, j)\in \ldsb 1,m\rdsb^2$,
\begin{eqnarray*}
&&\ \E\left[\sup_{(\theta, \tau)\in \overline{B_L}} \vert \e^{\tau'\psi(X_1,\theta)}\psi(X_1, \theta)\tau' \frac{\partial \psi(X_1, \theta)}{\partial \theta'}\vert\right]\\
& \leqslant & (\sup_{(\theta,\tau) \in \overline{B_L} }\vert \tau \vert)  \E\left[\sup_{(\theta, \tau)\in \overline{B_L}}  \e^{\tau'\psi(X_1,\theta)}\vert \psi(X_1, \theta)\vert \vert \frac{\partial \psi(X_1, \theta)}{\partial \theta'}\vert\right]\\
& \stackrel{(a)}{\leqslant} & (\sup_{(\theta,\tau) \in \overline{B_L} }\vert \tau \vert) \E\left[ \sup_{\theta\in \mathcal{N} } \sup_{\tau \in \Tbf(\theta)}\e^{ \tau'\psi(X_1, \theta)} b(X_1)^{2}\right] \stackrel{(b)}{<} \infty.
\end{eqnarray*}
 \textit{(a)}   Firstly, under Assumption \ref{Assp:ExistenceConsistency}(a)-(e) and (g)-(h), by Lemma \ref{Lem:LogESPExistenceInS}ii (p. \pageref{Lem:LogESPExistenceInS}), $\Sbf$ contains an open ball centered at $(\theta_0, \tau(\theta_0))$ Thus, under Assumption \ref{Assp:ExistenceConsistency}(a)-(e) and (g)-(h),  for $\oBL$ of sufficiently small radius, by definition of $\Sbf$, $\oBL \subset\{(\theta, \tau): \theta \in \mathcal{N}\wedge \tau \in \Tbf(\theta) \}\subset \Sbf$, because $ \mathcal{N}\subset \T$  by Assumption \ref{Assp:AsymptoticNormality}(a). Secondly, by Assumption \ref{Assp:AsymptoticNormality}(b), $\sup_{\theta\in \mathcal{N} }\vert \psi(X_1, \theta)\vert \leqslant b(X)$ and  $\sup_{\theta\in \mathcal{N} }\vert \frac{\partial \psi(X_1, \theta)}{\partial \theta'}\vert\leqslant b(X)$.   \textit{(b)}   Firstly,  $\sup_{(\theta,\tau) \in \overline{B_L} }\vert \tau \vert^{2}< \infty$ because    $ \overline{B_L}$ is bounded. Secondly, by Assumption \ref{Assp:AsymptoticNormality}(b), $\E\left[ \sup_{\theta\in \mathcal{N} } \sup_{\tau \in \Tbf(\theta)}\e^{ \tau'\psi(X_1, \theta)} b(X_1)^{2}\right]< \infty$.
\end{proof}

\subsection{Proof of Theorem \ref{theorem:TrinityPlus1}: Trinity$+1$} \label{Ap:PfTrinityPlus1}

The proof of Theorem \ref{theorem:TrinityPlus1} adapts the traditional way of deriving the trinity along the lines of \cite{2011Smith}. As in the proof of Theorem \ref{theorem:ConsistencyAsymptoticNormality}, the main difference comes from the complexity of the variance term $\vert \Sigma_T(\theta)\vert_{\det}$. 
\begin{proof}[Core of the proof of Theorem \ref{theorem:TrinityPlus1}]
\textit{Asymptotic distribution of $\mathrm{Wald}_T$.} By Assumption \ref{Assp:Trinity}(a), $r: \T \rightarrow \R^q$ is continuously differentiable. Thus, under Assumptions \ref{Assp:ExistenceConsistency} and \ref{Assp:AsymptoticNormality}, if  the test hypothesis \eqref{Eq:HypParameterRestriction} on p. \pageref{Eq:HypParameterRestriction} holds, a first-order Taylor-Lagrange expansion at $\theta_0$ evaluated at $\hat{\theta}_T$, $\omega$ by $\omega$, yields, $\P$-a.s. as $T \rightarrow \infty$,
\begin{eqnarray*}
r(\hat{\theta}_T) & = & r(\theta_0)+R(\bar{\theta}_T)(\hat{\theta}_T- \theta_0) \text{ ,  where $\bar{\theta}_T$ is between $\hat{\theta}_T$ and $\theta_0$; }\\
& \stackrel{(a)}{=} & R(\bar{\theta}_T)(\hat{\theta}_T- \theta_0)\\
& \underset{(b)}{\stackrel{D}{\rightarrow }} & R(\theta_0)\mathcal{N}(0, \Sigma(\theta_0))\\
&\stackrel{D}{=}& \mathcal{N} \left(0, R(\theta_0)\Sigma(\theta_0)R(\theta_0)' \right).
\end{eqnarray*}
\textit{(a)} By definition, if  the test hypothesis \eqref{Eq:HypParameterRestriction} on p. \pageref{Eq:HypParameterRestriction} holds, $r(\theta_0)=0_{q \times 1}$. \textit{(b)} Under Assumptions \ref{Assp:ExistenceConsistency} and \ref{Assp:AsymptoticNormality}, by Theorem \ref{theorem:ConsistencyAsymptoticNormality}ii (p. \pageref{theorem:ConsistencyAsymptoticNormality}), $\P$-a.s. as $T \rightarrow \infty$, $\sqrt{T}(\hat{\theta}_T- \theta_0) \stackrel{D}{\rightarrow}\mathcal{N}(0, \Sigma(\theta_0))$, which also implies that $\bar{\theta}_T \rightarrow \theta_0$.
Thus, under Assumptions \ref{Assp:ExistenceConsistency},  \ref{Assp:AsymptoticNormality} and \ref{Assp:Trinity}(a),  by continuity of $R(.)$, $\P$-a.s. as $T \rightarrow \infty$, $ R(\bar{\theta})\rightarrow R(\theta_0)$, so that the result follows by the Slutsky's theorem.

 Now, under Assumptions \ref{Assp:ExistenceConsistency}, \ref{Assp:AsymptoticNormality} and \ref{Assp:Trinity}(a), by Lemma \ref{Lem:ConstrainedEstLagrangian}i (p. \pageref{Lem:ConstrainedEstLagrangian}), $\P$-a.s. as $T \rightarrow \infty$, $\check{\theta}_T\rightarrow \theta_0$, so that $R(\check{\theta}_T)\rightarrow R(\theta_0)$ by Assumption \ref{Assp:Trinity}(a). Moreover, by the theorem's assumption, as $T \rightarrow \infty $, $\widehat{\Sigma(\theta_0)}\stackrel{\P}{\rightarrow} \Sigma(\theta_0)$. In addition, by Assumption \ref{Assp:ExistenceConsistency}(h) and \ref{Assp:Trinity}(b), $\Sigma(\theta_0)$ and $R(\theta_0)$ are full rank, so that $\widehat{\Sigma(\theta_0)} $ and $R(\check{\theta}_T)$ are full rank w.p.a.1 as $T \rightarrow \infty$ (Lemma \ref{Lem:UniFiniteSampleInvertibilityFromUniAsInvertibility} p. \pageref{Lem:UniFiniteSampleInvertibilityFromUniAsInvertibility}). Then, the result follows from the Cochran's theorem.

\textit{Asymptotic distribution of $\mathrm{LM}_T$.} Under Assumptions \ref{Assp:ExistenceConsistency}, \ref{Assp:AsymptoticNormality} and \ref{Assp:Trinity}, by Proposition \ref{Prop:ConstrainedEstAsNormality}iii (p. \pageref{Prop:ConstrainedEstAsNormality}), if the test hypothesis \eqref{Eq:HypParameterRestriction} on p. \pageref{Eq:HypParameterRestriction} holds, as $T \rightarrow \infty$, $\check{\gamma}_T\stackrel{D}{\rightarrow }\mathcal{N}(0,(R(\theta_0) \Sigma(\theta_0) R(\theta_0)')^{-1})$. Now, under Assumptions \ref{Assp:ExistenceConsistency}, \ref{Assp:AsymptoticNormality} and \ref{Assp:Trinity}(a), by Lemma \ref{Lem:ConstrainedEstLagrangian}i (p. \pageref{Lem:ConstrainedEstLagrangian}), $\P$-a.s. as $T \rightarrow \infty$, $\check{\theta}_T \rightarrow \theta_0$, so that $R(\check{\theta}_T)\rightarrow R(\theta_0)$ by Assumption \ref{Assp:Trinity}(a).
 Moreover, by the theorem's assumption, as $T \rightarrow \infty $, $\widehat{\Sigma(\theta_0)}\stackrel{\P}{\rightarrow} \Sigma(\theta_0)$. In addition, by Assumption \ref{Assp:ExistenceConsistency}(h) and \ref{Assp:Trinity}(b), $\Sigma(\theta_0)$ and $R(\theta_0)$ are full rank, so that $\widehat{\Sigma(\theta_0)} $ and $R(\check{\theta}_T)$ are full rank w.p.a.1 as $T \rightarrow \infty$ (Lemma \ref{Lem:UniFiniteSampleInvertibilityFromUniAsInvertibility} p. \pageref{Lem:UniFiniteSampleInvertibilityFromUniAsInvertibility}). Then, by the Cochran's theorem, as $T \rightarrow \infty$, $T \check{\gamma}_T'[R(\check{\theta}_T)\widehat{\Sigma(\theta_0)}R(\check{\theta}_T)']\check{\gamma}_T \stackrel{D}{\rightarrow} \chi_q^2$.
Finally, under Assumptions \ref{Assp:ExistenceConsistency}, \ref{Assp:AsymptoticNormality} and \ref{Assp:Trinity}, by
 Lemma \ref{Lem:ConstrainedEstLagrangian}iii (p. \pageref{Lem:ConstrainedEstLagrangian}), $R(\check{\theta}_T)'\check{\gamma}_T=-\left. \frac{\partial L_T(\theta, \tau_T(\theta))}{\partial \theta} \right \vert_{\theta= \check{\theta}_T} $, so $T \check{\gamma}_T'[R(\check{\theta}_T)\widehat{\Sigma(\theta_0)}R(\check{\theta}_T)']\check{\gamma}_T =T [R(\check{\theta}_T)'\check{\gamma}_T]'\widehat{\Sigma(\theta_0)}[R(\check{\theta}_T)'\check{\gamma}_T]=\negthickspace T\left. \frac{\partial L_T(\theta, \tau_T(\theta))}{\partial \theta} \right \vert_{\theta= \check{\theta}_T}'\negthickspace\widehat{\Sigma(\theta_0)}\left. \frac{\partial L_T(\theta, \tau_T(\theta))}{\partial \theta} \right \vert_{\theta= \check{\theta}_T}\negthickspace\negthickspace\negthickspace\negthickspace= \frac{\partial\ln[\hat{f}_{\theta^{*}_T}(\check{\theta}_T)]}{\partial \theta' }\widehat{\Sigma(\theta_0)}\frac{\partial\ln[\hat{f}_{\theta^{*}_T}(\check{\theta}_T)]}{\partial \theta }$, where the last equality follows from the definition of the LogESP in Lemma \ref{Lem:LogESPDecomposition} (p. \pageref{Lem:LogESPDecomposition}), i.e.,   

\noindent
$L_T(\theta, \tau_T(\theta)):=\ln  \left[\frac{1}{T}\sum_{t=1}^T \e^{\tau_T(\theta)'\psi_t(\theta)}\right]- \frac{1}{2T} \ln \vert \Sigma_T(\theta)\vert_{\det}=\frac{1}{T}[ \ln(\hat{f}_{\theta^{*}_T}(\theta))-\frac{m}{2}\ln(\frac{T}{2\pi})].$ 

\textit{Asymptotic distribution of $\mathrm{ALR}_T$.}  Under Assumptions \ref{Assp:ExistenceConsistency},\ref{Assp:AsymptoticNormality} and \ref{Assp:Trinity},   if the test hypothesis \eqref{Eq:HypParameterRestriction} on p. \pageref{Eq:HypParameterRestriction} holds, by Lemma \ref{Lem:AsStatLR_T} (p. \pageref{Lem:AsStatLR_T}), $\P$-a.s. as $T \rightarrow \infty$,
 \begin{eqnarray*}
& & 2\{\ln[\hat{f}_{\theta^{*}_T}(\hat{\theta}_T)]-\ln[\hat{f}_{\theta^{*}_T}(\check{\theta}_T)]\}\\
&=&-\left[\sqrt{T}(\hat{\theta}_T-\check{\theta}_T)\right]'\Sigma^{-1}\left[\sqrt{T}(\hat{\theta}_T-\check{\theta}_T)\right]+o_{\P}(1)\\
& \stackrel{(a)}{=} & -\left[\Sigma R'(R\Sigma R')^{-1}RM^{-1}\frac{1}{\sqrt{T}}\sum_{t=1}^T \psi_t(\theta_0)\negthickspace+\negthickspace o_{\P}(1)\right]'\negthickspace \negthickspace\Sigma^{-1}\negthickspace\left[\Sigma R'(R\Sigma R')^{-1}RM^{-1}\frac{1}{\sqrt{T}}\sum_{t=1}^T \psi_t(\theta_0)\negthickspace+\negthickspace o_{\P}(1)\right]\negthickspace+o_{\P}(1) \\
& \stackrel{(b)}{=} & -\left[\frac{1}{\sqrt{T}}\sum_{t=1}^T \psi_t(\theta_0)\right]' (M')^{-1}R'(R\Sigma R')^{-1}R\Sigma\Sigma^{-1}\Sigma R'(R\Sigma R')^{-1}RM^{-1}\left[\frac{1}{\sqrt{T}}\sum_{t=1}^T \psi_t(\theta_0)\right]+o_{\P}(1)\\
& \stackrel{(c)}{=} & -\left[\frac{1}{\sqrt{T}}\sum_{t=1}^T \psi_t(\theta_0)\right]' (M')^{-1}R'(R\Sigma R')^{-1}RM^{-1}\left[\frac{1}{\sqrt{T}}\sum_{t=1}^T \psi_t(\theta_0)\right]\negthickspace+o_{\P}(1)\\
&  \stackrel{(d)}{=}& -\left[V^{-1/2}\frac{1}{\sqrt{T}}\sum_{t=1}^T \psi_t(\theta_0)\right]' V^{1/2'}(M')^{-1}R'(R\Sigma R')^{-1}RM^{-1}V^{1/2}\left[V^{-1/2}\frac{1}{\sqrt{T}}\sum_{t=1}^T \psi_t(\theta_0)\right]\negthickspace+o_{\P}(1)\\
&  \stackrel{(e)}{=} & -\left[V^{-1/2}\frac{1}{\sqrt{T}}\sum_{t=1}^T \psi_t(\theta_0)\right]'P_{\Sigma^{1/2}R'} \left[V^{-1/2}\frac{1}{\sqrt{T}}\sum_{t=1}^T \psi_t(\theta_0)\right]\negthickspace+o_{\P}(1)\\
& \underset{(f)}{ \stackrel{D}{\rightarrow}} &\chi^2_q
 \end{eqnarray*}
 \textit{(a)} Under Assumptions \ref{Assp:ExistenceConsistency}, \ref{Assp:AsymptoticNormality} and \ref{Assp:Trinity},  if the test hypothesis \eqref{Eq:HypParameterRestriction} on p. \pageref{Eq:HypParameterRestriction} holds, by Lemma \ref{Lem:AsStatLR_T}ii (p. \pageref{Lem:AsStatLR_T}), $\P$-a.s. as $T \rightarrow \infty$,$
\sqrt{T}(\hat{\theta}_T-\check{\theta}_T )=\Sigma R'(R\Sigma R')^{-1}RM^{-1}\frac{1}{\sqrt{T}}\sum_{t=1}^T \psi_t(\theta_0)+o_{\P}(1)$. \textit{(b)} Transpose the content of the first square bracket, and then note that $\Sigma=\Sigma'$ by symmetry. \textit{(c)} Note that $R\Sigma\Sigma^{-1}\Sigma R'(R\Sigma R')^{-1}=I$. \textit{(d)} Use that $V^{1/2} V^{-1/2}=I$.
\textit{(e)} Note that $ P_{\Sigma^{1/2}R'}=V^{1/2'}(M')^{-1}R(R'\Sigma R)^{-1}R'M^{-1}V^{1/2}$. \textit{(f)} Under Assumption \ref{Assp:ExistenceConsistency}(a)-(c) and (g), by the Lindeberg-L{\'e}vy CLT theorem, as $T \rightarrow \infty$, $ \frac{1}{\sqrt{T}}\sum_{t=1}^T\psi_t(\theta_0)\stackrel{D}{\rightarrow} \mathcal{N}(0,V)$ where $V:=\E[\psi(X_1,\theta_{0}) \psi(X_1,\theta_{0})' ]$. Moreover, the orthogonal projection matrix $P_{\Sigma^{1/2}R'} $ has rank $q$  because $R$ is of rank $q$ and $\Sigma$ has full rank by Assumptions \ref{Assp:Trinity}(b) and \ref{Assp:ExistenceConsistency}(h), respectively.
Thus, the result follows from the Cochran's theorem. \textit{}

\textit{Asymptotic distribution of $\mathrm{ET}_T$.} Under Assumptions \ref{Assp:ExistenceConsistency}, \ref{Assp:AsymptoticNormality} and \ref{Assp:Trinity},  if the test hypothesis \eqref{Eq:HypParameterRestriction} on p. \pageref{Eq:HypParameterRestriction} holds, by Proposition \ref{Prop:ConstrainedEstAsNormality}ii (p. \pageref{Prop:ConstrainedEstAsNormality}),   $\P$-a.s. as $T \rightarrow \infty$,
\begin{eqnarray*}
\sqrt{T}\tau_T(\check{\theta}_T) &=& (M')^{-1}\Sigma^{-1/2}P_{\Sigma^{1/2}R'}\Sigma^{-1/2'}M^{-1}\frac{1}{\sqrt{T}}\sum_{t=1}^T\psi_t(\theta_0)+o_{\P}(1)\\
& \stackrel{(a)}{=} &  (M')^{-1}[V^{1/2}(M')^{-1}]^{-1}P_{\Sigma^{1/2}R'}[M^{-1}V^{1/2'}]^{-1}M^{-1}\frac{1}{\sqrt{T}}\sum_{t=1}^T\psi_t(\theta_0)+o_{\P}(1)\\
& \stackrel{(b)}{=} & (M')^{-1}M'V^{-1/2}P_{\Sigma^{1/2}R'}V^{-1/2'}MM^{-1}\frac{1}{\sqrt{T}}\sum_{t=1}^T\psi_t(\theta_0)+o_{\P}(1)\\
& = & V^{-1/2}P_{\Sigma^{1/2}R'}V^{-1/2'}\frac{1}{\sqrt{T}}\sum_{t=1}^T\psi_t(\theta_0)+o_{\P}(1)
\end{eqnarray*}
\textit{(a)} By definition, $M^{-1}V(M')^{-1}=:\Sigma=\Sigma^{1/2'}\Sigma^{1/2} $, so that $\Sigma^{-1/2}:=(\Sigma^{1/2})^{-1}=[V^{1/2}(M')^{-1}]^{-1}$ and $\Sigma^{-1/2'}:=(\Sigma^{1/2'})^{-1}=[M^{-1}V^{1/2'}]^{-1}$. \textit{(b)} By standard property of inverses, $[V^{1/2}(M')^{-1}]^{-1} $ and $[M^{-1}V^{1/2'}]^{-1}=V^{-1/2'}M$.

Thus, under Assumptions \ref{Assp:ExistenceConsistency}, \ref{Assp:AsymptoticNormality} and \ref{Assp:Trinity},  if the test hypothesis \eqref{Eq:HypParameterRestriction} on p. \pageref{Eq:HypParameterRestriction} holds, by Proposition \ref{Prop:ConstrainedEstAsNormality},   $\P$-a.s. as $T \rightarrow \infty$,
\begin{eqnarray*}
& &  T \tau_T(\check{\theta}_T)'\widehat{V}\tau_T(\check{\theta}_T)\\
& = & \left[ V^{-1/2}P_{\Sigma^{1/2}R'}V^{-1/2'}\frac{1}{\sqrt{T}}\sum_{t=1}^T\psi_t(\theta_0)+o_{\P}(1)\right]'\widehat{V}_T\left[V^{-1/2}P_{\Sigma^{1/2}R'}V^{-1/2'}\frac{1}{\sqrt{T}}\sum_{t=1}^T\psi_t(\theta_0)+o_{\P}(1) \right]\\
& \stackrel{(a)}{=} & \left[ V^{-1/2}P_{\Sigma^{1/2}R'}V^{-1/2'}\frac{1}{\sqrt{T}}\sum_{t=1}^T\psi_t(\theta_0)\right]'\widehat{V}_T\left[V^{-1/2}P_{\Sigma^{1/2}R'}V^{-1/2'}\frac{1}{\sqrt{T}}\sum_{t=1}^T\psi_t(\theta_0) \right]\\
& & 2  \left[ V^{-1/2}P_{\Sigma^{1/2}R'}V^{-1/2'}\frac{1}{\sqrt{T}}\sum_{t=1}^T\psi_t(\theta_0)\right]'\widehat{V}_T\left[o_{\P}(1) \right ]+ \left[ o_{\P}(1)\right]'\widehat{V}\left[o_{\P}(1) \right ]\\
& \stackrel{(b)}{=} &  \left[ P_{\Sigma^{1/2}R'}V^{-1/2'}\frac{1}{\sqrt{T}}\sum_{t=1}^T\psi_t(\theta_0)\right]'V^{-1/2'}\widehat{V}_TV^{-1/2}\left[P_{\Sigma^{1/2}R'}V^{-1/2'}\frac{1}{\sqrt{T}}\sum_{t=1}^T\psi_t(\theta_0) \right]+o_{\P}(1)\\
& \stackrel{(c)}{=} & \left[ P_{\Sigma^{1/2}R'}V^{-1/2'}\frac{1}{\sqrt{T}}\sum_{t=1}^T\psi_t(\theta_0)\right]'\left[P_{\Sigma^{1/2}R'}V^{-1/2'}\frac{1}{\sqrt{T}}\sum_{t=1}^T\psi_t(\theta_0) \right]\\
& & \left[ P_{\Sigma^{1/2}R'}V^{-1/2'}\frac{1}{\sqrt{T}}\sum_{t=1}^T\psi_t(\theta_0)\right]'\left(V^{-1/2'}\widehat{V}_TV^{-1/2}-I\right)\left[P_{\Sigma^{1/2}R'}V^{-1/2'}\frac{1}{\sqrt{T}}\sum_{t=1}^T\psi_t(\theta_0) \right]+o_{\P}(1)\\
& \stackrel{(d)}{=} & \left[ P_{\Sigma^{1/2}R'}V^{-1/2'}\frac{1}{\sqrt{T}}\sum_{t=1}^T\psi_t(\theta_0)\right]'\left[P_{\Sigma^{1/2}R'}V^{-1/2'}\frac{1}{\sqrt{T}}\sum_{t=1}^T\psi_t(\theta_0) \right]+o_{\P}(1)\\
& \underset{(d)}{\stackrel{D}{\rightarrow}} & \chi^2_q
\end{eqnarray*}
\textit{(a)} Use the bilinearity and symmetry of the quadratic form defined by the matrix $\widehat{V}_T$, which is symmetric by the theorem's assumption. \textit{(b)} Firstly, by theorem's assumption, as $T \rightarrow \infty$,  $\widehat{V}_T \stackrel{\P}{\rightarrow} V$, so that $\widehat{V}_T=O_{\P}(1)$. Secondly, under Assumption \ref{Assp:ExistenceConsistency}(a)-(c) and (g), by the Lindeberg-L{\'e}vy CLT theorem, as $T \rightarrow \infty$, $ \frac{1}{\sqrt{T}}\sum_{t=1}^T\psi_t(\theta_0)\stackrel{D}{\rightarrow} \mathcal{N}(0,V)$ where $V:=\E[\psi(X_1,\theta_{0}) \psi(X_1,\theta_{0})' ]$, so that, as $T \rightarrow \infty$,
 $\frac{1}{\sqrt{T}}\sum_{t=1}^T\psi_t(\theta_0)=O_{\P}(1)$. Thus, the second and third terms are $o_{\P}(1)$.   \textit{(c)} Add and subtract $\left[ P_{\Sigma^{1/2}R'}V^{-1/2'}\frac{1}{\sqrt{T}}\sum_{t=1}^T\psi_t(\theta_0)\right]'\left[P_{\Sigma^{1/2}R'}V^{-1/2'}\frac{1}{\sqrt{T}}\sum_{t=1}^T\psi_t(\theta_0) \right]$. \textit{(d)} Denoting the convergence in probability with $\stackrel{\P}{\rightarrow}$,  by the present theorem assumption, as $T \rightarrow \infty$,  $\widehat{V}_T \stackrel{\P}{\rightarrow} V$,  where $V$ is a positive definite symmetric matrix by Assumption \ref{Assp:ExistenceConsistency}(h). Thus, by Lemma \ref{Lem:PdmAsMatrix} (p. \pageref{Lem:PdmAsMatrix}),  w.p.a.1 as $T \rightarrow \infty$, $\hat{V}_T$ is p-d.m, so that it has a square root s.t.  $\hat{V}_T=\hat{V}_T^{1/2'}\hat{V}_T^{1/2}$, where $\hat{V}_T^{1/2}\stackrel{\P}{\rightarrow}V^{1/2} $.\textit{(d)} Under Assumption \ref{Assp:ExistenceConsistency}(a)-(c) and (g), by the Lindeberg-L{\'e}vy CLT theorem, as $T \rightarrow \infty$, $ \frac{1}{\sqrt{T}}\sum_{t=1}^T\psi_t(\theta_0)\stackrel{D}{\rightarrow} \mathcal{N}(0,V)$ where $V:=\E[\psi(X_1,\theta_{0}) \psi(X_1,\theta_{0})' ]$. Moreover, the orthogonal projection matrix $P_{\Sigma^{1/2}R'} $ has rank $q$  because $R$ is of rank $q$ and $\Sigma$ has full rank by Assumptions \ref{Assp:Trinity}(b) and \ref{Assp:ExistenceConsistency}(h), respectively.
Thus, the result follows from the Cochran's theorem.
\end{proof}
 \begin{lem}[Asymptotic expansions for $\mathrm{ALR}_T$]\label{Lem:AsStatLR_T} Under Assumptions \ref{Assp:ExistenceConsistency},\ref{Assp:AsymptoticNormality} and \ref{Assp:Trinity},  if the test hypothesis \eqref{Eq:HypParameterRestriction} on p. \pageref{Eq:HypParameterRestriction} holds, $\P$-a.s. as $T \rightarrow \infty$,
\begin{itemize}
\item[(i)] $\displaystyle  2\{\ln[\hat{f}_{\theta^{*}_T}(\hat{\theta}_T)]-\ln[\hat{f}_{\theta^{*}_T}(\check{\theta}_T)]\}=T(\hat{\theta}_T-\check{\theta}_T)'\Sigma^{-1}(\hat{\theta}_T-\check{\theta}_T)+o_{\P}(1)
 $;
\item[(ii)] $ \sqrt{T}(\hat{\theta}_T-\check{\theta}_T )=\Sigma R'(R\Sigma R')^{-1}RM^{-1}\frac{1}{\sqrt{T}}\sum_{t=1}^T \psi_t(\theta_0)+o_{\P}(1)$
\end{itemize}
where $\Sigma:=\Sigma(\theta_0):= M^{-1}V(M')^{-1}$, $M:=\E \left[\frac{\partial \psi(X_1, \theta_0)}{\partial \theta'}\right]$, $V:= \E[ \psi(X_1, \theta_0)\psi(X_1, \theta_0)']$, and $R:=\frac{\partial r(\theta_0)}{\partial \theta'}$.
\end{lem}
 \begin{proof}
\textit{ (i)} Under Assumption \ref{Assp:ExistenceConsistency}, by Lemma \ref{Lem:LogESPDecomposition} (p. \pageref{Lem:LogESPDecomposition}), $\P$-a.s. for $T$ big enough, for all $(\theta, \tau)$ in a  neighborhood of $(\theta_0, \tau(\theta_0))$, $L_T(\theta, \tau)$ exists. Moreover, under Assumptions \ref{Assp:ExistenceConsistency}, \ref{Assp:AsymptoticNormality} and \ref{Assp:Trinity}(a),  if the test hypothesis \eqref{Eq:HypParameterRestriction} on p. \pageref{Eq:HypParameterRestriction} holds,  by Theorem \ref{theorem:ConsistencyAsymptoticNormality}i (p. \pageref{theorem:ConsistencyAsymptoticNormality}), Lemma \ref{Lem:ConstrainedEstLagrangian}i (p. \pageref{Lem:ConstrainedEstLagrangian}) and  Lemma \ref{Lem:Schennachtheorem10PfFirstSteps}iii (p. \pageref{Lem:Schennachtheorem10PfFirstSteps}), $\hat{\theta}_T \rightarrow \theta_0$, $\check{\theta}_T \rightarrow \theta_0$, $\tau_T(\hat{\theta}_{T})\rightarrow \tau(\theta_0)$ and $\tau_T(\check{\theta}_T)\rightarrow \tau(\theta_0)$, $\P$-a.s. as $T \rightarrow \infty$. Thus, noting that $\ln[\hat{f}_{\theta^{*}_T}(\theta)]=L_T(\theta,\tau_T(\theta)) $,under Assumptions \ref{Assp:ExistenceConsistency}, \ref{Assp:AsymptoticNormality} and \ref{Assp:Trinity}(a),  if the test hypothesis \eqref{Eq:HypParameterRestriction} on p. \pageref{Eq:HypParameterRestriction} holds,  $\P$-a.s. for $T$  big enough,
\begin{eqnarray*}
2\{\ln[\hat{f}_{\theta^{*}_T}(\hat{\theta}_T)]-\ln[\hat{f}_{\theta^{*}_T}(\check{\theta}_T)]\}=-2T[L_T(\check{\theta}_T,\tau_T(\check{\theta}_T))-L_T(\hat{\theta}_T, \tau_T(\hat{\theta}_{T}))].
\end{eqnarray*}Now, under Assumptions \ref{Assp:ExistenceConsistency} and \ref{Assp:AsymptoticNormality}(a), by subsection \ref{Sec:LTAndDerivatives} (p. \pageref{Sec:LTAndDerivatives}), $L_T(.,.)$ is twice continuously differentiable in a neighborhood of $(\theta_0'\; \tau(\theta_0)')$   $\P$-a.s. for $T$ big enough, so that   a stochastic second-order Taylor-Lagrange expansion \citep[e.g.,][Theorem 18.18]{1999AliprantisBorder} around  $(\hat{\theta}_T, \tau_T(\hat{\theta}_T))$ and evaluated $(\check{\theta}_T, \tau_T(\check{\theta}_T)) $ yields, $\P$-a.s. for $T$ big enough,
\begin{eqnarray*}
 & & L_T(\check{\theta}_T,\tau_T(\check{\theta}_T)) =  L_T(\hat{\theta}_T, \tau_T(\hat{\theta}_{T}))+\begin{bmatrix}\frac{\partial L_T(\hat{\theta}_T, \tau_T(\hat{\theta}_{T})) }{\partial \theta'} & \frac{\partial L_T(\hat{\theta}_T, \tau_T(\hat{\theta}_{T})) }{\partial \tau'} \\
\end{bmatrix}\begin{bmatrix}\check{\theta}_T-\hat{\theta}_T  \\
\tau_T(\check{\theta}_T)-\tau_T(\hat{\theta}_{T}) \\
\end{bmatrix} \\
& &\ + \frac{1}{2}\begin{bmatrix}(\check{\theta}_T- \hat{\theta}_T)' & (\tau_T(\check{\theta}_T) - \tau_T(\hat{\theta}_{T}) )' \\
\end{bmatrix}
\left[ \begin{array}{c c } \frac{\partial^2 { L}_T(\bar{\theta}_T, \bar{\tau}_T) }{\partial \theta' \partial \theta}
                         & \frac{\partial^2 { L}_T(\bar{\theta}_T, \bar{\tau}_T) }{\partial \tau' \partial \theta} \\
                         \frac{\partial^2 { L}_T(\bar{\theta}_T, \bar{\tau}_T) }{\partial \tau' \partial \theta}   & \frac{\partial^2 { L}_T(\bar{\theta}_T, \bar{\tau}_T) }{\partial \tau' \partial \tau} \end{array} \right]\begin{bmatrix}\check{\theta}_T - \hat{\theta}_T \\
\tau_T(\check{\theta}_T) - \tau_T(\hat{\theta}_{T})  \\
\end{bmatrix}\\
& & \text{where $(\bar{\theta}_T, \bar{\tau}_T)$ is between $(\hat{\theta}_T, \tau_T(\hat{\theta}_{T})) $ and $(\check{\theta}_T, \tau(\check{\theta}_T))$ ;}\\
& \Rightarrow & L_T(\check{\theta}_T,\tau_T(\check{\theta}_T)) -  L_T(\hat{\theta}_T, \tau_T(\hat{\theta}_{T}))\\
& & =\frac{\partial L_T(\hat{\theta}_T, \tau_T(\hat{\theta}_{T})) }{\partial \theta'}(\check{\theta}_T-\hat{\theta}_T)+\frac{\partial L_T(\hat{\theta}_T, \tau_T(\hat{\theta}_{T})) }{\partial \tau'}(\tau_T(\check{\theta}_T)-\tau_T(\hat{\theta}_{T})) \\
& & + \frac{1}{2}(\check{\theta}_T - \hat{\theta}_T)'\frac{\partial^2 { L}_T(\bar{\theta}_T, \bar{\tau}_T) }{\partial \theta' \partial \theta}(\check{\theta}_T - \hat{\theta}_T)+\frac{1}{2}(\tau_T(\check{\theta}_T) - \tau_T(\hat{\theta}_{T}) )'\frac{\partial^2 { L}_T(\bar{\theta}_T, \bar{\tau}_T) }{\partial \tau' \partial \tau} (\tau_T(\check{\theta}_T) - \tau_T(\hat{\theta}_{T}) )
\\
& & +(\check{\theta}_T - \hat{\theta}_T)'\frac{\partial^2 { L}_T(\bar{\theta}_T, \bar{\tau}_T) }{\partial \tau' \partial \theta}(\tau_T(\check{\theta}_T) - \tau_T(\hat{\theta}_{T}) ),
\end{eqnarray*}
where
\begin{itemize}
\item Under Assumptions \ref{Assp:ExistenceConsistency}, \ref{Assp:AsymptoticNormality} and \ref{Assp:Trinity}, if the test hypothesis \eqref{Eq:HypParameterRestriction} on p. \pageref{Eq:HypParameterRestriction} holds,
 by Theorem \ref{theorem:ConsistencyAsymptoticNormality}ii (p. \pageref{theorem:ConsistencyAsymptoticNormality}) and Proposition \ref{Prop:ConstrainedEstAsNormality}iii (p. \pageref{Prop:ConstrainedEstAsNormality}), $\P$-a.s. as $T \rightarrow \infty$, $\check{\theta}_T- \hat{\theta}_T=(\check{\theta}_T-\theta_0)-( \hat{\theta}_T-\theta_0) =O_{\P}(T^{-\frac{1}{2}})+O_{\P}(T^{-\frac{1}{2}})=O_{\P}(T^{-\frac{1}{2}})$;

\item
Under Assumptions \ref{Assp:ExistenceConsistency} and \ref{Assp:AsymptoticNormality}, by Lemma \ref{Lem:ApproximateFOC} (p. \pageref{Lem:ApproximateFOC}), $\P$-a.s. as $T \rightarrow \infty$, $\frac{\partial L_T(\hat{\theta}_T, \tau_T(\hat{\theta}_{T})) }{\partial \theta'}=O(T^{-1})$, so that $\frac{\partial L_T(\hat{\theta}_T, \tau_T(\hat{\theta}_{T})) }{\partial \theta'}(\check{\theta}_T- \hat{\theta}_T)=O_{\P}(T^{-\frac{3}{2}})$ by the first bullet point;

\item Under Assumptions \ref{Assp:ExistenceConsistency}, \ref{Assp:AsymptoticNormality} and \ref{Assp:Trinity}, if the test hypothesis \eqref{Eq:HypParameterRestriction} on p. \pageref{Eq:HypParameterRestriction} holds, by Lemma  \ref{Lem:DerivativeImplicitFunctionImplicit}iv (p. \pageref{Lem:DerivativeImplicitFunctionImplicit}) and Theorem \ref{theorem:ConsistencyAsymptoticNormality}i (p. \pageref{theorem:ConsistencyAsymptoticNormality}) and Lemma \ref{Lem:ConstrainedEstLagrangian}i (p. \pageref{Lem:ConstrainedEstLagrangian}), $\P$-a.s. as $T \rightarrow \infty$, there exists $\tilde{\theta}_T$ between $\hat{\theta}_T$ and $\theta_0$ s.t. $\sqrt{T} [\tau_T(\check{\theta}_T)- \tau_T(\hat{\theta}_T)]=\sqrt{T} [\tau_T(\check{\theta}_T)- \tau_T(\theta_0)]-\sqrt{T} [\tau_T(\hat{\theta}_T)- \tau_T(\theta_0)]=-V^{-1} M(\check{\theta}_T- \theta_0)+o_{\P}(1)-\left[-V^{-1} M(\hat{\theta}_T- \theta_0)+o_{\P}(T^{-1/2})\right]=-V^{-1} M(\check{\theta}_T- \theta_0)+V^{-1} M(\hat{\theta}_T- \theta_0)+o_{\P}(1)=-V^{-1} M(\check{\theta}_T-\hat{\theta}_T)+o_{\P}(1) $;\item Under Assumptions \ref{Assp:ExistenceConsistency} and \ref{Assp:AsymptoticNormality}, by Lemma \ref{Lem:DLDTau}iv (p. \pageref{Lem:DLDTau}), $\P$-a.s. as $T \rightarrow \infty$, $\frac{\partial L_T(\hat{\theta}_T, \tau_T(\hat{\theta}_{T})) }{\partial \tau'}=O(T^{-1})$, so that, by the first and third bullet point, $\frac{\partial L_T(\hat{\theta}_T, \tau_T(\hat{\theta}_{T})) }{\partial \tau'}(\tau_T(\check{\theta}_T)-\tau_T(\hat{\theta}_{T}))=O_{\P}(T^{-3/2})$, under Assumptions \ref{Assp:ExistenceConsistency}, \ref{Assp:AsymptoticNormality} and \ref{Assp:Trinity}, if the test hypothesis \eqref{Eq:HypParameterRestriction} on p. \pageref{Eq:HypParameterRestriction} holds;
\item Under Assumptions \ref{Assp:ExistenceConsistency}, \ref{Assp:AsymptoticNormality} and \ref{Assp:Trinity}, if the test hypothesis \eqref{Eq:HypParameterRestriction} on p. \pageref{Eq:HypParameterRestriction} holds, by Lemma  \ref{Lem:DerivativeImplicitFunctionImplicit}iv (p. \pageref{Lem:DerivativeImplicitFunctionImplicit}) and Theorem \ref{theorem:ConsistencyAsymptoticNormality}i (p. \pageref{theorem:ConsistencyAsymptoticNormality}) and Lemma \ref{Lem:ConstrainedEstLagrangian}i (p. \pageref{Lem:ConstrainedEstLagrangian}), $\P$-a.s. as $T \rightarrow \infty$, there exists $\tilde{\theta}_T$ between $\hat{\theta}_T$ and $\theta_0$ s.t. $\sqrt{T} [\tau_T(\check{\theta}_T)- \tau_T(\hat{\theta}_T)]=\sqrt{T} [\tau_T(\check{\theta}_T)- \tau_T(\theta_0)]-\sqrt{T} [\tau_T(\hat{\theta}_T)- \tau_T(\theta_0)]=-V^{-1} M(\check{\theta}_T- \theta_0)+o_{\P}(1)-\left[-V^{-1} M(\hat{\theta}_T- \theta_0)+o_{\P}(T^{-1/2})\right]=-V^{-1} M(\check{\theta}_T- \theta_0)+V^{-1} M(\hat{\theta}_T- \theta_0)+o_{\P}(1)=-V^{-1} M(\check{\theta}_T-\hat{\theta}_T)+o_{\P}(1) $; and
\item under Assumptions \ref{Assp:ExistenceConsistency} and \ref{Assp:AsymptoticNormality}, by Lemma \ref{Lem:PartialLTheta0Tau0}ii (p. \pageref{Lem:PartialLTheta0Tau0}), $\P$-a.s. as $T \rightarrow \infty$  $\displaystyle\left \vert\frac{\partial^{2}  L_{T}(\theta_{T}, \tau_{T} ) }{\partial \theta_j \partial \theta_\ell}\right\vert=o(1) $, so that, by Theorem \ref{theorem:ConsistencyAsymptoticNormality}ii (p. \pageref{theorem:ConsistencyAsymptoticNormality}),  $(\check{\theta}_T- \hat{\theta}_T)'\frac{\partial^2 { L}_T(\bar{\theta}_T, \bar{\tau}_T) }{\partial \theta' \partial \theta}(\check{\theta}_T- \hat{\theta}_T)=o_{\P}(T^{-1})$ by the first bullet point.
\end{itemize}
Therefore, Assumptions \ref{Assp:ExistenceConsistency}, \ref{Assp:AsymptoticNormality} and \ref{Assp:Trinity}, if the test hypothesis \eqref{Eq:HypParameterRestriction} on p. \pageref{Eq:HypParameterRestriction} holds, \begin{eqnarray*}
& & 2\{\ln[\hat{f}_{\theta^{*}_T}(\hat{\theta}_T)]-\ln[\hat{f}_{\theta^{*}_T}(\check{\theta}_T)]\}=-2T\left[L_T(\check{\theta}_T,\tau_T(\check{\theta}_T)) -  L_T(\hat{\theta}_T, \tau_T(\hat{\theta}_{T}))\right]\\
& =& -\left[V^{-1} M\sqrt{T}(\check{\theta}_T - \hat{\theta}_T)\right]'\frac{\partial^2 { L}_T(\bar{\theta}_T, \bar{\tau}_T) }{\partial \tau' \partial \tau} \left[V^{-1} M\sqrt{T}(\check{\theta}_T- \hat{\theta}_T)\right]
\\
& & +2\sqrt{T}(\theta_0 - \hat{\theta}_T)'\frac{\partial^2 { L}_T(\bar{\theta}_T, \bar{\tau}_T) }{\partial \tau' \partial \theta}\left[V^{-1} M\sqrt{T}(\theta_0 - \hat{\theta}_T)\right]+o_{\P}(1)\\
& = & -\sqrt{T}(\check{\theta}_T- \hat{\theta}_T)'\left[ M'V^{-1} \frac{\partial^2 { L}_T(\bar{\theta}_T, \bar{\tau}_T) }{\partial \tau' \partial \tau}V^{-1} M-2\frac{\partial^2 { L}_T(\bar{\theta}_T, \bar{\tau}_T) }{\partial \tau' \partial \theta} V^{-1} M\right] \sqrt{T}(\check{\theta}_T- \hat{\theta}_T)+o_{\P}(1) \\
& = & \sqrt{T}(\check{\theta}_T- \hat{\theta}_T)'\Sigma(\theta_0)^{-1} \sqrt{T}(\check{\theta}_T- \hat{\theta}_T)+o_{\P}(1),
\end{eqnarray*}
where the explanations for the convergence are as follow.  Firstly,  under Assumptions \ref{Assp:ExistenceConsistency} and \ref{Assp:AsymptoticNormality}, by Lemma \ref{Lem:PartialLTheta0Tau0}ii (p. \pageref{Lem:PartialLTheta0Tau0}),  for any sequence $(\theta_T, \tau_T)_{T \in \N}$ converging to $(\theta_0, \tau(\theta_0))$, $\P$-a.s. as $T \rightarrow \infty$,
 $\frac{\partial^{2}  L_{T}(\theta_{T}, \tau_{T} ) }{\partial \theta' \partial \tau}\rightarrow \E\left[  \frac{\partial \psi(X_{1},\theta_{0})}{\partial \theta' }\right]=:M $. Secondly, under  Assumptions \ref{Assp:ExistenceConsistency} and \ref{Assp:AsymptoticNormality}, by Lemma \ref{Lem:DLDTauDTau}iv (p. \pageref{Lem:DLDTauDTau}), $\P$-a.s. as $T \rightarrow \infty$, $\frac{\partial L_T(\theta_T, \tau_T) }{\partial \tau\partial \tau'}\rightarrow  \E[\psi(X_1, \theta_{0})\psi(X_1, \theta_{0})']=:V$. Therefore, $\P$-a.s. as $T \rightarrow \infty$,
\begin{eqnarray*}
& &\ M'(V')^{-1} \frac{\partial^2 { L}_T(\bar{\theta}_T, \bar{\tau}_T) }{\partial \tau' \partial \tau}V^{-1} M-2\frac{\partial^2 { L}_T(\bar{\theta}_T, \bar{\tau}_T) }{\partial \tau' \partial \theta} V^{-1} M\\
& \rightarrow &M'(V')^{-1} VV^{-1} M -2M'V^{-1}M\\
&= & - M'V^{-1}M=-\Sigma(\theta_0)^{-1}.
\end{eqnarray*}

\textit{(ii)} Under Assumptions \ref{Assp:ExistenceConsistency}, \ref{Assp:AsymptoticNormality} and \ref{Assp:Trinity},  if the test hypothesis \eqref{Eq:HypParameterRestriction} on p. \pageref{Eq:HypParameterRestriction} holds, $\P$-a.s. as $T \rightarrow \infty$, addition and subtraction of $\sqrt{T}\theta_0$ yield
\begin{eqnarray*}
\sqrt{T}(\hat{\theta}_T-\check{\theta}_T )& = &\sqrt{T}(\hat{\theta}_T-\theta_0 )-\sqrt{T}(\check{\theta}_T-\theta_0)\\
& \stackrel{}{=} & -M^{-1}\frac{1}{\sqrt{T}}\sum_{t=1}^T \psi_t(\theta_0)+o_{\P}(1)-\left[ M^{-1}-\Sigma R'(R\Sigma R')^{-1}RM^{-1}\frac{1}{\sqrt{T}}\sum_{t=1}^T \psi_t(\theta_0)+o_{\P}(1)\right]\\
& \stackrel{}{=} &\Sigma R'(R\Sigma R')^{-1}RM^{-1}\frac{1}{\sqrt{T}}\sum_{t=1}^T \psi_t(\theta_0)+o_{\P}(1)
\end{eqnarray*}
where the explanations for the second equality are the following. Firstly, under Assumptions \ref{Assp:ExistenceConsistency} and \ref{Assp:AsymptoticNormality}, by Proposition \ref{Prop:AsExpansionEstimator} (p. \pageref{Prop:AsExpansionEstimator}), $\P$-a.s. as $T \rightarrow \infty$, $\sqrt{T}(\hat{\theta}_T- \theta_0)=-M\frac{1}{\sqrt{T}}\sum_{t=1}^T \psi_t(\theta_0)+o_{\P}(1)$. Secondly, under Assumptions \ref{Assp:ExistenceConsistency}, \ref{Assp:AsymptoticNormality} and \ref{Assp:Trinity}, by Proposition \ref{Prop:ConstrainedEstAsNormality}i (p. \pageref{Prop:ConstrainedEstAsNormality}), if the test hypothesis \eqref{Eq:HypParameterRestriction} on p. \pageref{Eq:HypParameterRestriction} holds, $\P$-a.s. as $T \rightarrow \infty$,
 $\sqrt{T}(\check{\theta}_T-\theta_0)=M^{-1}-\Sigma R'(R\Sigma R')^{-1}RM^{-1}\frac{1}{\sqrt{T}}\sum_{t=1}^T \psi_t(\theta_0)+o_{\P}(1) $.
\end{proof}

\begin{lem}[Asymptotic limit of $\frac{\partial ^{2}L_T(\hat{\theta}_T, \tau_T(\hat{\theta}_{T})) }{\partial \tau'\partial \tau}$]\label{Lem:DLDTauDTau} Under Assumptions \ref{Assp:ExistenceConsistency} and \ref{Assp:AsymptoticNormality}, for any sequence $(\theta_T, \tau_T)_{T \in \N}$ converging to $(\theta_0, \tau(\theta_0))$, $\P$-a.s. as $T \rightarrow \infty$, for all $(h,k) \in \ldsb 1,m\rdsb^2$, $\P$-a.s. as $T \rightarrow \infty$,
\begin{enumerate}
\item[(i)] $\frac{\partial^{2} M_{1,T}(\theta_T,\tau_T)}{\partial \tau_{h}\partial \tau_{k}}\rightarrow \E[\psi_{k}(X_1, \theta_{0})\psi_{h}(X_1, \theta_{0})]$;
\item[(ii)] $\frac{\partial^{2} M_{2,T}(\theta_T,\tau_T)}{\partial \tau_{h}\partial \tau_{k}}=O(T^{-1})$;
\item[(iii)] $\frac{\partial ^{2}M_{3,T}(\theta_T,\tau_T)}{\partial \tau_{h}\partial \tau_{k}}=O(T^{-1})$; and
\item[(iv)] $\frac{\partial ^{2}L_T(\theta_T, \tau_T) }{\partial \tau\partial \tau'}\rightarrow  \E[\psi(X_1, \theta_{0})\psi(X_1, \theta_{0})']$.
\end{enumerate}

\end{lem}
\begin{proof} \textit{(i)} By equation \eqref{Eq:LogETTermDTauDTau} on p. \pageref{Eq:LogETTermDTauDTau}, for all $(h,k) \in \ldsb 1,m\rdsb^2$,
\begin{eqnarray*}
& &
\frac{\partial^2 M_{1,T}(\theta_T,\tau_T)}{ \partial \tau_{h}\partial \tau_{k} }\nonumber \\
& = &
\negthickspace\left(1 - \frac{m}{2T}  \right) \frac{1}{ \left[ \frac{1}{T} \sum_{i=1}^T \e^{\tau_T'\psi_i(\theta_T)}  \right]^2 }\left\{\negthickspace
  \left[\frac{1}{T} \sum_{i=1}^T \e^{\tau_T'\psi_i(\theta_T)}\right]\negthickspace \left[\frac{1}{T} \sum_{t=1}^T \e^{\tau_T'\psi_t(\theta_T)}
 \psi_{t,h} (\theta_T)
 \psi_{t,k} (\theta_T)\right]\right.\nonumber
\\
& &
\left.
 \negthickspace-\negthickspace \left[\frac{1}{T} \sum_{t=1}^T \e^{\tau_T'\psi_t(\theta_T)} \psi_{t,h} (\theta_T)\right] \left[ \frac{1}{T} \sum_{i=1}^T \e^{\tau_T'\psi_i(\theta_T)}   \psi_{i,k}(\theta_T)\right]\negthickspace
\right\}.
\end{eqnarray*}
 where, as $T \rightarrow \infty$, $(\theta_T\quad  \tau_T) \rightarrow (\theta_0 \quad \tau(\theta_0))$ by assumption. Now, under Assumptions \ref{Assp:ExistenceConsistency} and \ref{Assp:AsymptoticNormality}, by Lemma \ref{Lem:dLdTaudTau}i-iii (p. \pageref{Lem:dLdTaudTau}), for $\overline{B_L}$ a ball around $(\theta_0, \tau(\theta_0))$ of sufficiently small radius, $\E\left[\sup_{(\theta, \tau)\in \overline{B_L}}  \e^{\tau'\psi(X_{1,}\theta)} \right] < \infty  $, $\E\left[\sup_{(\theta, \tau)\in \overline{B_L}}  \vert\e^{\tau'\psi(X_{1,}\theta)} \psi_{k}(X_1, \theta)\psi_{h}(X_1, \theta)\vert \right] < \infty  $, and\\
$\E\left[\sup_{(\theta, \tau)\in \overline{B_L}}  \vert\e^{\tau'\psi(X_{1,}\theta)} \psi_{k}(X_1, \theta)\vert \right] $. Thus, by Assumption \ref{Assp:ExistenceConsistency}(a)(b) and (d), the  ULLN (uniform law of large numbers) \`a la  Wald  \citep[e.g.,][pp. 24-25, Theorem 1.3.3]{2003GhoRam}, implies that, for all $(h,k)\in \ldsb 1,m\rdsb^{2}$, $\P$-a.s. as $T \rightarrow \infty$,
 \begin{eqnarray*}
& & \frac{\partial^2 M_{1,T}(\theta_T, \tau_T)}{ \partial \tau_{k} \partial \tau_{h}}\\
 & \rightarrow & \frac{1}{\E[\e^{\tau(\theta_0)'\psi(X_1, \theta_0)}]^2}\left\{\E[\e^{\tau(\theta_0)'\psi(X_1, \theta_0)}]\E[\e^{\tau(\theta_0)'\psi(X_1, \theta_0)}\psi_{k}(X_1, \theta_{0})\psi_{h}(X_1, \theta_{0})]\right. \\
& & -\left. \E[\e^{\tau(\theta_0)'\psi(X_1, \theta_0)}\psi_{h}(X_1, \theta_{0})]\E[\e^{\tau(\theta_0)'\psi(X_1, \theta_0)}\psi_{k}(X_1, \theta_{0})]\right\}\\
& = &  \E[\psi_{k}(X_1, \theta_{0})\psi_{h}(X_1, \theta_{0})]
 \end{eqnarray*}
 because $\E [\psi(X_1, \theta_0)]=0_{m \times 1}$ by Assumption \ref{Assp:ExistenceConsistency}(c), and  $\tau(\theta_0)=0_{m \times 1}$ by Lemma \ref{Lem:AsTiltingFct}iv (p. \pageref{Lem:AsTiltingFct}) under Assumption \ref{Assp:ExistenceConsistency}(a)-(e) and (g)-(h).

\textit{(ii)} Under Assumptions \ref{Assp:ExistenceConsistency}, by equation \eqref{Eq:DerivativeTermDTauDTau} on p. \pageref{Eq:DerivativeTermDTauDTau}, $\P$-a.s. for $T$ big enough, for all $(h,k) \in \ldsb 1,m\rdsb^2$,
\begin{eqnarray*}
& & \frac{\partial^2 M_{2,T}(\theta_T, \tau_T)}{ \partial \tau_{h}\partial \tau_k }\nonumber \\
& &  =- \frac{1}{T}
\tr\left\{  \left[ \frac{1}{T} \sum_{t=1}^T                \e^{\tau_T'\psi_t(\theta_T)}
      \frac{\partial \psi_t(\theta_T)}{\partial \theta' }  \right]^{-1}
 \left[ \frac{1}{T} \sum_{t=1}^T                \e^{\tau_T'\psi_t(\theta_T)}
 \psi_{t,k}(\theta_T)
      \frac{\partial \psi_t(\theta_T)}{\partial \theta' }  \right]^{}
     \right.\nonumber
     \\
     &  &
     \hspace{.5in} \times \left[ \frac{1}{T} \sum_{t=1}^T                \e^{\tau_T'\psi_t(\theta_T)}
      \frac{\partial \psi_t(\theta_T)}{\partial \theta' }  \right]^{-1}
     \left.
      \left[
      \frac{1}{T} \sum_{t=1}^T                \e^{\tau_T'\psi_t(\theta_T)}
    \psi_{t,h}(\theta_T)  \frac{\partial \psi_t(\theta_T)}{\partial \theta' }
      \right]
\right\} \nonumber\\
& &  + \frac{1}{T}
\tr\left\{  \left[ \frac{1}{T} \sum_{t=1}^T                \e^{\tau_T'\psi_t(\theta_T)}
      \frac{\partial \psi_t(\theta_T)}{\partial \theta' }  \right]^{-1}
      \left[
      \frac{1}{T} \sum_{t=1}^T                \e^{\tau_T'\psi_t(\theta_T)}
      \psi_{t,k}(\theta_T)
    \psi_{t,h}(\theta_T)  \frac{\partial \psi_t(\theta_T)}{\partial \theta' }
      \right]
\right\}.\
\end{eqnarray*}
where, as $T \rightarrow \infty$, $(\theta_T\quad  \tau_T) \rightarrow (\theta_0 \quad \tau(\theta_0))$ by assumption.
Now, under Assumptions \ref{Assp:ExistenceConsistency} and \ref{Assp:AsymptoticNormality}, by Lemma \ref{Lem:dLdTaudTau}iv-vi (p. \pageref{Lem:dLdTaudTau}), for $\overline{B_L}$ a ball around $(\theta_0, \tau(\theta_0))$ of sufficiently small radius, $\E\left[\sup_{(\theta, \tau)\in \overline{B_L}}  \vert\e^{\tau'\psi(X_{1,}\theta)} \frac{\partial\psi(X_1, \theta)}{\partial \theta'}\vert \right] $, $\E\left[\sup_{(\theta, \tau)\in \overline{B_L}}  \vert\e^{\tau'\psi(X_{1,}\theta)} \psi_{k}(X_1, \theta)\frac{\partial\psi(X_1, \theta)}{\partial \theta'}\vert \right] $, and\\ $\E\left[\sup_{(\theta, \tau)\in \overline{B_L}}  \vert\e^{\tau'\psi(X_{1,}\theta)} \psi_{k}(X_1, \theta)\psi_{h}(X_1, \theta)\frac{\partial\psi(X_1, \theta)}{\partial \theta'}\vert \right] $. Thus, by Assumptions \ref{Assp:ExistenceConsistency}(a)(b) and (d),   the ULLN (uniform law of large numbers) \`a la  Wald  \citep[e.g.,][pp. 24-25, Theorem 1.3.3]{2003GhoRam},  implies that, for all $(h,k)\in \ldsb 1,m\rdsb^{2}$, $\P$-a.s. as $T \rightarrow \infty$,
\begin{eqnarray*}
& & T\frac{\partial ^{2}M_{2,T}(\theta_T, \tau_T)}{\partial \tau_{h}\partial \tau_{k}} \\
& \rightarrow & \tr\left\{   \E \left[\e^{\tau(\theta_0)'\psi(X_{1,}\theta_{0})} \frac{\partial \psi(X_1, \theta_0)}{\partial \theta'}\right]^{-1} \E \left[\e^{\tau(\theta_0)'\psi(X_{1,}\theta_{0})} \psi_{k}(X_1,\theta_{0})\frac{\partial \psi(X_1, \theta_0)}{\partial \theta'}\right] \right.\\
& & \left.\times \E \left[\e^{\tau(\theta_0)'\psi(X_{1,}\theta_{0})} \frac{\partial \psi(X_1, \theta_0)}{\partial \theta'}\right]^{-1}\E \left[\e^{\tau(\theta_0)'\psi(X_{1,}\theta_{0})} \psi_{h}(X_1,\theta_{0})\frac{\partial \psi(X_1, \theta_0)}{\partial \theta'}\right]\right\}\\
& &+ \tr\left\{  \E \left[\e^{\tau(\theta_0)'\psi(X_{1,}\theta_{0})} \frac{\partial \psi(X_1, \theta_0)}{\partial \theta'}\right]^{-1} \E \left[\e^{\tau(\theta_0)'\psi(X_{1,}\theta_{0})} \psi_{k}(X_1,\theta_{0})\psi_{h}(X_1,\theta_{0})\frac{\partial \psi(X_1, \theta_0)}{\partial \theta'}\right] \right\}\\
& = & \tr\left\{   \E \left[ \frac{\partial \psi(X_1, \theta_0)}{\partial \theta'}\right]^{-1} \E \left[ \psi_{k}(X_1,\theta_{0})\frac{\partial \psi(X_1, \theta_0)}{\partial \theta'}\right] \right.\\
& & \left.\times \E \left[ \frac{\partial \psi(X_1, \theta_0)}{\partial \theta'}\right]^{-1}\E \left[\psi_{h}(X_1,\theta_{0})\frac{\partial \psi(X_1, \theta_0)}{\partial \theta'}\right]\right\}\\
& &+ \tr\left\{  \E \left[ \frac{\partial \psi(X_1, \theta_0)}{\partial \theta'}\right]^{-1} \E \left[ \psi_{k}(X_1,\theta_{0})\psi_{h}(X_1,\theta_{0})\frac{\partial \psi(X_1, \theta_0)}{\partial \theta'}\right] \right\}
\end{eqnarray*}
because  $\tau(\theta_0)=0_{m \times 1}$ by Lemma \ref{Lem:AsTiltingFct}iv (p. \pageref{Lem:AsTiltingFct}) under Assumption \ref{Assp:ExistenceConsistency}(a)-(e) and (g)-(h).
Therefore, $\P$-a.s. as $T \rightarrow \infty$, $\frac{\partial ^{2}M_{2,T}(\theta_T, \tau_T)}{\partial \tau_{h}\partial \tau_{k}}=O(T^{-1})$.

\textit{(iii)} Under Assumptions \ref{Assp:ExistenceConsistency}(a)(b)(e)(g)(h),  by equation \eqref{Eq:VarianceTermDTauDtau}  (p. \pageref{Eq:VarianceTermDTauDtau}), for all $(h,k) \in \ldsb 1,m\rdsb^2$,
\begin{eqnarray*}
& & \frac{\partial^{2}M_{3,T}(\theta_T, \tau_T)}{\partial \tau_{h}\partial \tau_{k}   } \nonumber
\\
& = &  \frac{1}{2T } \tr \left\{
\left[  \frac{1}{T} \sum_{t=1}^T                  \e^{\tau_T'\psi_t(\theta_T)}
    \psi_t(\theta_T)  \psi_t(\theta_T)' \right]^{-1}
  \left[ \frac{1}{T} \sum_{t=1}^T                  \e^{\tau_T'\psi_t(\theta_T)}
     \psi_{t, k} (\theta_T) \psi_t(\theta_T)  \psi_t(\theta_T)'     \right]\nonumber
    \right.
    \\
    & & \hspace{.5in} \left.
    \times \left[  \frac{1}{T} \sum_{t=1}^T                  \e^{\tau_T'\psi_t(\theta_T)}
    \psi_t(\theta_T)  \psi_t(\theta_T)' \right]^{-1}
     \left[ \frac{1}{T} \sum_{t=1}^T                  \e^{\tau_T'\psi_t(\theta_T)}
     \psi_{t, h} (\theta_T) \psi_t(\theta_T)  \psi_t(\theta_T)'     \right]
\right\}\nonumber
\\
&  &   -\frac{1}{2T } \tr \left\{
\left[  \frac{1}{T} \sum_{t=1}^T                  \e^{\tau_T'\psi_t(\theta_T)}
    \psi_t(\theta_T)  \psi_t(\theta_T)' \right]^{-1}
     \left[ \frac{1}{T} \sum_{t=1}^T                  \e^{\tau_T'\psi_t(\theta_T)}
    \psi_{t,k} (\theta_T)  \psi_{t, h} (\theta_T) \psi_t(\theta_T)  \psi_t(\theta_T)'     \right]
\right\}
\end{eqnarray*}
where, as $T \rightarrow \infty$, $(\theta_T\quad  \tau_T) \rightarrow (\theta_0 \quad \tau(\theta_0))$ by assumption.  Now,
under Assumptions \ref{Assp:ExistenceConsistency} and \ref{Assp:AsymptoticNormality}, by Lemma \ref{Lem:dLdTaudTau}vii-ix (p. \pageref{Lem:dLdTaudTau}), there exists a closed ball $\overline{B_L}\subset \Sbf$ centered at $(\theta_0, \tau(\theta_0))$ with strictly positive radius s.t.,   for all $k \in \ldsb 1,m\rdsb$, $\E\left[\sup_{(\theta, \tau)\in \overline{B_L}}  \vert\e^{\tau'\psi(X_{1,}\theta)} \psi(X_1, \theta)\psi(X_1, \theta)'\vert \right] < \infty$,
$\E\left[\sup_{(\theta, \tau)\in \overline{B_L}}  \vert\e^{\tau'\psi(X_{1,}\theta)} \psi_{k}(X_1, \theta)\psi(X_1, \theta)\psi(X_1, \theta)'\vert \right] < \infty$, and $\E[\sup_{(\theta, \tau)\in \overline{B_L}}  \vert\e^{\tau'\psi(X_{1,}\theta)} \psi_{h}(X_1, \theta) \\\psi_{k}(X_1, \theta)\psi(X_1, \theta)\psi(X_1, \theta)'\vert ] < \infty$.
Thus,
by Assumptions \ref{Assp:ExistenceConsistency}(a)(b) and (d),   the ULLN (uniform law of large numbers) \`a la  Wald  \citep[e.g.,][pp. 24-25, Theorem 1.3.3]{2003GhoRam},  implies that, for all $(h,k)\in \ldsb 1,m\rdsb^{2}$, $\P$-a.s. as $T \rightarrow \infty$,
\begin{eqnarray*}
& & T \frac{\partial^{2}M_{3,T}(\theta_T, \tau_T)}{\partial \tau_{h}\partial \tau_{k}   }\\
 & \rightarrow &    \frac{1}{2} {\rm tr} \left\{
\E \left[ \e^{\tau(\theta_0)'\psi(X_{1,}\theta_{0})} \psi(X_1, \theta_{0})\psi(X_1, \theta_{0})'\right]^{-1}
    \E \left[  \e^{\tau(\theta_0)'\psi(X_{1,}\theta_{0})} \psi_{k}(X_1, \theta_{0})\psi(X_1, \theta_{0})\psi(X_1, \theta_{0})'\right]\right.\\
& & \left.\times \E \left[ \e^{\tau(\theta_0)'\psi(X_{1,}\theta_{0})} \psi(X_1, \theta_{0})\psi(X_1, \theta_{0})'\right]^{-1}
    \E \left[  \e^{\tau(\theta_0)'\psi(X_{1,}\theta_{0})} \psi_{h}(X_1, \theta_{0})\psi(X_1, \theta_{0})\psi(X_1, \theta_{0})'\right]\right\} \\
&  &  -  \frac{1}{2} {\rm tr} \left\{
\E\negthickspace \left[ \e^{\tau(\theta_0)'\psi(X_{1,}\theta_{0})}\psi(X_1, \theta_{0})\psi(X_1, \theta_{0})'\right]^{-1}
\right.
\\ & & \hspace{1.8in} \left. \times
    E \left[  \e^{\tau(\theta_0)'\psi(X_{1,}\theta_{0})}\psi_{k}(X_1, \theta_{0})\psi_{h}(X_1, \theta_{0})\psi(X_1, \theta_{0})\psi(X_1, \theta_{0})'\right] 
\right\}\\
 & = &    \frac{1}{2} {\rm tr} \left\{
\E \left[  \psi(X_1, \theta_{0})\psi(X_1, \theta_{0})'\right]^{-1}
    \E \left[   \psi_{k}(X_1, \theta_{0})\psi(X_1, \theta_{0})\psi(X_1, \theta_{0})'\right]\right.\\
& & \left. \times \E \left[  \psi(X_1, \theta_{0})\psi(X_1, \theta_{0})'\right]^{-1}
    \E \left[  \psi_{h}(X_1, \theta_{0})\psi(X_1, \theta_{0})\psi(X_1, \theta_{0})'\right]\right\} \\
&  &  -  \frac{1}{2} {\rm tr} \left\{
\E \left[ \psi(X_1, \theta_{0})\psi(X_1, \theta_{0})'\right]^{-1}
    \E \left[  \psi_{k}(X_1, \theta_{0})\psi_{h}(X_1, \theta_{0})\psi(X_1, \theta_{0})\psi(X_1, \theta_{0})'\right]
\right\}
\end{eqnarray*}
because $\tau(\theta_0)=0_{m \times 1}$ by Lemma \ref{Lem:AsTiltingFct}iv (p. \pageref{Lem:AsTiltingFct}) under Assumption \ref{Assp:ExistenceConsistency}(a)-(e) and (g)-(h).
 Therefore, $\P$-a.s. as $T \rightarrow \infty$,  $\frac{\partial M_{3,T}(\hat{\theta}_T, \tau_T(\hat{\theta}_T))}{\partial \tau_{k}}=O(T^{-1})$.

\textit{(iv)} Under Assumption \ref{Assp:ExistenceConsistency}(a)-(b) and (d)-(h),  by Lemma \ref{Lem:LogESPDecomposition} (p. \pageref{Lem:LogESPDecomposition}), for all $(\theta, \tau)$ in a neighborhood of $(\theta_0, \tau(\theta_0))$, $L_T(\theta, \tau)=M_{1,T}(\theta, \tau)+M_{2,T}(\theta, \tau)+M_{3,T}(\theta, \tau)$, so that the result follows from the statement (i)-(iii) of the present lemma.
\end{proof}

\begin{prop}[Asymptotic normality of $\check{\theta}_T$, $\tau_T(\check{\theta}_T)$ and $\check{\gamma}_T$]\label{Prop:ConstrainedEstAsNormality} Under Assumptions \ref{Assp:ExistenceConsistency}, \ref{Assp:AsymptoticNormality} and \ref{Assp:Trinity}, if the test hypothesis \eqref{Eq:HypParameterRestriction} on p. \pageref{Eq:HypParameterRestriction} holds, $\P$-a.s. as $T \rightarrow \infty$,
\begin{enumerate}
\item[(i)] $\displaystyle \sqrt{T}\begin{bmatrix}\check{\theta}_T- \theta_0 \\
\tau_T(\check{\theta}_T) \\
\check{\gamma}_T \\
\end{bmatrix}
=
 \begin{bmatrix} \quad M^{-1}-\Sigma R'(R\Sigma R')^{-1}RM^{-1} \\
(M')^{-1}R'(R\Sigma R')^{-1}RM^{-1}  \\
 -(R\Sigma R')^{-1}RM^{-1}  \\
\end{bmatrix}\frac{1}{\sqrt{T}}\sum_{t=1}^T\psi_t(\theta_0)+ o_{\P}(1)$;\ and
\item[(ii)] $\displaystyle \sqrt{T}\begin{bmatrix}\check{\theta}_T- \theta_0 \\
\tau_T(\check{\theta}_T) \\
\check{\gamma}_T \\
\end{bmatrix}
=
 \begin{bmatrix} \quad
\Sigma^{1/2'} P_{\Sigma^{1/2}R'}^{\perp}\Sigma^{-1/2'}M^{-1}
 \\
(M')^{-1}\Sigma^{-1/2}P_{\Sigma^{1/2}R'}\Sigma^{-1/2'}M^{-1} &   \\
 -(R\Sigma R')^{-1}RM^{-1}  \\
\end{bmatrix}\frac{1}{\sqrt{T}}\sum_{t=1}^T\psi_t(\theta_0)+ o_{\P}(1)$
\item[(iii)] $ \displaystyle \sqrt{T}\begin{bmatrix}\check{\theta}_T- \theta_0 \\
\tau_T(\check{\theta}_T) \\
\check{\gamma}_T \\
\end{bmatrix}
\stackrel{D}{\rightarrow} \mathcal{N}\left( 0, \begin{bmatrix}(\Sigma^{1/2})'P^{\perp}_{\Sigma^{1/2}R'}\Sigma^{1/2} & 0_{m \times m} & 0_{m \times q} \\
0_{m \times m} & (V^{1/2})^{-1}P_{\Sigma^{1/2}R'}(V^{1/2'})^{-1} & -(M')^{-1}R'(R\Sigma R')^{-1} \\
0_{q \times m} & -(R\Sigma R')^{-1}RM^{-1} & (R \Sigma R')^{-1} \\
\end{bmatrix}\right)$,
\end{enumerate}
where $\Sigma:=\Sigma(\theta_0):= M^{-1}V(M')^{-1}$, $M:=\E \left[\frac{\partial \psi(X_1, \theta_0)}{\partial \theta'}\right]$, $V:= \E[ \psi(X_1, \theta_0)\psi(X_1, \theta_0)']$, and $R:=\frac{\partial r(\theta_0)}{\partial \theta'}$.
\end{prop}
\begin{proof} \textit{(i)-(ii)}  The function  $L_T(\theta, \tau)$ is well-defined and twice continuously differentiable in a neighborhood of $(\theta_0'\; \tau(\theta_0)')$   $\P$-a.s. for $T$ big enough by subsection \ref{Sec:LTAndDerivatives} (p. \pageref{Sec:LTAndDerivatives}),  under Assumptions \ref{Assp:ExistenceConsistency} and \ref{Assp:AsymptoticNormality}(a). Similarly, the function  $S_T(\theta, \tau):= \frac{1}{T}\sum_{t=1}^T \e^{\tau'\psi_t(\theta)}\psi_t(\theta)$ and $\theta \mapsto r(\theta) $ are  continuously differentiable in a neighborhood of $(\theta_0' \; \tau(\theta_0)')$ by Assumption \ref{Assp:ExistenceConsistency}(a)(b) and \ref{Assp:Trinity}(a). Now, under Assumptions \ref{Assp:ExistenceConsistency}, \ref{Assp:AsymptoticNormality}, and \ref{Assp:Trinity}(a), by Lemma \ref{Lem:ConstrainedEstLagrangian}i (p. \pageref{Lem:ConstrainedEstLagrangian}),  Lemma \ref{Lem:Schennachtheorem10PfFirstSteps}iii (p. \pageref{Lem:Schennachtheorem10PfFirstSteps}), $\P$-a.s., $\check{\theta}_T \rightarrow \theta_0$ and $\tau_T(\check{\theta}_T )\rightarrow \tau(\theta_0)$,  so that $\P$-a.s. for $T$ big enough, $(\check{\theta}_T ' \; \tau_T(\check{\theta}_T)')$ is in any arbitrary small neighborhood of $(\theta_0'\; \tau(\theta_0)')$. Therefore,  under Assumptions \ref{Assp:ExistenceConsistency}, \ref{Assp:AsymptoticNormality} and \ref{Assp:Trinity} (a),  stochastic first-order Taylor-Lagrange expansions \cite[Lemma 3]{1969Jen} around $(\theta_0, \tau(\theta_0))$ evaluated at $(\check{\theta}_T, \tau_T(\check{\theta}_T))$ yield, $\P$-a.s. for $T$ big enough
\begin{eqnarray*}
\frac{\partial { L}_T(\check{\theta}_T, \tau_T(\check{\theta}_T)) }{\partial \theta} & =& \frac{\partial { L}_T(\theta_0, \tau(\theta_0)) }{\partial \theta} + \frac{\partial^2 { L}_T(\bar{\theta}_T, \bar{\tau}_T) }{\partial \theta' \partial \theta}\left( \check{\theta}_T  - \theta_0 \right)+\frac{\partial^2 { L}_T(\bar{\theta}_T, \bar{\tau}_T) }{\partial \tau' \partial \theta} \tau_T(\check{\theta}_T)\\
S_T(\check{\theta}_T, \tau_T(\check{\theta}_T))& = &S_T(\theta_0, \tau(\theta_0) )+ \frac{\partial  S_T(\bar{\theta}_T, \bar{\tau}_T) }{\partial \theta'}\left( \check{\theta}_T  - \theta_0 \right)+ \frac{ \partial S_T(\bar{\theta}_T, \bar{\tau}_T) }{\partial \tau'} \tau_T(\check{\theta}_T)\\
r(\check{\theta}_T) & =& r(\theta_0)+ \frac{\partial r(\bar{\theta}_T) }{\partial \theta'} \left( \check{\theta}_T  - \theta_0 \right)
\end{eqnarray*}
because $\tau(\theta_0)=0_{m \times 1}$ by Lemma \ref{Lem:AsTiltingFct}iv (p. \pageref{Lem:AsTiltingFct}), and where  $\bar{\theta}_T$ and $\bar{\tau}_T$ are between $\check{\theta}_T$ and $\theta_0$, and between  $\tau_T(\check{\theta}_T)$ and $\tau(\theta_0)$, respectively.  Now, under Assumptions \ref{Assp:ExistenceConsistency} and \ref{Assp:AsymptoticNormality},  by definition of $\check{\theta}_T$ and   by definition of $\tau_T(.)$ (equation \ref{Eq:ESPTiltingEquation} on p. \pageref{Eq:ESPTiltingEquation}), $r(\check{\theta}_T)=0_{q \times 1} $ and  $S_T(\check{\theta}_T, \tau_T(\check{\theta}_T))=0$, respectively. Moreover, under Assumptions \ref{Assp:ExistenceConsistency},  \ref{Assp:AsymptoticNormality} and \ref{Assp:Trinity}, by Lemma \ref{Lem:ConstrainedEstLagrangian}iv (p. \pageref{Lem:ConstrainedEstLagrangian}), $\P$-a.s. as $T \rightarrow \infty$, $\frac{\partial L_T(\check{\theta}_T, \tau_T(\check{\theta}_T))}{\partial \theta}=-\frac{\partial r(\check{\theta}_T)'}{\partial \theta}\check{\gamma}_T+O(T^{-1}) $. Therefore, under Assumptions \ref{Assp:ExistenceConsistency},  \ref{Assp:AsymptoticNormality} and \ref{Assp:Trinity}, $\P$-a.s. as $T \rightarrow \infty$,
\begin{eqnarray*}
\begin{cases}
 O(T^{-1})  =  \frac{\partial { L}_T(\theta_0, \tau(\theta_0)) }{\partial \theta} + \frac{\partial^2 { L}_T(\bar{\theta}_T, \bar{\tau}_T) }{\partial \theta' \partial \theta}\left( \check{\theta}_T  - \theta_0 \right)+\frac{\partial^2 { L}_T(\bar{\theta}_T, \bar{\tau}_T) }{\partial \tau' \partial \theta} \tau_T(\check{\theta}_T)+ \frac{\partial r(\check{\theta}_T)'}{\partial \theta}\check{\gamma}_T\\
0_{m \times 1}  = S_T(\theta_0, \tau(\theta_0) )+ \frac{\partial  S_T(\bar{\theta}_T, \bar{\tau}_T) }{\partial \theta'}\left( \check{\theta}_T  - \theta_0 \right)+ \frac{ \partial S_T(\bar{\theta}_T, \bar{\tau}_T) }{\partial \tau'} \tau_T(\check{\theta}_T)\\
0_{q \times 1}  = r(\theta_0)+ \frac{\partial r(\bar{\theta}_T) }{\partial \theta'} \left( \check{\theta}_T  - \theta_0 \right)
\end{cases},
\end{eqnarray*}
which in matrix form is\begin{eqnarray*}
\begin{bmatrix}O(T^{-1}) \\
0_{m \times 1} \\
0_{q \times 1} \\
\end{bmatrix}
=
\begin{bmatrix}\frac{\partial { L}_T(\theta_0, \tau(\theta_0)) }{\partial \theta} \\
S_T(\theta_0, \tau(\theta_0) ) \\
r(\theta_0) \\
\end{bmatrix}
+
\begin{bmatrix}\frac{\partial^2 { L}_T(\bar{\theta}_T, \bar{\tau}_T) }{\partial \theta' \partial \theta} & \frac{\partial^2 { L}_T(\bar{\theta}_T, \bar{\tau}_T) }{\partial \tau' \partial \theta} & \frac{\partial r(\check{\theta}_T)'}{\partial \theta'} \\
\frac{\partial  S_T(\bar{\theta}_T, \bar{\tau}_T) }{\partial \theta'} & \frac{ \partial S_T(\bar{\theta}_T, \bar{\tau}_T) }{\partial \tau'} & 0 \\
\frac{\partial r(\bar{\theta}_T) }{\partial \theta'} & 0 & 0 \\
\end{bmatrix}
\begin{bmatrix}\check{\theta}_T  - \theta_0 \\
\tau_T(\check{\theta}_T) \\
\check{\gamma}_T \\
\end{bmatrix}.
\end{eqnarray*}
Now, under Assumptions \ref{Assp:ExistenceConsistency},  \ref{Assp:AsymptoticNormality} and \ref{Assp:Trinity}, by Lemma \ref{Lem:PartialLAndSAndR}ii (p. \pageref{Lem:PartialLAndSAndR}), $\P$-a.s. for $T$ big enough, the matrix $\begin{bmatrix}\frac{\partial^2 { L}_T(\bar{\theta}_T, \bar{\tau}_T) }{\partial \theta' \partial \theta} & \frac{\partial^2 { L}_T(\bar{\theta}_T, \bar{\tau}_T) }{\partial \tau' \partial \theta} & \frac{\partial r(\check{\theta}_T)'}{\partial \theta'} \\
\frac{\partial  S_T(\bar{\theta}_T, \bar{\tau}_T) }{\partial \theta'} & \frac{ \partial S_T(\bar{\theta}_T, \bar{\tau}_T) }{\partial \tau'} & 0 \\
\frac{\partial r(\bar{\theta}_T) }{\partial \theta'} & 0 & 0 \\
\end{bmatrix} $ is invertible. Then, under Assumptions \ref{Assp:ExistenceConsistency},  \ref{Assp:AsymptoticNormality} and \ref{Assp:Trinity}, solving for the parameters  and multiplying by $\sqrt{T}$ yield, $\P$-a.s. as $T \rightarrow \infty$,
\begin{eqnarray*}
& &\sqrt{T}\begin{bmatrix}\check{\theta}_T  - \theta_0 \\
\tau_T(\check{\theta}_T) \\
\check{\gamma}_T \\
\end{bmatrix}\\
&=&
-\begin{bmatrix}\frac{\partial^2 { L}_T(\bar{\theta}_T, \bar{\tau}_T) }{\partial \theta' \partial \theta} & \frac{\partial^2 { L}_T(\bar{\theta}_T, \bar{\tau}_T) }{\partial \tau' \partial \theta} & \frac{\partial r(\check{\theta}_T)'}{\partial \theta} \\
\frac{\partial  S_T(\bar{\theta}_T, \bar{\tau}_T) }{\partial \theta'} & \frac{ \partial S_T(\bar{\theta}_T, \bar{\tau}_T) }{\partial \tau'} & 0 \\
\frac{\partial r(\bar{\theta}_T) }{\partial \theta'} & 0 & 0 \\
\end{bmatrix}^{-1}
\sqrt{T}\begin{bmatrix}\frac{\partial { L}_T(\theta_0, \tau(\theta_0)) }{\partial \theta} +O(T^{- 1})\\
S_T(\theta_0, \tau(\theta_0) ) \\
r(\theta_0) \\
\end{bmatrix}\\
& \stackrel{(a)}{=} &
- \begin{bmatrix}\frac{\partial^2 { L}_T(\bar{\theta}_T, \bar{\tau}_T) }{\partial \theta' \partial \theta} & \frac{\partial^2 { L}_T(\bar{\theta}_T, \bar{\tau}_T) }{\partial \tau' \partial \theta} & \frac{\partial r(\check{\theta}_T)'}{\partial \theta} \\
\frac{\partial  S_T(\bar{\theta}_T, \bar{\tau}_T) }{\partial \theta'} & \frac{ \partial S_T(\bar{\theta}_T, \bar{\tau}_T) }{\partial \tau'} & 0 \\
\frac{\partial r(\bar{\theta}_T) }{\partial \theta'} & 0 & 0 \\
\end{bmatrix}^{-1}
\begin{bmatrix}O(T^{- \frac{1}{2}})\\
\sqrt{T}\frac{1}{T}\sum_{t=1}^T\psi_t(\theta_0) \\
0 \\
\end{bmatrix}\\
&\underset{}{ \stackrel{(b)}{
=}} & \negthickspace-\negthickspace \begin{bmatrix}- \Sigma+\Sigma R'(R\Sigma R')^{-1}R\Sigma & \quad M^{-1}\negthickspace-\negthickspace\Sigma R'(R\Sigma R')^{-1}RM^{-1} & \Sigma R'(R\Sigma R')^{-1} \\
(M')^{-1}\negthickspace-\negthickspace(M')^{-1}R'(R\Sigma R')^{-1}R\Sigma & (M')^{-1}R'(R\Sigma R')^{-1}RM^{-1} & \; -(M')^{-1}R'(R\Sigma R')^{-1} \\
(R\Sigma R')^{-1} R\Sigma& -(R\Sigma R')^{-1}RM^{-1} & (R\Sigma R')^{-1} \\
\end{bmatrix}\negthickspace
\\
& & \hspace{4in} \times
\begin{bmatrix}O(T^{- \frac{1}{2}})\\
\frac{1}{\sqrt{T}}\sum_{t=1}^T\psi_t(\theta_0) \\
0 \\
\end{bmatrix}
\\
& & +\left\{  \begin{bmatrix}- \Sigma+\Sigma R'(R\Sigma R')^{-1}R\Sigma & \quad M^{-1}\negthickspace-\negthickspace\Sigma R'(R\Sigma R')^{-1}RM^{-1} & \Sigma R'(R\Sigma R')^{-1} \\
(M')^{-1}\negthickspace-\negthickspace(M')^{-1}R'(R\Sigma R')^{-1}R\Sigma & (M')^{-1}R'(R\Sigma R')^{-1}RM^{-1} & \; -(M')^{-1}R'(R\Sigma R')^{-1} \\
(R\Sigma R')^{-1} R\Sigma& -(R\Sigma R')^{-1}RM^{-1} & (R\Sigma R')^{-1} \\
\end{bmatrix}\right.\\
& &- \left. \begin{bmatrix}\frac{\partial^2 { L}_T(\bar{\theta}_T, \bar{\tau}_T) }{\partial \theta' \partial \theta} & \frac{\partial^2 { L}_T(\bar{\theta}_T, \bar{\tau}_T) }{\partial \tau' \partial \theta} & \frac{\partial r(\check{\theta}_T)'}{\partial \theta} \\
\frac{\partial  S_T(\bar{\theta}_T, \bar{\tau}_T) }{\partial \theta'} & \frac{ \partial S_T(\bar{\theta}_T, \bar{\tau}_T) }{\partial \tau'} & 0 \\
\frac{\partial r(\bar{\theta}_T) }{\partial \theta'} & 0 & 0 \\
\end{bmatrix}^{-1} \right\} \begin{bmatrix}O(T^{- \frac{1}{2}})\\
\frac{1}{\sqrt{T}}\sum_{t=1}^T\psi_t(\theta_0) \\
0 \\
\end{bmatrix}\\
&\underset{}{ \stackrel{(c)}{
=}} & \begin{bmatrix} \quad M^{-1}-\Sigma R'(R\Sigma R')^{-1}RM^{-1} \\
(M')^{-1}R'(R\Sigma R')^{-1}RM^{-1}  \\
 -(R\Sigma R')^{-1}RM^{-1}  \\
\end{bmatrix}\frac{1}{\sqrt{T}}\sum_{t=1}^T\psi_t(\theta_0)+ o_{\P}(1) \\
&\underset{}{ \stackrel{(d)}{
=}} & \begin{bmatrix} \quad
\Sigma^{1/2'} P_{\Sigma^{1/2}R'}^{\perp}\Sigma^{-1/2'}M^{-1}
 \\
(M')^{-1}\Sigma^{-1/2}P_{\Sigma^{1/2}R'}\Sigma^{-1/2'}M^{-1} &   \\
 -(R\Sigma R')^{-1}RM^{-1}  \\
\end{bmatrix}\frac{1}{\sqrt{T}}\sum_{t=1}^T\psi_t(\theta_0)+ o_{\P}(1)
\end{eqnarray*}
\textit{(a)} Firstly, under Assumptions \ref{Assp:ExistenceConsistency} and \ref{Assp:AsymptoticNormality}, by Lemma \ref{Lem:PartialLTheta0Tau0}i (p. \pageref{Lem:PartialLTheta0Tau0}), $\P$-a.s. as $T \rightarrow \infty$,  $\frac{\partial  L_{T}(\theta_0, \tau(\theta_0) ) }{\partial \theta_j}=O(T^{-1})$, so that $\sqrt{T}\left[\frac{\partial  L_{T}(\theta_0, \tau(\theta_0) ) }{\partial \theta_j}+O(T^{-1})\right]=O(T^{-\frac{1}{2}})$. Secondly, note that $S_T(\theta_0, \tau(\theta_0) )=\frac{1}{T}\sum_{t=1}^T\psi_t(\theta_0)$ because $\tau(\theta_0)=0_{m \times 1}$ by Lemma \ref{Lem:AsTiltingFct}iv (p. \pageref{Lem:AsTiltingFct}) under Assumption \ref{Assp:ExistenceConsistency}(a)-(e) and (g)-(h). Finally, if the test hypothesis \eqref{Eq:HypParameterRestriction} on p. \pageref{Eq:HypParameterRestriction} holds, then $r(\theta_0)=0_{q \times 1}$.    \textit{(b)}  Add and subtract the matrix $\begin{bmatrix}- \Sigma+\Sigma R'(R\Sigma R')^{-1}R\Sigma & \quad M^{-1}\negthickspace-\negthickspace\Sigma R'(R\Sigma R')^{-1}RM^{-1} & \Sigma R'(R\Sigma R')^{-1} \\
(M')^{-1}\negthickspace-\negthickspace(M')^{-1}R'(R\Sigma R')^{-1}R\Sigma & (M')^{-1}R'(R\Sigma R')^{-1}RM^{-1} & \; -(M')^{-1}R'(R\Sigma R')^{-1} \\
(R\Sigma R')^{-1} R\Sigma& -(R\Sigma R')^{-1}RM^{-1} & (R\Sigma R')^{-1} \\
\end{bmatrix} $. \textit{(c)} Firstly, the first and third column of the first square matrix cancel out because the first element and third element of the vector are zeros.
 Secondly, under Assumptions \ref{Assp:ExistenceConsistency},  \ref{Assp:AsymptoticNormality} and \ref{Assp:Trinity}, by Lemma \ref{Lem:PartialLAndSAndR}iii (p. \pageref{Lem:PartialLAndSAndR})
and Theorem \ref{theorem:ConsistencyAsymptoticNormality}i (p. \pageref{theorem:ConsistencyAsymptoticNormality}), $\P$-a.s. as $T \rightarrow \infty$, the curly bracket is $o(1)$, and, under Assumption \ref{Assp:ExistenceConsistency}(a)-(c) and (g), by the Lindeberg-L\'evy CLT,  $\frac{1}{\sqrt{T}}\sum_{t=1}^T\psi_t(\theta_0)=O_{\P}(1) $, as $T \rightarrow \infty$.  \textit{(d)} By definition $\Sigma=\Sigma^{1/2'}\Sigma^{1/2}$ and  $\Sigma^{-1/2'}=[\Sigma^{1/2'}]^{-1}$. Thus,
\begin{itemize}
\item  $M^{-1}-\Sigma R'(R\Sigma R')^{-1}RM^{-1}=\Sigma^{1/2'} [I-\Sigma^{1/2} R'(R\Sigma R')^{-1}R\Sigma^{1/2'} ]\Sigma^{-1/2'}M^{-1}$

$=\Sigma^{1/2'} P_{\Sigma^{1/2}R'}^{\perp}\Sigma^{-1/2'}M^{-1}$ where $P_{\Sigma^{1/2}R'}^{\perp} $ denotes the orthogonal projection on the orthogonal of the space spanned by the columns of $\Sigma^{1/2}R'$.
\item  $(M')^{-1}R'(R\Sigma R')^{-1}RM^{-1}=(M')^{-1}\Sigma^{-1/2}[\Sigma^{1/2}R'(R\Sigma R')^{-1}R\Sigma^{1/2'}] \Sigma^{-1/2'}M^{-1}\\=(M')^{-1}\Sigma^{-1/2}P_{\Sigma^{1/2}R'}\Sigma^{-1/2'}M^{-1}
=(M')^{-1}\Sigma^{-1/2}P_{\Sigma^{1/2}R'}\Sigma^{-1/2'}\Sigma\Sigma^{-1/2}P_{\Sigma^{1/2}R'}\Sigma^{-1/2'}M^{-1} \\=(M')^{-1}\Sigma^{-1/2}P_{\Sigma^{1/2}R'}P_{\Sigma^{1/2}R'}\Sigma^{-1/2'}M^{-1}=(M')^{-1}\Sigma^{-1/2}P_{\Sigma^{1/2}R'}\Sigma^{-1/2'}M^{-1}\\= (M')^{-1}[V^{1/2}(M')^{-1}]^{-1}P_{\Sigma^{1/2}R'}[M^{-1}V^{1/2'}]^{-1}M^{-1}=(V^{1/2})^{-1}P_{\Sigma^{1/2}R'}(V^{1/2'})^{-1} $ because $M^{-1}V(M')^{-1}=:\Sigma=\Sigma^{1/2'}\Sigma^{1/2} $, so that $\Sigma^{-1/2}:=(\Sigma^{1/2})^{-1}=[V^{1/2}(M')^{-1}]^{-1}=M'V^{-1/2}$ and $\Sigma^{-1/2'}:=(\Sigma^{1/2'})^{-1}=[M^{-1}V^{1/2'}]^{-1}=V^{-1/2'}M$.
\end{itemize}

\textit{(iii)} Under Assumptions \ref{Assp:ExistenceConsistency},  \ref{Assp:AsymptoticNormality} and \ref{Assp:Trinity}, by the statement (ii) of the present proposition, $\P$-a.s. as $T \rightarrow \infty$,

  \begin{eqnarray*}
& &\sqrt{T}\begin{bmatrix}\check{\theta}_T  - \theta_0 \\
\tau_T(\check{\theta}_T) \\
\check{\gamma}_T \\
\end{bmatrix}\\
&\underset{}{ \stackrel{}{
=}} & \begin{bmatrix} \quad \Sigma^{1/2'} P_{\Sigma^{1/2}R'}^{\perp}\Sigma^{-1/2'}M^{-1} \\
(M')^{-1}\Sigma^{-1/2}P_{\Sigma^{1/2}R'}\Sigma^{-1/2'}M^{-1} &\quad    \\
 -(R\Sigma R')^{-1}RM^{-1}  \\
\end{bmatrix}\frac{1}{\sqrt{T}}\sum_{t=1}^T\psi_t(\theta_0)+ o(1) \\
& \underset{(a)}{ \stackrel{D}{\rightarrow }} &- \begin{bmatrix}\Sigma^{1/2'} P_{\Sigma^{1/2}R'}^{\perp}\Sigma^{-1/2'}M^{-1} \\
(M')^{-1}\Sigma^{-1/2}P_{\Sigma^{1/2}R'}\Sigma^{-1/2'}M^{-1} &\quad   \\
-(R\Sigma R')^{-1}RM^{-1} \\
\end{bmatrix}
\mathcal{N}(0, V)\\
& \underset{(b)}{ \stackrel{D}{=}} & \mathcal{N}\left(0, \begin{bmatrix}\Sigma^{1/2'} P_{\Sigma^{1/2}R'}^{\perp}\Sigma^{-1/2'}M^{-1} \\
(M')^{-1}\Sigma^{-1/2}P_{\Sigma^{1/2}R'}\Sigma^{-1/2'}M^{-1} &\quad   \\
-(R\Sigma R')^{-1}RM^{-1} \\
\end{bmatrix}V\right.\\
& &  \left.\begin{array}{c}
 \\
 \\
 \\
\end{array}\times\begin{bmatrix}(M')^{-1}\Sigma^{-1/2} P_{\Sigma^{1/2}R'}^{\perp}\Sigma^{1/2} & \quad (M')^{-1}\Sigma^{-1/2}P_{\Sigma^{1/2}R'}\Sigma^{-1/2'}M^{-1} & \quad-(M')^{-1}R'(R\Sigma R')^{-1}\\
\end{bmatrix}\right)\\
& \underset{(c)}{ \stackrel{D}{=}} &  \mathcal{N}\left(0,\begin{bmatrix}(\Sigma^{1/2})'P^{\perp}_{\Sigma^{1/2}R'}\Sigma^{1/2} & 0_{m \times m} & 0_{m \times q} \\
0_{m \times m} & (V^{1/2})^{-1}P_{\Sigma^{1/2}R'}(V^{1/2'})^{-1} & -(M')^{-1}R'(R\Sigma R')^{-1} \\
0_{q \times m} & -(R\Sigma R')^{-1}RM^{-1} & (R\Sigma R')^{-1} \\
\end{bmatrix}\right)
\end{eqnarray*}
 \textit{(a)} Under Assumption \ref{Assp:ExistenceConsistency}(a)-(c) and (g), by the Lindeberg-L{\'e}vy CLT theorem, as $T \rightarrow \infty$, $ \frac{1}{\sqrt{T}}\sum_{t=1}^T\psi_t(\theta_0)\stackrel{D}{\rightarrow} \mathcal{N}(0,V)$ where $V:=\E[\psi(X_1,\theta_{0}) \psi(X_1,\theta_{0})' ]$. \textit{(b)} Firstly, the minus sign can be discarded because of the symmetry of the Gaussian distribution.
 Secondly,  if $X$ is a random vector and $F$ is a matrix, then $\V(FX)=F \V(X)F' $. \textit{(c)} Denote  the final asymptotic variance matrix  with $\Gamma$, and its $(i,j)$ block components with $\Gamma_{i,j}$. Then,
\begin{itemize}
\item $\Gamma_{1,1}= \Sigma^{1/2'} P_{\Sigma^{1/2}R'}^{\perp}\Sigma^{-1/2'}M^{-1}V(M')^{-1}\Sigma^{-1/2} P_{\Sigma^{1/2}R'}^{\perp}\Sigma^{1/2} \\= \Sigma^{1/2'} P_{\Sigma^{1/2}R'}^{\perp}\Sigma^{-1/2'}\Sigma\Sigma^{-1/2} P_{\Sigma^{1/2}R'}^{\perp}\Sigma^{1/2}=(\Sigma^{1/2})'P^{\perp}_{\Sigma^{1/2}R'}\Sigma^{1/2}$ because $M^{-1}V(M')^{-1}=:\Sigma=\Sigma^{1/2'}\Sigma^{1/2} $, $\Sigma^{-1/2}:=(\Sigma^{1/2})^{-1}$, $\Sigma^{-1/2'}:=(\Sigma^{1/2'})^{-1}$, and $P^{\perp}_{\Sigma^{1/2}R'}P^{\perp}_{\Sigma^{1/2}R'}=P^{\perp}_{\Sigma^{1/2}R'} $ by idempotence of projections on linear spaces;

\item  $\Gamma_{2,2}=(M')^{-1}\Sigma^{-1/2}P_{\Sigma^{1/2}R'}\Sigma^{-1/2'}M^{-1}V(M')^{-1}\Sigma^{-1/2}P_{\Sigma^{1/2}R'}\Sigma^{-1/2'}M^{-1}=\\(M')^{-1}\Sigma^{-1/2}P_{\Sigma^{1/2}R'}\Sigma^{-1/2'}\Sigma\Sigma^{-1/2}P_{\Sigma^{1/2}R'}\Sigma^{-1/2'}M^{-1} \\ =(M')^{-1}\Sigma^{-1/2}P_{\Sigma^{1/2}R'}P_{\Sigma^{1/2}R'}\Sigma^{-1/2'}M^{-1}=(M')^{-1}\Sigma^{-1/2}P_{\Sigma^{1/2}R'}\Sigma^{-1/2'}M^{-1} \\ = (M')^{-1}[V^{1/2}(M')^{-1}]^{-1}P_{\Sigma^{1/2}R'}[M^{-1}V^{1/2'}]^{-1}M^{-1}\\=(V^{1/2})^{-1}P_{\Sigma^{1/2}R'}(V^{1/2'})^{-1} $ because $M^{-1}V(M')^{-1}=:\Sigma=\Sigma^{1/2'}\Sigma^{1/2} $, $\Sigma^{-1/2}:=(\Sigma^{1/2})^{-1}=[V^{1/2}(M')^{-1}]^{-1}=M'V^{-1/2}$, $\Sigma^{-1/2'}:=(\Sigma^{1/2'})^{-1}=[M^{-1}V^{1/2'}]^{-1}=V^{-1/2'}M$, and $P^{}_{\Sigma^{1/2}R'}P^{}_{\Sigma^{1/2}R'}=P^{\perp}_{\Sigma^{1/2}R'} $ by idempotence;
\item $\Gamma_{3,3}=(R\Sigma R')^{-1}RM^{-1}V(M')^{-1}R'(R\Sigma R')^{-1}=(R\Sigma R')^{-1}R\Sigma R'(R\Sigma R')^{-1}=(R\Sigma R')^{-1}$ because $M^{-1}V(M')^{-1}=:\Sigma$;

\item $\Gamma_{1,2}=\Sigma^{1/2'} P_{\Sigma^{1/2}R'}^{\perp}\Sigma^{-1/2'}M^{-1}V(M')^{-1}\Sigma^{-1/2}P_{\Sigma^{1/2}R'}\Sigma^{-1/2'}M^{-1}=\\\Sigma^{1/2'} P_{\Sigma^{1/2}R'}^{\perp}\Sigma^{-1/2'}\Sigma \Sigma^{-1/2}P_{\Sigma^{1/2}R'}\Sigma^{-1/2'}M^{-1}=\Sigma^{1/2'} P_{\Sigma^{1/2}R'}^{\perp}P_{\Sigma^{1/2}R'}\Sigma^{-1/2'}M^{-1}=0_{}$ because $M^{-1}V(M')^{-1}=:\Sigma=\Sigma^{1/2'}\Sigma^{1/2} $, $\Sigma^{-1/2}:=(\Sigma^{1/2})^{-1}$, $\Sigma^{-1/2'}:=(\Sigma^{1/2'})^{-1}$, and $P^{\perp}_{\Sigma^{1/2}R'}P^{}_{\Sigma^{1/2}R'}=0_{m \times m} $;
\item $\Gamma_{1,3}=- \Sigma^{1/2'} P_{\Sigma^{1/2}R'}^{\perp}\Sigma^{-1/2'}M^{-1}V(M')^{-1}R'(R\Sigma R')^{-1} \\= - \Sigma^{1/2'} P_{\Sigma^{1/2}R'}^{\perp}\Sigma^{-1/2'}\Sigma R'(R\Sigma R')^{-1}=- \Sigma^{1/2'} P_{\Sigma^{1/2}R'}^{\perp}\Sigma^{1/2} R'(R\Sigma R')^{-1}=0$ because  $M^{-1}V(M')^{-1}=:\Sigma=\Sigma^{1/2'}\Sigma^{1/2} $,  $\Sigma^{-1/2'}:=(\Sigma^{1/2'})^{-1}$, and $P_{\Sigma^{1/2}R'}^{\perp}\Sigma^{1/2} R'=0_{m \times q}$;
\item $\Gamma_{2,3}=-(M')^{-1}\Sigma^{-1/2}P_{\Sigma^{1/2}R'}\Sigma^{-1/2'}M^{-1}V(M')^{-1}R'(R\Sigma R')^{-1}\\=-(M')^{-1}\Sigma^{-1/2}P_{\Sigma^{1/2}R'}\Sigma^{-1/2'}\Sigma R'(R\Sigma R')^{-1}\\=-(M')^{-1}\Sigma^{-1/2}[\Sigma^{1/2}R'(R\Sigma R')^{-1}R\Sigma^{1/2'}]\Sigma^{-1/2'}\Sigma R'(R\Sigma R')^{-1}\\=-(M')^{-1}[R'(R\Sigma R')^{-1}R]\Sigma R'(R\Sigma R')^{-1}=-(M')^{-1}R'(R\Sigma R')^{-1}$ 

\noindent
because $M^{-1}V(M')^{-1}=:\Sigma$ and $P_{\Sigma^{1/2}R'}=[\Sigma^{1/2}R'(R\Sigma R')^{-1}R\Sigma^{1/2'}]$.
\end{itemize}
\end{proof}

\begin{lem}\label{Lem:PartialLAndSAndR} Using the notation of Proposition \ref{Prop:ConstrainedEstAsNormality} (p. \pageref{Prop:ConstrainedEstAsNormality}), under Assumptions \ref{Assp:ExistenceConsistency},  \ref{Assp:AsymptoticNormality} and \ref{Assp:Trinity},  \begin{enumerate}
\item[(i)]for any sequence $(\theta_T, \tau_T)_{T \in \N}$ converging to $(\theta_0, \tau(\theta_0))$, $\P$-a.s. as $T \rightarrow \infty$, \\
$\begin{bmatrix}\frac{\partial^2 { L}_T(\bar{\theta}_T, \bar{\tau}_T) }{\partial \theta' \partial \theta} & \frac{\partial^2 { L}_T(\bar{\theta}_T, \bar{\tau}_T) }{\partial \tau' \partial \theta} & \frac{\partial r(\bar{\theta}_T)'}{\partial \theta} \\
\frac{\partial  S_T(\bar{\theta}_T, \bar{\tau}_T) }{\partial \theta'} & \frac{ \partial S_T(\bar{\theta}_T, \bar{\tau}_T) }{\partial \tau'} & 0 \\
\frac{\partial r(\bar{\theta}_T) }{\partial \theta'} & 0 & 0 \\
\end{bmatrix}
\rightarrow \begin{bmatrix}0_{m \times m} & \E\left[  \frac{\partial \psi(X_1,\theta_{0})}{\partial \theta' }\right]' & \frac{\partial r(\theta_0)'}{\partial \theta} \\
\E\left[  \frac{\partial \psi(X_1,\theta_{0})}{\partial \theta' }\right] & \E\left[\psi(X_1,\theta_{0}) \psi(X_1,\theta_{0})' \right] & 0 \\
\frac{\partial r(\theta_0) }{\partial \theta'} & 0 & 0 \\
\end{bmatrix}
$
;
\item[(ii)] $\begin{bmatrix}0_{m \times m} & \E\left[  \frac{\partial \psi(X_1,\theta_{0})}{\partial \theta' }\right]' & \frac{\partial r(\theta_0)'}{\partial \theta} \\
\E\left[  \frac{\partial \psi(X_1,\theta_{0})}{\partial \theta' }\right] & \E\left[\psi(X_1,\theta_{0}) \psi(X_1,\theta_{0})' \right] & 0 \\
\frac{\partial r(\theta_0) }{\partial \theta'} & 0 & 0 \\
\end{bmatrix}$ is invertible,  so that, for any sequence $(\theta_T, \tau_T)_{T \in \N}$ converging to $(\theta_0, \tau(\theta_0))$, $\P$-a.s., for $T$ big enough, the matrix 

\noindent
 $\begin{bmatrix}\frac{\partial^2 { L}_T(\bar{\theta}_T, \bar{\tau}_T) }{\partial \theta' \partial \theta} & \frac{\partial^2 { L}_T(\bar{\theta}_T, \bar{\tau}_T) }{\partial \tau' \partial \theta} & \frac{\partial r(\bar{\theta}_T)'}{\partial \theta} \\
\frac{\partial  S_T(\bar{\theta}_T, \bar{\tau}_T) }{\partial \theta'} & \frac{ \partial S_T(\bar{\theta}_T, \bar{\tau}_T) }{\partial \tau'} & 0 \\
\frac{\partial r(\bar{\theta}_T) }{\partial \theta'} & 0 & 0 \\
\end{bmatrix}$ is invertible; and
\item[(iii)] for any sequence $(\theta_T, \tau_T)_{T \in \N}$ converging to $(\theta_0, \tau(\theta_0))$, $\P$-a.s. as $T \rightarrow \infty$,\\  $\begin{bmatrix}\frac{\partial^2 { L}_T(\bar{\theta}_T, \bar{\tau}_T) }{\partial \theta' \partial \theta} & \frac{\partial^2 { L}_T(\bar{\theta}_T, \bar{\tau}_T) }{\partial \tau' \partial \theta} & \frac{\partial r(\bar{\theta}_T)'}{\partial \theta} \\
\frac{\partial  S_T(\bar{\theta}_T, \bar{\tau}_T) }{\partial \theta'} & \frac{ \partial S_T(\bar{\theta}_T, \bar{\tau}_T) }{\partial \tau'} & 0 \\
\frac{\partial r(\bar{\theta}_T) }{\partial \theta'} & 0 & 0 \\
\end{bmatrix}^{-1}\\
\rightarrow \begin{bmatrix}- \Sigma+\Sigma R'(R\Sigma R')^{-1}R\Sigma & \quad M^{-1}\negthickspace-\negthickspace\Sigma R'(R\Sigma R')^{-1}RM^{-1} & \Sigma R'(R\Sigma R')^{-1} \\
(M')^{-1}\negthickspace-\negthickspace(M')^{-1}R'(R\Sigma R')^{-1}R\Sigma & (M')^{-1}R'(R\Sigma R')^{-1}RM^{-1} & \; -(M')^{-1}R'(R\Sigma R')^{-1} \\
(R\Sigma R')^{-1} R\Sigma& -(R\Sigma R')^{-1}RM^{-1} & (R\Sigma R')^{-1} \\
\end{bmatrix}$.

\end{enumerate}
\end{lem}
\begin{proof}\textit{(i)} Under Assumptions \ref{Assp:ExistenceConsistency},  \ref{Assp:AsymptoticNormality} and \ref{Assp:Trinity}(a), it follows from  the continuity of $ \frac{\partial r(.)}{\partial \theta'}$, which is implied by Assumption \ref{Assp:Trinity}(a), and Lemma \ref{Lem:PartialLTheta0Tau0}ii and iii (p. \pageref{Lem:PartialLTheta0Tau0}) and Lemma \ref{Lem:ETEquationFirstDerivatives} (p. \pageref{Lem:ETEquationFirstDerivatives}), given that $\tau(\theta_0)=0_{m \times 1}$ by  Lemma \ref{Lem:AsTiltingFct}ii (p. \pageref{Lem:AsTiltingFct}) and Assumption \ref{Assp:ExistenceConsistency}(c), under  Assumption \ref{Assp:ExistenceConsistency}(a)(b)(d)(e)(g) and (h).

\textit{(ii)} It is sufficient to check the assumptions of Corollary \ref{Cor:For2x2BlockMatrix}i (p. \pageref{Cor:For2x2BlockMatrix}) with $A=\left[ \begin{array}{c c } 0_{m \times m}
                         & M' \\
                         M   &\ V \end{array} \right] $ and $B= \begin{bmatrix}R' \\
0_{m \times q} \\
\end{bmatrix} $  in order to establish the first part of the statement. Firstly, under Assumptions \ref{Assp:ExistenceConsistency} and \ref{Assp:AsymptoticNormality}, by Lemma \ref{Lem:PartialLAndS}iii (p. \pageref{Lem:PartialLAndS}), $A=\left[ \begin{array}{c c } 0_{m \times m}
                         & M' \\
                         M   &\ V \end{array} \right]$ is invertible. Secondly, by  Assumptions \ref{Assp:ExistenceConsistency}(h) and \ref{Assp:Trinity}(b), $(B'A^{-1}B)=-(R \Sigma R') $  is also invertible.   Then, the second
part of the statement follows from a trivial case of the Lemma \ref{Lem:UniFiniteSampleInvertibilityFromUniAsInvertibility} (p. \pageref{Lem:UniFiniteSampleInvertibilityFromUniAsInvertibility}).\\
\textit{(iii)} Under Assumption \ref{Assp:ExistenceConsistency}(a)(b)(c)(d)(e)(g)(h), by the statement (ii) of the present lemma, the limiting matrix is invertible. Thus, using the notation of Proposition \ref{Prop:ConstrainedEstAsNormality} (p. \pageref{Prop:ConstrainedEstAsNormality}),\begin{eqnarray*}
& &\ \begin{bmatrix}0_{m \times m} & \E\left[  \frac{\partial \psi(X_1,\theta_{0})}{\partial \theta' }\right]' & \frac{\partial r(\theta_0)'}{\partial \theta} \\
\E\left[  \frac{\partial \psi(X_1,\theta_{0})}{\partial \theta' }\right] & \E\left[\psi(X_1,\theta_{0}) \psi(X_1,\theta_{0})' \right] & 0 \\
\frac{\partial r(\theta_0) }{\partial \theta'} & 0 & 0 \\
\end{bmatrix}^{-1}\\
& = &
\begin{bmatrix}0_{m \times m} & M' & R' \\
M & V & 0_{m \times q} \\
R & 0_{q \times m} & 0_{q \times q} \\
\end{bmatrix}^{-1}\\
& \stackrel{}{=} & \begin{bmatrix}- \Sigma+\Sigma R'(R\Sigma R')^{-1}R\Sigma & \quad M^{-1}-\Sigma R'(R\Sigma R')^{-1}RM^{-1} & \Sigma R'(R\Sigma R')^{-1} \\
(M')^{-1}-(M')^{-1}R'(R\Sigma R')^{-1}R\Sigma & (M')^{-1}R'(R\Sigma R')^{-1}RM^{-1} & \quad -(M')^{-1}R'(R\Sigma R')^{-1} \\
(R\Sigma R')^{-1} R\Sigma& -(R\Sigma R')^{-1}RM^{-1} & (R\Sigma R')^{-1} \\
\end{bmatrix}\\
\end{eqnarray*}
where the explanation for the last equality is as follows. Apply Corollary \ref{Cor:For2x2BlockMatrix}ii (p. \pageref{Cor:For2x2BlockMatrix}) with $A= \begin{bmatrix}0_{m \times m} & M'  \\
M & V
\end{bmatrix}$ and $B= \begin{bmatrix}R' \\
0_{m \times q} \\
\end{bmatrix} $, and note that, by Lemma \ref{Lem:MatrixInversionCalculus}iii, iv and vi (p. \pageref{Lem:MatrixInversionCalculus}),
\begin{eqnarray*}
 A^{-1}-A^{-1}B(B'A^{-1}B)B'A^{-1} & = &
 \begin{bmatrix}- \Sigma+\Sigma R'(R\Sigma R')^{-1}R\Sigma & \quad M^{-1}-\Sigma R'(R\Sigma R')^{-1}RM^{-1} \\
(M')^{-1}-(M')^{-1}R'(R\Sigma R')^{-1}R\Sigma &\quad  (M')^{-1}R'(R\Sigma R')^{-1}RM^{-1} \\
\end{bmatrix} \\
 A^{-1} B(B'A^{-1}B)^{-1} & = & \begin{bmatrix}\Sigma R'(R\Sigma R')^{-1} \\
-(M')^{-1}R'(R\Sigma R')^{-1} \\
\end{bmatrix}\\
(B'A^{-1}B)^{-1}& =& -(R\Sigma R')^{-1}.
\end{eqnarray*}

Then, the result follows from the continuity of the inverse transformation  \citep[e.g.,][Theorem 9.8]{1953Rudin}.
\end{proof}

\begin{lem}\label{Lem:MatrixInversionCalculus} Let $A  =  \begin{bmatrix}0_{m \times m} & M'  \\
M & V
\end{bmatrix} $ and $B=\begin{bmatrix}R' \\
0_{m \times q} \\
\end{bmatrix} $ where
$\Sigma:=\Sigma(\theta_0):= M^{-1}V(M')^{-1}$, $M:=\E \left[\frac{\partial \psi(X_1, \theta_0)}{\partial \theta'}\right]$, $V:= \E[ \psi(X_1, \theta_0)\psi(X_1, \theta_0)']$, and $R=\frac{\partial r(\theta_0)}{\partial \theta'}$. Then, under Assumption \ref{Assp:ExistenceConsistency}(a)(b)(h) and \ref{Assp:Trinity}(b), the following equalities hold
\begin{enumerate}
\item[(i)] $\displaystyle A^{-1}=\begin{bmatrix}- \Sigma & M^{-1} \\
(M')^{-1} & 0_{m \times m} \\
\end{bmatrix}$;
\item[(ii)] $\displaystyle A^{-1}B = \begin{bmatrix}-\Sigma R' \\
(M')^{-1}R' \\
\end{bmatrix}$, so that $B'A^{-1}=\begin{bmatrix}-R\Sigma  & \quad RM^{-1} \\
\end{bmatrix}$;

\item[(iii)] $\displaystyle (B'A^{-1}B)^{-1}=-(R\Sigma R')^{-1}$;
\item[(iv)] $\displaystyle A^{-1} B(B'A^{-1}B)^{-1}= \begin{bmatrix}\Sigma R'(R\Sigma R')^{-1} \\
-(M')^{-1}R'(R\Sigma R')^{-1} \\
\end{bmatrix}$;

\item[(v)] $\displaystyle A^{-1}B(B'A^{-1}B)B'A^{-1}= \begin{bmatrix}-\Sigma R'(R\Sigma R')^{-1}R\Sigma & \qquad\Sigma R'(R\Sigma R')^{-1}RM^{-1} \\
(M')^{-1}R'(R\Sigma R')^{-1}R\Sigma &\quad  -(M')^{-1}R'(R\Sigma R')^{-1}RM^{-1} \\
\end{bmatrix}$; and
\item[(vi)] $\displaystyle A^{-1}-A^{-1}B(B'A^{-1}B)B'A^{-1} \\=  \begin{bmatrix}- \Sigma+\Sigma R'(R\Sigma R')^{-1}R\Sigma & \quad M^{-1}-\Sigma R'(R\Sigma R')^{-1}RM^{-1} \\
(M')^{-1}-(M')^{-1}R'(R\Sigma R')^{-1}R\Sigma &\quad  (M')^{-1}R'(R\Sigma R')^{-1}RM^{-1} \\
\end{bmatrix}$.

\end{enumerate}
\end{lem}
\begin{proof} 
\textit{(i)} It corresponds to a part of Lemma \ref{Lem:PartialLAndS}iii (p. \pageref{Lem:PartialLAndS}) under Assumptions \ref{Assp:ExistenceConsistency} and \ref{Assp:AsymptoticNormality}.

\textit{(ii)}
\begin{eqnarray*}
\begin{bmatrix}- \Sigma & M^{-1} \\
(M')^{-1} & \qquad0_{m \times m} \\
\end{bmatrix} 
 \begin{bmatrix}R' \\
0_{m \times q} \\
\end{bmatrix}
= \begin{bmatrix}-\Sigma R' \\
(M')^{-1}R' \\
\end{bmatrix}
&=A^{-1}B \\
\end{eqnarray*}

\textit{(iii)}
\begin{eqnarray*}
\begin{bmatrix}R & 0_{q \times m} \\
\end{bmatrix}
 \begin{bmatrix}-\Sigma R' \\
(M')^{-1}R' \\
\end{bmatrix}
=
\begin{bmatrix} -R\Sigma R'
\end{bmatrix} & =B'A^{-1}B \\
\end{eqnarray*}

\textit{(iv)}
\begin{eqnarray*}
\begin{bmatrix}-\Sigma R' \\
(M')^{-1}R' \\
\end{bmatrix} 
 \begin{bmatrix} -(R\Sigma R')^{-1}
\end{bmatrix} 
=
\begin{bmatrix}\Sigma R'(R\Sigma R')^{-1} \\
-(M')^{-1}R'(R\Sigma R')^{-1} \\
\end{bmatrix}  & =A^{-1} B(B'A^{-1}B)^{-1}\\
\end{eqnarray*}

\textit{(v)}
\begin{eqnarray*}
& &\begin{bmatrix}\Sigma R'(R\Sigma R')^{-1} \\
-(M')^{-1}R'(R\Sigma R')^{-1} \\
\end{bmatrix}  
 \begin{bmatrix}-R\Sigma & \qquad RM^{-1} \\
\end{bmatrix} 
\\
& = &
 \begin{bmatrix}-\Sigma R'(R\Sigma R')^{-1}R\Sigma & \qquad\Sigma R'(R\Sigma R')^{-1}RM^{-1} \\
(M')^{-1}R'(R\Sigma R')^{-1}R\Sigma &\quad  -(M')^{-1}R'(R\Sigma R')^{-1}RM^{-1} \end{bmatrix}  =A^{-1}B(B'A^{-1}B)B'A^{-1} 
\end{eqnarray*}

\end{proof}
\begin{lem}[Constrained estimator and its Lagrangian]\label{Lem:ConstrainedEstLagrangian} Under Assumptions \ref{Assp:ExistenceConsistency}, \ref{Assp:AsymptoticNormality}
and \ref{Assp:Trinity}(a), if the test hypothesis \eqref{Eq:HypParameterRestriction} on p. \pageref{Eq:HypParameterRestriction} holds, $\P$-a.s.  for $T$ big enough, \begin{enumerate}
\item[(i)]  the constrained estimator $\check{\theta}_T $ exists, and $\check{\theta}_T \rightarrow \theta_0$, as $T \rightarrow \infty$;

\item[(ii)]  $\theta \mapsto L_T(\theta, \tau_T(\theta))$ is continuously differentiable in a neighborhood of  $\check{\theta}_T$;

\item[(iii)] under additional Assumption \ref{Assp:Trinity}(b), there exists a unique vector, $\check{\gamma}_T$, called the Lagrangian multiplier, s.t.  $\left. \frac{\partial L_T(\theta, \tau_T(\theta))}{\partial \theta} \right \vert_{\theta= \check{\theta}_T}+ \frac{\partial r(\check{\theta}_T)'}{\partial \theta}\check{\gamma}_T=0_{m \times 1}$; 
\item[(iv)] under additional Assumption \ref{Assp:Trinity}(b),   $ \frac{\partial L_T(\check{\theta}_T, \tau_T(\check{\theta}_T))}{\partial \theta} + \frac{\partial r(\check{\theta}_T)'}{\partial \theta}\check{\gamma}_T= O(T^{-1})$, as $T\rightarrow \infty$, where $\frac{\partial L_T(\check{\theta}_T, \tau_T(\check{\theta}_T))}{\partial \theta}:= \left. \frac{\partial L_T(\theta, \tau)}{\partial \theta} \right \vert_{(\theta, \tau)= (\check{\theta}_T, \tau_T(\check{\theta}_T))}$.
\end{enumerate}
\end{lem}
\begin{proof}\textit{(i)}  The constrained set $\tilde{\T}:=\{\theta \in \T:r(\theta)=0 \}$ is bounded as a subset of the compact (and thus bounded) set $\T$. The constrained  set  $\tilde{\T}$ is also closed: For all $(\theta_n)_{n \in \N} \in \tilde{\T}^\N$ s.t. $\lim_{n \rightarrow }\theta_n = \bar{\theta}$, $\bar{\theta}\in \tilde{\T} $ because (i) by compactness of $\T$, $\bar{\theta}\in \T $;  and (ii) by the continuity of $r: \T \rightarrow \R^q $ (i.e., Assumption \ref{Assp:Trinity}(a)),  $r(\bar{\theta})=\lim_{n \rightarrow \infty }r(\theta_n)=\lim_{n \rightarrow \infty }0=0$. Therefore, the constrained set $\tilde{\T} $ is itself compact. Moreover, under Assumption \ref{Assp:ExistenceConsistency}(a)(b) and (d)-(h), by Lemma \ref{Lem:ESPExistence}ii-iii (p. \pageref{Lem:ESPExistence}), $\P$-a.s. for $T$ big enough, $\theta \mapsto \hat{f}_{\theta^{*}_T}(\theta)$ is continuous and, for all $\theta \in \T$, $ \omega \mapsto \hat{f}_{\theta^{*}_T}(\theta)$ is measurable. Thus, the existence and the measurability of the constrained estimator $\check{\theta}_T$  follows from the Schmetterer-Jennrich lemma (\citealt{1966Sch} Chap. 5 Lemma 3.3;   \citealt{1969Jen} Lemma 2).

In order to establish the consistency of $ \check{\theta}_T$, it remains to check  the other assumptions of the standard consistency theorem \citep[e.g.][pp. 2121-2122 Theorem 2.1, which is also valid in an almost-sure sense]{1994NewMcF}, where the constrained set $\tilde{\T}:=\{\theta \in \T:r(\theta)=0 \}$ is the parameter space. Because $\tilde{\T}\subset \T$,  $\P$-a.s. as $T \rightarrow \infty$,
 \begin{eqnarray*}
& & \sup_{\theta \in \tilde{\T}}\left\vert\ln  \left[\frac{1}{T}\sum_{t=1}^T \e^{\tau_T(\theta)'\psi_t(\theta)}\right]- \frac{1}{2T} \ln \vert \Sigma_T(\theta)\vert_{\det}-\ln \E[ \e^{\tau(\theta)'\psi(X_1, \theta)}]\right\vert \\
&\leqslant & \sup_{\theta \in \T}\left\vert\ln  \left[\frac{1}{T}\sum_{t=1}^T \e^{\tau_T(\theta)'\psi_t(\theta)}\right]- \frac{1}{2T} \ln \vert \Sigma_T(\theta)\vert_{\det}-\ln \E[ \e^{\tau(\theta)'\psi(X_1, \theta)}]\right\vert\rightarrow 0  \nonumber
\end{eqnarray*}
where the convergence to zero follows from equation \eqref{Eq:UniformCVOfESPObjFct} on p. \pageref{Eq:UniformCVOfESPObjFct},  under Assumption \ref{Assp:ExistenceConsistency}. In addition,  under Assumption \ref{Assp:ExistenceConsistency} (a)-(e) and (g)-(h), by Lemma \ref{Lem:AsTiltingFct}iv (p. \pageref{Lem:AsTiltingFct}), $\theta \mapsto \ln \E[ \e^{\tau(\theta)'\psi(X_1, \theta)}] $ is uniquely maximized at $\theta_0$, i.e.,  for all $\theta \in \T \setminus\{ \theta_0\} $, $\ln \E[ \e^{\tau(\theta)'\psi(X_1, \theta)}]<\ln \E[ \e^{\tau(\theta_0)'\psi(X_1, \theta_{0})}]=0$, and, under Assumptions \ref{Assp:ExistenceConsistency} (a)(b)(d)(e)(g) and (h), by Lemma \ref{Lem:ExpTauPsiStrictlyPositive} (p. \pageref{Lem:ExpTauPsiStrictlyPositive}),    $\theta \mapsto \ln \E[ \e^{\tau(\theta)'\psi(X_1, \theta)}] $  is continuous in $\tilde{\T}\subset \T$.

\textit{(ii)} Under Assumptions \ref{Assp:ExistenceConsistency} and \ref{Assp:AsymptoticNormality}(a), by subsection \ref{Sec:LTAndDerivatives} (p. \pageref{Sec:LTAndDerivatives}), the function  $L_T(\theta, \tau)$ is well-defined and twice continuously differentiable in a neighborhood of $(\theta_0'\; \tau(\theta_0)')$   $\P$-a.s. for $T$ big enough. Moreover, under Assumption \ref{Assp:ExistenceConsistency}(a)(b) and (d)-(h), by Lemma \ref{Lem:DerivativeImplicitFunctionImplicit}i (p. \pageref{Lem:DerivativeImplicitFunctionImplicit}),  $\tau_T(.)$ is continuously differentiable in $\T$. Now,  under Assumption \ref{Assp:ExistenceConsistency}, by  the statement (i) of the present lemma and  Lemma \ref{Lem:Schennachtheorem10PfFirstSteps}iii (p. \pageref{Lem:Schennachtheorem10PfFirstSteps}), $\P$-a.s., $\check{\theta}_T \rightarrow \theta_0$ and $\tau_T(\check{\theta}_T )\rightarrow \tau(\theta_0)$, so that $\P$-a.s. for $T$ big enough, $(\check{\theta}_T ' \; \tau_T(\check{\theta}_T)')$ is in any arbitrary small neighborhood of $(\theta_0'\; \tau(\theta_0)')$. Therefore,  under Assumption \ref{Assp:ExistenceConsistency} and \ref{Assp:AsymptoticNormality}(a),  by the chain rule theorem \citep[e.g.,][Chap. 5 sec. 11]{1988MagnusNeudecker}, $\P$-a.s. for $T$ big enough, $\theta \mapsto L_T(\theta, \tau_T(\theta))$ is continuously differentiable  at $\check{\theta}_T$.

\textit{(iii)} It is a consequence of the  Lagrange theorem \citep[e.g.,][Chap. 7 sec. 12]{1988MagnusNeudecker}. Check its assumptions. Firstly, under Assumptions \ref{Assp:ExistenceConsistency} and \ref{Assp:AsymptoticNormality},  $\P$-a.s. by the statement (i) of the present lemma, $\P$-a.s. for $T$ big enough, the constrained estimator $\check{\theta}_T$ exists and that it is in the interior of $\T$ by consistency and Assumption \ref{Assp:ExistenceConsistency}(c). Then, we should check the other assumptions of the Lagrange theorem $\omega$ by $\omega$ on the subset of $\Omegabf$ where $\check{\theta}_T$ exists.
Firstly, by Assumption \ref{Assp:Trinity}(a), $r:\T \rightarrow \R^q$ is continuously differentiable. Secondly, under Assumptions \ref{Assp:ExistenceConsistency},\ref{Assp:AsymptoticNormality} and \ref{Assp:Trinity}(a), if the test hypothesis \eqref{Eq:HypParameterRestriction} on p. \pageref{Eq:HypParameterRestriction} holds, $\P$-a.s.  as $T \rightarrow \infty$,  $\check{\theta}_T \rightarrow \theta_0$, and,   by Assumption \ref{Assp:Trinity}(b), $\frac{\partial r(\theta_0)}{\partial \theta'}$ is full  rank, Thus, $\P$-a.s. for $T$ big enough, $\frac{\partial r(\check{\theta}_T)}{\partial \theta'}$ is full  rank
by continuity of the determinant function. Finally, by the statement (iv) of the present lemma  $\theta \mapsto L_T(\theta, \tau_T(\theta))$ is  differentiable  at $\check{\theta}_T$.

\textit{(iv)}  First of all, note that it does \textit{not} immediately follow from the statement (iii)  because   $\frac{\partial L_T(\hat{\theta}_T,\tau_T(\hat{\theta}_T))}{\partial \theta }$  denotes $  \left. \frac{\partial L_T(\theta, \tau)}{\partial \theta} \right \vert_{(\theta, \tau)= (\check{\theta}_T, \tau_T(\check{\theta}_T) )}$instead of $\left.\frac{\partial L_T(\theta, \tau_T(\theta))}{\partial \theta }\right\vert_{\theta=\hat{\theta}_T}$ (see footnote \ref{FootN:AmbiguousNotation} on p. \pageref{FootN:AmbiguousNotation}). Under Assumption \ref{Assp:ExistenceConsistency}(a)(b) and (d)-(h), by Lemma \ref{Lem:DerivativeImplicitFunctionImplicit}i (p. \pageref{Lem:DerivativeImplicitFunctionImplicit}),  $\tau_T(.)$ is continuously differentiable in $\T$. Moreover,  under Assumptions \ref{Assp:ExistenceConsistency}, \ref{Assp:AsymptoticNormality} and \ref{Assp:Trinity}(a), if  by the statement (ii) of the present lemma, $\P$-a.s. for $T$ big enough, $\theta \mapsto L_T(\theta, \tau_T(\theta))$ is continuously differentiable in a neighborhood of $\check{\theta}_T$.
Thus,
 by  an immediate and standard implication  of the chain rule \citep[e.g.,][chap. 5, sec. 12, exercise 3]{1988MagnusNeudecker}, $\P$-a.s. for $T$ big enough,  for all $j \in \ldsb 1,m\rdsb$,
\begin{eqnarray}
 \left.\frac{\partial L_T(\theta, \tau_{T}(\theta))}{\partial \theta_{j} }\right\vert_{\theta=\check{\theta}_T} & = & \left.\frac{\partial L_T(\theta, \tau)}{\partial \theta_{j} }\right\vert_{(\theta, \tau)=(\check{\theta}_T, \tau_T(\check{\theta}_T))}+ \left.\frac{\partial L_T(\theta, \tau)}{\partial \tau' }\right\vert_{(\theta, \tau)=(\check{\theta}_T, \tau_T(\check{\theta}_T))}\left.\frac{\partial \tau(\theta)}{\partial \theta_{j}}\right\vert_{\theta=\check{\theta}_T} \nonumber\\
& \stackrel{}{=} & \left.\frac{\partial L_T(\theta, \tau)}{\partial \theta_{j} }\right\vert_{(\theta, \tau)=(\check{\theta}_T, \tau_T(\check{\theta}_T))}+O(T^{-1})\label{Eq:DLDThetaTildeOhP}
\end{eqnarray}
where the explanations for the last equality are as follow.
Firstly, under Assumptions \ref{Assp:ExistenceConsistency}, \ref{Assp:AsymptoticNormality} and \ref{Assp:Trinity}(a),  by  Lemma \ref{Lem:DLDTau}iv  (p. \pageref{Lem:DLDTau}), $\P$-a.s. as $T \rightarrow \infty$, $\left.\frac{\partial L_T(\theta, \tau)'}{\partial \tau }\right\vert_{(\theta, \tau)=(\check{\theta}_T, \tau_T(\check{\theta}_T))}=O(T^{-1})$ because $\check{\theta}_T\rightarrow \theta_0$, $\P$-a.s. as $T\rightarrow \infty$, by the second part of the statement (i) of the present lemma. Secondly, under Assumptions \ref{Assp:ExistenceConsistency}, \ref{Assp:AsymptoticNormality} and \ref{Assp:Trinity}(a), by the second part  of the statement (i) of the present lemma  and Lemma  \ref{Lem:DerivativeImplicitFunctionImplicit}iii (p. \pageref{Lem:DerivativeImplicitFunctionImplicit}), $\P$-a.s. as $T \rightarrow \infty$,$\left.\frac{\partial \tau(\theta)}{\partial \theta_{j}}\right\vert_{\theta=\check{\theta}_T}=O(1)$.

Now the results follows by plugging the above equation \eqref{Eq:DLDThetaTildeOhP} into the Lagrangian FOC of the statement (iii) of the present lemma.
\end{proof}

\section{On the assumptions}

\subsection{Discussion }\label{Sec:DiscussionAsspSchennach}
 Assumptions \ref{Assp:ExistenceConsistency} and \ref{Assp:AsymptoticNormality}  are mainly adapted from the entropy literature.   Assumption \ref{Assp:ExistenceConsistency}(a)  ensures the basic requirement for inference, that  is, data contain  different pieces of information (independence) about   the same phenomenon (identically distributed). The conditions ``independence and identically distributed"  are much stronger than needed, and can be relaxed  to allow for time dependence along the  lines of \cite{1997KitStu}.
We restrain ourself to the i.i.d. case for brevity and clarity.
Assumption \ref{Assp:ExistenceConsistency}(a) also requires completeness of the probability space so that we can define functions  only a probability-one subset of $\Omegabf$ without generating potential measurability complications. The completeness of the probability space is without significant loss of generality \citep[e.g., ][p. 13]{2002Kal}, and it is often implicitly or explicitly required in the literature.

 Assumption \ref{Assp:ExistenceConsistency}(b) mainly requires  standard regularity conditions for the moment function $\psi(.,.)$. As usual in nonlinear econometrics, the existence of the estimator relies on such regularity conditions.  An alternative would be to rely on empirical process theory, but it seems here inappropriate as the  implicit nature of the definition of the ESP approximation requires smooth functions.
We require Assumption \ref{Assp:ExistenceConsistency}(b), as well as some of the  following assumptions,  to hold in  an $\epsilon$-neighborhood of the parameter space $\T$, so that we can deal with its boundary  $\partial \T$   in the same way as with its interior.  In particular, it ensures that $\Sigma(\theta)$ is invertible for $\theta\in \partial\T$ under probability measures equivalent to $\P$ (Corollary \ref{Cor:ChgOfMeasureInvertibilityPDP}ii on p. \pageref{Cor:ChgOfMeasureInvertibilityPDP}), and it  allows to apply an implicit function theorem to $\tau(\theta) $, also for $\theta \in \partial \T$ (Lemma \ref{Lem:AsTiltingFct} on p. \pageref{Lem:AsTiltingFct}). For the latter reason, the entropy literature  often appears to also (implicitly) assume that  assumptions hold in an $\epsilon$-neighborhood of the parameter space.
In applications, this is often innocuous as the boundary of the parameter space is often loosely specified. However, in some specific situations, which we rule out, this may be problematic \cite[e.g.,][and references therein]{1999Andrews}.

 Assumption \ref{Assp:ExistenceConsistency}(c)
 requires global identification, which is a necessary condition to prove the consistency of an estimator. If we were interested in the ESP approximation instead of its maximizer (i.e., the ESP estimator), global identification could be relaxed as   \cite{2012Hol} and a companion paper show. Assumption \ref{Assp:ExistenceConsistency}(c)
also requires equality between the dimension of the parameter space and the number of moment conditions, i.e., just-identified moment conditions. We impose the latter for mainly three reasons. Firstly, it appears reasonable to investigate the ESP estimator in the  just-identified case before moving to the over-identified case, which requires to generalize the ESP approximation. Secondly, the just-identified case makes clear the difference between the ESP estimator and the existing alternatives, which are all equal in this case (see section \ref{Sec:ESPvsMMandOthers}). Thirdly,  this is a standard assumption  in the saddlepoint literature. However, note that (i) this assumption is less restrictive than it seems at first sight because, in the linear case, over-identified  moment conditions correspond to just-identified moment conditions through the FOCs, and, in the nonlinear case, we can transform over-identified estimating equations into just-identified estimating equations through an extension of
the parameter space \citep[e.g.,][p. 2232]{1994NewMcF}; (ii) ongoing work show how to generalize the ESP approximation to over-identified moment conditions.
 
 Assumption \ref{Assp:ExistenceConsistency}(d)  requires the compactness of the parameter space $\T$, and the existence of a solution     $\tau(\theta)\in \R^m$    that solves the equation    $\E\left[ \e^{\tau'\psi(X_1,\theta)}\psi(X_1,\theta)\right]=0 $,  for all $\theta \in \T$. \cite{2005Sch} also makes this assumption.  Compactness of the parameter space is a convenient standard mathematical assumption that is
often relevant in practice. A computer can only handle a bounded parameter space ---finite memory of a computer. Regarding the existence of $\tau(\theta)$, it is necessary to ensure the asymptotic existence of the ESP approximation. From a theoretical point of view, the existence of $\tau(\theta)$   looks like a reasonable assumption\,: If, for some $\theta \in \T$, $0_{m \times 1}$ is outside the convex hull of the support of $\psi(X_1, \theta)$, there is not such a solution $\tau(\theta)$, which  also means that $\theta$ cannot be $\theta_0$, so that it should be excluded from the parameter space.  However, the existence of $\tau(\theta)$   might be difficult to check  in practice.  A way to get around this assumption is to (i) assume the existence of $\tau(\theta)$ only in a neighbohood of $\theta_0$; and (ii) to  set the ESP approximation  to zero for the $\theta$ values that do not have a solution to the finite-sample moment conditions \eqref{Eq:ESPTiltingEquation}. \cite{2012Hol} follows such an approach. We do not follow such an approach because it  significantly complicates the proofs and the presentation. 

Assumptions \ref{Assp:ExistenceConsistency}(e) and \ref{Assp:AsymptoticNormality}(b) rule out fat-tailed distributions. More precisely, they require the existence of exponential moments. They are necessary to apply the the  ULLN (uniform law of large numbers) \`a la  Wald  \citep[e.g.,][pp. 24-25, Theorem 1.3.3]{2003GhoRam} to components of the ESP approximation. Assumptions \ref{Assp:ExistenceConsistency}(e) and \ref{Assp:AsymptoticNormality}(b) are stronger than the
moment existence assumption in \cite{1982Han}, but they are a common type of assumptions   in the
entropy literature \citep[e.g.,][]{1984Haberman,1997KitStu, 2007Schennach},  the saddlepoint literature \citep[e.g.,][]{2000AlmFieRob} and the literature on exponential models \citep[e.g.,][]{1972Ber}. In particular, Assumptions \ref{Assp:ExistenceConsistency}(d) and \ref{Assp:AsymptoticNormality}(b) are a convenient   variant of Assumptions 3.4 and 3.5   in \cite{2007Schennach}. Both in \cite{2007Schennach} and in the present paper, the successful estimation of the Hall and Horowitz model, which does not satisfy Assumptions \ref{Assp:ExistenceConsistency}(e) and \ref{Assp:AsymptoticNormality}(b), suggests that the latter can be relaxed.  In practice, Assumptions \ref{Assp:ExistenceConsistency}(e) and \ref{Assp:AsymptoticNormality}(b) are  not as strong as it may appear because   observable quantities  have finite support (finite memory of computers), which, in turn, implies that they have all finite moments. Moreover, in the case in which unboundedness  is a concern (e.g., moment conditions derived from a likelihood), \cite{2001RonchettiTrojani}  provide   a way to bound  moment functions.

 Assumptions \ref{Assp:ExistenceConsistency}(f) and (g) play the same role as Assumptions \ref{Assp:ExistenceConsistency}(e) and \ref{Assp:AsymptoticNormality}(b), although they are less stringent. Assumption \ref{Assp:ExistenceConsistency}(h)  requires the invertibility of the asymptotic variance
of standard estimators (scaled by $\sqrt{T}$)
of any  solution to the tilted moment condition. In the present paper, this assumption has two main roles. Firstly, it ensures that  the determinant term $\left|\Sigma_T(\theta) \right|_{\det}^{-\frac{1}{2}} $ in the ESP approximation \eqref{Eq:ESPApproximationDefn} does not explode, asymptotically.
Secondly, it  ensures the positive definiteness of the symmetric matrix $\E\left[\e^{\tau' \psi(X_1,\theta)}\psi(X_1,\theta) \psi(X_1,\theta)' \right]$  for all $(\theta, \tau) \in \Sbf $, so that the $\min_{\tau \in \R^m} \E[\e^{\tau'\psi(X_1,\theta)}]$ is a strictly convex problem, which, in turn, implies the unicity of its solution $\tau(\theta)$.    In the setup of the present paper, Assumption  \ref{Assp:ExistenceConsistency}(g) is equivalent to the invertibility of $\E\left[\e^{\tau(\theta)' \psi(X_{1},\theta)}  \frac{\partial \psi(X_{1},\theta)}{\partial \theta'}\right]$  and $\E\left[\psi(X_1,\theta) \psi(X_1,\theta)' \right]$,  for all $\theta \in \T $ (Lemma \ref{Lem:ChgOfMeasureInvertibilityPDP} on p. \pageref{Lem:ChgOfMeasureInvertibilityPDP} with $\Prm =\P$ and $\frac{\d \Qrm}{\d\Prm}=\frac{1}{\e^{\tau(\theta)'\psi(X_1, \theta)}}$). In this way, it is stronger  than the Assumption 4 in   \citet[]{1997KitStu}, but it is close to  \citet[Assumption C]{2000StoWri}.  Note  that    \citet[][]{2007Schennach}  also implicitly assumes that $\E\left[\e^{\tau' \psi(X_{1},\theta)}\psi(X_{1},\theta) \psi(X_{1},\theta)' \right]$ is full rank for all $(\theta, \tau) \in \Sbf $, because    \citet[][p. 649]{2007Schennach} regards  $\tau(\theta)$ as a solution to a strictly convex problem \citep[e.g.,][chap. 4, Theorem 4.3.1]{1993Hir-UrLem}. Assumption \ref{Assp:ExistenceConsistency}(g) should often be reasonable because the set of singular matrices has zero Lebesgue measure in the space of square matrices.\footnote{The set of singular matrices corresponds to the set of zeros of the determinant, which is nonzero polynomial in several variables. Moreover, by induction over the number of variables with  the fundamental theorem of algebra for the base step, a nonzero polynomials has a finite number of zeros.}

\subsection{Implications of Assumption \ref{Assp:ExistenceConsistency}(h)}
\begin{lem}\label{Lem:ChgOfMeasureInvertibilityPDP} Let $(\Omegabf_A, \mathcal{A}) $ be a measurable space, $Z: \Omegabf \rightarrow \R^k$ be a  $k$-dimensional random vectors with $k \in \ldsb 1, \infty \ldsb$ and $\Prm$ and $\Qrm $  two probability measures on $(\Omegabf_{A}, \mathcal{A}) $. Denote the expectation and the variance under $\Prm $ with $\E_{\Prm}$ and $\V_{\Prm}$, respectively.
\begin{enumerate}
\item[(i)] For all $\tau \in \R^k$, $ \E_{\Prm}\left( \e^{\tau' Z}Z Z'\right)\geqslant 0 $, it is a positive  semi-definite symmetric matrix.
\item[(ii)] If $\Prm \sim\Qrm$ (i.e., they are equivalent), $\E_\Prm(\vert Z Z'\vert )< \infty$ and $\E_\Qrm(\vert Z Z'\vert)< \infty $, then
\begin{eqnarray*}
\text{ $\E_\Prm(Z Z')$ invertible $\Leftrightarrow$ $\E_\Qrm(Z Z')$ invertible }
 \end{eqnarray*}
\end{enumerate}
\end{lem}
\begin{proof} \textit{(i)} Symmetry follows from the invariance under transposition of $\E_{\Prm}\left( Z Z'\e^{\tau' Z}\right)$. It remains to show positive semi-definiteness.
For all $ y\in\R^k $,
\begin{eqnarray*}
& & \forall \omega \in \Omegabf,\quad   y'\e^{\tau' Z}ZZ' y  =  \e^{\tau' Z}[y'Z]^2 \geqslant 0\\
& \Rightarrow & y'\E_{\Prm }\left[\e^{\tau' Z}ZZ' \right]y=\E_{\Prm }\left[y'\e^{\tau' Z}ZZ'y\right] \geqslant 0.
\end{eqnarray*}
where the implication follows from the monotonicity of the Lebesgue integral \citep[e.g.,][p. 47]{1997Mon}.

\textit{(ii)} By contraposition, it is equivalent to prove  that  $\E_\Prm(Z Z')$ noninvertible iff $\E_\Qrm(Z Z')$ noninvertible. By statement (i),
\begin{eqnarray*}
&  & \text{ $\E_\Prm(Z Z')$ noninvertible}\\
& \Leftrightarrow & \text{$\exists y \in \R^k\setminus\{0_{k \times 1}\}: $ $y'\E_\Prm(Z Z')y=0$ }\\
& \stackrel{(a)}{\Leftrightarrow} & \text{$\exists y \in \R^k\setminus\{0_{k \times 1}\}: $ $\E_\Prm[  (y'Z)^2]=0$ } \\
& \stackrel{(b)}{\Leftrightarrow} & \exists y \in \R^k\setminus\{0_{k \times 1}\}:   (y'Z)^2=0 \text{ $\Prm$-a.s.}\\
& \stackrel{(c)}{\Leftrightarrow} & \exists y \in \R^k\setminus\{0_{k \times 1}\}:   (y'Z)^2=0 \text{ $\Qrm$-a.s.}\\
& \stackrel{(d)}{\Leftrightarrow} & \text{$\exists y \in \R^k\setminus\{0_{k \times 1}\}: $ $\E_\Qrm[  (y'Z)^2]=0$ } \\
& \stackrel{(e)}{\Leftrightarrow} & \text{$\exists y \in \R^k\setminus\{0_{k \times 1}\}: $ $y'\E_\Qrm(Z Z')y=0$ }\\
& \Leftrightarrow & \text{ $\E_\Qrm(Z Z')$ noninvertible}
\end{eqnarray*}
\textit{(a)} $y'\E_\Prm(Z Z')y=\E_\Prm[y'Z (y'Z)']=\E_\Prm[  (y'Z)^2]$
\textit{(b)} The integral of a positive function w.r.t a measure is null iff the function is null almost-surely \citep[e.g.,][Lemma 1.24]{2002Kal}.   \textit{(c)} By assumption, $\Prm \sim \Qrm$.\textit{(d)} Same as (b). \textit{(a)} Same as (a) with $\Qrm$ instead of $\Prm$.

\end{proof}

 \begin{cor}[Implication of Assumption \ref{Assp:ExistenceConsistency}(h)]\label{Cor:ChgOfMeasureInvertibilityPDP}
 Under Assumptions \ref{Assp:ExistenceConsistency}(a)-(b), (e) and (g), Assumption \ref{Assp:ExistenceConsistency}(h) implies that, for all $(\theta, \tau)\in \Sbf $, $\E\left[\e^{\tau' \psi(X_1, \theta)}\psi(X_1, \theta) \psi(X_1, \theta)' \right] $  is a positive definite symmetric matrix.

\end{cor}
\begin{proof} By Lemma \ref{Lem:ChgOfMeasureInvertibilityPDP}i (p. \pageref{Lem:ChgOfMeasureInvertibilityPDP}) with $Z=\psi(X_{1},\theta)$, it is a positive semi-definite matrix. Thus, it remains to show that it is invertible, i.e., definite instead of only semi-definite.

Under  Assumption \ref{Assp:ExistenceConsistency} (a)(b)(d)(e)(g) and (h), by Lemma  \ref{Lem:ExpTauPsiStrictlyPositive} (p. \pageref{Lem:ExpTauPsiStrictlyPositive}) and Assumption \ref{Assp:ExistenceConsistency}(d)(e), for all $\theta \in \T $, $0<\E_{}[\e^{\tau(\theta)'\psi(X_1, \theta)}]<\infty$. Moreover, by Assumption \ref{Assp:ExistenceConsistency}(h), for all $\theta\in  \T$, $\E\left[\e^{\tau(\theta)' \psi(X_1,\theta)}\psi(X_1, \theta) \psi(X_1, \theta)' \right] $ is invertible, so that $\frac{1}{\E_{}[\e^{\tau(\theta)'\psi(X_1, \theta)}]}\E\left[\e^{\tau(\theta)' \psi(X_1,\theta)}\psi(X_1, \theta) \psi(X_1, \theta)' \right] $ is also invertible. For every $(\theta, \tau) \in \Sbf$, check the assumptions of Lemma \ref{Lem:ChgOfMeasureInvertibilityPDP}ii (p. \pageref{Lem:ChgOfMeasureInvertibilityPDP})  with $Z=\psi(X_{1},\theta)$,  $\frac{\d \Prm_{\theta}}{\d \P}=\frac{\e^{\tau(\theta)'\psi(X_1, \theta)}}{\E_{}[\e^{\tau(\theta)'\psi(X_1, \theta)}]} $ and $\frac{\d \Qrm_{(\theta, \tau)}}{\d \Prm_{\theta}}=\frac{\e^{\tau'\psi(X_1, \theta)}\E_{}[\e^{\tau(\theta)'\psi(X_1, \theta)}]}{\E_{}[\e^{\tau'\psi(X_1, \theta)}]\e^{\tau(\theta)'\psi(X_1, \theta)}}$, so that $  \frac{\d \Qrm_{(\theta, \tau)}}{\d \P}=\frac{\e^{\tau'\psi(X_1, \theta)}}{\E_{}[\e^{\tau'\psi(X_1, \theta)}]} $.  Firstly, for all $(\omega, \tau, \theta)\in \Omegabf \times \Tbf \times \T $, $0<\frac{\d \Qrm_{(\theta, \tau)}}{\d \Prm_\theta}$ and $0<\frac{\d \Prm_{\theta}}{\d \P}$, so that $\Qrm_{(\theta, \tau)} \sim \Prm_\theta \sim \P$. Secondly,  by monotonicity of integration and the Cauchy-Schwarz inequality, for all $\dot{\theta}\in \T$,  $\E [\vert \psi(X_1,\dot \theta) \psi(X_1, \dot\theta)'\vert] \leqslant \E [\sup_{\theta \in \T}\vert \psi(X_1, \theta) \psi(X_1, \theta)'\vert]<\sqrt{\E [\sup_{\theta \in \T}\vert \psi(X_1, \theta) \psi(X_1, \theta)'\vert^2]}<\infty$, where the last inequality follows from   Assumption \ref{Assp:ExistenceConsistency}(g).  Thirdly, under Assumption \ref{Assp:ExistenceConsistency} (a)(b)(d)(e)(g) awnd (h), by Lemma  \ref{Lem:ExpTauPsiStrictlyPositive} (p. \pageref{Lem:ExpTauPsiStrictlyPositive}) and Assumption \ref{Assp:ExistenceConsistency}(d)(e), for all $(\theta, \tau) \in \Sbf $, $0<\E_{}[\e^{\tau '\psi(X_1, \theta)}]<\infty$. Moreover, under Assumptions \ref{Assp:ExistenceConsistency}(a)-(b), (e) and (g),
by Lemma \ref{Lem:TiltedVariance}i (p. \pageref{Lem:TiltedVariance}), $ \E \left[ \sup_{(\theta, \tau)\in \Sbf}\vert\e^{\tau'\psi(X_1, \theta) }\psi(X_{1},\theta) \psi(X_{1},\theta)' \vert\right]< \infty$, so that, for all $(\theta, \tau)\in \Sbf$, 

\noindent
$\E \left[ \frac{\e^{\tau'\psi(X_1, \theta) }}{\E_{}[\e^{\tau'\psi(X_1, \theta)}]}\vert\psi(X_{1},\theta) \psi(X_{1},\theta)' \vert\right]<\infty$.
Thus, for each $(\theta, \tau)\in \Sbf $, apply Lemma \ref{Lem:ChgOfMeasureInvertibilityPDP}ii (p. \pageref{Lem:ChgOfMeasureInvertibilityPDP}) to show the result.
\end{proof}

\section{Remaining technical results}

\begin{lem}[Asymptotic invertibility of sequence of matrix functions]\label{Lem:UniFiniteSampleInvertibilityFromUniAsInvertibility} Let $A(\gamma)$ be a family of invertible matrices indexed by $\gamma \in \Gammabf$ s.t. $\gamma \mapsto A(\gamma)$ is continuous, and where $\Gammabf$ is a compact subset of a Euclidean space.  Let $(A_T(\gamma))_{T \in \ldsb 1, \infty \ldsb}$ be a sequence of square matrices. If, as $T \rightarrow \infty$, $\sup_{\gamma \in \Gammabf}\vert A_T(\gamma)-A(\gamma) \vert\rightarrow 0$, then there exist  a constant $\varepsilon_A>0$ and  $T_A\in \N$ s.t. for all $T \in \ldsb T_A, \infty\ldsb $, for all $\gamma \in \Gammabf$, $ \left\vert\vert A_T(\gamma)\vert_{\det}\right\vert\geqslant \varepsilon_A$. \end{lem}
\begin{proof} The function $A \mapsto \vert \vert A \vert_{\det} \vert$ is a continuous function. Moreover, by assumption, for all $\gamma \in \Gammabf$, $\vert \vert A(\gamma)\vert_{\det}\vert>0$. Thus, by continuity of $\gamma \mapsto A(\gamma)$ and compactness of $\Gammabf$, there exists $\varepsilon_A$ s.t. $\min_{\gamma \in \Gammabf}\left\vert\vert A(\gamma)\vert_{\det}\right\vert> 2\varepsilon_A$. Now continuity of $A \mapsto \vert \vert A \vert_{\det} \vert$ on the compact set $\Gammabf$ implies uniform continuity \citep[e.g.,][Theorem 4.19]{1953Rudin}, so that  there exists $T_{\varepsilon_A}\in \N$ s.t., for all $T \in \ldsb T_{\varepsilon_A}, \infty\ldsb$, $\sup_{\gamma \in \Gammabf}\left\vert\vert \vert A_T(\gamma) \vert_{\det} \vert -\vert \vert A(\gamma) \vert_{\det} \vert  \right\vert\leqslant \varepsilon_A $. Then, for all $\gamma \in \Gammabf$,  the triangle inequality $\vert \vert A(\gamma) \vert_{\det} \vert \leqslant  \left\vert\vert \vert A(\gamma) \vert_{\det} \vert-\vert \vert A_T(\gamma) \vert_{\det} \vert  \right\vert +\vert \vert A_T(\gamma) \vert_{\det} \vert $ implies that $\varepsilon_A= 2\varepsilon_A- \varepsilon_A\leqslant \vert \vert A(\gamma) \vert_{\det} \vert-\left\vert\vert \vert A(\gamma) \vert_{\det} \vert-\vert \vert A_T(\gamma) \vert_{\det} \vert  \right\vert \leqslant \vert \vert A_T(\gamma) \vert_{\det} \vert  $.
\end{proof}

\begin{lem}[Asymptotic positivity and definiteness of matrices] \label{Lem:PdmAsMatrix}
Let $(A_T )_{ T\geqslant1} $ a  sequence of  square matrices  converging to $A$ as $T \rightarrow \infty $.\footnote{Note that we do not need to specify the norm as all norms are equivalent in finite-dimensional spaces.} Then, if $(A_T )_{ T\geqslant1}$ is a sequence of symmetric matrices and $A$ is a positive-definite matrix (p-d.m), then  there exists $\dot{T}\in \N$ such that  $ T \geqslant \dot{T}$ implies $A_T$ is p-d.m.

\end{lem}
\begin{proof}
On one hand, $A_T $ is a p-d.m. if and only if all its eigenvalues are strictly positive \citep[e.g.,][Ch. 1 Sec. 13 Theorem 8]{1988MagnusNeudecker}.  On the other hand, $\min \sp A_T =\min_{z: \Vert z \Vert=1} z'A_Tz $,     where $\sp A_T $ denotes the set of eigenvalues of $A$ \citep[e.g.,][Ch. 11 Sec. 5]{1988MagnusNeudecker}.  Thus, it is sufficient to prove that $\lim_{T \rightarrow \infty} \min_{z: \Vert z \Vert =1} z'A_Tz = \min_{z: \Vert z \Vert=1}z'A z  $, which in turn implies that it is sufficient to prove that $\sup_{z: \Vert z \Vert=1 }|z'A_Tz-z'Az| \rightarrow 0 $ , as $T \rightarrow \infty$. Prove this last result by contradiction.

 Assume that $\sup_{z: \Vert z \Vert=1 }|z'A_Tz-z'Az|$ does not converge to $0 $ as $T \rightarrow \infty $. Then, there exists $\varepsilon>0 $ and an increasing function $\alpha_{1}:\N \mapsto \N $ defining a subsequence of vectors of norm $1$, $\left(z_{\alpha_1(T)} \right)_{T \geqslant 1}  $,  and  a subsequence of matrices,$\left(A_{\alpha_1(T)} \right)_{T \geqslant1} $, such that
\begin{eqnarray*}
\varepsilon & < & \left|z_{\alpha_1(T)}' A_{\alpha_1(T)}z_{\alpha_1(T)}- z_{\alpha_1(T)}' Az_{\alpha_1(T)}\right|\\
 & = & \left|z_{\alpha_1(T)}' \left(A_{\alpha_1(T)}-A \right)z_{\alpha_1(T)}\right| \leqslant  \sum_{(k,l)\in \ldsb 1,m\rdsb^2} \left|\left[a_{\alpha_1(T)}^{(k,l)}-a^{(k,l)}\right]z^{(k)}_{\alpha_1(T)} z^{(l)}_{\alpha_1(T)}\right|\\
& & \qquad  \leqslant m^{2} \times \max_{(k,l)\in \ldsb1,m\rdsb^2} \left|a_{\alpha_1(T)}^{(k,l)}-a^{(k,l)} \right|
\end{eqnarray*}
where $m $ is the size of the matrix $A$ and $a^{(k,l)}$ denotes the component of the matrix $A$ in the $k $th row and $l$th column. Now,  by assumption, using the max norm, $ \max_{(k,l)\in \ldsb1,m\rdsb^2} \left|a_{\alpha_1(T)}^{(k,l)}-a^{(k,l)} \right|\rightarrow 0 $ as $T\rightarrow  \infty $. Thus, there is a contradiction.
\end{proof}

\begin{lem} [Differential of a log of a squared determinant]\label{Lem:DeterminantDifferential} Let $\mathbf{G} $ be an open set of  $\R^q$ with $q\in \ldsb 1, \infty \ldsb $, and $F: \mathbf{G}\rightarrow \R^{m \times m}$ a differentiable function on $ \mathbf{G}$. Then $\vert F\vert_{\det}:  \mathbf{G}\rightarrow \R$ is also differentiable on $\mathbf{G}$. Moreover, if $\vert F(x) \vert_{\det}\neq 0$ where $x \in \Gbf$, then
\begin{enumerate}
\item[(i)]   $D \vert F(x)\vert_{\det} =\vert F (x)\vert_{\det} \tr[ F(x)^{-1}DF(x)]$;
\item[(ii)]  $D \ln[\vert F(x)\vert_{\det}^2] =2\tr[ F(x)^{-1}DF(x)]$.
 \end{enumerate}

\end{lem}
\begin{proof}
\textit{(i)} It is a consequence of the so-called Jacobi's formula   \citep[e.g.,][chap. 8 sec. 3]{1988MagnusNeudecker}.

\textit{(ii)} First of all, note that the logarithm is well-defined as its argument is strictly positive by assumption. Then, by the statement (i) of the present lemma and the chain rule, \begin{eqnarray*}
D\ln[\vert F(x)\vert_{\det}^2]=\frac{1}{\vert F(x)\vert_{\det}^2}2\vert F(x)\vert_{\det} \vert F (x)\vert_{\det} \tr[ F(x)^{-1}DF(x)].
\end{eqnarray*}
\end{proof}

\begin{lem}[Inverse of a $2\times 2$ partitioned matrix]\label{Lem:BlockMatrixInverse} Let $F$ be a square matrix s.t.
\begin{eqnarray*}
F=\begin{bmatrix}A & B \\
C & D \\
\end{bmatrix}
\end{eqnarray*}
where $A$ and $D$ are square matrices. Then, the following statements hold.
\begin{enumerate}
\item[(i)] If $A$ is invertible, then $F$ invertible $\Leftrightarrow $ $(D-CA^{-1}B)$ invertible. Moreover,
\begin{eqnarray*}
F^{-1}=\begin{bmatrix}A^{-1}+A^{-1}B(D-CA^{-1}B)^{-1}CA^{-1} & \quad -A^{-1}B(D-CA^{-1}B)^{-1} \\
-(D-CA^{-1}B)^{-1}CA^{-1} & (D-CA^{-1}B)^{-1} \\
\end{bmatrix}.
\end{eqnarray*}

\item[(ii)] If $D$ is invertible, then $F$ invertible $\Leftrightarrow $ $(A-BD^{-1}C)^{-1}$ invertible. Moreover,
\begin{eqnarray*}
F^{-1}=\begin{bmatrix}(A-BD^{-1}C)^{-1} & -(A-BD^{-1}C)^{-1}BD^{-1} \\
-D^{-1}C(A-BD^{-1}C)^{-1} & \quad D^{-1}+D^{-1}C(A-BD^{-1}C)^{-1}BD^{-1} \\
\end{bmatrix}.
\end{eqnarray*}
\end{enumerate}
\end{lem}
\begin{proof} This is a standard result \citep[e.g.,][Chap. 1 sec. 11]{1988MagnusNeudecker}.
\end{proof}

\begin{cor}[Inverse of a $2 \times 2$ partitioned matrix in a special case]\label{Cor:For2x2BlockMatrix} Let $E$ be a square matrix s.t.
\begin{eqnarray*}
E=\begin{bmatrix}A & B \\
B' & 0 \\
\end{bmatrix}.
\end{eqnarray*}
Then,
\begin{itemize}
\item[(i)]  If $A$  and $B'A^{-1}B $ are invertible, then $E$ in invertible; and
\item[(ii)] $\displaystyle \begin{bmatrix}A & B \\
B' & 0 \\
\end{bmatrix}^{-1}
=
\left[ \begin{array}{c c} A^{-1} - A^{-1}B \left( B' A^{-1} B \right)^{-1} B' A^{-1} &   A^{-1}B \left( B' A^{-1} B \right)^{-1} \\
\left( B' A^{-1} B \right)^{-1} B' A^{-1}  & - \left( B' A^{-1} B \right)^{-1} \end{array} \right]$.
\end{itemize}
\end{cor}
\begin{proof} Apply the above Lemma \ref{Lem:BlockMatrixInverse}i with $F=E$, $C=B'$ and $D=0$.
\end{proof}

\section{More on the numerical example}\label{Ap:NumericalExample}
 The simulations were performed in R.   Each model parameterization is simulated 10,000 times. The robustness of the simulation results was checked with different optimization algoritheorems, starting values and tolerance parameter values.  The estimation for a single sample is typically performed in less than a few  seconds. The calculations were done on a 24 CPU cores of a Dell server with 4 AMD Opteron 8425 HE processors running at 2.1 GHz.  We numerically checked that  the reported statistics have a converging behaviour  as we increase the number of simulated samples to 10,000.%\footnote{Of course, as a computer can only handle a bounded parameter space (finite memory of a computer),  the statistics would necessarily converge if could sufficiently increase the number of similated sample} 

\section{More on the empirical example}\label{Ap:EmpiricalExample}

In empirical consumption-based asset pricing,  the literature has found little common ground about the value of the   relative risk aversion (RRA)  of the representative agent: In most studies, point estimates from economically similar moment conditions  are  generally outside  of each other's confidence intervals. Section \ref{Sec:EmpExample} (p. \pageref{Sec:EmpExample}) and the present appendix  revisit the estimation of the RRA. The popularity of moment-based estimation in consumption-based asset pricing, and more generally in economics is due to the fact that moment-based estimation does not necessarily require the specification of a family of distributions for the data \citep[e.g.][sec. 3]{2013HanN}. Typically, an economic  model does not imply  such family of distributions, except for tractability reasons. Imposing a family of distributions makes it difficult to disentangle the part of the inference results due to the empirical relevance of the economic model from the part due to these additional restrictions. Under  regularity  conditions, assuming a distribution corresponds to imposing an infinite number of extra moment restrictions \citep[e.g.,][1971/1966, chap. VII, sec. 3]{1966Fel}.

In Section \ref{Sec:EmpExample} (p. \pageref{Sec:EmpExample}) and the present appendix, we rely on the moment condition \eqref{Eq:KeyEmpMomCond} on p. \pageref{Eq:KeyEmpMomCond}.
This  moment condition has several advantages. Firstly, it is  as consistent with \cite{1978Luc} as with more recent consumption-based  asset-pricing models, such as \cite{2006Bar}  or \cite{2012Gab}. In other words, despite its simplicity it also correspond to sophisticated models, and it allows us to obtain estimates  that are robust to different variations of consumption-based asset pricing theory.  Secondly, without loss of generality, it does not require to estimate the time discount rate, about which there is little debate: The time discount rate of the representative agent is consistently found to be between .9 and 1.
Note also that  it has been common to use moment conditions with  a separate parameter for 
the so-called intertemporal  elasticity of substitution, i.e., use Epstein-Zin-Weil preferences \citep[e.g.][]{1991EpsZin}. However, \cite{2017BommierKochovLeGrand} show that such a specification  makes the economic interpretation of the parameters difficult. In particular, they show that an increase of the so-called RRA (relative risk-aversion) parameter does not yield a  behaviour that would  be considered more risk averse. E.g., All other things being equal, savings can be a decreasing function of  the so-called RRA parameter for an agent  with Epstein-Zin-Weil  preferences   \citep[e.g.,][sec. 6]{2017BommierKochovLeGrand}.  This difficulty of interpretation comes from a violation of the monotonicity axiom according to which an agent does not choose an action if another available action is preferable in every state of the world.    \\

\subsection{Additional empirical evidence}

\begin{table}[ht!] \caption{\textbf{ET vs. ESP inference (1890--2009)}} \label{Tab:ETvsESP1890}
 \medskip
 \begin{tabular}{ll}
% \hline
%  \hline
% \multicolumn{2}{c}{\textbf{GMM inference with CRRA preferences} }  \\ 
% %
\hline\multicolumn{2}{l}{Empirical moment condition: $\frac{1}{2009-1889}\sum_{t=1890}^{2009}\left[  \left(\frac{C_{t}}{C_{t-1}}\right)^{-\theta}(R_{m,t}-R_{f,t})\right]=0 $, where} \\
\multicolumn{2}{l}{$R_{m,t}:=$ gross market  return,\, $R_{f,t}:=$risk-free asset gross return,\, $C_t:=$ consumption, }\\
\multicolumn{2}{l}{ and $\theta:=$relative risk aversion;}\\
\multicolumn{2}{l}{\text{Normalized ET:=}$\exp\negthickspace\left\{T\ln\left[ \frac{1}{T}\sum_{t=1}^T\e^{\tau_T(.) ' \psi_t(.)}\right]\right\}\negthickspace/\negthickspace\int_{\Theta} \exp\negthickspace\left\{T\ln\left[ \frac{1}{T}\sum_{t=1}^T\e^{\tau_T(\theta) ' \psi_t(\theta)}\right]\right\}\d \theta$;    }\\
\multicolumn{2}{l}{\text{Normalized ESP:=}$ \hat{f}_{\theta^*_T}(.)/\negthickspace\int_{\Theta} \hat{f}_{\theta^*_T}(\theta)\d \theta$;    }\\
% \multicolumn{2}{l}{Case with support restricted to $\R_{+}$: $\hat{I}_{.05}\negthickspace %=\negthickspace[10.50,  188.85]$ (stripe on A), ESP  support $=\negthickspace[0,289.0]$}\\
\multicolumn{2}{l}{$\hat{\theta}_{\mathrm{ET},T}=\hat{\theta}_{\mathrm{MM},T}=50.3$ (bullet) and $\hat{\theta}_{\mathrm{ESP},T}=32.21$ (bullet); }\\
\multicolumn{2}{l}{  ET and ESP  support $=[-218.2,289.0]$; 95\% ET ALR conf. region=$[18.3, 289.0] $ (stripe);  }\\
\multicolumn{2}{l}{   95\% ESP ALR conf. region=$[15.0,112.7]$ (stripe). }
  \\ 
\hline
%  \multicolumn{2}{c}{ \includegraphics[scale=1.85]{1890_2009_ESP_ET_Norm.png} }  \\ 
%   \multicolumn{2}{c}{\begin{small}(A)  Normalized ET   (light green) vs. normalized ESP (dark blue). \end{small} }\\ 
% \hline
% \includegraphics[scale=1.1]{1890_2009_ET_Norm_ALR.png} & \includegraphics[scale=1.1]{1890_2009_ESP_Norm_ALR.png} 

 \multicolumn{2}{c}{ \includegraphics[scale=.5886]{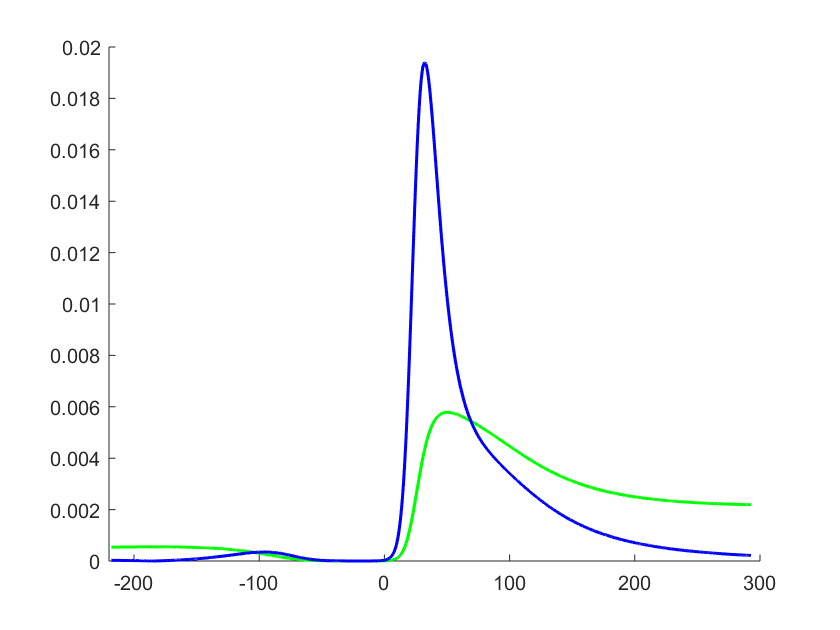} }  \\ 
  \multicolumn{2}{c}{\begin{small}(A)  Normalized ET   (light green) vs. normalized ESP (dark blue). \end{small} }\\ 
\hline
\includegraphics[scale=.351]{1890_2009_ET_Norm_ALR.png} & \includegraphics[scale=.351]{1890_2009_ESP_Norm_ALR.png} 
\\ 
  \begin{small}(A) ET est. and ALR conf. region.\end{small} &  \begin{small}(B) ESP est. and ALR conf. region.\end{small} \\ 
\hline
\hline
\end{tabular}
\end{table}

\begin{table} \caption{\textbf{ET vs. ESP inference (1930--2009)}} \label{Tab:ETvsESP1930}
%\medskip
 \begin{tabular}{ll}
% \hline
%  \hline
% \multicolumn{2}{c}{\textbf{GMM inference with CRRA preferences} }  \\ 
% %
\hline\multicolumn{2}{l}{Empirical moment condition: $\frac{1}{2009-1889}\sum_{t=1890}^{2009}\left[  \left(\frac{C_{t}}{C_{t-1}}\right)^{-\theta}(R_{m,t}-R_{f,t})\right]=0 $, where} \\
\multicolumn{2}{l}{$R_{m,t}:=$ gross market  return,\, $R_{f,t}:=$risk-free asset gross return,\, $C_t:=$ consumption, }\\
\multicolumn{2}{l}{ and $\theta:=$relative risk aversion;}\\
\multicolumn{2}{l}{\text{Normalized ET:=}$\exp\negthickspace\left\{T\ln\left[ \frac{1}{T}\sum_{t=1}^T\e^{\tau_T(.) ' \psi_t(.)}\right]\right\}\negthickspace/\negthickspace\int_{\Theta} \exp\negthickspace\left\{T\ln\left[ \frac{1}{T}\sum_{t=1}^T\e^{\tau_T(\theta) ' \psi_t(\theta)}\right]\right\}\d \theta$;    }\\
\multicolumn{2}{l}{\text{Normalized ESP:=}$ \hat{f}_{\theta^*_T}(.)/\negthickspace\int_{\Theta} \hat{f}_{\theta^*_T}(\theta)\d \theta$;    }\\
% \multicolumn{2}{l}{Case with support restricted to $\R_{+}$: $\hat{I}_{.05}\negthickspace %=\negthickspace[10.50,  188.85]$ (stripe on A), ESP  support $=\negthickspace[0,289.0]$}\\
\multicolumn{2}{l}{$\hat{\theta}_{\mathrm{ET},T}=35.0$ (bullet) and $\hat{\theta}_{\mathrm{ESP},T}=32.5$ (bullet); ET and ESP  support$=[-202.8,813.3]$}\\
\multicolumn{2}{l}{95\% ET ALR conf. region=$[-202.8,-76.0]\cup[17.7,197.8] $ (stripe);  }
  \\ 
  \multicolumn{2}{l}{95\% ESP ALR conf. region=$[17.7, 58.7]$ (stripe). }
  \\
\hline
 \multicolumn{2}{c}{ \includegraphics[scale=.5886]{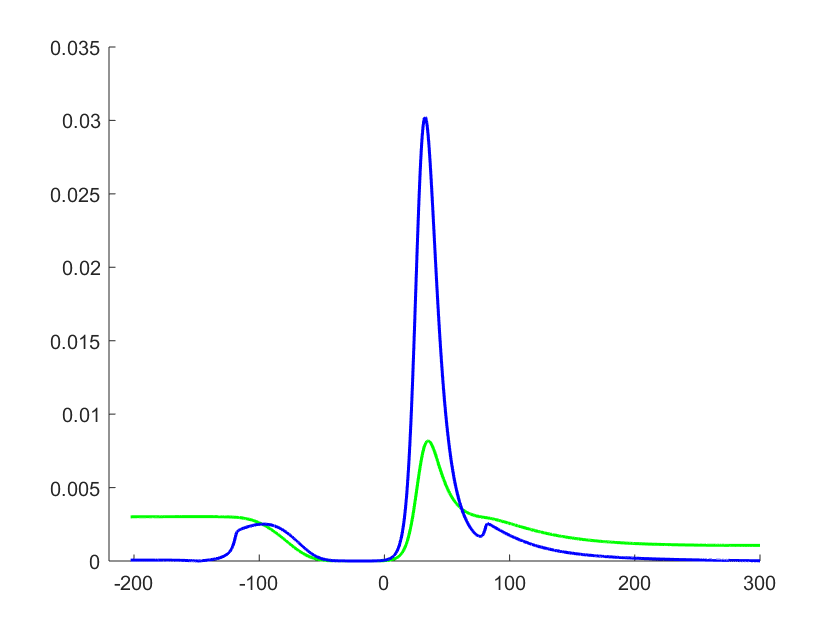}
  }  \\ 
  \multicolumn{2}{c}{\begin{small}(A)   Normalized ET   (light green) vs. normalized ESP (dark blue). \end{small} }\\ 
\hline
\includegraphics[scale=.351]{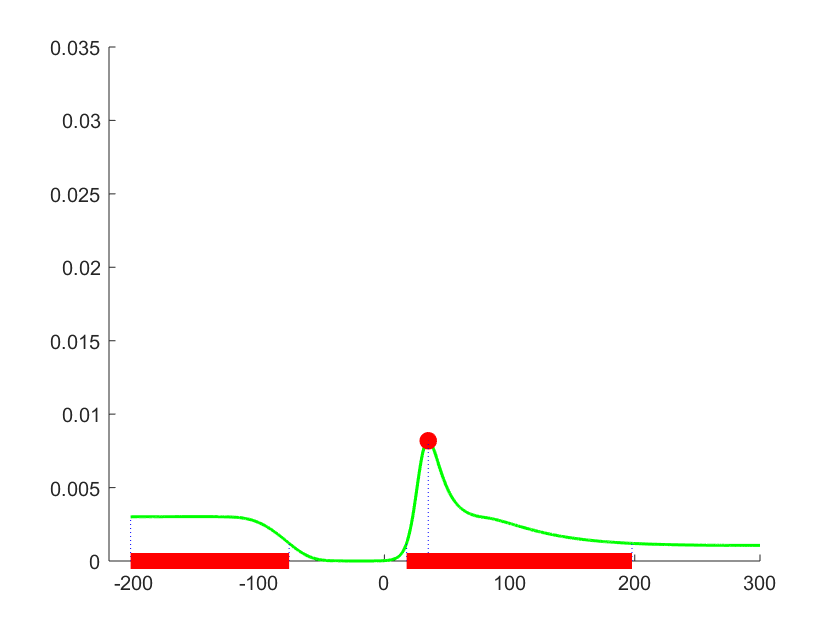}
 & \includegraphics[scale=.351]{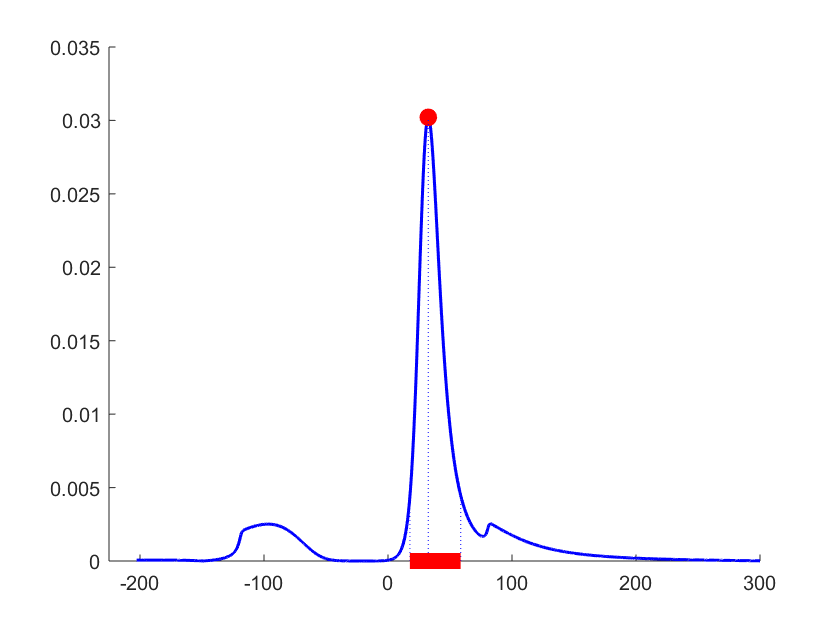} 
 \\ 
  \begin{small}(B)  ET  est. and ALR conf. region.\end{small} &  \begin{small}(C) ESP  est. and ALR conf. region.\end{small} \\ 
\hline
\hline
\end{tabular}
\end{table}
\qquad 
  
Table  \ref{Tab:ETvsESP1890} (p. \pageref{Tab:ETvsESP1890}) is the same as Table  \ref{Tab:ETvsESP1890Short} (p. \pageref{Tab:ETvsESP1890Short})
with the additional Table  \ref{Tab:ETvsESP1890} Figures (A). The  latter clearly shows   that the normalized ESP is relatively sharp around the ESP estimator.    \\ 
Table \ref{Tab:ETvsESP1930} (p. \pageref{Tab:ETvsESP1930}) is the counterpart of Table  \ref{Tab:ETvsESP1890} (p. \pageref{Tab:ETvsESP1890}) for  the 1930-2009 data set. The 95\% ET ALR confidence region is based on the inversion of the ALR ET statistic  
$2T\left[ \mathrm{LogET}(\hat{\theta}_T)-\mathrm{LogET}(\theta_0)\right]\\=2T\mathrm{LogET}(\theta_0)\rightarrow \chi^2_1 $  \cite[Theorem 4 with $K=0$ and $\mathrm{H}_0: \theta= \theta_0$]{1997KitStu}, where LogET$(\theta):=\negthickspace\ln  \left[\frac{1}{T}\sum_{t=1}^T \e^{\tau_T(\theta)'\psi_t(\theta)}\right]$ and $\mathrm{LogET}(\hat{\theta}_T)= \ln  \left[\frac{1}{T}\sum_{t=1}^T \e^{\tau_T(\hat{\theta}_T)'\psi_t(\hat{\theta}_T)}\right] =0$ because, in the just-identified case, $\frac{1}{T}\sum_{t=1}^T\psi_t(\hat{\theta}_T)=0 $ so that $\tau_T(\hat{\theta}_T)=0_{m \times 1}$. The ET and ESP support correspond to the parameter values $\theta\in \T$ for which there exists a solution $\tau_T(\theta)$ to the equation \eqref{Eq:ESPTiltingEquation} on p. \pageref{Eq:ESPTiltingEquation}.
Table \ref{Tab:ETvsESP1930} confirms the findings of  Table \ref{Tab:ETvsESP1890} (p. \pageref{Tab:ETvsESP1890}) in  Section \ref{Sec:EmpExample}\,:  The ESP is sharper than the ET around its maximum, so that the ESP confidence region is also shorter. Note also that  the ESP estimate is almost the same as for the data set 1890-2009.  These results are in line with the ESP shrinkage-like behaviour documented in the Monte-Carlo simulations of the section \ref{Sec:NumericalExp}.     

Tables \ref{Tab:MM1890} (p. \pageref{Tab:MM1890}) and \ref{Tab:MM1930} (p. \pageref{Tab:MM1930}) report the MM estimates and the  confidence regions  based on  the inversion of the MM\ ALR\ test statistic  
$T\left[ Q_{\mathrm{MM}, T}(\theta_{0})- Q_{\mathrm{MM}, T}(\hat{\theta}_{\mathrm{MM}, T}) \right]=T Q_{\mathrm{MM}, T}(\theta_{0}) \stackrel{D}{\rightarrow} \chi^2_1  $, as $T \rightarrow \infty$, \citep[e.g.,][Theorem 9.2]{1994NewMcF},
where $Q_{\mathrm{MM}, T}(\theta):=\left[ \frac{1}{T}\sum_{t=1}^T\psi_t(\theta)\right]'\\ \times\left[\frac{1}{T}\sum_{t=1}^T\psi_t(\hat{\theta}_{\mathrm{MM}, T})\psi_t(\hat{\theta}_{\mathrm{MM}, T})'\right]^{-1}\left[ \frac{1}{T}\sum_{t=1}^T\psi_t(\theta)\right]$
and $Q_{\mathrm{MM}, T}(\hat{\theta}_{\mathrm{MM}, T})=0$ because  

\noindent
 $\frac{1}{T}\sum_{t=1}^T\psi_t(\hat{\theta}_T)=0 $ in the just-identified case. The MM objective function is sharper around its minimum for the 1930-2009 data set than for the 1890-2009. However, the former sharpness appears misleading as it yields a confidence region that does not include the MM estimate of the 1890-2009 data set.   

Tables \ref{Tab:CU1890} (p. \pageref{Tab:CU1890}) and \ref{Tab:CU1930} (p. \pageref{Tab:CU1930}) report the  CU (continuously updating) MM estimates and the confidence regions based on the inversion of the  CU ALR test statistic

\noindent
$T\left[ Q_{\mathrm{CU}, T}(\theta_{0})- Q_{\mathrm{CU}, T}(\hat{\theta}_{\mathrm{MM}, T}) \right]=T Q_{\mathrm{CU}, T}(\theta_{0}) \rightarrow \chi^2_1  $, as $T \rightarrow \infty$, 
where 

\noindent
$Q_{\mathrm{CU}, T}(\theta):=\left[\frac{1}{T}\sum_{t=1}^T\psi_t(\theta)\right]'\left[\frac{1}{T}\sum_{t=1}^T\psi_t(\theta)\psi_t(\theta)'\right]^{-1}\left[\frac{1}{T}\sum_{t=1}^T\psi_t(\theta)\right]$
and $Q_{\mathrm{CU}, T}(\hat{\theta}_{\mathrm{CU}, T})=0$   because   $\frac{1}{T}\sum_{t=1}^T\psi_t(\hat{\theta}_T)=0 $ in the just-identified case.
In the just-identified case, which is the case addressed in the present paper,  such confidence regions correspond to the $S$-sets, which  were proposed by \cite{2000StoWri} ---following \citet[(c) Constrained-Minimized]{1996HanHeaYar}--- as a solution to the flatness of GMM objective functions.  As previously documented in the literature \citep[e.g.,][]{1996HanHeaYar}, CU GMM objective functions tend to be flat and low in the tails. Thus,  the CU\ ALR confidence regions (and $S$-sets  in the just-identified case) are  huge,  and hardly  informative.\\

\begin{table} \caption{\textbf{MM inference (1890--2009)}} \label{Tab:MM1890}
 \medskip
 \begin{tabular}{ll}
% \hline
%  \hline
% \multicolumn{2}{c}{\textbf{GMM inference with CRRA preferences} }  \\ 
% %
\hline\multicolumn{2}{l}{Empirical moment condition: $\frac{1}{2009-1889}\sum_{t=1890}^{2009}\left[  \left(\frac{C_{t}}{C_{t-1}}\right)^{-\theta}(R_{m,t}-R_{f,t})\right]=0 $, where} \\
\multicolumn{2}{l}{$R_{m,t}:=$ gross market  return,\, $R_{f,t}:=$risk-free asset gross return,\, $C_t:=$ consumption, and}\\
\multicolumn{2}{l}{ $\theta:=$relative risk aversion.}\\
\multicolumn{2}{l}{$\hat{\theta}_{\mathrm{GMM},T}=  50.3 $ (bullet);\, 95\% ALR confidence region$=[-41.7,71.5]$ (stripe).}
  \\ \hline
\includegraphics[scale=.35]{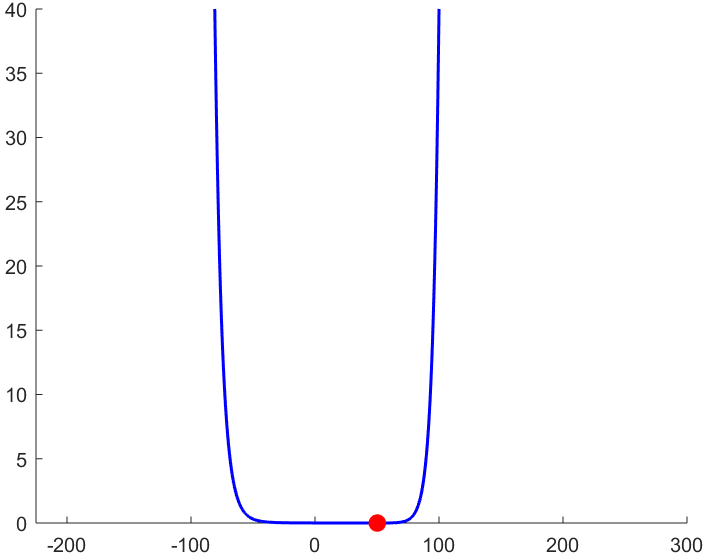} & \includegraphics[scale=.35]{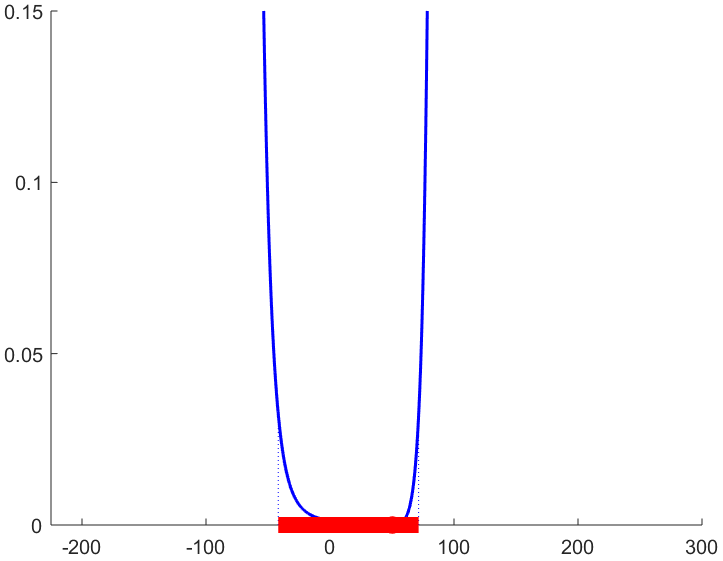} \\ 
  \begin{small}(A) MM objective function and point estimate.\end{small} &  \begin{small}(A zoom) MM obj. function ALR conf. region. \end{small} \\ 
% \hline
%  \multicolumn{2}{c}{ \includegraphics[scale=2]{GMM_Gau_CRRA_1890.png} }  \\
% % 
% \multicolumn{2}{c}{ \begin{small}(B) Gaussian distribution,  point estimate  and  confidence interval.  \end{small} }\\ 
\hline
\hline
\end{tabular}
\end{table}
   
\bigskip
\bigskip

 \begin{table}\caption{\textbf{Continuously updated (CU) GMM inference (1890--2009)}} \label{Tab:CU1890}
 \medskip
 \begin{tabular}{ll}
% \hline
%  \hline
% \multicolumn{2}{c}{\textbf{GMM inference with CRRA preferences} }  \\ 
% %
\hline\multicolumn{2}{l}{Empirical moment condition: $\frac{1}{2009-1890}\sum_{t=1890}^{2009}\left[  \left(\frac{C_{t}}{C_{t-1}}\right)^{-\theta}(R_{m,t}-R_{f,t})\right]=0 $, where} \\
\multicolumn{2}{l}{$R_{m,t}:=$ gross market  return,\, $R_{f,t}:=$risk-free asset gross return,\, $C_t:=$ consumption, and }\\
\multicolumn{2}{l}{ $\theta:=$relative risk aversion.}\\
\multicolumn{2}{l}{$\hat{\theta}_T^{\mathrm{CU}}\negthickspace=\negthickspace  50.3 $ (bullet); 95\% ALR confidence region (and $S$-set) $\negthickspace=]\ldots, -59.1]\negthickspace\cup\negthickspace[18.2,\ldots[$ (stripe).}\\

\multicolumn{2}{l}{Rk: We constrain the numerical search for point estimate to discard large values of $\theta$. }   \\ \hline
% \multicolumn{2}{c}{ \includegraphics[scale=1.5]{1890_2009_CU_GMM.png} } \\ 
%  \multicolumn{2}{c}{\begin{small}(A) Objective function and point estimate.\end{small}} \\ 
% \hline
\includegraphics[scale=.35]{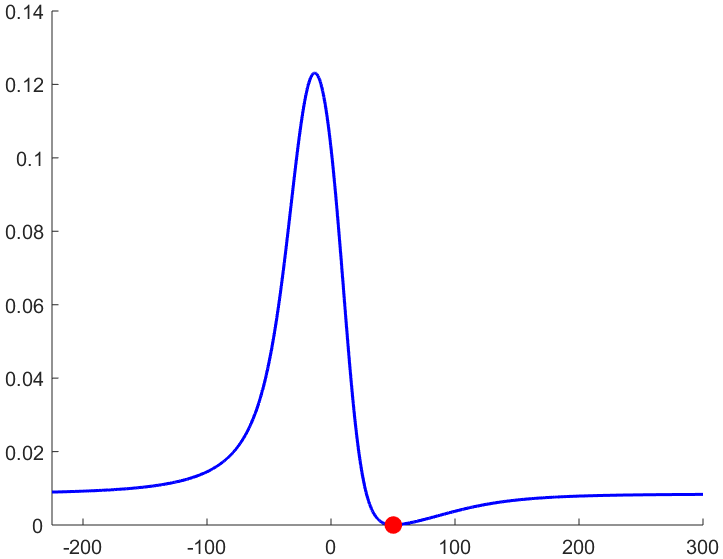} & \includegraphics[scale=.35]{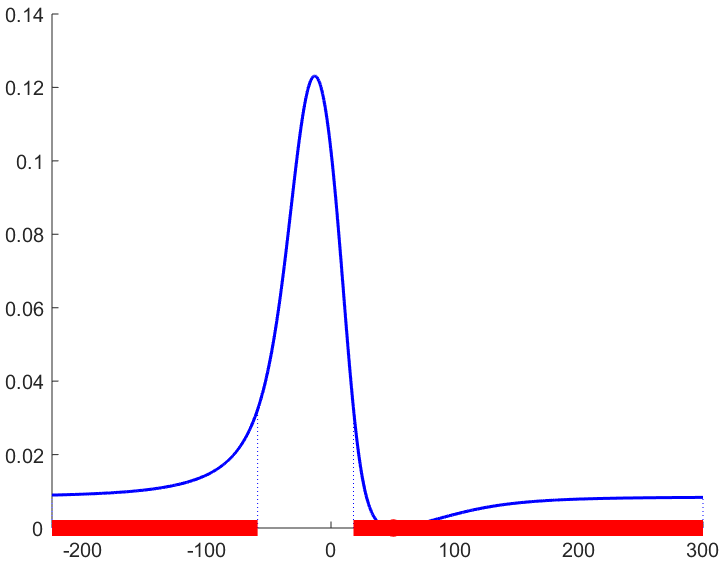} \\ 
  \begin{small}(A) Objective function and point estimate.\end{small} &  \begin{small}(B) Truncated ALR conf. region (and $S$-set).  \end{small} \\ 
\hline
\hline
\end{tabular}
\end{table}

\begin{table} \caption{\textbf{MM inference (1930-2009)}}\label{Tab:MM1930}
 \medskip
 \begin{tabular}{ll}
% \hline
%  \hline
% \multicolumn{2}{c}{\textbf{GMM inference with CRRA preferences} }  \\ 
% %
\hline\multicolumn{2}{l}{Empirical moment condition: $\frac{1}{2009-1930}\sum_{t=1930}^{2009}\left[  \left(\frac{C_{t}}{C_{t-1}}\right)^{-\theta}(R_{m,t}-R_{f,t})\right]=0 $, where} \\
\multicolumn{2}{l}{$R_{m,t}:=$ gross market  return,\, $R_{f,t}:=$risk-free asset gross return,\, $C_t:=$ consumption, }\\
\multicolumn{2}{l}{ $\theta:=$relative risk aversion.}\\
\multicolumn{2}{l}{$\hat{\theta}_{\mathrm{MM},T}=  35.0 $ (bullet),\, ALR confidence region$=[-10.4,46.5]$ (stripe)}
  \\ \hline
\includegraphics[scale=.35]{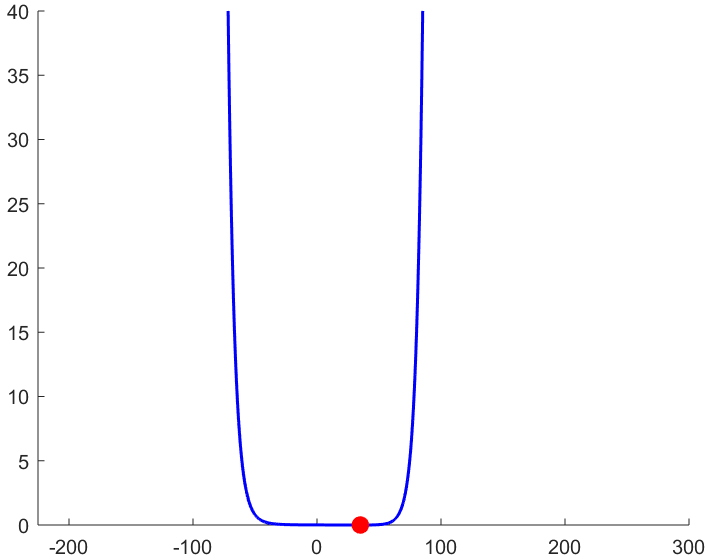} & \includegraphics[scale=.35]{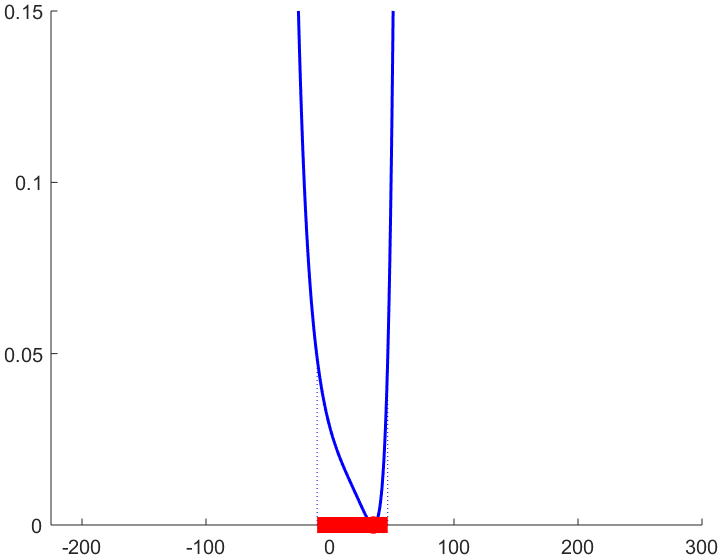} \\ 
  \begin{small}(A) MM objective function and point estimate.\end{small} &  \begin{small}(A zoom) Objective function and point estimate. \end{small} \\ 
% \hline
%  \multicolumn{2}{c}{\includegraphics[scale=2]{GMM_Gau_CRRA_0_300_1930_2009.png}}  \\ 
%   \multicolumn{2}{c}{\begin{small}(C) Gaussian distribution, point estimate  and $5\%$ confidence interval. \end{small} }\\ 
\hline
\hline
\end{tabular}
\end{table}

\begin{table}\caption{\textbf{Continuously updated (CU) GMM inference (1930--2009)}} \label{Tab:CU1930} 
 \medskip
 \begin{tabular}{ll}
% \hline
%  \hline
% \multicolumn{2}{c}{\textbf{GMM inference with CRRA preferences} }  \\ 
% %
\hline\multicolumn{2}{l}{Empirical moment condition: $\frac{1}{2009-1890}\sum_{t=1890}^{2009}\left[  \left(\frac{C_{t}}{C_{t-1}}\right)^{-\theta}(R_{m,t}-R_{f,t})\right]=0 $, where} \\
\multicolumn{2}{l}{$R_{m,t}:=$ gross market  return,\, $R_{f,t}:=$risk-free asset gross return,\, $C_t:=$ consumption, }\\
\multicolumn{2}{l}{ $\theta:=$relative risk aversion.}\\
\multicolumn{2}{l}{$\hat{\theta}_T^{\mathrm{CU}}=  50.3 $ (bullet);\, ALR confidence region (and $S$-set) $=]\ldots, -35.8]\cup[17.9,\ldots[$ (stripe).}\\
\multicolumn{2}{l}{Rk: We constrain the numerical search for point estimate to discard large values of $\theta$. }   \\ \hline

\includegraphics[scale=.35]{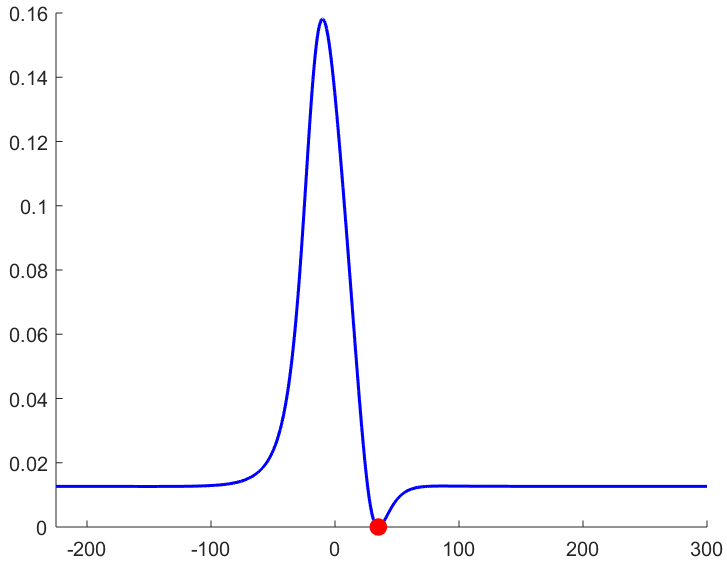} & \includegraphics[scale=.35]{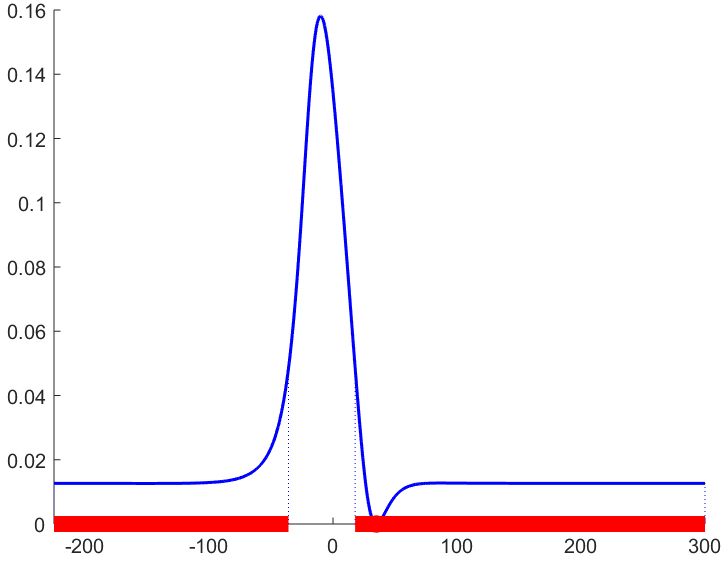} \\ 
  \begin{small}(A) Objective function and point estimate.\end{small} &  \begin{small}(B) Truncated ALR conf. region (and S-set).  \end{small} \\ 
\hline
\hline
\end{tabular}
\end{table}

\subsection{Data description} As in  \cite{2012JulGho},
our data are  standard. For the 1890-2009 data set, our source is the Robert Shiller's web site. The prime commercial paper   and the S\&P
stock price index play the role of proxies for the risk-less asset and the market return.
\begin{table}[ht]\caption{\textbf{Descriptive statistics.}}\label{Tab:DescriptiveStat}
\begin{tabular}{lll}
%\hline
\hline
 & \multicolumn{2}{c}{Mean (Variance)} \\ 
\cline{2-3}
Variable & \multicolumn{1}{c}{1890-2009} & 1930-2009 \\ 
\hline
$C_{t}/C_{t-1}$ & \multicolumn{1}{c}{$1.0182$} & \multicolumn{1}{c}{$1.014$} \\ 
 & \multicolumn{1}{c}{(.0009)} & \multicolumn{1}{c}{(.0007)} \\ 
$R_{m,t}-R_{f,t} $ & \multicolumn{1}{c}{$.0630$} & \multicolumn{1}{c}{$.074$} \\ 
 & \multicolumn{1}{c}{(.0367)} & \multicolumn{1}{c}{(.0424)} \\ 
\hline
\hline
\end{tabular}
\end{table}

\begin{table}\caption{\textbf{Excess returns: $R_{m,t}-R_{f,t}$}}\label{Tab:Rm_Rf}
\begin{tabular}{cc}\hline %\hline
1890-2009 & 1930-2009 \\\hline
\includegraphics[scale=.5]{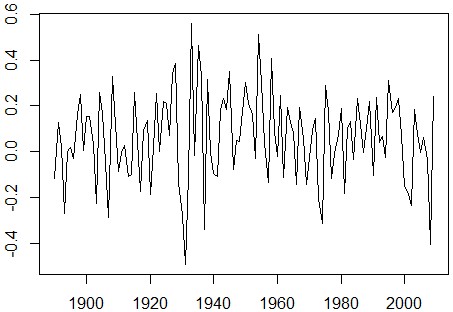} & \includegraphics[scale=.5]{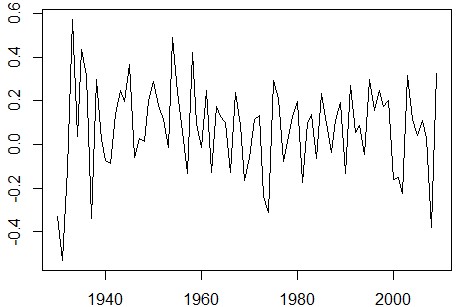} \\

\multicolumn{1}{l}{(A) Time series}  & \multicolumn{1}{l}{(B) Time series}  \\\hline
\includegraphics[scale=.5]{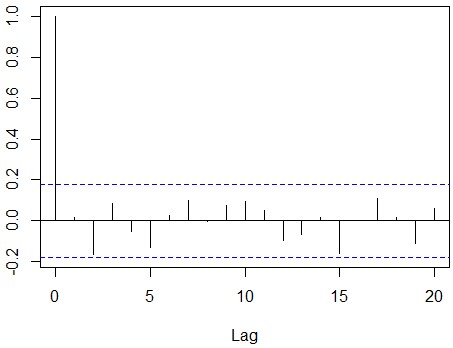}a & \includegraphics[scale=.5]{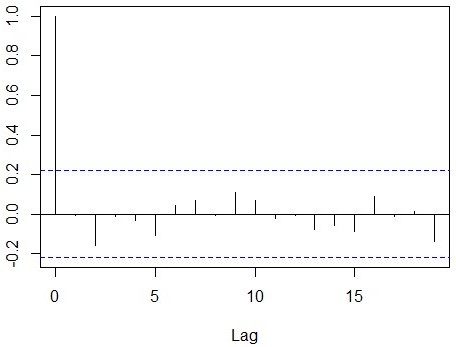} \\
\multicolumn{1}{l}{(C) Autocorr. function of $R_{m,t}-R_{f,t}$}  & \multicolumn{1}{l}{(D) Autocorr. function of $R_{m,t}-R_{f,t}$}\\\hline
\includegraphics[scale=.5]{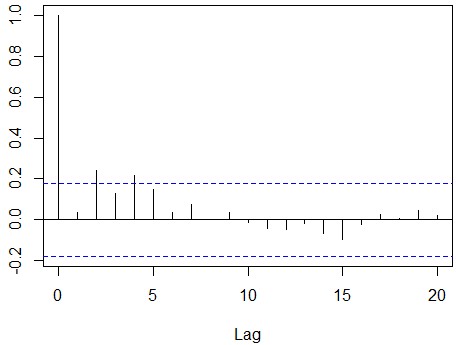} &  \includegraphics[scale=.5]{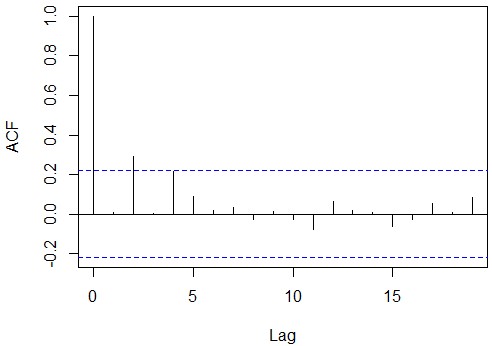}\\
\multicolumn{1}{l}{(E) Autocorr. function of $(R_{m,t}\negthickspace-\negthickspace R_{f,t})^2$}  & \multicolumn{1}{l}{(F) Autocorr. function of $(R_{m,t}\negthickspace-\negthickspace R_{f,t})^2$} \\\hline \hline
\end{tabular}
\end{table}

\begin{table}\caption{\textbf{Growth consumption: $C_t/C_{t-1}$.}}\label{Tab:gc}
\begin{tabular}{cc}\hline %\hline
1890-2009 & 1930-2009 \\\hline
\includegraphics[scale=.5]{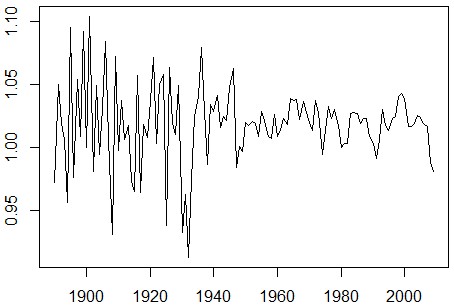} & \includegraphics[scale=.5]{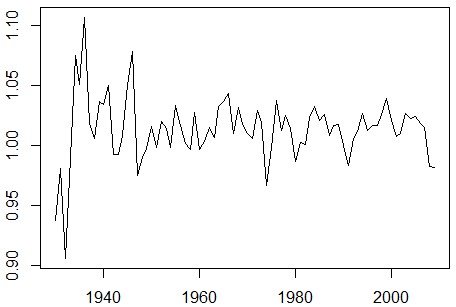} \\
\multicolumn{1}{l}{(A) Time series}  & \multicolumn{1}{l}{(B) Time series}  \\\hline
\includegraphics[scale=.5]{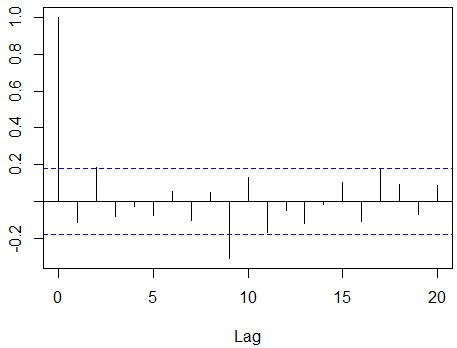} & \includegraphics[scale=.5]{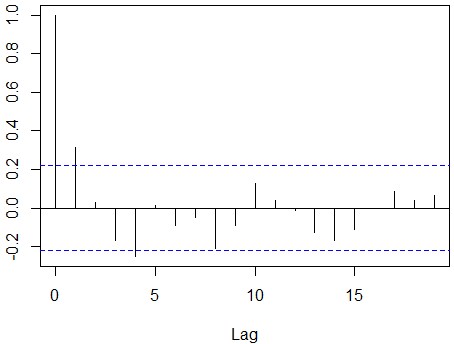} \\
\multicolumn{1}{l}{(C) Autocorr. function of $C_t/C_{t-1}$}  & \multicolumn{1}{l}{(D) Autocorr. function of $C_t/C_{t-1}$} \\\hline \hline
\end{tabular}
\end{table}
 For the 1930-2009 data set, the proxies for the risk-less asset and the market return are the one month
Treasury-bill  and the Center
for Research in Security Prices (CRSP) value-weighted index of all stocks on
the NYSE, AMEX, and NASDAQ. The computation of the growth consumption is based per capita
real personal consumption expenditures on nondurable goods from the
National Income and Product Accounts (NIPA). Quantities are deflated from the inflation.

Tables \ref{Tab:DescriptiveStat} and \ref{Tab:Rm_Rf} indicate that there is no significant autocorrelation for the excess returns, and only a mild clustering effect (Figures (E) and (F) in Table \ref{Tab:Rm_Rf} on p. \pageref{Tab:Rm_Rf}).  Thus, the i.i.d. assumption (Assumption \ref{Assp:ExistenceConsistency}(a)) appears to be a good approximation for the excess returns for both data set. For the growth consumption, the i.i.d. assumption may appear less appropriate. Table \ref{Tab:gc}  indicates  a mild autocorrelation for the growth consumption, and, more strikingly, a change of variance at the end of WWII. However, in the moment function, the growth consumption is multiplied by the   excess returns, whose variance is several orders of magnitude higher (Table \ref{Tab:DescriptiveStat} on p. \pageref{Tab:DescriptiveStat}), so that the change of variance is dampened.  %The similarity of   the LogESP for both data sets, which correspond to different time spans, also indicates that the  change in variance of the growth consumption does not invalidate our observation regarding the sharpness  of the LogESP w.r.t the LogET. Analysis on subsamples, namely 1890-1929 and 1950-2009, confirm the latter point.\footnote{We do not focus on these subsample analysis because they are probably to short, and they are prone to sample  selection bias\,: The great depression and the WWII are often regarded as key  events from an asset pricing point of view \citep[e.g.,][]{2006Bar}. }    

\end{small}
%%%%%%%%%%%%%%%%%%%%%%%%%%%%%%%%%%%%%%%%%%
%%%%%%%%%%%%%%%%%%%%%%%%%%%%%%%%%%%%%%%%%%
% \newpage
%
%
%
% \include{appendix_to_think_about}

\end{document}